\UseRawInputEncoding
\documentclass[10pt, reqno]{amsart}
\usepackage{amsfonts,latexsym,enumerate}
\usepackage{amsmath}
\usepackage{amscd}
\usepackage{float,amsmath,amssymb,mathrsfs,bm,multirow,graphics}
\usepackage[dvips]{graphicx}
\usepackage[percent]{overpic}
\usepackage{amsaddr}
\usepackage[numbers,sort&compress]{natbib}
\usepackage{todonotes}
\usetikzlibrary{shapes.misc}

\usepackage{cancel}
\usepackage{ulem}

\newcommand{\R}{{\Bbb R}}

\newcommand{\C}{{\Bbb C}}
\newcommand{\D}{{\Bbb D}}

\newcommand{\diag}{\text{\upshape diag\,}}
\newcommand{\re}{\text{\upshape Re\,}}
\newcommand{\im}{\text{\upshape Im\,}}
\newcommand{\Ai}{\text{\upshape Ai}}

\newcommand{\HM}{\text{\upshape HM}}
\newcommand{\sym}{\text{\upshape sym}}

\DeclareMathOperator{\res}{Res}

\DeclareMathOperator{\dist}{dist}

\DeclareMathOperator{\erf}{erf}

\newcommand{\II}{\text{\upshape II}}
\newcommand{\III}{\text{\upshape III}}
\newcommand{\IV}{\text{\upshape IV}}
\newcommand{\V}{\text{\upshape V}}

\usepackage{tikz}
\usepackage{pict2e}
\usetikzlibrary{arrows,calc,chains, positioning, shapes.geometric,shapes.symbols,decorations.markings,arrows.meta}
\usetikzlibrary{patterns,angles,quotes}
\usetikzlibrary{decorations.markings}
\tikzset{middlearrow/.style={
			decoration={markings,
				mark= at position 0.6 with {\arrow{#1}} ,
			},
			postaction={decorate}
		}
	}
\tikzset{->-/.style={decoration={
				markings,
				mark=at position #1 with {\arrow{latex}}},postaction={decorate}}}
	
\tikzset{-<-/.style={decoration={
				markings,
				mark=at position #1 with {\arrowreversed{latex}}},postaction={decorate}}}
				
				\tikzset{
	master/.style={
		execute at end picture={
			\coordinate (lower right) at (current bounding box.south east);
			\coordinate (upper left) at (current bounding box.north west);
		}
	},
	slave/.style={
		execute at end picture={
			\pgfresetboundingbox
			\path (upper left) rectangle (lower right);
		}
	}
}
\tikzset{cross/.style={cross out, draw, 
         minimum size=2*(#1-\pgflinewidth), 
         inner sep=0pt, outer sep=0pt}}

\def\Xint#1{\mathchoice
{\XXint\displaystyle\textstyle{#1}}%
{\XXint\textstyle\scriptstyle{#1}}%
{\XXint\scriptstyle\scriptscriptstyle{#1}}%
{\XXint\scriptscriptstyle\scriptscriptstyle{#1}}%
\!\int}
\def\XXint#1#2#3{{\setbox0=\hbox{$#1{#2#3}{\int}$}
\vcenter{\hbox{$#2#3$}}\kern-.5\wd0}}

\def\dashint{\;\Xint-}

\allowdisplaybreaks

\newtheorem{theorem}{Theorem}[section]
\newtheorem{proposition}[theorem]{Proposition}

\newtheorem{lemma}[theorem]{Lemma}
\newtheorem{definition}[theorem]{Definition}
\newtheorem{assumption}[theorem]{Assumption}
\newtheorem{remark}[theorem]{Remark}

\newtheorem{RHproblem}[theorem]{RH problem}
\newtheorem{figuretext}{Figure}

\numberwithin{equation}{section}

\usepackage[colorlinks=true]{hyperref}
\hypersetup{urlcolor=blue, citecolor=red, linkcolor=blue}

\title[On Boussinesq's equation for water waves]
{On Boussinesq's equation for water waves}

\author{C. Charlier$^{1}$ and J. Lenells$^{2}$}

\address{$^{1}$Centre for Mathematical Sciences, Lund University,
22100 Lund, Sweden. \\
$^{2}$Department of Mathematics, KTH Royal Institute of Technology, \\
10044 Stockholm, Sweden.}
\email{christophe.charlier@math.lu.se}
\email{jlenells@kth.se}

\begin{document}

\begin{abstract} 
A century and a half ago, J. Boussinesq derived an equation for the propagation of water waves in a channel. 
Despite the fundamental importance of this equation
for a number of physical phenomena, mathematical results on it remain scarce. One reason for this is that the equation is ill-posed. In this paper, we establish several results on the Boussinesq equation. First, by solving the direct and inverse problems for an associated third-order spectral problem, we develop an Inverse Scattering Transform (IST) approach to the initial value problem. Using this approach, we establish a number of existence, uniqueness, and blow-up results. 
For example, the IST approach allows us to identify physically meaningful global solutions and to construct, for each $T > 0$, solutions that blow up exactly at time $T$. 
Our approach also yields 
an expression for the solution of the initial value problem for the Boussinesq equation in terms of the solution of a Riemann--Hilbert problem. By analyzing this Riemann--Hilbert problem, we arrive at 
asymptotic formulas for the solution. We identify ten main asymptotic sectors in the $(x,t)$-plane; in each of these sectors, we compute an exact expression for the leading asymptotic term together with a precise error estimate. 
\end{abstract}

\maketitle



\setcounter{tocdepth}{1}
\tableofcontents

\section{Introduction}
In 1872, the mathematician and physicist Joseph Boussinesq published a paper where he derived an equation for water waves propagating in a rectangular channel \cite{B1872}. This equation---now known as the Boussinesq equation---is a fundamental equation in water wave theory modeling nonlinear dispersive long waves of small amplitude, see e.g. \cite{J1997}. It is given in nondimensional units by
\begin{align}\label{badboussinesq}
u_{tt} = u_{xx} + (u^2)_{xx} + u_{xxxx},
\end{align}
where $u(x,t)$ is a real-valued function and subscripts denote partial derivatives.
In addition to its relevance in fluid dynamics, equation (\ref{badboussinesq}) describes a range of other physical phenomena, including nonlinear lattice waves in the continuum limit \cite{T1975}, the propagation of ion-sound waves in a uniform isotropic plasma \cite{M1978}, and the dynamics of the anharmonic lattice in the Fermi--Pasta--Ulam (FPU) problem \cite{M1978}. Equation (\ref{badboussinesq}) is also known as the ``nonlinear string equation'' \cite{Z1974}.

In this paper, we establish several results on the Boussinesq equation (\ref{badboussinesq}). First, by solving the direct and inverse problems for an associated third-order spectral problem, we develop an Inverse Scattering Transform (IST) approach to the initial value problem. The IST approach yields an expression for the solution of the initial value problem for the Boussinesq equation in terms of the solution of a $3 \times 3$-matrix Riemann--Hilbert (RH) problem. By establishing a vanishing lemma for this RH problem, we obtain existence, uniqueness, and blow-up results for (\ref{badboussinesq}). Furthermore, by performing a steepest descent analysis of the RH problem, we arrive at asymptotic formulas for the solution. The asymptotic picture that emerges consists, roughly speaking, of two nonlinearly coupled copies of the corresponding picture for the (unidirectional) KdV equation, one copy for right-moving and one for left-moving waves. Of particular interest are the sectors which describe the interaction of right and left moving waves, which present qualitatively new phenomena. For simplicity, we only consider initial data in the Schwartz class $\mathcal{S}(\R)$, although it will be clear from the arguments that a finite degree of regularity and decay is sufficient.

\subsection{Background}
Substitution of $u(x,t) = e^{i(kx - \omega t)}$ into the linearized version of (\ref{badboussinesq}) yields the dispersion relation
$$\omega = \pm \sqrt{k^2(1 - k^2)}.$$
Since $\omega$ is nonreal for all $k \in \R$ with $|k| > 1$, it follows that (\ref{badboussinesq}) is linearly unstable, i.e., Fourier modes of large frequency grow (or decay) exponentially in time. Because of this feature, which makes the initial value problem ill-posed, equation (\ref{badboussinesq}) is sometimes referred to as the ``bad'' Boussinesq equation. The ``good'' Boussinesq equation is in this context the equation obtained by changing the sign of $u_{xxxx}$ in (\ref{badboussinesq}). This change of signs makes the equation linearly stable and amenable to various classical PDE techniques. Accordingly, there exists a multitude of works dedicated to the study of the ``good'' Boussinesq equation; see e.g. \cite{BFH2019, BS1988, F2009, F2011, HM2015, KT2010, TM1991, X2006} for local existence and well-posedness results, \cite{BS1988, L1993, TM1991, X2006} for global existence results, and \cite{LS1995, F2011, L1997} for scattering properties and asymptotic results.
On the other hand, the literature on equation (\ref{badboussinesq}) is limited, at least from a mathematical standpoint. Multi-soliton solutions of (\ref{badboussinesq}) were constructed by Hirota \cite{H1973} and further evidence for the integrability of (\ref{badboussinesq}) was put forward by Zakharov in \cite{Z1974} where a Lax pair was constructed.
The solitons were shown to be unstable under small perturbations in the linear approximation by Berryman \cite{B1976} (but the question of their nonlinear stability was left open). In \cite{KL1977}, the authors examined an initial-boundary value problem on a finite interval for (\ref{badboussinesq}) and showed that solutions may blow up in finite time; this initial-boundary value problem was studied further in \cite{LS1985, Y2002}. The solitons of (\ref{badboussinesq}) were analyzed using a $\bar{\partial}$-dressing method in \cite{BZ2002}. 
There are also various results on modified versions of (\ref{badboussinesq}). For example, the sixth-order Boussinesq equation obtained by adding an additional term with a sixth order derivative to (\ref{badboussinesq}) (which makes it linearly stable) has been extensively studied, see e.g. \cite{EF2012}. Another version of (\ref{badboussinesq}) that has been investigated (see \cite{M1981, DTT1982}) is the equation obtained by ignoring the $u_{xx}$-term in (\ref{badboussinesq}), i.e.,
 \begin{align}\label{badboussinesqnouxx}
u_{tt} = (u^2)_{xx} + u_{xxxx}.
\end{align}

Some steps towards an IST formalism for the Boussinesq equation were outlined in \cite{Z1974}. These steps were made complete in \cite{DTT1982} in the case of the simplified equation (\ref{badboussinesqnouxx}). 
A key contribution of \cite{DTT1982} was the establishment of a so-called vanishing lemma for the RH problem associated to (\ref{badboussinesqnouxx}). 
Ever since the publication of \cite{DTT1982} in 1982, there has been an expectation in the integrable systems community that it should be possible to develop similar results for the Boussinesq equation (\ref{badboussinesq}). After all, the equation (\ref{badboussinesqnouxx}) considered in \cite{DTT1982} differs from (\ref{badboussinesq}) only in that the $u_{xx}$-term has been removed. 
The relevance of an IST approach for (\ref{badboussinesq}) became increasingly clear in 1993 when Deift and Zhou introduced their nonlinear steepest descent method for RH problems \cite{DZ1993}, the expectation being that Deift--Zhou ideas in combination with IST techniques would make it possible to derive detailed asymptotic formulas for the Boussinesq equation.   
However, it turns out that the removal of the $u_{xx}$-term from (\ref{badboussinesq}) thoroughly changes the quantitative and qualitative structure of solutions, and has far-reaching consequences also for the proofs; see further discussion below. 
The problem of developing an IST approach and of deriving asymptotic formulas for the solution of (\ref{badboussinesq}) has therefore remained an outstanding open problem in the field, see e.g. \cite{D2008} where P. Deift notes that ``The long-time behavior of the solutions of the Boussinesq equation with general initial data is a very interesting problem with many challenges.'' 
The purpose of this paper is to address this problem.

\subsection{Importance and implications of the $u_{xx}$-term}
The $u_{xx}$-term in (\ref{badboussinesq}) can be removed by the simple shift $u \to u - 1/2$, and the resulting equation, equation (\ref{badboussinesqnouxx}), possesses a similar, but much simpler, algebraic structure. So why is it important to include the $u_{xx}$-term in (\ref{badboussinesq})? To answer this question, we note that in physical units (\ref{badboussinesq}) takes the form (see \cite[Eq. (26)]{B1872})
\begin{align}\label{badboussinesqphysical}
\eta_{\tau \tau} - g h \eta_{\xi\xi} - g h \frac{\partial^2}{\partial\xi^2} \bigg(\frac{3}{2 h} \eta^2 + \frac{h^2}{3} \eta_{\xi\xi}\bigg) = 0,
\end{align}
where $h$ is the mean water depth, $g$ is the gravitational acceleration, and $\eta(\xi,\tau)$ is the surface elevation measured with respect to the mean water level at the horizontal position  $\xi$ at time $\tau$. The dimensionless equation (\ref{badboussinesq}) is obtained from (\ref{badboussinesqphysical}) by letting
$$u(x,t) = \frac{3}{2h}\eta(\xi, \tau), \quad x = \xi \frac{\sqrt{3}}{h}, \quad
t = \tau \sqrt{\frac{3g}{h}}.$$ 
In order to be relevant for water waves, the deviation $\eta$ from the mean water level, and hence also $u$, must tend to $0$ at spatial infinity, at least in average. The shift $u \to u - 1/2$ does not preserve this condition and is therefore not allowed if the connection to water waves is to be maintained. In fact, numerical simulations, as well as analytic calculations, show that the solutions (\ref{badboussinesq}) and (\ref{badboussinesqnouxx}) behave very differently both quantitatively and qualitatively. 

Let us also comment on why the $u_{xx}$-term in (\ref{badboussinesq}) makes such a big difference for the analysis of (\ref{badboussinesq}). First note that equation (\ref{badboussinesq}) can be rewritten as the system 
\begin{align}\label{boussinesqsystem}
& \begin{cases}
 v_{t} = u_{x} + (u^2)_{x} + u_{xxx},
 \\
 u_t = v_x.
\end{cases}
\end{align}
The Lax pair found in \cite{Z1974} does not apply to (\ref{badboussinesq}) directly, but to the system (\ref{boussinesqsystem}). The isospectral problem (i.e., the $x$-part of the Lax pair) associated to (\ref{boussinesqsystem}) is given by \cite{Z1974} 
\begin{align}\label{isospectral}
(D^3 + p_1D + p_0)\varphi = z^3 \varphi,
\end{align}
where $z \in \C$ is a spectral parameter, $D = -i\partial_x$, and 
$$p_1 := - \frac{u}{2} - \frac{1}{4}, \qquad
p_0 := i\Big(\frac{u_{x}}{4}+\frac{iv}{4\sqrt{3}}\Big).$$
We note that (\ref{isospectral}) is a higher-order spectral problem in the sense that the differential operator is not (first or) second order. The treatment of the direct and inverse problems for higher-order operators requires a different and more sophisticated theory compared to the second-order case. Building on the advances made in \cite{DTT1982}, such a theory was developed in \cite{BDT1988} in the case when all the coefficients $p_j$ of the higher-order operator lie in the Schwartz class $\mathcal{S}(\R)$. This is the case for the simplified equation (\ref{badboussinesqnouxx}) whose spectral problem has the same form as (\ref{isospectral}) except that the term $-1/4$ in the expression for $p_1$ is absent. In this case, the jump contour for the associated $3 \times 3$ RH problem consists of six rays going from $0$ to $\infty$, and on each ray the jump matrix is effectively $2 \times 2$ in the sense that it acts nontrivially only on a two-dimensional subspace.  
In contrast, the problem (\ref{isospectral}) does not fit into the framework of \cite{DTT1982} because the coefficient $p_1$ in (\ref{isospectral}) tends to $-1/4$ as $x \to \pm \infty$ (assuming that $u$ has decay at infinity). This has the effect of introducing an additional jump into the RH problem along the unit circle. This new jump is truly $3 \times 3$ in the sense that it acts nontrivially in all three dimensions. Handling this jump requires new techniques; in fact, most of the work in the present paper is related to the handling of this jump and its consequences. For example, it is this new jump that is responsible for the asymptotic behavior of solutions of (\ref{badboussinesq}) in the sense that all main asymptotic contributions originate from saddle points on (or very close to) the unit circle. Another example is that our steepest descent analysis of (\ref{badboussinesq}) leads to a new kind of local parametrix which is nontrivial in all three dimensions, see Section \ref{transitionsec}. The jump on the unit circle also has far-reaching implications for the direct and inverse problems. Our development of a direct and inverse scattering theory for (\ref{isospectral}) can be viewed as a first step towards the development of a general scattering theory for higher-order operators with non-decaying coefficients.

\subsection{Global solutions}\label{globalsubsec}
The linear instability of the Boussinesq equation (\ref{badboussinesq}) leads to the exponential amplification of high frequency Fourier modes. In particular, it implies that solutions will grow exponentially unless the unstable modes vanish identically. This exponential growth may at first sight seem to contradict the physical relevance of the equation as a model for water waves. However, the presence of unstable modes is consistent with Boussinesq's derivation which assumes that the waves are long, in the sense that the wavelength is long compared to the depth of the water. Physically relevant solutions correspond to initial data without high-frequency modes, and the equation only remains valid as a model as long as the unstable high-frequency modes are heavily suppressed. However, since the equation is nonlinear, high-frequency modes will be generated by the time evolution, even if none are present initially. This makes it difficult to analyze solutions of (\ref{badboussinesq}) by means of traditional techniques, even at a numerical level \cite{DH1999}. 

The IST approach is particularly well suited to address this problem. Indeed, the IST formalism can be viewed as a nonlinear Fourier transform which linearizes the flow of the equation. Thus it is possible within the IST framework to identify the unstable modes in a nonlinearly precise way: the condition that these modes vanish amounts to the spectral functions being zero for certain ranges of the spectral parameter. By projecting out these ``unstable nonlinear Fourier modes'', it is possible to identify physically meaningful solutions that do not blow up. These solutions exist globally and the asymptotic formulas we obtain for their long-time behavior can be expected to capture the physical characteristics of water waves according to Boussinesq's model.

\subsection{Description of main results}
For conciseness, we restrict ourselves throughout this paper to generic solitonless solutions, see Assumptions \ref{solitonlessassumption} and \ref{originassumption}. 
Our main results are presented in the form of seven theorems as follows:\footnote{Our approach can be used to obtain analogous theorems also in the case when solitons are present; since this case is notationally more cumbersome, details are deferred to the companion papers \cite{CLscatteringsolitons, CLsectorII, CLsectorIV}. The inclusion of solitons leads to a proof of the soliton resolution conjecture for (\ref{badboussinesq}) in all of the $(x,t)$-plane except in a number of small transition zones \cite{CLsectorII}.}

\begin{enumerate}[$-$]
\item Theorem \ref{directth} addresses the direct problem for the spectral problem (\ref{isospectral}). It constructs an appropriate set of scattering data associated to the given potentials $u$ and $v$ and establishes properties that characterize these data. 

\item Theorem \ref{inverseth} solves the inverse problem of reconstructing $u$ and $v$ from the scattering data. It expresses $u$ and $v$ in terms of a row vector RH problem, whose solution is shown to exist with the help of an associated vanishing lemma. By including the explicit time dependence of the scattering data in the definition of the jump matrix, the theorem constructs solutions $\{u(x,t), v(x,t)\}$ of the system (\ref{boussinesqsystem}) from some given scattering data. It is shown that the solution is unique and exists for $t \in [0,T)$ where $T$ is determined by the decay rate of the reflection coefficients. 

\item Theorem \ref{IVPth} provides the solution of the initial value problem for (\ref{badboussinesq}).
By combining the direct and inverse scattering results, it expresses the solution $u(x,t)$ of the initial value problem in terms of the given initial data $u_0(x) = u(x,0)$. It shows that the solution is unique and exists for $t \in [0,T)$ where $T$ is determined by the decay rate of the associated reflection coefficients. 

\item Theorems \ref{blowupth} and \ref{existenceblowupth} are blow-up results; the former shows that the solution obtained via inverse scattering actually ceases to exist at the time $T$ determined by the reflection coefficients if $T < \infty$, while the latter shows that there exist solutions of (\ref{badboussinesq}) that blow up at each $T \in (0, \infty)$. 

\item Theorem \ref{globalth} shows that the solution of (\ref{badboussinesq}) exists globally in the physically relevant case where the ``unstable nonlinear Fourier modes'' are assumed to vanish.

\item Theorem \ref{asymptoticsth} provides long-time asymptotic formulas for the global solutions.

\end{enumerate}


\begin{figure}
\bigskip\begin{center}
\begin{overpic}[width=.7\textwidth]{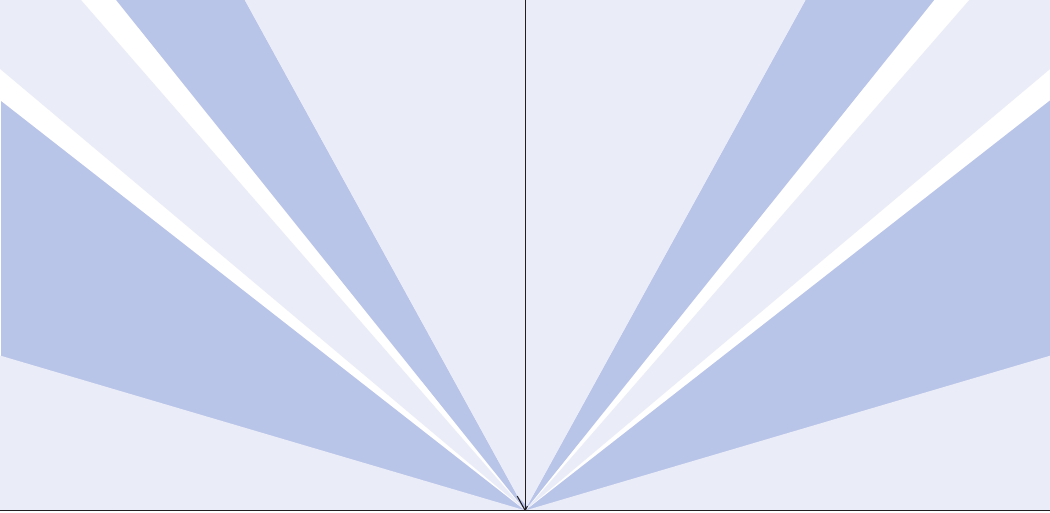}
      \put(102,-.3){\small $x$}
      \put(49.7,50.8){\small $t$}
      \put(90,4){\small I}
      \put(90,21){\small II}
      \put(86,36){\small III}
      \put(77,41){\small IV}
      \put(60.5,41){\small V}
      \put(37,41){\small VI}
      \put(20,41){\small VII}
      \put(10,36){\small VIII}
      \put(8,21){\small IX}
      \put(8,4){\small X}
    \end{overpic}
     \begin{figuretext}\label{sectors.pdf}
       The asymptotic sectors \upshape{I}--{X} in the $xt$-plane.
     \end{figuretext}
     \end{center}
\end{figure}

\subsection{Asymptotics and comparison to KdV}
Apart from the Boussinesq equation, another famous equation modeling nonlinear dispersive long waves of small amplitude is the Korteweg--de Vries (KdV) equation (which, incidentally, was also first introduced by Boussinesq \cite[p. 360]{B1877}). Whereas the KdV equation only describes waves traveling in one direction, the Boussinesq equation is bidirectional, i.e., it models waves propagating in both the positive and negative $x$-directions, see e.g. \cite[Section 3.2.5]{J1997}. 
It is therefore not surprising that the asymptotic picture that we find for the Boussinesq equation in Theorem \ref{asymptoticsth} is, roughly speaking, a superposition of two nonlinearly coupled copies of the corresponding picture for the KdV equation, one copy for right-moving and one for left-moving waves. 
More precisely, we identify ten main asymptotic sectors in the $(x,t)$-plane for the solution of (\ref{badboussinesq}), see Figure \ref{sectors.pdf}. Since (\ref{badboussinesq}) is invariant under the transformation $x \mapsto -x$, it is enough to consider the five sectors I--V corresponding to $x \geq 0$. Sector I (characterized by $x/t \geq M \gg 1$, $M$ constant) describes the rapid decay of the solution near the $x$-axis. In Sector II (characterized by $1 < x/t \leq M$), the solution approaches a modulated sinusoidal traveling wave decaying as $1/\sqrt{t}$.
Sector III (characterized by $|x -  t|\leq M t^{1/3}$) is a transition sector describing the wave front; the leading behavior in this sector is given in terms of the Hastings--McLeod solution of the Painlev\'e II equation. In Sectors IV and V (characterized by $1/\sqrt{3}< x/t < 1$ and $0< x/t < 1/\sqrt{3}$), the solution asymptotes to a sum of two modulated sine waves of order $O(t^{-1/2})$, one moving to the left and one to the right. 

To compare the above with the corresponding picture for the KdV equation, we note that KdV also features a transition sector where the leading term is expressed in terms of the Hastings--McLeod solution of Painlev\'e II, see e.g. \cite{DVZ1994}. However, while the steepest descent analysis for KdV involves two saddle points and the Painlev\'e II sector is a result of these two points merging, the analysis for (\ref{badboussinesq}) involves twelve saddle points and Sector III is a result of these twelve points merging in three different groups of four. This means that the analysis near the wave front at $x/t \approx 1$ is much more involved. More precisely, it involves a new error function local parametrix nonlinearly superimposed on a more classical Painlev\'e II local parametrix. (As far as we know, this is the first time such a phenomenon has been seen in a steepest descent calculation for a RH problem.) In particular, there are subleading asymptotic terms of a new type for $x/t \approx 1$. This also means that there are multiple narrow transition regions on either side of Sector III. Just like for KdV, we expect that there exists some kind of collisionless shock region \cite{DVZ1994} between Sectors III and IV. 
But in addition to this collisionless shock region, there will be further narrow transition regions corresponding to different phases of the merging of the various saddle points. 

We also note that Sectors IV and V involve two modulated sine waves moving in opposite directions. This feature is a result of the Boussinesq equation being bidirectional and has no analog for the KdV equation.

\subsection{Organization of the paper}
All the main theorems are stated in Section \ref{mainsec}.
Sections \ref{directsec} and \ref{inversesec} treat the direct and inverse problems, respectively; in particular, these sections include proofs of Theorems \ref{directth} and \ref{inverseth}. 
In Section \ref{IVPsec}, we consider the solution of the initial value problem for (\ref{badboussinesq}) and prove Theorem \ref{IVPth}.
Blow-up is considered in Section \ref{blowupsec}. Sections \ref{overviewsec}--\ref{othersectorssec} are devoted to the proof of the asymptotic result Theorem \ref{asymptoticsth}. After providing an overview of the proof in Section \ref{overviewsec}, we give a detailed derivation of the asymptotics in Sector III in Section \ref{transitionsec}. The remaining sectors are discussed in Section \ref{othersectorssec}.

\subsection{Notation}\label{notationsubsec}
The following notation will be used throughout the article.

\begin{enumerate}[$-$]
\item $C>0$ and $c>0$ will denote generic constants that may change within a computation.

\item $[A]_1$, $[A]_2$, and $[A]_3$ will denote the first, second, and third columns of a $3 \times 3$ matrix $A$.

\item If $A$ is an $n \times m$ matrix, we define $|A| \ge 0$ by $|A|^2=\Sigma_{i,j}|A_{ij}|^2$. Note that $|A + B| \leq |A| + |B|$ and $|AB| \leq |A| |B|$.
For a piecewise smooth contour $\gamma \subset \C$, we write $A \in L^p(\gamma)$ and define $\|A\|_{L^p(\gamma)} := \| |A|\|_{L^p(\gamma)}$ if $|A|$ belongs to $L^p(\gamma)$, $1 \le p \le \infty$.

\item $\D = \{k \in \C \, | \, |k| < 1\}$ will denote the open unit disk and $\partial \D = \{k \in \C \, | \, |k| = 1\}$ will denote the unit circle. 

\item $f^*(k):= \overline{f(\bar{k})}$ will denote the Schwarz conjugate of a function $f(k)$.

\item $D_\epsilon(k)$ will denote the open disk of radius $\epsilon$ centered at a point $k \in \C$.

\item $\mathcal{S}(\R)$ will denote the Schwartz space of all smooth functions $f$ on $\R$ such that $f$ and all its derivatives have rapid decay as $x \to \pm \infty$. 

\item $\kappa_{j} = e^{\frac{\pi i(j-1)}{3}}$, $j=1,\ldots,6$, will denote the sixth roots of unity, see Figure \ref{fig: Dn}.

\item We let $\mathcal{Q} := \{\kappa_{j}\}_{j=1}^{6}$ and $\hat{\mathcal{Q}} := \mathcal{Q} \cup \{0\}$.

\item $D_n$, $n = 1, \dots, 6$, will denote the open subsets of the complex plane shown in Figure \ref{fig: Dn}.

\item $\Gamma = \cup_{j=1}^9 \Gamma_j$ will denote the contour shown and oriented as in Figure \ref{fig: Dn}.

\item $\hat{\Gamma}_{j} = \Gamma_{j} \cup \partial \D$ will denote the union of $\Gamma_j$ and the unit circle.

\end{enumerate}

\section{Main results}\label{mainsec}

\subsection{The direct problem}
Our first theorem concerns the direct problem, that is, the construction of an appropriate set of scattering data from the given initial data. It turns out that the scattering data for the system (\ref{boussinesqsystem}) involves two spectral functions $r_1(k)$  and $r_2(k)$. These functions can viewed as the ``reflection coefficients'' for (\ref{boussinesqsystem}) corresponding to the initial data 
$$u_0(x) := u(x,0), \qquad v_0(x) := v(x,0).$$ 
Before stating the theorem, we need to define the functions $r_1$ and $r_2$. We refer to Section \ref{directsec} for the origin of the following definitions and for proofs that all quantities are well-defined.

Let $u_0, v_0 \in \mathcal{S}(\R)$ be real-valued.
Let $\omega := e^{\frac{2\pi i}{3}}$ and define $\{l_j(k), z_j(k)\}_{j=1}^3$ by
\begin{align}\label{lmexpressions intro}
& l_{j}(k) = i \frac{\omega^{j}k + (\omega^{j}k)^{-1}}{2\sqrt{3}}, \qquad z_{j}(k) = i \frac{(\omega^{j}k)^{2} + (\omega^{j}k)^{-2}}{4\sqrt{3}}, \qquad k \in \C\setminus \{0\}.
\end{align}
Let the matrix-valued function $\mathsf{U}(x,k)$ be given by
\begin{align}\label{mathsfUdef intro}
\mathsf{U}(x,k) = P(k)^{-1} \begin{pmatrix}
0 & 0 & 0 \\
0 & 0 & 0 \\
-\frac{u_{0x}}{4}-\frac{iv_{0}}{4\sqrt{3}} & -\frac{u_{0}}{2} & 0
\end{pmatrix} P(k),
\end{align} 
where
\begin{align}\label{Pdef intro}
P(k) = \begin{pmatrix}
1 & 1 & 1  \\
l_{1}(k) & l_{2}(k) & l_{3}(k) \\
l_{1}(k)^{2} & l_{2}(k)^{2} & l_{3}(k)^{2}
\end{pmatrix}.
\end{align}
Let $X(x,k)$ and $X^A(x,k)$ be the unique $3 \times 3$-matrix valued solutions of the Volterra integral equations
\begin{subequations}\label{XXAdef intro}
\begin{align}  
 & X(x,k) = I - \int_x^{\infty} e^{(x-x')\widehat{\mathcal{L}(k)}} (\mathsf{U}X)(x',k) dx',
	\\\label{XXAdefb intro}
 & X^A(x,k) = I + \int_x^{\infty} e^{-(x-x')\widehat{\mathcal{L}(k)}} (\mathsf{U}^T X^A)(x',k) dx',	
\end{align}
\end{subequations}
where $\mathsf{U}^T$ denotes the transpose of $\mathsf{U}$, $\mathcal{L} = \diag(l_1 , l_2 , l_3)$, and $\hat{\mathcal{L}}$ denotes the operator which acts on a $3 \times 3$ matrix $A$ by $\hat{\mathcal{L}}A = [\mathcal{L}, A]$ (i.e., $e^{\hat{\mathcal{L}}}A = e^\mathcal{L} A e^{-\mathcal{L}}$). Define $s(k)$ and $s^A(k)$ by 
\begin{align}\label{sdef intro}
& s(k) = I - \int_\R e^{-x\widehat{\mathcal{L}(k)}}(\mathsf{U}X)(x,k)dx,
 	\\ \label{sAdef intro}
& s^A(k) = I + \int_\R e^{x\widehat{\mathcal{L}(k)}}(\mathsf{U}^T X^A)(x,k)dx.
\end{align}
The two spectral functions $\{r_j(k)\}_1^2$ are defined by
\begin{align}\label{r1r2def}
\begin{cases}
r_1(k) = \frac{(s(k))_{12}}{(s(k))_{11}}, & k \in \hat{\Gamma}_{1}\setminus \hat{\mathcal{Q}},
	\\ 
r_2(k) = \frac{(s^A(k))_{12}}{(s^A(k))_{11}}, \quad & k \in \hat{\Gamma}_{4}\setminus \hat{\mathcal{Q}},
\end{cases}
\end{align}	
where the contours $\hat{\Gamma}_j$ and the set $\hat{\mathcal{Q}}$ are as defined in Section \ref{notationsubsec}.

\begin{figure}
\begin{center}
\begin{tikzpicture}[scale=0.7]
\node at (0,0) {};
\draw[black,line width=0.45 mm,->-=0.4,->-=0.85] (0,0)--(30:4);
\draw[black,line width=0.45 mm,->-=0.4,->-=0.85] (0,0)--(90:4);
\draw[black,line width=0.45 mm,->-=0.4,->-=0.85] (0,0)--(150:4);
\draw[black,line width=0.45 mm,->-=0.4,->-=0.85] (0,0)--(-30:4);
\draw[black,line width=0.45 mm,->-=0.4,->-=0.85] (0,0)--(-90:4);
\draw[black,line width=0.45 mm,->-=0.4,->-=0.85] (0,0)--(-150:4);

\draw[black,line width=0.45 mm] ([shift=(-180:2.5cm)]0,0) arc (-180:180:2.5cm);
\draw[black,arrows={-Triangle[length=0.2cm,width=0.18cm]}]
($(3:2.5)$) --  ++(90:0.001);
\draw[black,arrows={-Triangle[length=0.2cm,width=0.18cm]}]
($(57:2.5)$) --  ++(-30:0.001);
\draw[black,arrows={-Triangle[length=0.2cm,width=0.18cm]}]
($(123:2.5)$) --  ++(210:0.001);
\draw[black,arrows={-Triangle[length=0.2cm,width=0.18cm]}]
($(177:2.5)$) --  ++(90:0.001);
\draw[black,arrows={-Triangle[length=0.2cm,width=0.18cm]}]
($(243:2.5)$) --  ++(330:0.001);
\draw[black,arrows={-Triangle[length=0.2cm,width=0.18cm]}]
($(297:2.5)$) --  ++(210:0.001);

\draw[black,line width=0.15 mm] ([shift=(-30:0.55cm)]0,0) arc (-30:30:0.55cm);

\node at (0.8,0) {$\tiny \frac{\pi}{3}$};

\node at (-1:2.9) {\footnotesize $\Gamma_8$};
\node at (60:2.9) {\footnotesize $\Gamma_9$};
\node at (120:2.9) {\footnotesize $\Gamma_7$};
\node at (181:2.9) {\footnotesize $\Gamma_8$};
\node at (240:2.83) {\footnotesize $\Gamma_9$};
\node at (300:2.83) {\footnotesize $\Gamma_7$};

\node at (105:1.45) {\footnotesize $\Gamma_1$};
\node at (138:1.45) {\footnotesize $\Gamma_2$};
\node at (223:1.45) {\footnotesize $\Gamma_3$};
\node at (-104:1.45) {\footnotesize $\Gamma_4$};
\node at (-42:1.45) {\footnotesize $\Gamma_5$};
\node at (43:1.45) {\footnotesize $\Gamma_6$};

\node at (97:3.3) {\footnotesize $\Gamma_4$};
\node at (144:3.3) {\footnotesize $\Gamma_5$};
\node at (217:3.3) {\footnotesize $\Gamma_6$};
\node at (-96:3.3) {\footnotesize $\Gamma_1$};
\node at (-35:3.3) {\footnotesize $\Gamma_2$};
\node at (36:3.3) {\footnotesize $\Gamma_3$};
\end{tikzpicture}
\hspace{1.7cm}
\begin{tikzpicture}[scale=0.7]
\node at (0,0) {};
\draw[black,line width=0.45 mm] (0,0)--(30:4);
\draw[black,line width=0.45 mm] (0,0)--(90:4);
\draw[black,line width=0.45 mm] (0,0)--(150:4);
\draw[black,line width=0.45 mm] (0,0)--(-30:4);
\draw[black,line width=0.45 mm] (0,0)--(-90:4);
\draw[black,line width=0.45 mm] (0,0)--(-150:4);

\draw[black,line width=0.45 mm] ([shift=(-180:2.5cm)]0,0) arc (-180:180:2.5cm);
\draw[black,line width=0.15 mm] ([shift=(-30:0.55cm)]0,0) arc (-30:30:0.55cm);

\node at (120:1.6) {\footnotesize{$D_{1}$}};
\node at (-60:3.7) {\footnotesize{$D_{1}$}};

\node at (180:1.6) {\footnotesize{$D_{2}$}};
\node at (0:3.7) {\footnotesize{$D_{2}$}};

\node at (240:1.6) {\footnotesize{$D_{3}$}};
\node at (60:3.7) {\footnotesize{$D_{3}$}};

\node at (-60:1.6) {\footnotesize{$D_{4}$}};
\node at (120:3.7) {\footnotesize{$D_{4}$}};

\node at (0:1.6) {\footnotesize{$D_{5}$}};
\node at (180:3.7) {\footnotesize{$D_{5}$}};

\node at (60:1.6) {\footnotesize{$D_{6}$}};
\node at (-120:3.7) {\footnotesize{$D_{6}$}};

\node at (0.8,0) {$\tiny \frac{\pi}{3}$};

\draw[fill] (0:2.5) circle (0.1);
\draw[fill] (60:2.5) circle (0.1);
\draw[fill] (120:2.5) circle (0.1);
\draw[fill] (180:2.5) circle (0.1);
\draw[fill] (240:2.5) circle (0.1);
\draw[fill] (300:2.5) circle (0.1);

\node at (0:2.9) {\footnotesize{$\kappa_1$}};
\node at (60:2.85) {\footnotesize{$\kappa_2$}};
\node at (120:2.85) {\footnotesize{$\kappa_3$}};
\node at (180:2.9) {\footnotesize{$\kappa_4$}};
\node at (240:2.85) {\footnotesize{$\kappa_5$}};
\node at (300:2.85) {\footnotesize{$\kappa_6$}};

\draw[dashed] (-6.3,-3.8)--(-6.3,3.8);

\end{tikzpicture}
\end{center}
\begin{figuretext}\label{fig: Dn}
The contour $\Gamma = \cup_{j=1}^9 \Gamma_j$ in the complex $k$-plane (left) and the open sets $D_{n}$, $n=1,\ldots,6$, together with the sixth roots of unity $\kappa_j$, $j = 1, \dots, 6$ (right).
\end{figuretext}
\end{figure}

\subsubsection{Assumption of no solitons}
We will show in Section \ref{directsec} that the entries $(s(k))_{11}$ and $(s(k))_{12}$ of $s(k)$ that appear in (\ref{r1r2def}) are smooth functions of $k \in \hat{\Gamma}_{1}\setminus \hat{\mathcal{Q}}$, and that $(s^A(k))_{11}$ and $(s^A(k))_{12}$ are smooth functions of $k \in \hat{\Gamma}_{4}\setminus \hat{\mathcal{Q}}$. Thus $r_1(k)$ and $r_2(k)$ are smooth on their respective domains, except possibly at points where $s_{11}$ and $s^A_{11}$ vanish. Propositions \ref{sprop} and \ref{sAprop} imply that $s_{11}$ and $s^A_{11}$ have smooth extensions to $\omega^{2}\hat{\mathcal{S}}\setminus \hat{\mathcal{Q}}$ and $-\omega^{2}\hat{\mathcal{S}}\setminus \hat{\mathcal{Q}}$, respectively, where $\hat{\mathcal{S}} := \partial \D \cup \bar{D}_{3}\cup\bar{D}_{4}$. 
The possible zeros of $s_{11}$ and $s^A_{11}$ are related to the presence of solitons. 
We will make the following assumption throughout this paper (see \cite{CLscatteringsolitons, CLsectorII, CLsectorIV} for the changes necessary when solitons are present).

\begin{assumption}[Absence of solitons]\label{solitonlessassumption}\upshape
Assume that $(s(k))_{11}$ is nonzero for $k \in (\bar{D}_2 \cup \partial \D) \setminus \hat{\mathcal{Q}}$.
\end{assumption}

The symmetries $s_{11}(k) = s_{11}(\omega/k)$ and $s_{11}^A(k) = \overline{s_{11}(\bar{k}^{-1})}$, which will be established in Propositions \ref{sprop} and \ref{sAprop}, show that if Assumption \ref{solitonlessassumption} holds, then $s_{11}$ is nonzero on $\omega^{2}\hat{\mathcal{S}} \setminus \hat{\mathcal{Q}}$ and $s_{11}^A$ is nonzero on $-\omega^{2}\hat{\mathcal{S}} \setminus \hat{\mathcal{Q}}$.

\subsubsection{Assumption of generic behavior at $k = \pm 1$}
It turns out that each of the four functions $s_{11}$, $s_{12}$, $s^A_{11}$, and $s^A_{12}$ in (\ref{r1r2def}) has at most a simple pole at $k = 1$ and at $k = -1$ (see Propositions \ref{sprop} and \ref{sAprop}). Moreover, $s_{12}$ has simple poles at $1$ and $-1$ if and only if $s_{11}$ has simple poles at $1$ and $-1$, and $s^A_{12}$ has simple poles at $1$ and $-1$ if and only if $s^A_{11}$ has simple poles at $1$ and $-1$. 
For simplicity, we will restrict ourselves to the case in which all of these four functions have simple poles at both $k=1$ and $k=-1$; this is the case for generic initial data. For similar reasons, we will assume that some other entries of $s$ and $s^{A}$ also have a generic behavior at $k= \pm 1$. More precisely, our results will be stated under the following assumption.

\begin{assumption}[Generic behavior at $k = \pm 1$]\label{originassumption}\upshape
Assume for $k_{\star} =1$ and $k_{\star}=-1$ that
\begin{align*}
& \lim_{k \to k_{\star}} (k-k_{\star}) s(k)_{11} \neq 0, & & \lim_{k \to k_{\star}} (k-k_{\star}) s(k)_{13} \neq 0, & & \lim_{k \to k_{\star}} s(k)_{31} \neq 0, & & \lim_{k \to k_{\star}} s(k)_{33} \neq 0, \\
& \lim_{k \to k_{\star}} (k-k_{\star}) s^{A}(k)_{11} \neq 0, & & \lim_{k \to k_{\star}} (k-k_{\star}) s^{A}(k)_{31} \neq 0, & & \lim_{k \to k_{\star}} s^{A}(k)_{13} \neq 0, & & \lim_{k \to k_{\star}} s^{A}(k)_{33} \neq 0.
\end{align*}
\end{assumption}

\subsubsection{Statement of the theorem}
We can now state our theorem on the direct problem.

\begin{figure}[t]
\begin{center}
\begin{tikzpicture}[master]
\node at (0,0) {\includegraphics[scale=0.20]{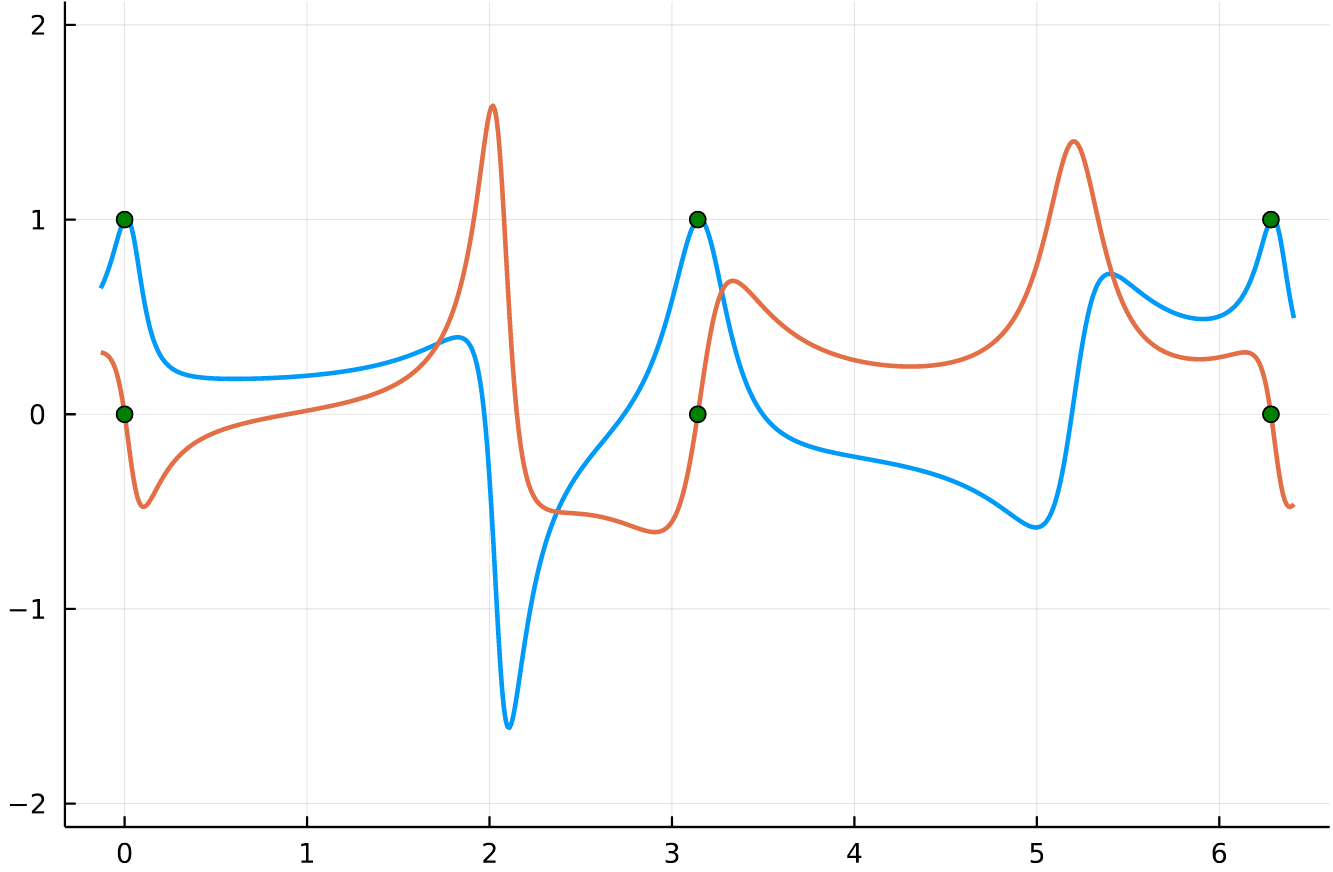}};
\node at (-2.3,2.53) {\tiny $\re r_{1}(e^{i\theta}), \im r_{1}(e^{i\theta})$};
\node at (3.65,-2.1) {\tiny $\theta$};
\end{tikzpicture}
\begin{tikzpicture}[slave]
\node at (0,0) {\includegraphics[scale=0.20]{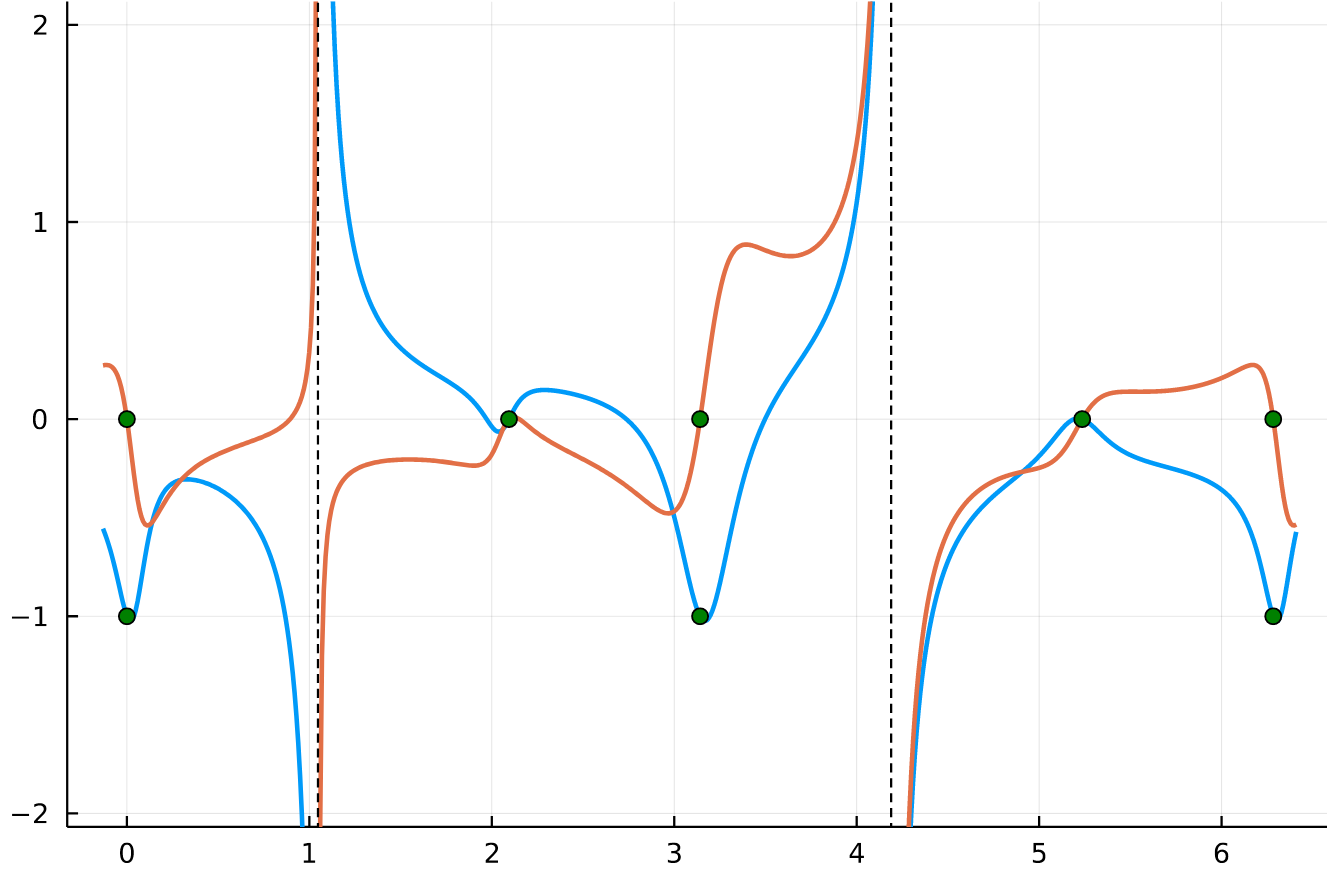}};
\node at (-2.3,2.53) {\tiny $\re r_{2}(e^{i\theta}), \im r_{2}(e^{i\theta})$};
\node at (3.65,-2.1) {\tiny $\theta$};
\end{tikzpicture}
\end{center}
\begin{figuretext}
\label{fig:r1 and r2}
The functions $\theta \mapsto \re r_{j}(e^{i\theta})$ (blue) and $\theta \mapsto \im r_{j}(e^{i\theta})$ (orange) for $j = 1,2$ in the case of the initial data (\ref{numericscompactinitialdata}). The dots and dashed vertical lines illustrate part $(ii)$ of Theorem \ref{directth}. 
\end{figuretext}
\end{figure}

\begin{theorem}[Properties of $r_1(k)$ and $r_2(k)$]\label{directth}
Suppose $u_0,v_0 \in \mathcal{S}(\R)$ are two real-valued functions such that Assumptions \ref{solitonlessassumption} and \ref{originassumption} hold.
Then the spectral functions $r_1:\hat{\Gamma}_{1}\setminus \hat{\mathcal{Q}} \to \C$ and $r_2: \hat{\Gamma}_{4}\setminus \hat{\mathcal{Q}} \to \C$ are well-defined by (\ref{r1r2def}) and have the following properties:
\begin{enumerate}[$(i)$]
 \item \label{Theorem2.3itemi}
 $r_1$ and $r_2$ admit extensions such that $r_1 \in C^\infty(\hat{\Gamma}_{1})$\footnote{At $k_\star = \pm i$, the notation $r_1 \in C^\infty(\hat{\Gamma}_{1})$ indicates that $r_1$ has derivatives to all orders along each subcontour of $\hat{\Gamma}_1$ emanating from $k_\star$ and that these derivatives have consistent finite limits as $k \to k_\star$.}
and $r_2 \in C^\infty(\hat{\Gamma}_{4}\setminus \{\omega^{2}, -\omega^{2}\})$.

\item \label{Theorem2.3itemii}
$r_{1}(k)$ is bounded on the whole unit circle, and $r_{1}(\kappa_{j})\neq 0$ for $j=1,\ldots,6$. $r_{2}(k)$ has simple poles at $k=\omega^2$ and $k = -\omega^2$, and simple zeros at $k=\omega$ and $k=-\omega$. Furthermore,
\begin{align}\label{r1r2at0}
r_{1}(1) = r_{1}(-1) = 1, \qquad r_{2}(1) = r_{2}(-1) = -1.
\end{align}

\item \label{Theorem2.3itemiii}
$r_1(k)$ and $r_2(k)$ are rapidly decreasing as $|k| \to \infty$, i.e., for each integer $N \geq 0$,
\begin{align}\label{r1r2rapiddecay}
& \max_{j=0,1,\dots,N}\sup_{k \in \Gamma_{1}} (1+|k|)^N |\partial_k^jr_1(k)| < \infty,  
\qquad
 \max_{j=0,1,\dots,N} \sup_{k \in \Gamma_{4}} (1+|k|)^N|\partial_k^jr_2(k)| < \infty.
\end{align}

\item \label{Theorem2.3itemiv}
For all $k \in \partial \D \setminus \{-\omega,\omega\}$, we have
\begin{align}\label{r1r2 relation on the unit circle}
r_{1}(\tfrac{1}{\omega k}) + r_{2}(\omega k) + r_{1}(\omega^{2} k) r_{2}(\tfrac{1}{k}) = 0.
\end{align}
In fact, \eqref{r1r2 relation on the unit circle} is also equivalent to any of the following two relations:
\begin{align}\label{r1r2 relation on the unit circle new}
r_{2}(k) = \frac{r_{1}( \omega k)r_{1}(\omega^{2} k)-r_{1}(\frac{1}{k})}{1-r_{1}(\omega k)r_{1}(\frac{1}{\omega k})}, \qquad r_{1}(k) = \frac{r_{2}(\omega k)r_{2}(\omega^{2}k)-r_{2}(\frac{1}{k})}{1-r_{2}(\omega k)r_{2}(\frac{1}{\omega k})}.
\end{align}

\item \label{Theorem2.3itemv}
$r_1$ and $r_2$ are related by the symmetry
\begin{align}\label{r1r2 relation with kbar symmetry}
& r_{2}(k) = \tilde{r}(k) \overline{r_{1}(\bar{k}^{-1})} & & \mbox{for all } k \in \hat{\Gamma}_{4}\setminus \{0,\omega^{2}, -\omega^{2}\}, 
\end{align}
where
\begin{align}\label{def of tilde r}
\tilde{r}(k):=\frac{\omega^{2}-k^{2}}{1-\omega^{2}k^{2}} & & \mbox{for } k \in \mathbb{C}\setminus \{\omega^{2},-\omega^{2}\}.
\end{align}


\end{enumerate} 
\end{theorem}

The proof of Theorem \ref{directth} is presented in Section \ref{directsec}.

\subsection{The inverse problem}
We next consider the inverse problem, i.e., the problem of recovering $\{u_0, v_0\}$ from the scattering data. Since we are assuming that no solitons are present, the scattering data consists only of the two reflection coefficients $r_1(k)$ and $r_2(k)$. We will show that the inverse problem can be solved by means of a RH problem for a $1 \times 3$-vector valued function $n(x,t,k)$ with jump contour $\Gamma$, whose jump matrix is expressed in terms of $r_1$ and $r_2$. As always in the implementation of the inverse scattering transform, a key point is that the time evolution of the scattering data is simple. By including this simple time-dependence in the definition of the jump matrix, we obtain the solution $\{u(x,t), v(x,t)\}$ at any later time $t$. We first give the definition of the RH problem.

Define the complex-valued functions $\Phi_{ij}(\zeta, k)$ for $1 \leq j<i \leq 3$ by
\begin{align}\label{def of Phi ij}
\Phi_{ij}(\zeta,k) = (l_{i}(k)-l_{j}(k))\zeta + (z_{i}(k)-z_{j}(k)),
\end{align}
where $\zeta := x/t$ and $\{l_j(k), z_j(k)\}_{j=1}^3$ are defined by (\ref{lmexpressions intro}).
Let $\theta_{ij}(x,t,k) = t \, \Phi_{ij}(\zeta,k)$ and define the jump matrix $v(x,t,k)$ for $k \in \Gamma$ by
\begin{align}\nonumber
&  v_1 = 
  \begin{pmatrix}  
 1 & - r_1(k)e^{-\theta_{21}} & 0 \\
  r_1(\frac{1}{k})e^{\theta_{21}} & 1 - r_1(k)r_1(\frac{1}{k}) & 0 \\
  0 & 0 & 1
  \end{pmatrix},
	\quad  v_2 = 
  \begin{pmatrix}   
 1 & 0 & 0 \\
 0 & 1 - r_2(\omega k)r_2(\frac{1}{\omega k}) & -r_2(\frac{1}{\omega k})e^{-\theta_{32}} \\
 0 & r_2(\omega k)e^{\theta_{32}} & 1 
    \end{pmatrix},
   	\\ \nonumber
  &v_3 = 
  \begin{pmatrix} 
 1 - r_1(\omega^2 k)r_1(\frac{1}{\omega^2 k}) & 0 & r_1(\frac{1}{\omega^2 k})e^{-\theta_{31}} \\
 0 & 1 & 0 \\
 -r_1(\omega^2 k)e^{\theta_{31}} & 0 & 1  
  \end{pmatrix},
	\quad   v_4 = 
  \begin{pmatrix}  
  1 - r_2(k)r_{2}(\frac{1}{k}) & -r_2(\frac{1}{k}) e^{-\theta_{21}} & 0 \\
  r_2(k)e^{\theta_{21}} & 1 & 0 \\
  0 & 0 & 1
   \end{pmatrix},
   	\\ \nonumber
&  v_5 = 
  \begin{pmatrix}
  1 & 0 & 0 \\
  0 & 1 & -r_1(\omega k)e^{-\theta_{32}} \\
  0 & r_1(\frac{1}{\omega k})e^{\theta_{32}} & 1 - r_1(\omega k)r_1(\frac{1}{\omega k}) 
  \end{pmatrix},
	\quad   v_6 = 
  \begin{pmatrix} 
  1 & 0 & r_2(\omega^2 k)e^{-\theta_{31}} \\
  0 & 1 & 0 \\
  -r_2(\frac{1}{\omega^2 k})e^{\theta_{31}} & 0 & 1 - r_2(\omega^2 k)r_2(\frac{1}{\omega^2 k})
   \end{pmatrix}, \nonumber \\
& v_{7} = \begin{pmatrix}
1 & -r_{1}(k)e^{-\theta_{21}} & r_{2}(\omega^{2}k)e^{-\theta_{31}} \\
-r_{2}(k)e^{\theta_{21}} & 1+r_{1}(k)r_{2}(k) & \big(r_{2}(\frac{1}{\omega k})-r_{2}(k)r_{2}(\omega^{2}k)\big)e^{-\theta_{32}} \\
r_{1}(\omega^{2}k)e^{\theta_{31}} & \big(r_{1}(\frac{1}{\omega k})-r_{1}(k)r_{1}(\omega^{2}k)\big)e^{\theta_{32}} & f(\omega^{2}k)
\end{pmatrix}, \nonumber \\
& v_{8} = \begin{pmatrix}
f(k) & r_{1}(k)e^{-\theta_{21}} & \big(r_{1}(\frac{1}{\omega^{2} k})-r_{1}(k)r_{1}(\omega k)\big)e^{-\theta_{31}} \\
r_{2}(k)e^{\theta_{21}} & 1 & -r_{1}(\omega k) e^{-\theta_{32}} \\
\big( r_{2}(\frac{1}{\omega^{2}k})-r_{2}(\omega k)r_{2}(k) \big)e^{\theta_{31}} & -r_{2}(\omega k) e^{\theta_{32}} & 1+r_{1}(\omega k)r_{2}(\omega k)
\end{pmatrix}, \nonumber \\
& v_{9} = \begin{pmatrix}
1+r_{1}(\omega^{2}k)r_{2}(\omega^{2}k) & \big( r_{2}(\frac{1}{k})-r_{2}(\omega k)r_{2}(\omega^{2} k) \big)e^{-\theta_{21}} & -r_{2}(\omega^{2}k)e^{-\theta_{31}} \\
\big(r_{1}(\frac{1}{k})-r_{1}(\omega k) r_{1}(\omega^{2} k)\big)e^{\theta_{21}} & f(\omega k) & r_{1}(\omega k)e^{-\theta_{32}} \\
-r_{1}(\omega^{2}k)e^{\theta_{31}} & r_{2}(\omega k) e^{\theta_{32}} & 1
\end{pmatrix}, \label{vdef}
\end{align}
where $v_j$ denotes the restriction of $v$ to $\Gamma_{j}$, and 
\begin{align}\label{def of f}
f(k) := 1+r_{1}(k)r_{2}(k) + r_{1}(\tfrac{1}{\omega^{2}k})r_{2}(\tfrac{1}{\omega^{2}k}), \qquad k \in \partial \mathbb{D}.
\end{align}
Let $\Gamma_\star = \{i\kappa_j\}_{j=1}^6 \cup \{0\}$ denote the set of intersection points of $\Gamma$.

\begin{RHproblem}[RH problem for $n$]\label{RHn}
Find a $1 \times 3$-row-vector valued function $n(x,t,k)$ with the following properties:
\begin{enumerate}[$(a)$]
\item\label{RHnitema} $n(x,t,\cdot) : \C \setminus \Gamma \to \mathbb{C}^{1 \times 3}$ is analytic.

\item\label{RHnitemb} The limits of $n(x,t,k)$ as $k$ approaches $\Gamma \setminus \Gamma_\star$ from the left and right exist, are continuous on $\Gamma \setminus \Gamma_\star$, and are denoted by $n_+$ and $n_-$, respectively. Furthermore, they are related by
\begin{align}\label{njump}
  n_+(x,t,k) = n_-(x, t, k) v(x, t, k) \qquad \text{for} \quad k \in \Gamma \setminus \Gamma_\star.
\end{align}

\item\label{RHnitemc} $n(x,t,k) = O(1)$ as $k \to k_{\star} \in \Gamma_\star$.

\item\label{RHnitemd} For $k \in \C \setminus \Gamma$, $n$ obeys the symmetries
\begin{align}\label{nsymm}
n(x,t,k) = n(x,t,\omega k)\mathcal{A}^{-1} = n(x,t,k^{-1}) \mathcal{B},
\end{align}
where $\mathcal{A}$ and $\mathcal{B}$ are the matrices defined by 
\begin{align}\label{def of Acal and Bcal}
\mathcal{A} := \begin{pmatrix}
0 & 0 & 1 \\
1 & 0 & 0 \\
0 & 1 & 0
\end{pmatrix} \qquad \mbox{ and } \qquad \mathcal{B} := \begin{pmatrix}
0 & 1 & 0 \\
1 & 0 & 0 \\
0 & 0 & 1
\end{pmatrix}.
\end{align}

\item\label{RHniteme} $n(x,t,k) = (1,1,1) + O(k^{-1})$ as $k \to \infty$.
\end{enumerate}
\end{RHproblem}

Our second theorem states that the solution $\{u(x,t), v(x,t)\}$ of the Boussinesq system \eqref{boussinesqsystem} can be recovered from the solution $n = (n_1, n_2, n_3)$ of RH problem \ref{RHn}.
Although it is possible to carry out all the arguments under more restricted regularity and decay assumptions, we will only deal with Schwartz class solutions for simplicity. 

\begin{definition}\label{Schwartzsolutiondef}\upshape
Let $T \in (0, \infty]$. We say that $\{u(x,t), v(x,t)\}$ is a {\it Schwartz class solution of \eqref{boussinesqsystem} on $\R \times [0,T)$ with initial data $u_0, v_0 \in \mathcal{S}(\R)$} if
\begin{enumerate}[$(i)$] 
  \item $u,v$ are smooth real-valued functions of $(x,t) \in \R \times [0,T)$.

\item $u,v$ satisfy \eqref{boussinesqsystem} for $(x,t) \in \R \times [0,T)$ and 
$$u(x,0) = u_0(x), \quad v(x,0) = v_0(x), \qquad x \in \R.$$ 

  \item $u,v$ have rapid decay as $|x| \to \infty$ in the sense that, for each integer $N \geq 1$ and each $\tau \in [0,T)$,
$$\sup_{\substack{x \in \R \\ t \in [0, \tau]}} \sum_{i =0}^N (1+|x|)^N(|\partial_x^i u| + |\partial_x^i v| ) < \infty.$$
\end{enumerate} 
\end{definition}

\begin{theorem}[Solution of the inverse problem]\label{inverseth}
Let $r_1:\hat{\Gamma}_{1}\setminus \hat{\mathcal{Q}} \to \C$ and $r_2: \hat{\Gamma}_{4}\setminus \hat{\mathcal{Q}} \to \C$ be two functions which satisfy properties $(i)$--$(v)$ of Theorem \ref{directth}.
Define $T \in (0, \infty]$ by 
\begin{align}\label{Tdef}
T := \sup \big\{t \geq 0 \, | \, \text{$e^{\frac{|k|^2t}{4}}r_1(1/k)$ and its derivatives are rapidly decreasing as $\Gamma_1 \ni k \to \infty$}\big\}.
\end{align}

Then RH problem \ref{RHn} has a unique solution $n(x,t,k)$ for each $(x,t) \in \R \times [0,T)$ and the function $n_{3}^{(1)}$ defined by
$$n_{3}^{(1)}(x,t) := \lim_{k\to \infty} k (n_{3}(x,t,k) -1)$$ 
is well-defined and smooth for $(x,t) \in \R \times [0,T)$. Moreover, $\{u(x,t), v(x,t)\}$ defined by
\begin{align}\label{recoveruvn}
\begin{cases}
u(x,t) = -i\sqrt{3}\frac{\partial}{\partial x}n_{3}^{(1)}(x,t),
	\\
v(x,t) = -i\sqrt{3}\frac{\partial}{\partial t}n_{3}^{(1)}(x,t),
\end{cases}
\end{align}
is a Schwartz class solution of (\ref{boussinesqsystem}) on $\R \times [0,T)$.
\end{theorem}

An important ingredient in the proof of Theorem \ref{inverseth} is a vanishing lemma for RH problem \ref{RHn}; such a lemma is established in Section \ref{vanishinglemmasubsec}.
The proof of Theorem \ref{inverseth} is then given in Section \ref{inversethsubsec}.

\subsection{Solution of the initial value problem for (\ref{badboussinesq})}
By combining the solutions of the direct and inverse problems and recalling the relation between (\ref{badboussinesq}) and the system (\ref{boussinesqsystem}), it is possible to solve the initial value problem for the Boussinesq equation (\ref{badboussinesq}) via inverse scattering.

\begin{definition}\upshape
Let $T \in (0, \infty]$. We say that $u(x,t)$ is a {\it Schwartz class solution of \eqref{badboussinesq} on $\R \times [0,T)$ with initial data $u_0, u_1 \in \mathcal{S}(\R)$} if
\begin{enumerate}[$(i)$] 
  \item $u$ is a smooth real-valued functions of $(x,t) \in \R \times [0,T)$.

\item $u$ satisfies \eqref{badboussinesq} for $(x,t) \in \R \times [0,T)$ and 
$$u(x,0) = u_0(x), \quad u_t(x,0) = u_1(x), \qquad x \in \R.$$ 

  \item $u$ has rapid decay as $|x| \to \infty$ in the sense that, for each integer $N \geq 1$ and each $\tau \in [0,T)$,
$$\sup_{\substack{x \in \R \\ t \in [0, \tau]}} \sum_{i =0}^N (1+|x|)^N |\partial_x^i u(x,t)|  < \infty.$$
\end{enumerate} 
\end{definition}

The proof of the following theorem is presented in Section \ref{IVPsec}.

\begin{theorem}[Solution of (\ref{badboussinesq}) via direct and inverse scattering]\label{IVPth}
Let $u_0,u_1 \in \mathcal{S}(\R)$ be real-valued and suppose that
\begin{align}\label{u1zeromean}
\int_\R u_1(x) dx = 0.
\end{align}
Let $v_0(x) = \int_{-\infty}^x u_1(x')dx'$ and suppose $u_0, v_0$ are such that Assumptions \ref{solitonlessassumption} and \ref{originassumption} hold. Define the spectral functions $r_j(k)$, $j = 1,2$, in terms of $u_0, v_0$ by (\ref{r1r2def}). 
Define $T \in (0,\infty]$ by (\ref{Tdef}).
Then the initial value problem for (\ref{badboussinesq}) with initial data $u_0, u_1$ has a unique Schwartz class solution $u(x,t)$ on $\R \times [0,T)$. Moreover, $u(x,t)$ can be recovered from the solution $n(x,t,k)$ of RH problem \ref{RHn} via the representation formula (\ref{recoveruvn}) for any $(x,t) \in \R \times [0,T)$.
\end{theorem}

The condition (\ref{u1zeromean}) is required to ensure that $v_0 \in \mathcal{S}(\R)$. Physically, it ensures that the total mass $\int_\R u dx$ does not grow linearly but is conserved in time. Indeed, (\ref{badboussinesq}) implies that 
\begin{align*}
\frac{d^2}{dt^2}\int_\R u dx = 0, \quad \text{i.e.} \quad \int_\R u dx = \bigg(\int_\R u_1 dx\bigg)t + \int_\R u_0 dx.
\end{align*}

\subsection{Blow-up}
Theorem \ref{inverseth} states that the solution of (\ref{boussinesqsystem}) exists at least as long as $t < T$ where $T$ is defined by (\ref{Tdef}). Our next result shows that the solution actually ceases to exist at the time $T$ if $T < \infty$.

\begin{theorem}[Blow-up]\label{blowupth}
Under the assumptions of Theorem \ref{inverseth}, if $T$ defined by (\ref{Tdef}) satisfies $T < \infty$, then the Schwartz class solution $\{u(x,t), v(x,t)\}$ defined in (\ref{recoveruvn}) blows up at time $T$.
\end{theorem}

More precisely, the conclusion of Theorem \ref{blowupth} is that there exists no Schwartz class solution $\{\tilde{u}, \tilde{v}\}$ of (\ref{boussinesqsystem}) on $[0,\tilde{T})$ with $\tilde{T} > T$ which coincides with $\{u, v\}$ for $t \in [0,T)$. An analogous definition of blow-up is used also in the next theorem. 

\begin{theorem}[Existence of blow-up solutions]\label{existenceblowupth}
For each $T > 0$, there exists a class of solutions of (\ref{badboussinesq}) that blow up at time $T$. 
\end{theorem}

The proofs of Theorems \ref{blowupth} and \ref{existenceblowupth} are presented in Section \ref{blowupsec}.

\subsection{Global solutions}
As explained in Section \ref{globalsubsec}, solutions of (\ref{badboussinesq}) without ``unstable nonlinear Fourier modes'' are of particular interest. Such solutions are obtained by assuming the following. 

\begin{assumption}\label{nounstablemodesassumption}
The function $r_{1}:\hat{\Gamma}_{1}\setminus \hat{\mathcal{Q}}\to \mathbb{C}$ satisfies $r_{1}(k) = 0$ for all $k \in [0,i]$, where $[0,i]$ denotes the vertical segment from $0$ to $i$.
\end{assumption}

Assumption \ref{nounstablemodesassumption} ensures that $T = \infty$ in (\ref{Tdef}) and leads to the existence of global solutions. The following theorem follows immediately from Theorem \ref{IVPth}.

\begin{theorem}[Global solutions]\label{globalth}
Let $u_0,u_1 \in \mathcal{S}(\R)$ be such that the assumptions of Theorem \ref{IVPth} are fulfilled. Assume also that Assumption \ref{nounstablemodesassumption} holds. 
Then the initial value problem for (\ref{badboussinesq}) with initial data $u_0, u_1$ has a unique global Schwartz class solution $u(x,t)$. Moreover, $u(x,t)$ can be recovered from the solution $n(x,t,k)$ of RH problem \ref{RHn} via the representation formula (\ref{recoveruvn}) for any $(x,t) \in \R \times [0,\infty)$.
\end{theorem}

\subsection{Asymptotics}
For the class of global solutions constructed in Theorem \ref{globalth}, it makes sense to consider the large $t$ behavior of the solution. Our next theorem describes this behavior in detail. To state it, we need the following lemma whose proof is given in Section \ref{inequalitiessubsec}.

\begin{figure}[h!]
\begin{center}
\vspace{-.2cm}
\begin{tikzpicture}[master]
\node at (0,0) {\includegraphics[scale=0.20]{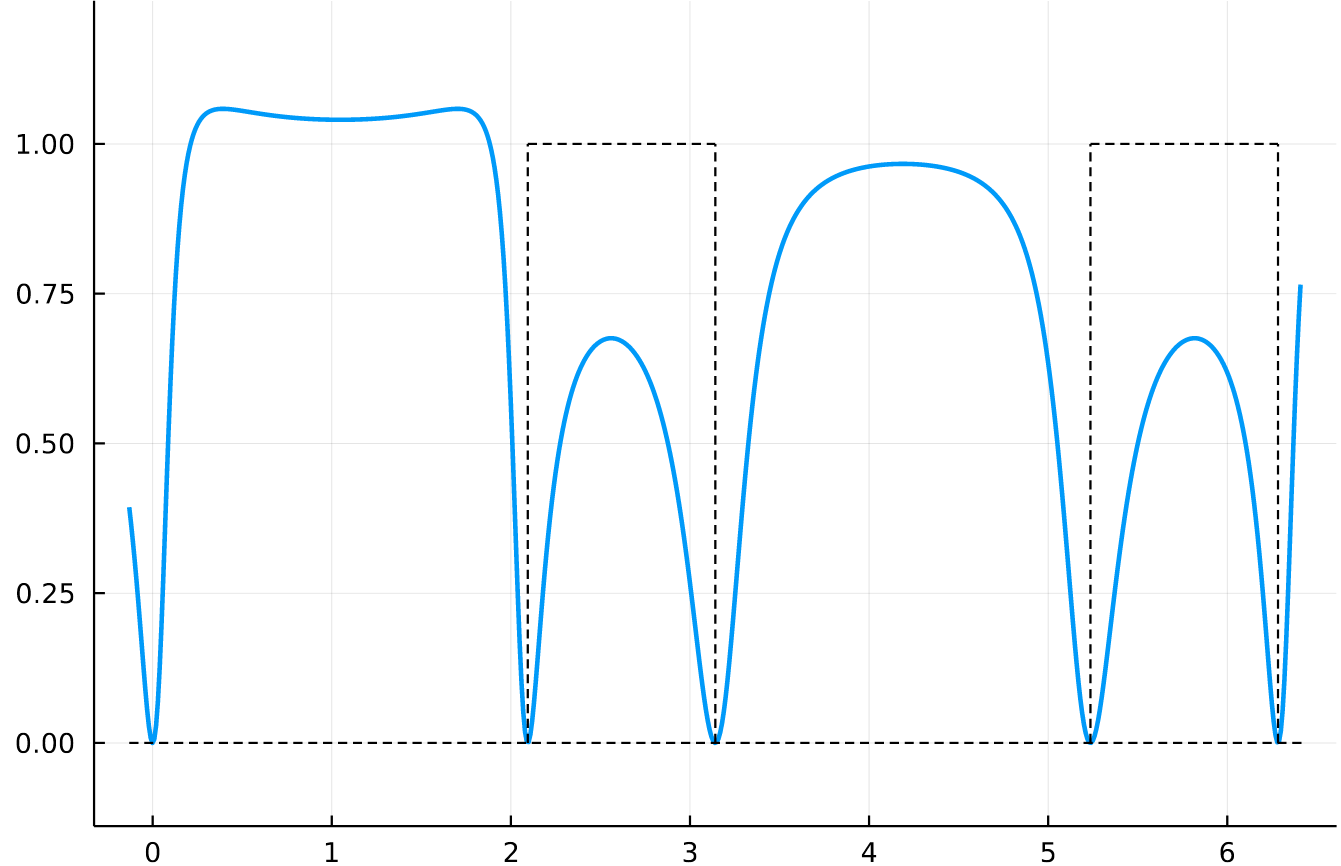}};
\node at (-2.95,2.53) {\tiny $f(e^{i\theta})$};
\node at (3.65,-2.1) {\tiny $\theta$};
\node at (-.25,1.1) {\tiny $(i)$};
\node at (2.72,1.1) {\tiny $(i)$};
\node at (3.56,-1.64) {\tiny $(i)$};
\end{tikzpicture}
\begin{tikzpicture}[slave]
\node at (0,0) {\includegraphics[scale=0.20]{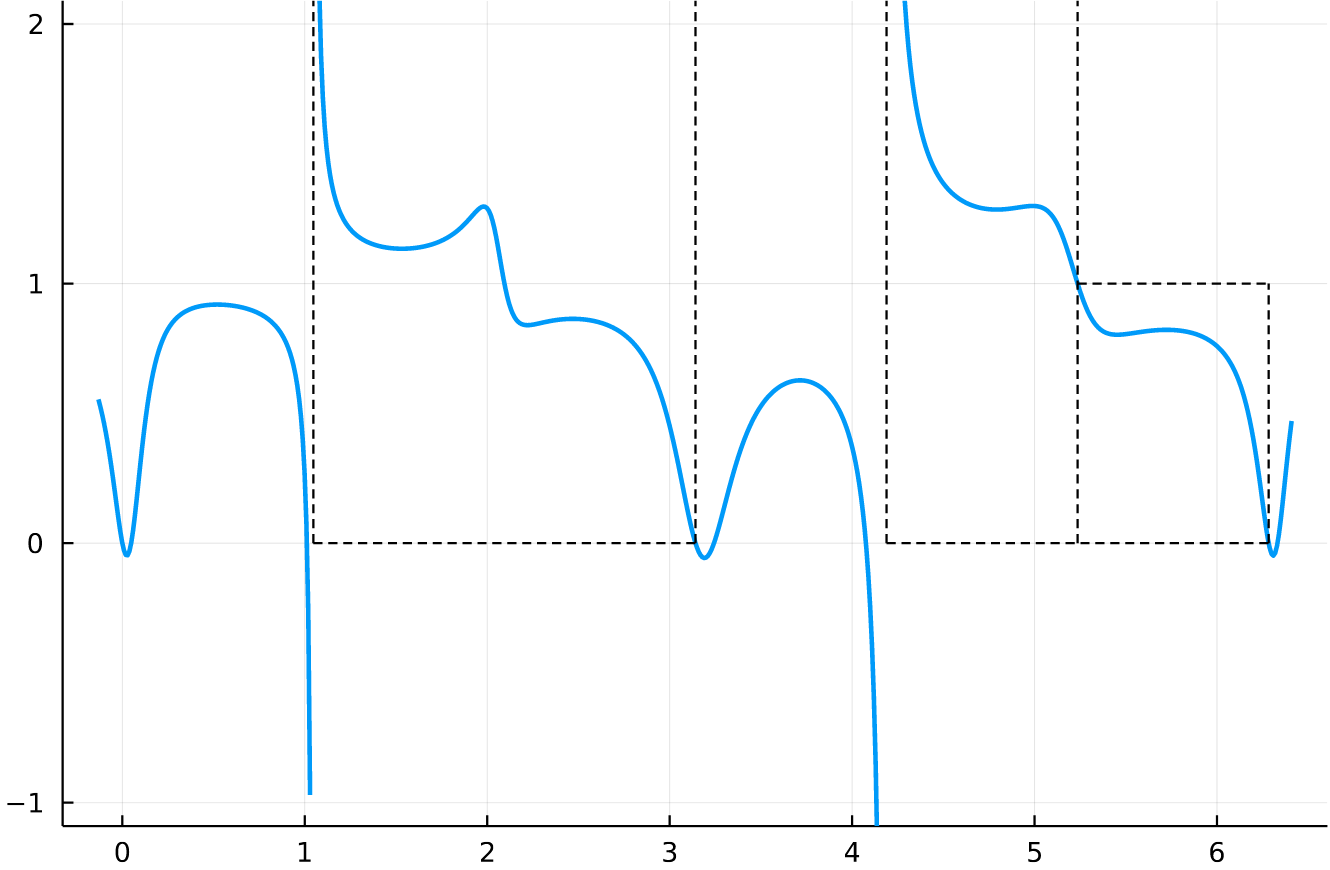}};
\node at (-2.95,2.53) {\tiny $1+r_{1}(e^{i\theta})r_{2}(e^{i\theta})$};
\node at (3.65,-2.1) {\tiny $\theta$};
\node at (-0.85,1.85) {\tiny $(ii)$};
\node at (1.67,1.85) {\tiny $(ii)$};
\node at (2.65,-0.35) {\tiny $(ii)$};
\node at (2.65,0.1) {\tiny $(iii)$};
\end{tikzpicture}
\end{center}
\begin{figuretext}
\label{fig:f}
The functions $\theta \mapsto f(e^{i\theta})$ (left) and $\theta \mapsto 1+r_{1}(e^{i\theta})r_{2}(e^{i\theta})$ (right) in the case of the initial data (\ref{numericscompactinitialdata}). The dashed lines illustrate assertions $(i)$--$(iii)$ of Lemma \ref{inequalitieslemma} as indicated.
\end{figuretext}
\end{figure}

\begin{lemma}[Inequalities satisfied by the spectral functions]\label{inequalitieslemma}
Suppose $u_0,v_0 \in \mathcal{S}(\R)$ are such that Assumptions \ref{solitonlessassumption} and \ref{originassumption} hold. Let $\{r_j\}_1^2$ be the associated reflection coefficients defined in (\ref{r1r2def}).
\begin{enumerate}[$(i)$]
\item\label{inequalitieslemmaitemi}
 The function $f:\partial \mathbb{D}\to \mathbb{C}$ defined in (\ref{def of f}) satisfies 
\begin{itemize}
\item[$(a)$] $f(k) \geq 0$ for all $k \in \partial \mathbb{D}$,

\item[$(b)$] $f(k)=0$ if and only if $k \in \{\pm 1, \pm \omega\}$, and 

\item[$(c)$] $f(k) \leq 1$ for all $k \in \partial \mathbb{D}$ with $\arg k \in (2\pi/3, \pi) \cup (5\pi/3, 2\pi)$.
\end{itemize}

\item\label{inequalitieslemmaitemii} $1+r_{1}(k)r_{2}(k) > 0$ for all $k \in \partial \mathbb{D}$ with $\arg k \in (\pi/3, \pi) \cup (4\pi/3, 2\pi)$.

\item\label{inequalitieslemmaitemiii} $-\frac{1}{2\pi}\ln(1+r_{1}(k)r_{2}(k)) \geq 0$ for all $k \in \partial \mathbb{D}$ with $\arg k \in (5\pi/3, 2\pi)$.

\item\label{inequalitieslemmaitemiv} The functions $\hat{\nu}_1, \hat{\nu}_{2}:\partial \mathbb{D}\to \mathbb{C}$ defined by
\begin{align}\label{hatnu12def}
\hat{\nu}_1(k) := \nu_3(k) - \nu_1(k), \qquad \hat{\nu}_{2}(k) = \nu_{2}(k) +\nu_3(k) -\nu_{4}(k),
\end{align}
where
\begin{align}\nonumber
& \nu_1(k) := - \frac{1}{2\pi}\ln(1+r_{1}(\omega k)r_{2}(\omega k)), 
&& \nu_2(k) := - \frac{1}{2\pi}\ln(1+r_{1}(\omega^{2} k)r_{2}(\omega^{2} k)), 
	\\ \label{nu12345def}
& \nu_3(k) := - \frac{1}{2\pi}\ln(f(\omega k)),  &&
\nu_4(k) := - \frac{1}{2\pi}\ln(f(\omega^{2} k)).
\end{align}
satisfy $\hat{\nu}_1(k) \geq 0$ for all $k \in \partial \mathbb{D}$ with $\arg k \in (5\pi/3, 2\pi)$, and $\hat{\nu}_{2}(k) \geq 0$ for all $k \in \partial \mathbb{D}$ with $\arg k \in (\pi, 4\pi/3)$.

\end{enumerate} 
\end{lemma}

\begin{figure}
\begin{center}
\vspace{-.3cm}
\begin{tikzpicture}[master]
\node at (0,0) {\includegraphics[width=2.9cm]{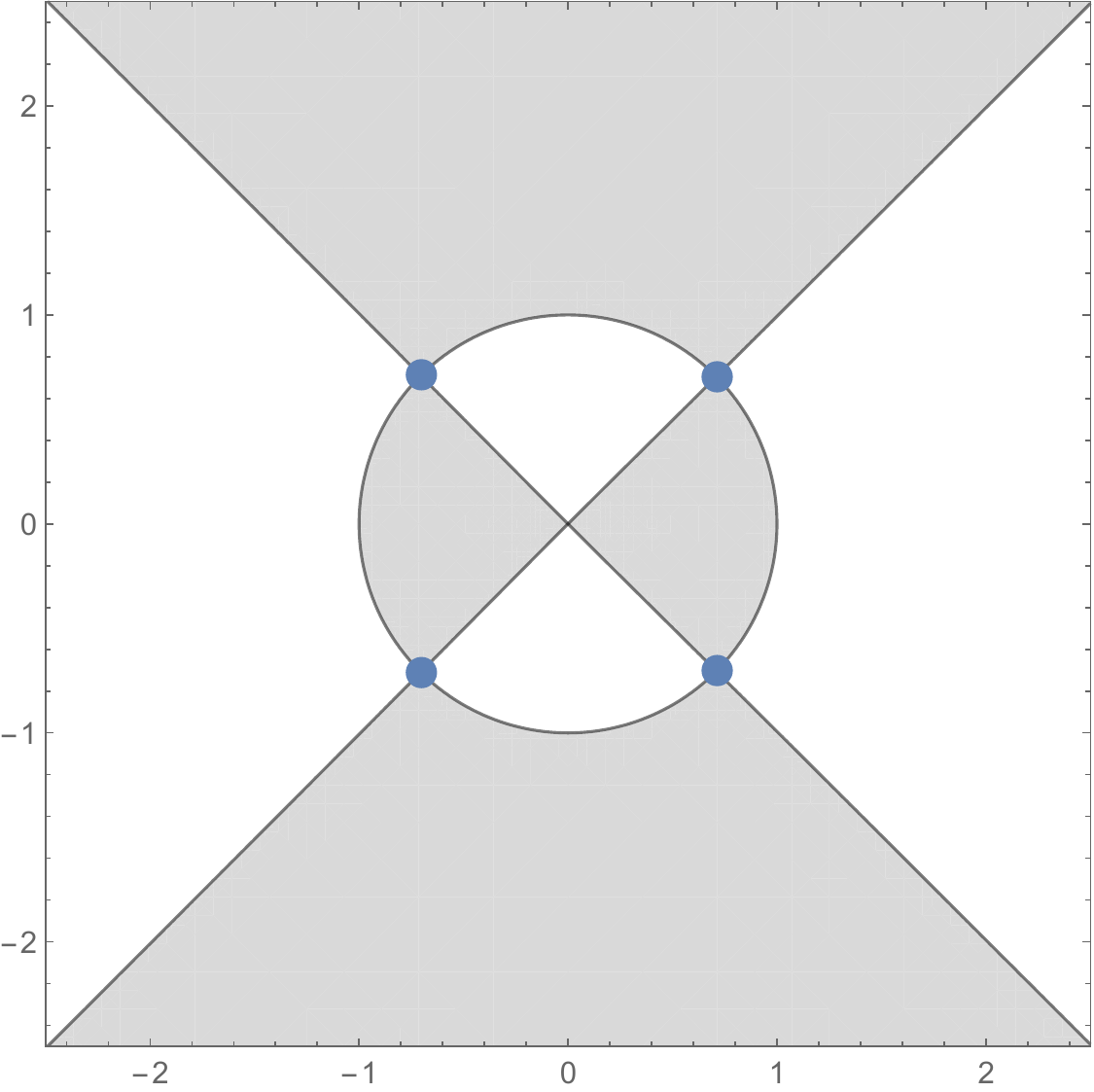}};
\node at (-0.33,0.65) {\footnotesize $k_{1}$};
\node at (-0.33,-0.5) {\footnotesize $k_{2}$};
\node at (0.43,0.65) {\footnotesize $k_{3}$};
\node at (0.46,-0.5) {\footnotesize $k_{4}$};
\end{tikzpicture} \hspace{-0.4cm}
\begin{tikzpicture}[slave]
\node at (0,0) {\includegraphics[width=2.9cm]{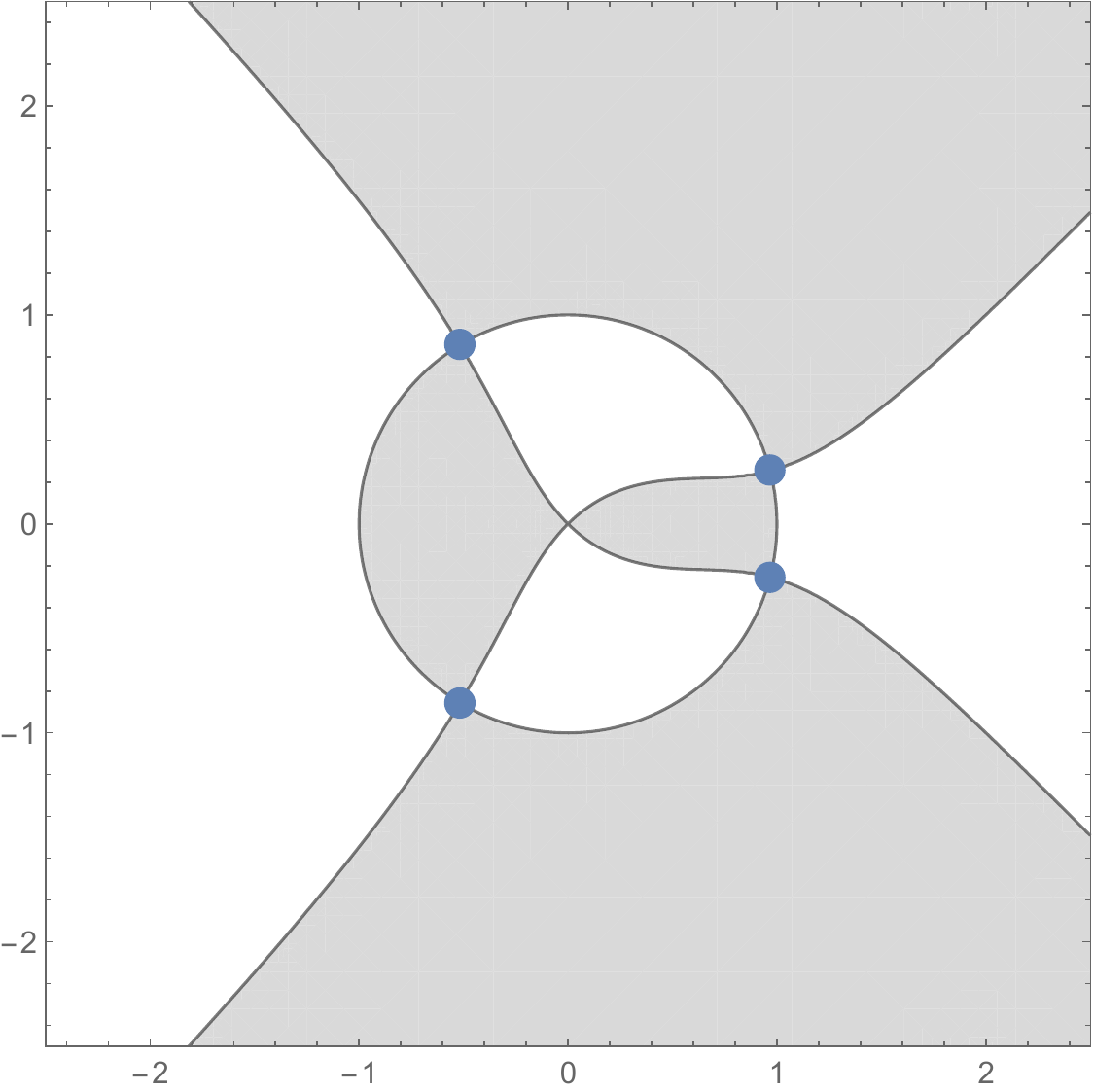}};
\node at (-0.23,0.73) {\footnotesize $k_{1}$};
\node at (-0.23,-0.55) {\footnotesize $k_{2}$};
\node at (0.6,0.4) {\footnotesize $k_{3}$};
\node at (0.6,-0.27) {\footnotesize $k_{4}$};
\end{tikzpicture} \hspace{-0.4cm} \begin{tikzpicture}[slave]
\node at (0,0) {\includegraphics[width=2.9cm]{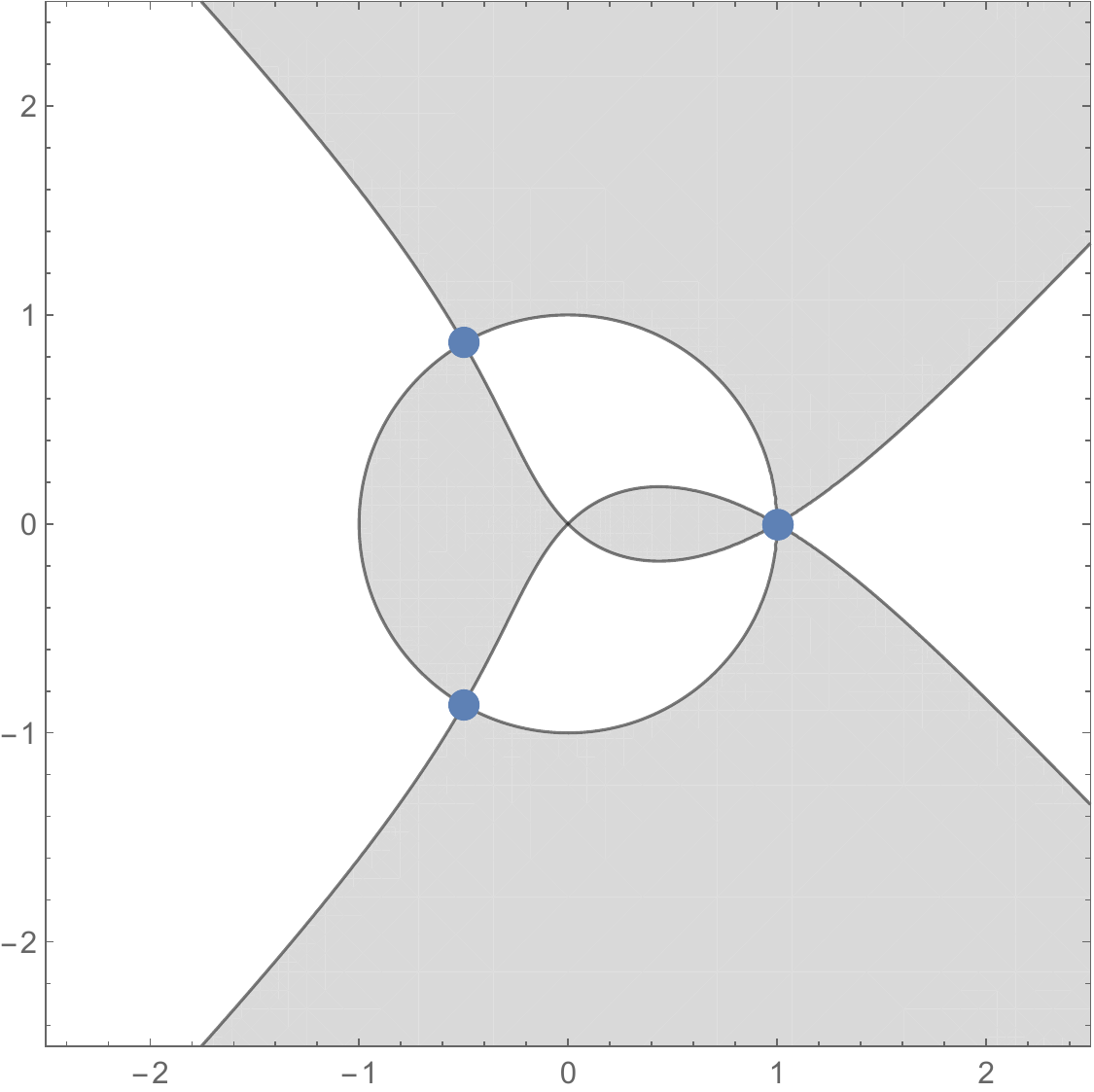}};
\node at (-0.2,0.75) {\footnotesize $k_{1}$};
\node at (-0.2,-0.58) {\footnotesize $k_{2}$};
\node at (0.7,0.3) {\footnotesize $k_{3} \hspace{-0.08cm}=\hspace{-0.08cm}k_{4}$};
\end{tikzpicture} \hspace{-0.4cm} \begin{tikzpicture}[slave]
\node at (0,0) {\includegraphics[width=2.9cm]{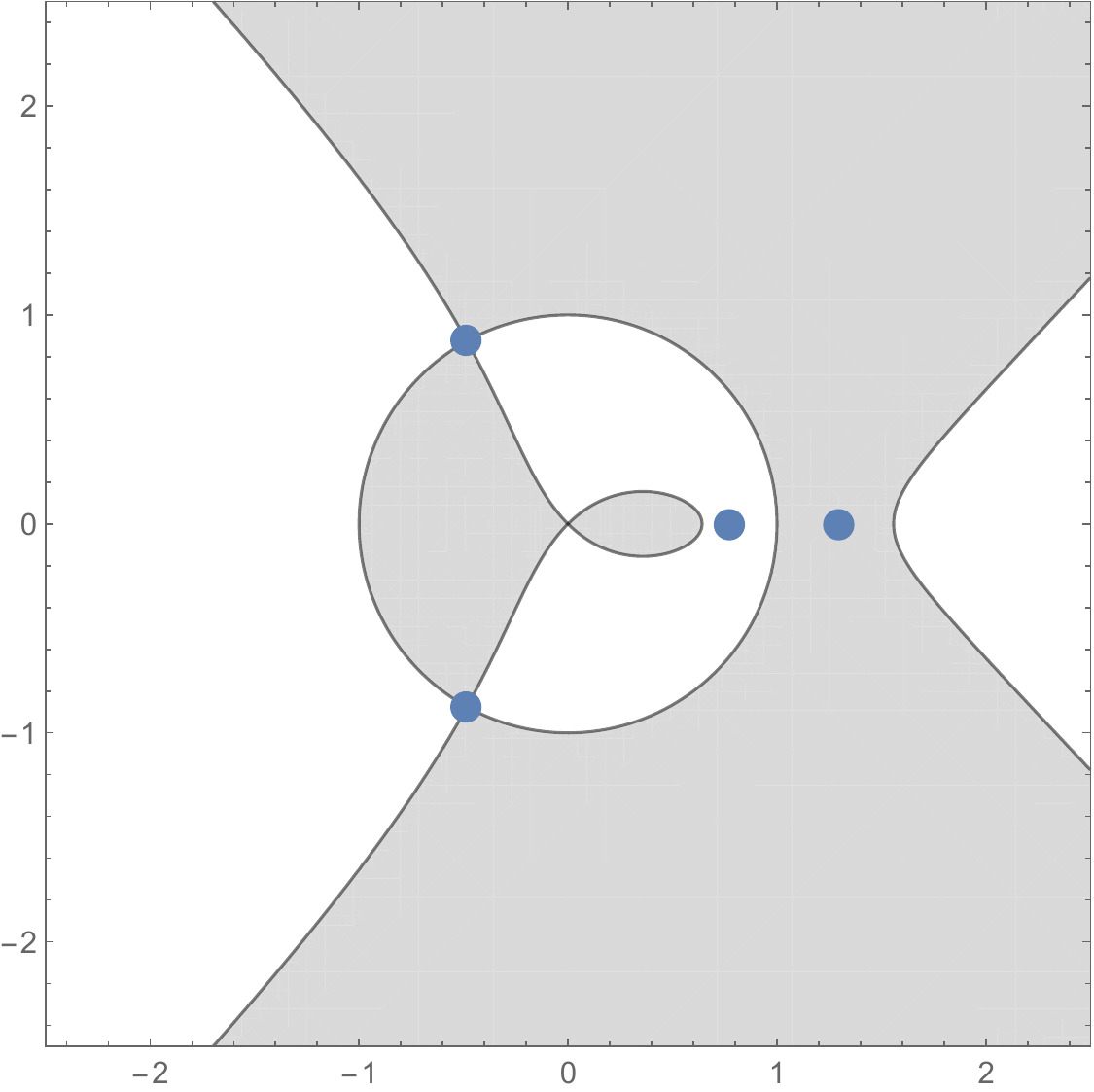}};
\node at (-0.17,0.75) {\footnotesize $k_{1}$};
\node at (-0.17,-0.62) {\footnotesize $k_{2}$};
\node at (0.82,0.25) {\footnotesize $k_{3}$};
\node at (0.45,0.25) {\footnotesize $k_{4}$};
\end{tikzpicture} \hspace{-0.4cm}
\begin{tikzpicture}[slave]
\node at (0,0) {\includegraphics[width=2.9cm]{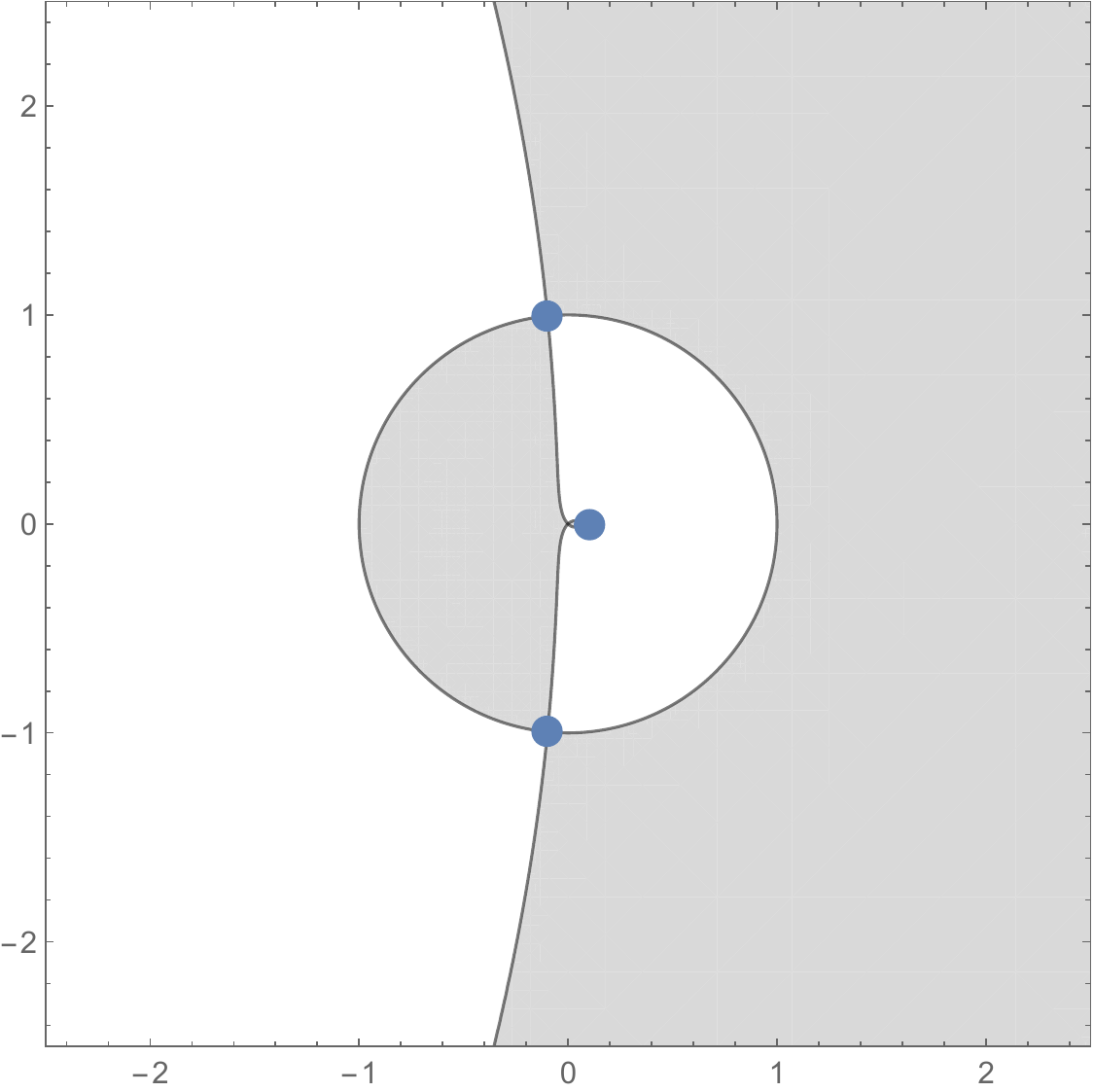}};
\node at (0,0.8) {\footnotesize $k_{1}$};
\node at (0,-0.67) {\footnotesize $k_{2}$};
\node at (0.35,0.05) {\footnotesize $k_{4}$};
\end{tikzpicture}
\end{center}
\begin{figuretext}\label{fig: Re Phi 21 for various zeta} From left to right: $\zeta = 0.01$, $\zeta=0.9$, $\zeta=1$, $\zeta=1.1$, and $\zeta = 10$. The shaded regions correspond to $\{k \, | \, \re \Phi_{21}(\zeta,k)>0\}$ and the white regions to $\{k \, | \, \re \Phi_{21}(\zeta,k)<0\}$. The dots represent the four saddle points $\{k_{j}(\zeta)\}_{j=1}^{4}$.
\end{figuretext}
\vspace{-.2cm}
\end{figure}

The statement of our next theorem, whose proof is given in Sections \ref{overviewsec}--\ref{othersectorssec}, involves several square roots and logarithms. Lemma \ref{inequalitieslemma} implies that the arguments of all these square roots and logarithms are $\geq 0$ and $> 0$, respectively. Furthermore, Theorem \ref{globalth} ensures that the solution $u(x,t)$ in the statement exists and is unique. The statement also involves  the saddle points $k_j = k_j(\zeta)$, $j = 1, \dots, 4$, of $\Phi_{21}$ given by (see Figure \ref{fig: Re Phi 21 for various zeta})
\begin{subequations}\label{def of kj}
\begin{align}
& k_{1} = k_{1}(\zeta) = \frac{1}{4}\bigg( \zeta - \sqrt{8+\zeta^{2}} + i \sqrt{2}\sqrt{4-\zeta^{2}+\zeta\sqrt{8+\zeta^{2}}} \bigg), \label{def of k1} \\
& k_{2} = k_{2}(\zeta) = \frac{1}{4}\bigg( \zeta - \sqrt{8+\zeta^{2}} - i \sqrt{2}\sqrt{4-\zeta^{2}+\zeta\sqrt{8+\zeta^{2}}} \bigg), \label{def of k2} \\
& k_{3} = k_{3}(\zeta) = \frac{1}{4}\bigg( \zeta + \sqrt{8+\zeta^{2}} + \sqrt{2}\sqrt{-4+\zeta^{2}+\zeta\sqrt{8+\zeta^{2}}} \bigg), \label{def of k3} \\
& k_{4} = k_{4}(\zeta) = \frac{1}{4}\bigg( \zeta + \sqrt{8+\zeta^{2}} - \sqrt{2}\sqrt{-4+\zeta^{2}+\zeta\sqrt{8+\zeta^{2}}} \bigg), \label{def of k4}
\end{align}
\end{subequations}
and the function $\tilde{r}$ defined in (\ref{def of tilde r}).
We let $u_{\HM}$ denote the Hasting--McLeod solution of the Painlev\'e II equation, i.e., $u_{\HM}$ is the unique solution of 
\begin{align}\label{painleveII}
u''(y) = yu(y) + 2u(y)^3
\end{align}
such that 
\begin{align}\label{uHMasymptotics}
u_{\HM}(y) = \begin{cases} \Ai(y)(1+o(1)), & y \to +\infty, \\
\sqrt{-y/2}(1+o(1)), & y \to -\infty.
\end{cases}
\end{align}

\begin{figure}[t]
\begin{center}
\begin{tikzpicture}[master]
\node at (0,0) {\includegraphics[scale=0.2]{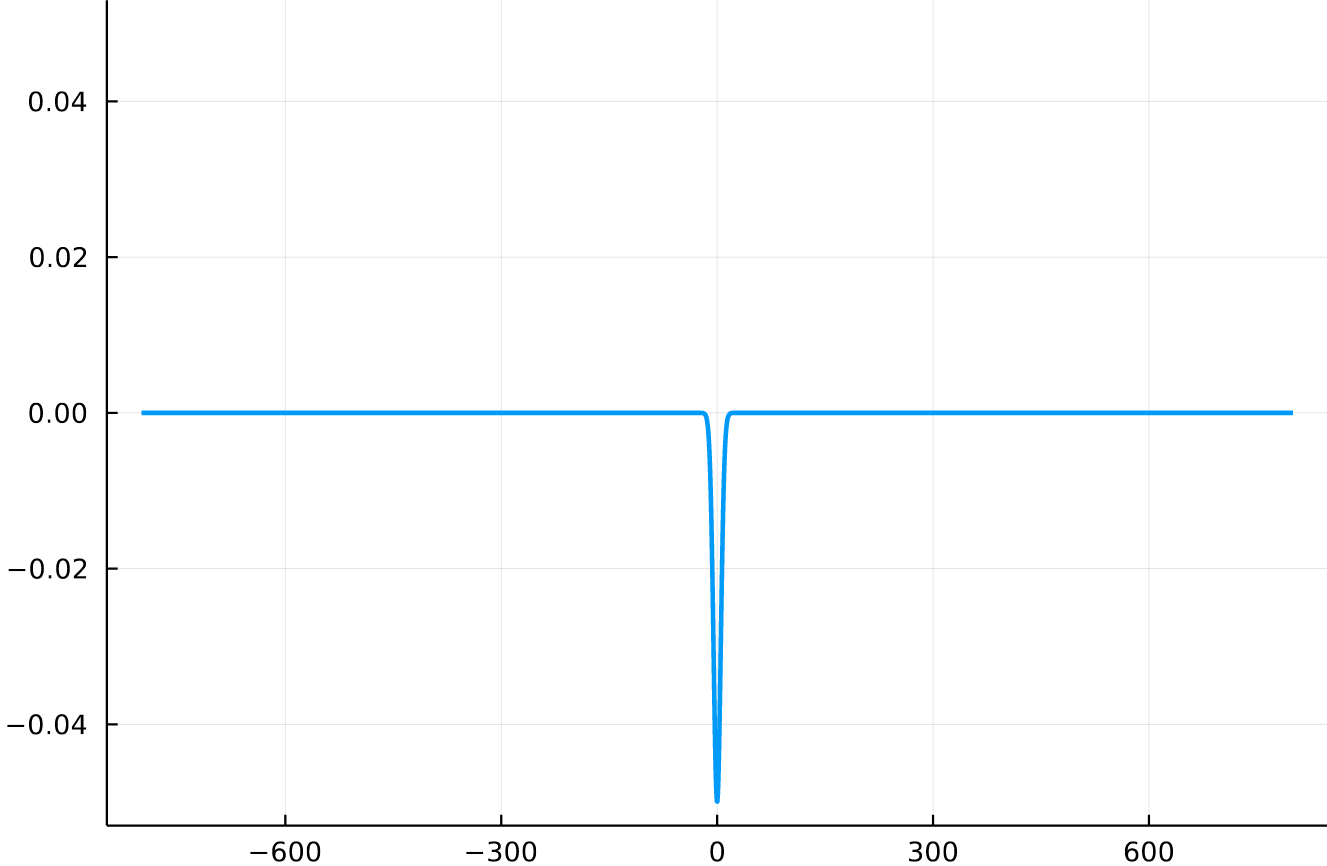}};
\node at (3.68,-2.07) {\tiny $x$};
\node at (-2.95,2.43) {\tiny $u$};
\node at (0.3,1.3) {\tiny $t=0$};
\end{tikzpicture}
\begin{tikzpicture}[slave]
\node at (0,0) {\includegraphics[scale=0.2]{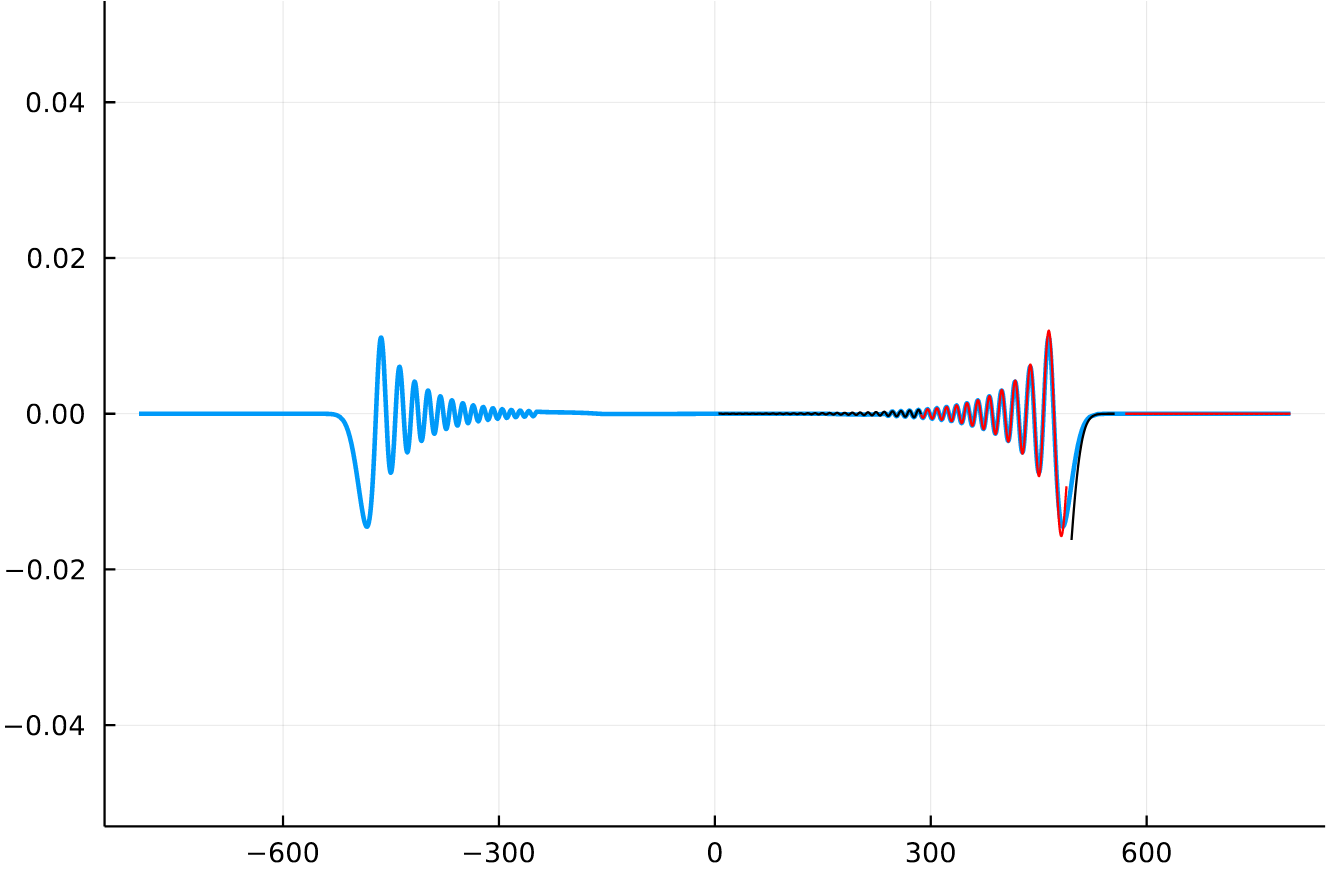}};
\node at (3.68,-2.07) {\tiny $x$};
\node at (-2.95,2.43) {\tiny $u$};
\node at (0.3,1.3) {\tiny $t=500$};
\node at (0.8,-.1) {\tiny Sector \V};
\node at (1.85,0.75) {\tiny \color{red} Sector \IV};

\draw[->-=1] (2.65,-0.97)--(2.22,-0.32);
\node at (2.7,-1.1) {\tiny Sector \III};

\node at (3,-0.1) {\tiny \color{red} Sector \II};
\end{tikzpicture}
\end{center}
\begin{figuretext}
\label{fig:asymp}
Numerical simulation of the solution $u(x,t)$ of (\ref{badboussinesq}) (blue) together with the asymptotic approximations provided by Theorem \ref{asymptoticsth} (black and red). The initial data is $u_{0}(x) = -0.05e^{-0.02x^{2}}$ and $u_1(x) = 0$. The approximations in Sectors $\III$ and $\V$ are drawn black, while those in Sectors $\II$ and $\IV$ are drawn red. 
\end{figuretext}
\end{figure}

\begin{figure}[t]
\begin{center}
\vspace{-.3cm}
\begin{tikzpicture}[master]
\hspace{-0.17cm}
\node at (0,0) {\includegraphics[scale=0.343]{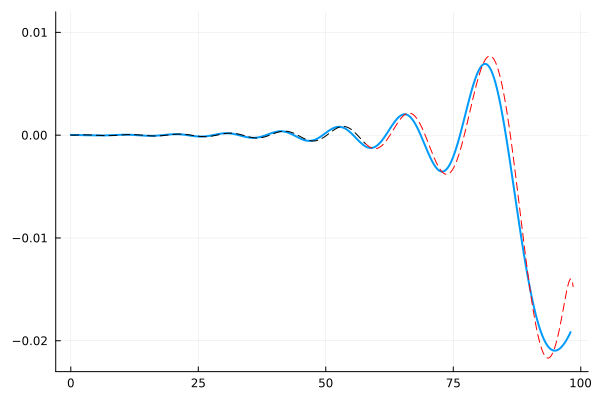}};
\node at (3.66,-2.08) {\tiny $x$};
\node at (-2.95,2.45) {\tiny $u$};
\node at (0.3,1.5) {\tiny $t=100$};
\node at (-1.2,.5) {\tiny Sector \V};
\node at (1.8,0) {\tiny \color{red} Sector \IV};
\end{tikzpicture}
\hspace{-0.4cm}
\begin{tikzpicture}[slave]
\node at (0,0) {\includegraphics[scale=0.343]{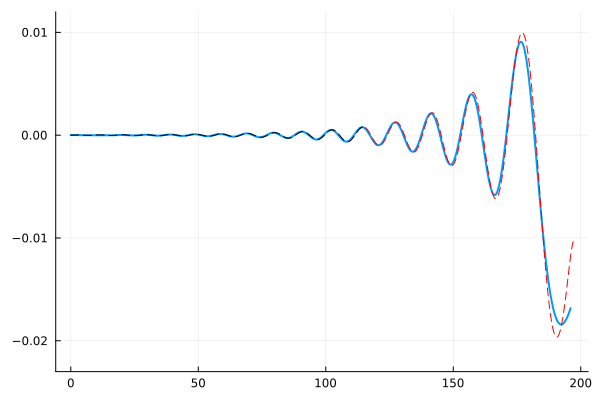}};
\node at (3.66,-2.08) {\tiny $x$};
\node at (-2.95,2.45) {\tiny $u$};
\node at (0.3,1.5) {\tiny $t=200$};
\node at (-1.2,.5) {\tiny Sector \V};
\node at (2.2,-0.25) {\tiny \color{red} Sector \IV};
\end{tikzpicture}
\end{center}
\begin{figuretext}
\label{fig:zoomasymp}
The same simulation as in Figure \ref{fig:asymp} but now zoomed in to show the solution $u(x,t)$ (blue) and the asymptotic approximations of Theorem \ref{asymptoticsth} in Sector $\IV$ (dashed red) and Sector $\V$ (dashed black) at times $t = 100$ and $t = 200$. 
\end{figuretext}
\end{figure}

\begin{theorem}[Asymptotics]\label{asymptoticsth}
Let $u_0,u_1 \in \mathcal{S}(\R)$ be such that the assumptions of Theorem \ref{globalth} are fulfilled. Then the global solution $u(x,t)$ of the initial value problem for (\ref{badboussinesq}) with initial data $u_0, u_1$ obeys the following asymptotic formulas as $(x,t) \to \infty$ in the half-plane $t \geq 0$ (see Figure \ref{sectors.pdf}):
\begin{align}\nonumber
\text{\upshape Sector I:} \ & u(x,t) = \frac{A(\zeta)}{\sqrt{t}} \cos \alpha(\zeta,t) 
+ O\bigg(\frac{1}{x^N} +\frac{C_N(\zeta) \ln x}{x}\bigg), && M \leq \zeta < \infty, 	
	\\\nonumber
\text{\upshape Sector II:} \ & u(x,t) = \frac{A(\zeta)}{\sqrt{t}} \cos \alpha(\zeta,t) + O\bigg(\frac{\ln t}{t}\bigg), && 1 + \frac{1}{M} \leq  \zeta \leq M, 
	\\\nonumber
\text{\upshape Sector III:} \ & u(x,t) = \frac{u_P\big((\frac{2}{3t})^{1/3}(x-t)\big)}{t^{2/3}} + O\bigg(\frac{1}{t^{5/6}}\bigg), && |\zeta -  1|\leq \frac{M}{t^{2/3}},
	\\\nonumber
\text{\upshape Sector IV:} \ & u(x,t) =  \frac{A_{1}(\zeta)}{\sqrt{t}}   \cos \alpha_{1}(\zeta,t) +\frac{A_{2}(\zeta)}{\sqrt{t}} \cos \alpha_{2}(\zeta,t) + O\bigg(\frac{\ln t}{t}\bigg),  &&
\frac{1}{\sqrt{3}}< \zeta < 1,
	\\\label{uasymptotics}
\text{\upshape Sector V:} \ & u(x,t) =
 \frac{\tilde{A}_{1}(\zeta)}{\sqrt{t}}   \cos \tilde{\alpha}_{1}(\zeta,t) +\frac{A_{2}(\zeta)}{\sqrt{t}} \cos \tilde{\alpha}_{2}(\zeta,t) + O\bigg(\frac{\ln t}{t}\bigg), && 0 < \zeta < \frac{1}{\sqrt{3}},
\end{align}
where 
\begin{itemize}
\item the formulas in Sectors I--III hold uniformly with respect to $\zeta = x/t$ in the stated intervals for any fixed $N \geq 1$ and $M > 1$;

\item the formulas in Sectors IV and V hold uniformly with respect to $\zeta = x/t$ in compact subsets of the stated intervals;

\item $C_N(\zeta)$ is rapidly decreasing as $\zeta\to \infty$ for each $N$;

\item in Sectors I and II, 
\begin{align}
& A(\zeta) := 2\sqrt{3} \frac{\sqrt{-\nu}\sqrt{-1-2\cos (2\arg k_{1})}}{-ik_{1}z_{\star}} \im k_1, 
	 \nonumber \\
& \alpha(\zeta,t) := \frac{3\pi}{4} + \arg r_{2}(k_{1}) + \arg \Gamma(i \nu) + \arg d_0 + t \, \im \Phi_{21}(\zeta,k_{1}), \nonumber
\end{align}
with $\Gamma(k)$ denoting the Gamma function, 
\begin{align*}
& \nu = \nu(k_1) := - \frac{1}{2\pi}\ln(1+r_{1}(k_{1})r_{2}(k_{1})) 
\leq 0, \\
& z_{\star} = z_{\star}(\zeta) := \sqrt{2}e^{\frac{\pi i}{4}} \sqrt{\frac{4-3k_{1} \zeta - k_{1}^{3} \zeta}{4k_{1}^{4}}}, \qquad -ik_{1}z_{\star}>0,
	\\
& \arg d_0 = \arg d_0(\zeta,t) :=  \nu  \ln \bigg| \frac{(\frac{1}{\omega^{2}k_{1}}-k_{1})(\frac{1}{\omega k_{1}}-k_{1})}{3(\frac{1}{k_{1}}-k_{1})^{2}z_{\star}^{2}} \bigg| - \nu\ln t \\
& \hspace{3.4cm} + \frac{1}{2\pi} \int_{i}^{k_{1}} \ln \bigg| \frac{(k_{1}-s)^{2}(\frac{1}{\omega^{2}k_{1}}-s)(\frac{1}{\omega k_{1}}-s)}{(\frac{1}{k_{1}}-s)^{2}(\omega k_{1}-s)(\omega^{2} k_{1}-s)} \bigg| d \ln(1+r_{1}(s)r_{2}(s)),
\end{align*}
and the path of the integral $\int_{i}^{k_{1}}$ starts at $i$, follows the unit circle in the counterclockwise direction, and ends at $k_{1}$;

\item in Sector III, $u_P:\R \to \R$ is a smooth function given in terms of the Hastings--McLeod solution $u_{\HM}$ of the Painlev\'e II equation by
$$u_P(y) := 2^{2/3} 3^{1/3}(u_{\HM}'(y) - u_{\HM}(y)^2);$$

\item in Sectors IV and V, 
\begin{align}\nonumber
& A_{1}(\zeta) := \frac{4\sqrt{3}\sqrt{\hat{\nu}_1(k_4)}\im k_4}{-i\omega k_{4} z_{1,\star}|\tilde{r}(\frac{1}{k_{4}})|^{\frac{1}{2}}}\sin(\arg(\omega k_{4})),  
	\\ \nonumber
& \alpha_{1}(\zeta,t) := \frac{3\pi}{4}+\arg q_{3}+\arg\Gamma(i\hat{\nu}_1(k_4))+\arg d_{1,0}-t\, \im \Phi_{31}(\zeta,\omega k_{4}), 
	\\ \nonumber
& A_{2}(\zeta) := \frac{-4\sqrt{3}\sqrt{\hat{\nu}_2(k_2)}|\tilde{r}(\frac{1}{k_{2}})|^{\frac{1}{2}}\im k_2}{-i\omega^{2} k_{2} z_{2,\star}}\sin(\arg(\omega^{2} k_{2})), 
	\\ \nonumber
& \alpha_{2}(\zeta,t) := \frac{3\pi}{4}-\arg (q_{6}-q_{2}q_{5})+\arg\Gamma(i\hat{\nu}_2(k_2))+\arg d_{2,0}-t \, \im \Phi_{32}(\zeta,\omega^{2} k_{2}), 
	\\ \nonumber 
& \tilde{A}_{1}(\zeta) := \frac{4\sqrt{3}\sqrt{\nu_{1}(\omega^2 k_4)}\im k_4}{-i\omega k_{4} z_{1,\star}|\tilde{r}(\frac{1}{k_{4}})|^{\frac{1}{2}}}\sin(\arg(\omega k_{4})),
	\\ \nonumber
& \tilde{\alpha}_1(\zeta, t) := \frac{3\pi}{4}-\arg \tilde{q}_{1}+\arg\Gamma(i\nu_{1}(\omega^2 k_4))+\arg \tilde{d}_{1,0}-t\, \im \Phi_{31}(\zeta,\omega k_{4}),
	\\ \nonumber
& \tilde{\alpha}_{2}(\zeta,t) := \frac{3\pi}{4}-\arg (q_{6}-q_{2}q_{5})+\arg\Gamma(i\hat{\nu}_2(k_2))+\arg \tilde{d}_{2,0}-t \, \im \Phi_{32}(\zeta,\omega^{2} k_{2}), 
\end{align}
where 
$\hat{\nu}_1(k)$, $\hat{\nu}_2(k)$, and $\nu_1(k)$ are the functions defined in (\ref{hatnu12def}) and (\ref{nu12345def}),
\begin{align}\nonumber
& 
\tilde{q}_1 := |\tilde{r}(k_{4})|^{\frac{1}{2}}r_{1}(k_{4}), \qquad
q_{2} := \tilde{r}(\omega^{2}k_{2})^{\frac{1}{2}}r_{1}(\omega^{2}k_{2}), \qquad q_{3} := |\tilde{r}(\tfrac{1}{k_{4}})|^{\frac{1}{2}}r_{1}(\tfrac{1}{k_{4}}), 
	\\\nonumber
& q_{5} := |\tilde{r}(\omega k_{2})|^{\frac{1}{2}}r_{1}(\omega k_{2}), 
\qquad
q_{6} := |\tilde{r}(\tfrac{1}{k_{2}})|^{\frac{1}{2}}r_{1}(\tfrac{1}{k_{2}}),
	\\\nonumber
& z_{1,\star} = z_{1,\star}(\zeta) := \sqrt{2}e^{\frac{\pi i}{4}} \sqrt{\omega \frac{4-3k_{4} \zeta - k_{4}^{3} \zeta}{4k_{4}^{4}}}, \qquad -i\omega k_{4}z_{1,\star}>0, 
	\\ \label{z1starz2stardef}
& z_{2,\star} = z_{2,\star}(\zeta) := \sqrt{2}e^{\frac{\pi i}{4}} \sqrt{-\omega^{2} \frac{4-3k_{2} \zeta - k_{2}^{3} \zeta}{4k_{2}^{4}}}, \qquad -i\omega^{2} k_{2}z_{2,\star}>0,
\end{align}
and $d_{1,0}$, $d_{2,0}$, $\tilde{d}_{1,0}$, and $\tilde{d}_{2,0}$ are defined in \eqref{def of d10}, \eqref{def of d20}, \eqref{def of dt10}, and \eqref{def of dt20}, respectively.

\end{itemize}

\end{theorem}

\subsection{Numerics}
Most of the formulas of this paper have been verified numerically in some way. 
Our numerical computations for the direct and inverse scattering have been based on the Julia package ``ApproxFun" developed by S. Olver,\footnote{See \texttt{https://github.com/JuliaApproximation/ApproxFun.jl} and \cite{OT2013}.} which allows for a fast and accurate evaluation of the spectral functions. For example, Figures \ref{fig:r1 and r2} and \ref{fig:f}, which correspond to the initial data
\begin{align}\label{numericscompactinitialdata}
\text{$\begin{cases} u_{0}(x) = -e^{-x^{2}}(1-x^{2})^{2}, \\ v_{0}(x) = 2(e^{-x^{2}}+5(x-0.2))(1-x^{2})^{2}, \end{cases}$ \hspace{-.3cm} for $x \in [-1,1]$; \quad $\begin{cases} u_{0}(x) = 0, \\ v_{0}(x) = 0, \end{cases}$ \hspace{-.3cm}  for $x \notin [-1,1]$}
\end{align}  
have been generated in this way. The asymptotic formulas of Theorem \ref{asymptoticsth} (see e.g. Figures \ref{fig:asymp} and \ref{fig:zoomasymp})
have also been evaluated with the help of this package.

\section{The direct problem}\label{directsec}

\subsection{Lax pair} 
The Boussinesq system (\ref{boussinesqsystem}) is the compatibility condition of the Lax pair equations
\begin{equation}\label{Lax pair tilde}
\begin{cases}
 \tilde{X}_{x} = \tilde{L}\tilde{X}, \\
 \tilde{X}_{t} = \tilde{Z}\tilde{X},
\end{cases}
\end{equation}
where $\tilde{X}(x,t,\lambda)$ is a $3 \times 3$-matrix valued eigenfunction, $\lambda \in \C$ is the spectral parameter, and
$$\tilde{L} := \begin{pmatrix}
0 & 1 & 0 \\
0 & 0 & 1 \\
\frac{\lambda}{12 i \sqrt{3}}-\frac{u_{x}}{4}-\frac{iv}{4\sqrt{3}} & -\frac{1+2u}{4} & 0
\end{pmatrix}, \quad
\tilde{Z} := \begin{pmatrix}
-i \frac{1+2u}{2\sqrt{3}} & 0 & -i \sqrt{3} \\
- \frac{\lambda}{12} - i \frac{u_{x}}{4\sqrt{3}} - \frac{v}{4} & i \frac{1+2u}{4\sqrt{3}} & 0 \\
-i \frac{u_{xx}}{4\sqrt{3}}-\frac{v_{x}}{4} & - \frac{\lambda}{12} + \frac{iu_{x}}{4\sqrt{3}}-\frac{v}{4} & i \frac{1+2u}{4\sqrt{3}}
\end{pmatrix}.$$
Note that the $x$-part of (\ref{Lax pair tilde}) is a reformulation of the spectral problem (\ref{isospectral}) as a first-order system with $\lambda = 12 \sqrt{3} z^{3}$, and that $\tilde{L}$ and $\tilde{Z}$ are traceless. We define a new spectral parameter $k \in \mathbb{C}$ by
\begin{align}\label{lambdadef}
\lambda = \frac{k^3 + k^{-3}}{2},
\end{align}
and let $\{l_j(k), z_j(k)\}_1^3$ be given by (\ref{lmexpressions intro}).
The transformation 
\begin{equation}\label{X hat to X tilde}
\tilde{X}(x,t,k) = P(k) \hat{X}(x,t,k),
\end{equation}
where
\begin{equation}
P(k) := \begin{pmatrix}
1 & 1 & 1  \\
l_{1}(k) & l_{2}(k) & l_{3}(k) \\
l_{1}^{2}(k) & l_{2}^{2}(k) & l_{3}^{2}(k)
\end{pmatrix}, 
\end{equation}
transforms \eqref{Lax pair tilde} into
\begin{equation}\label{Xhatlax}
\begin{cases}
\hat{X}_{x} = L\hat{X}, \\
\hat{X}_{t} = Z\hat{X},
\end{cases} \qquad \text{where} \qquad 
\begin{cases}
L := P^{-1}\tilde{L}P, \\
Z := P^{-1}\tilde{Z}P.
\end{cases}
\end{equation}
The above transformation is valid only for $k \in \mathbb{C}\setminus \hat{\mathcal{Q}}$, because
\begin{equation}
\det P(k) = \frac{1 - k^6}{8k^3}.
\end{equation}
As $k \to \infty$, we have
\begin{equation}\label{limits of L and Z at k infty}
L(x,t,k) = \mathcal{L} + O(1/k), \qquad Z(x,t,k) = \mathcal{Z} + O(1), 
\end{equation}
where the diagonal matrices $\mathcal{L} = \mathcal{L}(k)$ and $\mathcal{Z} = \mathcal{Z}(k)$ are given by
\begin{equation}\label{def of mathcal L and mathcal Z}
\mathcal{L} = \lim_{x \to \pm\infty} L = \mbox{diag}(l_{1},l_{2},l_{3}), \qquad \mathcal{Z} = \lim_{x \to \pm\infty} Z = \mbox{diag}(z_{1},z_{2},z_{3}).
\end{equation}
Defining $\mathsf{U} := L - \mathcal{L}$ and $\mathsf{V} := Z - \mathcal{Z}$, we can rewrite \eqref{Xhatlax} as
\begin{equation}\label{Lax pair hat}
\begin{cases}
 \hat{X}_{x} - \mathcal{L}\hat{X} = \mathsf{U}\hat{X}, \\
 \hat{X}_{t} - \mathcal{Z}\hat{X} = \mathsf{V}\hat{X}.
\end{cases}
\end{equation}
The transformation \eqref{X hat to X tilde} diagonalizes the highest-order terms in $k$ as $k \to \infty$ of the Lax pair \eqref{Lax pair tilde} and ensures that the lower-order terms decay as $x \to \pm \infty$. The transformation 
\begin{equation}\label{XhatX}
\hat{X} = X e^{\mathcal{L}x + \mathcal{Z}t}
\end{equation}
transforms \eqref{Lax pair hat} into
\begin{equation}\label{Xlax}
\begin{cases}
 X_{x} - [\mathcal{L},X] = \mathsf{U} X, \\
 X_{t} - [\mathcal{Z},X] = \mathsf{V} X.
\end{cases}
\end{equation}

\subsubsection{Symmetries of $\mathcal{L}$, $\mathcal{Z}$, $\mathsf{U}$, and $\mathsf{V}$} The matrices $\mathcal{L}$, $\mathcal{Z}$, $\mathsf{U}$, and $\mathsf{V}$ satisfy the $\mathbb{Z}_{3}$-symmetry
\begin{align}
& \mathcal{L}(k) = \mathcal{A}\mathcal{L}(\omega k) \mathcal{A}^{-1}, & & \mathcal{Z}(k) = \mathcal{A}\mathcal{Z}(\omega k) \mathcal{A}^{-1}, \nonumber \\
& \mathsf{U}(x,t,k) = \mathcal{A}\mathsf{U}(x,t,\omega k) \mathcal{A}^{-1}, & & \mathsf{V}(x,t,k) = \mathcal{A}\mathsf{V}(x,t,\omega k) \mathcal{A}^{-1}, \label{Acaldef}
\end{align}
and the $\mathbb{Z}_{2}$-symmetry
\begin{align}
& \mathcal{L}(k) = \mathcal{B}\mathcal{L}(k^{-1}) \mathcal{B}^{-1}, & & \mathcal{Z}(k) = \mathcal{B}\mathcal{Z}(k^{-1}) \mathcal{B}^{-1}, \nonumber \\
& \mathsf{U}(x,t,k) = \mathcal{B}\mathsf{U}(x,t,k^{-1}) \mathcal{B}^{-1}, & & \mathsf{V}(x,t,k) = \mathcal{B}\mathsf{V}(x,t,k^{-1}) \mathcal{B}^{-1}, \label{Bcaldef}
\end{align}
where $\mathcal{A}$ and $\mathcal{B}$ are the matrices in (\ref{def of Acal and Bcal}). Since $L=\mathcal{L}+\mathsf{U}$ and $Z = \mathcal{Z} + \mathsf{V}$, the matrices $L$ and $Z$ also satisfy these symmetries.

\subsubsection{Pole structure of $\mathsf{U}$ and $\mathsf{V}$} 
By \eqref{limits of L and Z at k infty}, we have $\mathsf{U}(x,t,k) = O(k^{-1})$ and $\mathsf{V}(x,t,k) = O(1)$ as $k \to \infty$.
The symmetry \eqref{Bcaldef} then implies that $\mathsf{U}(x,t,k) = O(k)$ and $\mathsf{V}(x,t,k) = O(1)$ as $k \to 0$. 
A computation shows that $\mathsf{U}(x,t,k)$ and $\mathsf{V}(x,t,k)$ have simple poles at $k=1$ and $k=-1$. The symmetry \eqref{Acaldef} then implies that $\mathsf{U}(x,t,k)$ and $\mathsf{V}(x,t,k)$ in fact have  simple poles at each of the six points $\kappa_{j}$, $j=1,\ldots,6$.

\begin{remark}\upshape
The symmetries (\ref{Acaldef})--\eqref{Bcaldef} for the Lax pair (\ref{Xlax}) are the same as the symmetries of the Lax pair for the ``good'' Boussinesq equation without $u_{xx}$-term studied in \cite{CLgoodboussinesq}, see \cite[Eqs. (3.19) and (3.20)]{CLgoodboussinesq}. However, the pole structure of (\ref{Xlax}) is very different from the pole structure of the Lax pair in \cite{CLgoodboussinesq}. Indeed, whereas the matrices $\mathsf{U}$ and $\mathsf{V}$ in (\ref{Xlax}) have six simple poles at $\kappa_{j}$, $j=1,\ldots,6$, the analogs of $\mathsf{U}$ and $\mathsf{V}$ in \cite{CLgoodboussinesq} have double poles at the origin.  In what follows, we will omit some proofs when they are similar to their counterparts in \cite{CLgoodboussinesq}. 
\end{remark}

\subsection{The eigenfunctions $X$ and $Y$} 
From now until the end of Section \ref{sec: n at t=0}, we fix $t=0$ and abuse notation by writing $\mathsf{U}(x,k)$ for $\mathsf{U}(x,0,k)$. Consider the $x$-part of the Lax pair \eqref{Xlax} evaluated at $t = 0$:
\begin{align}\label{xpart}
X_x - [\mathcal{L}, X] = \mathsf{U} X.
\end{align}
We define two $3 \times 3$-matrix valued solutions $X(x,k)$ and $Y(x,k)$ of (\ref{xpart}) as the solutions of the linear Volterra integral equations
\begin{subequations}\label{XYdef}
\begin{align}  \label{XYdefa}
 & X(x,k) = I - \int_x^{\infty} e^{(x-x')\widehat{\mathcal{L}(k)}} (\mathsf{U}X)(x',k) dx',
  	\\ \label{XYdefb}
&  Y(x,k) = I + \int_{-\infty}^x e^{(x-x')\widehat{\mathcal{L}(k)}} (\mathsf{U}Y)(x',k) dx'.
\end{align}
\end{subequations}
We decompose the complex $k$-plane into the six open subsets $\{D_n\}_1^6$ defined by (see Figure \ref{fig: Dn})
\begin{align*}
&D_1 = \{k \in \C\,|\, \re l_1 < \re l_2 < \re l_3\},
	\qquad
D_2 = \{k \in \C\,|\, \re l_1 < \re l_3 < \re l_2\},
	\\
&D_3 = \{k \in \C\,|\, \re l_3 < \re l_1 < \re l_2\},
	\qquad
D_4 = \{k \in \C\,|\, \re l_3 < \re l_2 < \re l_1\},
	\\
&D_5 = \{k \in \C\,|\, \re l_2 < \re l_3 < \re l_1\},
	\qquad
D_6 = \{k \in \C\,|\, \re l_2 < \re l_1 < \re l_3\},
\end{align*}
and let $\hat{\mathcal{S}} = \partial \D \cup \bar{D}_{3}\cup\bar{D}_{4}$. Let $\mathcal{S}$ denote the interior of $\bar{D}_3 \cup \bar{D}_4$ (see Figure \ref{fig: S}), and recall that $\mathcal{Q}=\{\kappa_{j}\}_{j=1}^{6}$ and $\hat{\mathcal{Q}} = \mathcal{Q} \cup \{0\}$.

\begin{figure}
\begin{center}
\begin{tikzpicture}[master, scale=0.6]
\node at (0,0) {};
\draw[black,line width=0.45 mm] (30:2.5)--(30:4);
\draw[black,line width=0.45 mm] (0,0)--(90:2.5);
\draw[black,line width=0.45 mm] (150:2.5)--(150:4);
\draw[black,line width=0.45 mm] (0,0)--(-30:2.5);
\draw[black,line width=0.45 mm] (-90:2.5)--(-90:4);
\draw[black,line width=0.45 mm] (0,0)--(-150:2.5);

\draw[black,line width=0.45 mm] ([shift=(-180:2.5cm)]0,0) arc (-180:180:2.5cm);

\node at (30:1.4) {$\omega \mathcal{S}$};
\node at (90:3.3) {$\mathcal{S}$};
\node at (150:1.4) {$\omega^{2}\mathcal{S}$};
\node at (210:3.6) {$\omega \mathcal{S}$};
\node at (270:1.4) {$\mathcal{S}$};
\node at (330:3.6) {$\omega^{2} \mathcal{S}$};

\draw[fill] (0:2.5) circle (0.1);
\draw[fill] (60:2.5) circle (0.1);
\draw[fill] (120:2.5) circle (0.1);
\draw[fill] (180:2.5) circle (0.1);
\draw[fill] (240:2.5) circle (0.1);
\draw[fill] (300:2.5) circle (0.1);

\end{tikzpicture} \hspace{2.5cm} \begin{tikzpicture}[slave, scale=0.6]
\node at (0,0) {};
\draw[black,line width=0.45 mm] (0,0)--(30:2.5);
\draw[black,line width=0.45 mm] (90:2.5)--(90:4);
\draw[black,line width=0.45 mm] (0,0)--(150:2.5);
\draw[black,line width=0.45 mm] (-30:2.5)--(-30:4);
\draw[black,line width=0.45 mm] (0,0)--(-90:2.5);
\draw[black,line width=0.45 mm] (-150:2.5)--(-150:4);

\draw[black,line width=0.45 mm] ([shift=(-180:2.5cm)]0,0) arc (-180:180:2.5cm);

\node at (30:3.6) {$-\omega \mathcal{S}$};
\node at (90:1.4) {$-\mathcal{S}$};
\node at (150:3.8) {$-\omega^{2}\mathcal{S}$};
\node at (210:1.4) {$-\omega \mathcal{S}$};
\node at (270:3.2) {$-\mathcal{S}$};
\node at (330:1.4) {$-\omega^{2} \mathcal{S}$};

\draw[fill] (0:2.5) circle (0.1);
\draw[fill] (60:2.5) circle (0.1);
\draw[fill] (120:2.5) circle (0.1);
\draw[fill] (180:2.5) circle (0.1);
\draw[fill] (240:2.5) circle (0.1);
\draw[fill] (300:2.5) circle (0.1);

\draw[dashed] (-5.7,-3.8)--(-5.7,3.8);
\end{tikzpicture}
\end{center}
\begin{figuretext}\label{fig: S}The domains of definition of the columns of $X$ and $Y$. The dots represent the singularities $\kappa_{j}$, $j=1,\ldots,6$.
\end{figuretext}
\end{figure}

The following proposition summarizes some basic properties of $X$ and $Y$. The statement is
similar to \cite[Proposition 3.1]{CLgoodboussinesq}; the  main difference is that in our case $X$ and $Y$ are singular at the points $\{\kappa_{j}\}_1^6$. The proof is omitted.

\begin{proposition}\label{XYprop}
Suppose $u_0, v_0 \in \mathcal{S}(\R)$. 
Then the equations (\ref{XYdef}) uniquely define two $3 \times 3$-matrix valued solutions $X$ and $Y$ of (\ref{xpart}) with the following properties:
\begin{enumerate}[$(a)$]
\item The function $X(x, k)$ is defined for $x \in \R$ and $k \in (\omega^2 \hat{\mathcal{S}}, \omega \hat{\mathcal{S}}, \hat{\mathcal{S}}) \setminus \hat{\mathcal{Q}}$. For each $k \in (\omega^2 \hat{\mathcal{S}}, \omega \hat{\mathcal{S}}, \hat{\mathcal{S}}) \setminus \hat{\mathcal{Q}}$, $X(\cdot, k)$ is smooth and satisfies (\ref{xpart}).

\item The function $Y(x, k)$ is defined for $x \in \R$ and $k \in (-\omega^2 \hat{\mathcal{S}}, -\omega \hat{\mathcal{S}}, -\hat{\mathcal{S}}) \setminus \hat{\mathcal{Q}}$. For each $k \in (-\omega^2 \hat{\mathcal{S}}, -\omega \hat{\mathcal{S}}, -\hat{\mathcal{S}}) \setminus \hat{\mathcal{Q}}$, $Y(\cdot, k)$ is smooth and satisfies (\ref{xpart}).

\item For each $x \in \R$, $X(x,\cdot)$ is continuous for $k \in (\omega^2 \hat{\mathcal{S}}, \omega \hat{\mathcal{S}}, \hat{\mathcal{S}}) \setminus \hat{\mathcal{Q}}$ and analytic for $k \in (\omega^2 \mathcal{S}, \omega \mathcal{S}, \mathcal{S})$, while $Y(x,\cdot)$ is continuous for $k \in (-\omega^2 \hat{\mathcal{S}}, -\omega \hat{\mathcal{S}}, -\hat{\mathcal{S}})\setminus \hat{\mathcal{Q}}$ and analytic for $k \in (-\omega^2 \mathcal{S}, -\omega \mathcal{S}, -\mathcal{S})$.

\item For each $x \in \R$ and each $j = 1, 2, \dots$, $\partial_k^j X(x, \cdot)$ is well defined for $(\omega^2 \hat{\mathcal{S}}, \omega \hat{\mathcal{S}}, \hat{\mathcal{S}})\setminus \hat{\mathcal{Q}}$, while $\partial_k^j Y(x, \cdot)$ is well defined for $(-\omega^2 \hat{\mathcal{S}}, -\omega \hat{\mathcal{S}}, -\hat{\mathcal{S}})\setminus \hat{\mathcal{Q}}$.

\item For each $n \geq 1$ and $\epsilon > 0$, there are bounded smooth positive functions $f_+(x)$ and $f_-(x)$ of $x \in \R$ with rapid decay as $x \to +\infty$ and $x \to -\infty$, respectively, such that the following estimates hold for $x \in \R$ and $ j = 0, 1, \dots, n$:
\begin{subequations}\label{Xest}
\begin{align}\label{Xesta}
& |\partial_k^j (X(x,k) - I) | \leq
f_+(x), \qquad k \in (\omega^2 \hat{\mathcal{S}}, \omega \hat{\mathcal{S}}, \hat{\mathcal{S}}), \ \dist(k,\hat{\mathcal{Q}}) > \epsilon,
	\\ \label{Xestb}
& |\partial_k^j (Y(x,k) - I) | \leq
f_-(x), \qquad k \in (-\omega^2 \hat{\mathcal{S}}, -\omega \hat{\mathcal{S}}, -\hat{\mathcal{S}}), \ \dist(k,\hat{\mathcal{Q}}) > \epsilon.
\end{align}
\end{subequations}

\item $X$ and $Y$ obey the following symmetries for each $x \in \R$:
\begin{subequations}\label{XYsymm}
\begin{align}
&  X(x, k) = \mathcal{A} X(x,\omega k)\mathcal{A}^{-1} = \mathcal{B} X(x,k^{-1})\mathcal{B}, \qquad k \in (\omega^2 \hat{\mathcal{S}}, \omega \hat{\mathcal{S}}, \hat{\mathcal{S}})\setminus \hat{\mathcal{Q}},
	\\
&  Y(x, k) = \mathcal{A} Y(x,\omega k)\mathcal{A}^{-1} = \mathcal{B} Y(x,k^{-1})\mathcal{B}, \qquad k \in (-\omega^2 \hat{\mathcal{S}}, -\omega \hat{\mathcal{S}}, -\hat{\mathcal{S}})\setminus \hat{\mathcal{Q}}.
\end{align}
\end{subequations}

\item If $u_0(x), v_0(x)$ have compact support, then, for each  $x \in \R$,  $X(x, k)$ and $Y(x, k)$ are defined and analytic for $k \in \C \setminus \hat{\mathcal{Q}}$ and $\det X = \det Y = 1$.
\end{enumerate}
\end{proposition}

\subsection{Asymptotics of $X$ and $Y$ as $k \to \infty$}
We first consider formal power series solutions of (\ref{xpart}). We then show that these series are good approximations of $X$ and $Y$ as $k \to \infty$. 

Consider the series
\begin{align*}
& X_{formal}(x,k) = I + \frac{X_1(x)}{k} + \frac{X_2(x)}{k^2} + \cdots,
\qquad
Y_{formal}(x,k) =  I + \frac{Y_1(x)}{k} + \frac{Y_2(x)}{k^2} + \cdots,
\end{align*}
normalized at $x = \infty$ and $x = -\infty$, respectively:
\begin{align}\label{Fjnormalization}
\lim_{x\to \infty} X_j(x) = \lim_{x\to -\infty} Y_j(x) = 0, \qquad j \geq 1.
\end{align}
For any integer $n \geq 1$, $\mathsf{U}$ has an expansion at $k=\infty$ of the form
$$\mathsf{U}(x,k) = \sum_{j=1}^{n} \frac{\mathsf{U}_{j}(x)}{k^{j}} + O(k^{-n-1}), \qquad k \to \infty,$$
where
$$ \mathsf{U}_{1}(x) = \frac{i \, u_{0}(x)}{\sqrt{3}} \begin{pmatrix}
\omega^{2} & 1 & \omega \\
1 & \omega & \omega^{2} \\
\omega & \omega^{2} & 1
\end{pmatrix}, \qquad \mathsf{U}_{2}(x) = \left( \frac{i \, v_{0}(x)}{\sqrt{3}} + u_{0}'(x) \right) \begin{pmatrix}
\omega & \omega & \omega \\
\omega^{2} & \omega^{2} & \omega^{2} \\
1 & 1 & 1
\end{pmatrix}.
$$
Moreover, $\mathcal{L}$ can be written as
\begin{align*}
& \mathcal{L} = \frac{\mathcal{L}_{-1}}{k^{-1}} + \frac{\mathcal{L}_{1}}{k}, \qquad \mathcal{L}_{-1} := \frac{i}{2\sqrt{3}}\begin{pmatrix}
\omega & 0 & 0 \\
0 & \omega^{2} & 0 \\
0 & 0 & 1
\end{pmatrix}, \qquad \mathcal{L}_{1} := \frac{i}{2\sqrt{3}}\begin{pmatrix}
\omega^{2} & 0 & 0 \\
0 & \omega & 0 \\
0 & 0 & 1
\end{pmatrix}.
\end{align*}
Hence, substituting $X_{formal}$ into (\ref{xpart}), the terms of order $k^{-j}$ yield the relations
\begin{align}\label{xrecursive}
\begin{cases}
[\mathcal{L}_{-1}, X_{j+1}] = \partial_x X_{j}^{(o)} -[\mathcal{L}_{1}, X_{j-1}] - \mathsf{U}_{j}^{(o)} - \sum_{i=1}^{j-1}(\mathsf{U}_{j-i}X_{i})^{(o)},
	\\
\partial_x X_{j}^{(d)} = \mathsf{U}_{j}^{(d)} + \sum_{i=1}^{j-1}(\mathsf{U}_{j-i}X_{i})^{(d)},$$
\end{cases}
\end{align}
where $X_{0}=I$, $A^{(d)}$ and $A^{(o)}$ denote the diagonal and off-diagonal parts of a $3 \times 3$ matrix $A$, respectively.
The coefficients $\{X_j(x)\}$ are uniquely determined from (\ref{Fjnormalization})--(\ref{xrecursive}), the initial values
$X_{-1} = 0, X_0 = I$, and the normalization conditions \eqref{Fjnormalization}. Similarly, the coefficients $\{Y_j(x)\}$ are found by solving \eqref{xrecursive} with $\{X_j\}$ replaced by $\{Y_j\}$, together with
$Y_{-1} = 0, Y_0 = I$ and \eqref{Fjnormalization}. The first coefficients $X_{1},X_{2},Y_{1},Y_{2}$ are given by
\begin{align}
X_1(x) = & \; \frac{i}{\sqrt{3}}  \int_{\infty}^{x} u_0(x^{\prime}) dx' \begin{pmatrix} \omega^2 & 0 & 0 \\ 
0 & \omega & 0 \\ 
0 & 0 & 1
\end{pmatrix}, \label{def of X1}
	\\ \label{def of X2}
X_2(x) = &\; \frac{-2u_{0}(x)}{1-\omega} \begin{pmatrix} 0 & \omega^{2} & -\omega \\ 
-\omega^{2} & 0 & 1 \\ 
\omega & -1 & 0 \end{pmatrix}
 + \int_{\infty}^x \left(\frac{iv_0}{\sqrt{3}} + u_{0x} + \frac{iu_0}{\sqrt{3}}  (X_1)_{33}\right) dx'  \begin{pmatrix} \omega & 0 & 0 \\ 0 & \omega^2 & 0 \\ 0 & 0 & 1 \end{pmatrix}, 
	\\ \nonumber
Y_1(x) = &\; \frac{i}{\sqrt{3}}  \int_{-\infty}^x u_0(x^{\prime}) dx' \begin{pmatrix} \omega^2 & 0 & 0 \\ 
0 & \omega & 0 \\ 
0 & 0 & 1
\end{pmatrix},
	\\ \nonumber
Y_2(x) = &\; \frac{-2 u_0(x)}{1-\omega}\begin{pmatrix} 0 & \omega^{2} & -\omega \\ 
-\omega^{2} & 0 & 1 \\ 
\omega & -1 & 0 \end{pmatrix}
+ \int_{-\infty}^x \left(\frac{iv_0}{\sqrt{3}} + u_{0x} + \frac{iu_0}{\sqrt{3}}  (Y_1)_{33}\right)dx'   \begin{pmatrix} \omega & 0 & 0 \\ 0 & \omega^2 & 0 \\ 0 & 0 & 1 \end{pmatrix}.
\end{align}

The next proposition shows that the matrices $X$ and $Y$ coincide with $X_{formal}$ and $Y_{formal}$ to all orders as $k \to \infty$. The proof is similar to the proof of \cite[Proposition 3.2]{CLgoodboussinesq} and is omitted.

\begin{proposition}\label{XYprop2}
Suppose $u_0, v_0 \in \mathcal{S}(\R)$. 
As $k \to \infty$, $X$ and $Y$ coincide to all orders with $X_{formal}$ and $Y_{formal}$, respectively. More precisely, let $p \geq 0$ be an integer. Then the functions
\begin{align}\label{Xpdef}
&X_{(p)}(x,k) := I + \frac{X_1(x)}{k} + \cdots + \frac{X_{p}(x)}{k^{p}},
	\\ \nonumber
&Y_{(p)}(x,k) := I + \frac{Y_1(x)}{k} + \cdots + \frac{Y_{p}(x)}{k^{p}},
\end{align}
are well-defined and, for each integer $j \geq 0$,
\begin{subequations}\label{Xasymptotics}
\begin{align}\label{Xasymptoticsa}
& \bigg|\frac{\partial^j}{\partial k^j}\big(X - X_{(p)}\big) \bigg| \leq
\frac{f_+(x)}{|k|^{p+1}}, \qquad x \in \R, \  k \in (\omega^2 \bar{\mathcal{S}}, \omega \bar{\mathcal{S}}, \bar{\mathcal{S}}), \ |k| \geq 2,
	\\ \label{Xasymptoticsb}
& \bigg|\frac{\partial^j}{\partial k^j}\big(Y - Y_{(p)}\big) \bigg| \leq
\frac{f_-(x)}{|k|^{p+1}}, \qquad x \in \R, \  k \in (-\omega^2 \bar{\mathcal{S}}, -\omega \bar{\mathcal{S}}, -\bar{\mathcal{S}}), \ |k| \geq 2,
\end{align}
\end{subequations}
where $f_+(x)$ and $f_-(x)$ are bounded smooth positive functions of $x \in \R$ with rapid decay as $x \to +\infty$ and $x \to -\infty$, respectively.
\end{proposition}

Using the symmetries $X(x, k) = \mathcal{B} X(x,k^{-1})\mathcal{B}$ and $Y(x, k) = \mathcal{B} Y(x,k^{-1})\mathcal{B}$, we can also deduce from Proposition \ref{XYprop2} the asymptotics for $X$ and $Y$ as $k \to 0$. In particular, for each $x \in \R$, $X$ and $Y$ have continuous extensions to $k=0$ given by $X(x,0)=I$ and $Y(x,0) = I$.

\subsection{Asymptotics of $X$ and $Y$ as $k \to \kappa_j$} 
In this subsection, we obtain the asymptotics of $X(x, k)$ and $Y(x, k)$ as $k \to \kappa_{j}$, $j=1,\ldots,6$. Thanks to the symmetries $X(x, k) = \mathcal{A} X(x,\omega k)\mathcal{A}^{-1}$ and $Y(x, k) = \mathcal{A} Y(x,\omega k)\mathcal{A}^{-1}$, it is sufficient to analyze the asymptotics as $k \to \pm 1$. 

For compactly supported data $u_0, v_0$, Proposition \ref{XYat1prop} below shows that $X$ and $Y$ have at most simple poles at $k = \pm 1$ with leading Laurent series coefficients of the form (\ref{Cjpm1p})--(\ref{Cjp0p}). If $u_0, v_0 \in \mathcal{S}(\R)$ are not compactly supported, then all columns of $X$ and $Y$ are, in general, not defined in neighborhoods of $k = \pm 1$, so we have to give a more careful statement. The proof is omitted; see \cite[Proposition 3.3]{CLgoodboussinesq} for the proof of a similar (although different) statement.

\begin{proposition}\label{XYat1prop}
Suppose $u_0, v_0 \in \mathcal{S}(\R)$ and let $p \geq 0$ be an integer. 
Then there are $3 \times 3$-matrix valued functions $C_i^{(l)}(x)$, $i = 1, 2,3,4$, $l = -1,0, \dots, p$, with the following properties:
\begin{itemize}
\item For $x \in \R$ and $k \in (\omega^2 \bar{\mathcal{S}}, \omega \bar{\mathcal{S}}, \bar{\mathcal{S}})$, the function $X$ satisfies
\begin{subequations}\label{XYat1}
\begin{align}
& \bigg|\frac{\partial^j}{\partial k^j}\big(X - I - \sum_{l=-1}^p C_1^{(l)}(x)(k-1)^l\big) \bigg| \leq
f_+(x)|k-1|^{p+1-j}, \qquad |k-1| \leq \frac{1}{2}, \label{Xat1} \\
& \bigg|\frac{\partial^j}{\partial k^j}\big(X - I - \sum_{l=-1}^p C_2^{(l)}(x)(k+1)^l\big) \bigg| \leq
f_+(x)|k+1|^{p+1-j}, \qquad |k+1| \leq \frac{1}{2}, \label{Xatm1}
\end{align}
while, for $x \in \R$ and $k \in (-\omega^2 \bar{\mathcal{S}}, -\omega \bar{\mathcal{S}}, -\bar{\mathcal{S}})$, the function $Y$ satisfies
\begin{align}
& \bigg|\frac{\partial^j}{\partial k^j}\big(Y - I - \sum_{l=-1}^p C_3^{(l)}(x) (k-1)^l\big) \bigg| \leq
f_-(x)|k-1|^{p+1-j}, \qquad |k-1| \leq \frac{1}{2}, \label{Yat1} \\
& \bigg|\frac{\partial^j}{\partial k^j}\big(Y - I - \sum_{l=-1}^p C_4^{(l)}(x) (k+1)^l\big) \bigg| \leq
f_-(x)|k+1|^{p+1-j}, \qquad |k+1| \leq \frac{1}{2}, \label{Yatm1}
\end{align}
\end{subequations}
where $f_+(x)$ and $f_-(x)$ are smooth positive functions of $x \in \R$ with rapid decay as $x \to +\infty$ and $x \to -\infty$, respectively, and $j \geq 0$ is any integer.

\item For each $l \geq -1$, $C_1^{(l)}(x),C_2^{(l)}(x)$ and $C_3^{(l)}(x),C_4^{(l)}(x)$ are smooth functions of $x \in \R$ which have rapid decay as $x \to +\infty$ and $x \to -\infty$, respectively.

\item The leading coefficients have the form
\begin{align}
  & C_{i}^{(-1)}(x)
=  \begin{pmatrix}
\alpha_{i}(x) & \alpha_{i}(x) & \beta_{i}(x) \\
-\alpha_{i}(x) & -\alpha_{i}(x) & -\beta_{i}(x) \\
0 & 0 & 0
\end{pmatrix}, \label{Cjpm1p} \\
& C_{i}^{(0)}(x) = - I + \begin{pmatrix}
\gamma_{i,3}(x) & \gamma_{i,4}(x) & \gamma_{i,5}(x) \\
\gamma_{i,4}(x)+(-1)^{i}\alpha_{i}(x) & \gamma_{i,3}(x)+(-1)^{i}\alpha_{i}(x) & \gamma_{i,5}(x)+(-1)^{i}\beta_{i}(x) \\
\gamma_{i,1}(x) & \gamma_{i,1}(x) & \gamma_{i,2}(x)
\end{pmatrix}, \label{Cjp0p}
\end{align}
where $\alpha_{i}(x)$, $\beta_{i}(x)$, $\gamma_{i,j}(x)$, $i = 1,2,3,4$, $j=1,\ldots,5$, are  complex-valued functions of $x \in \R$. Furthermore, the functions $\{\alpha_{i},\beta_{i},\gamma_{i,1},\gamma_{i,2}-1,\gamma_{i,3}-1,\gamma_{i,4},\gamma_{i,5}\}$ have rapid decay at $+\infty$ for $i=1,2$, and rapid decay at $-\infty$ for $i=3,4$.
\end{itemize}
\end{proposition}

\subsection{The spectral function $s(k)$}
The spectral function $s(k)$ was defined in \eqref{sdef intro}. The next proposition collects some properties of $s$. The proof is similar to that of \cite[Proposition 3.4]{CLgoodboussinesq}. 

\begin{proposition}\label{sprop}
Suppose $u_0,v_0 \in \mathcal{S}(\R)$. 
Then the spectral function $s(k)$ defined in \eqref{sdef intro} has the following properties:
\begin{enumerate}[$(a)$]
\item The entries of $s(k)$ are defined and continuous for $k$ in
\begin{align}\label{sdomainofdefinition}
 \begin{pmatrix}
 \omega^2 \hat{\mathcal{S}} & \hat{\Gamma}_{1} & \hat{\Gamma}_{3} \\
 \hat{\Gamma}_{1} & \omega \hat{\mathcal{S}} & \hat{\Gamma}_{5} \\
\hat{\Gamma}_{3} & \hat{\Gamma}_{5} & \hat{\mathcal{S}}
 \end{pmatrix}\setminus \hat{\mathcal{Q}},
\end{align}
that is, the $(11)$-entry of $s(k)$ is defined and continuous for $k \in \omega^2 \hat{\mathcal{S}}\setminus \hat{\mathcal{Q}}$, etc. 
 
\item The diagonal entries of $s(k)$  are analytic in the interior of their domains of definition as given in (\ref{sdomainofdefinition}). 
 
\item For $j = 1, 2, \dots$, the derivative $\partial_k^js(k)$ is well-defined and continuous for $k$ in (\ref{sdomainofdefinition}).

\item $s(k)$ obeys the symmetries
\begin{align}\label{symmetries of s}
&  s(k) = \mathcal{A} s(\omega k)\mathcal{A}^{-1} = \mathcal{B} s(k^{-1})\mathcal{B}.
\end{align}

\item $s(k)$ approaches the identity matrix as $k \to \infty$. More precisely, there are diagonal matrices $\{s_j\}_1^\infty$ such that, for any $N \geq 1$,
\begin{align*}\nonumber
\begin{cases}
\big|\partial_k^j \big(s(k) - I - \sum_{j=1}^N \frac{s_j}{k^j}\big)\big| = O(k^{-N-1}), & k \to \infty, \\
\big|\partial_k^j \big(s(k) - I - \sum_{j=1}^N \mathcal{B}s_j \mathcal{B}k^{j}\big)\big| = O(k^{N+1}), & k \to 0,
\end{cases}
\end{align*}
uniformly for $k$ as in \eqref{sdomainofdefinition} and $j = 0, 1, \dots, N$. In particular, the off-diagonal entries of $s(k)$ have rapid decay as $k \to \infty$ and as $k \to 0$.

\item For $k$ as in \eqref{sdomainofdefinition}, we have
\begin{align}
& s(k) = \frac{s_{1}^{(-1)}}{k-1} + s_{1}^{(0)} + s_{1}^{(1)}(k-1) + \cdots & & \mbox{as } k \to 1, \label{s at 1} \\
& s(k) = \frac{s_{-1}^{(-1)}}{k+1} + s_{-1}^{(0)} + s_{-1}^{(1)}(k+1) + \cdots & & \mbox{as } k \to -1, \label{s at -1}
\end{align}
where
\begin{align}
& s_{1}^{(-1)} = \begin{pmatrix}
\mathfrak{s}_{1} & \mathfrak{s}_{1} & \mathfrak{s}_{2} \\
-\mathfrak{s}_{1} & -\mathfrak{s}_{1} & -\mathfrak{s}_{2} \\
0 & 0 & 0
\end{pmatrix}, & & s_{-1}^{(-1)} = \begin{pmatrix}
\mathfrak{s}_{3} & \mathfrak{s}_{3} & \mathfrak{s}_{4} \\
-\mathfrak{s}_{3} & -\mathfrak{s}_{3} & -\mathfrak{s}_{4} \\
0 & 0 & 0
\end{pmatrix}, \label{s1pm1p and sm1pm1p} \\
& (s_{1}^{(0)})_{31} = (s_{1}^{(0)})_{32}, & & (s_{-1}^{(0)})_{31} = (s_{-1}^{(0)})_{32},
\end{align}
for certain constants $\mathfrak{s}_{1},\mathfrak{s}_{2},\mathfrak{s}_{3},\mathfrak{s}_{4} \in  \mathbb{C}$, and the expansions can be differentiated termwise any number of times.
\item If $u_0(x), v_0(x)$ have compact support, then $s(k)$ is defined and analytic for $k \in \C \setminus \hat{\mathcal{Q}}$, $\det s = 1$ for $k \in \C \setminus \hat{\mathcal{Q}}$, and
\begin{align}\label{XYs} 
X(x,k) = Y(x,k)e^{x\widehat{\mathcal{L}(k)}} s(k), \qquad k \in \C  \setminus \hat{\mathcal{Q}}.
\end{align}
\end{enumerate}
\end{proposition}


\subsection{The eigenfunctions $X^A$ and $Y^A$} 
Let $X^A(x,k)$ and $Y^A(x,k)$ be the solutions of the linear Volterra integral equations
\begin{subequations}\label{XAYAdef}
\begin{align}  
 & X^A(x,k) = I + \int_x^{\infty} e^{-(x-x')\widehat{\mathcal{L}(k)}} (\mathsf{U}^T X^A)(x',k) dx',	
  	\\ 
 & Y^A(x,k) = I - \int_{-\infty}^x e^{-(x-x')\widehat{\mathcal{L}(k)}} (\mathsf{U}^T Y^A)(x',k) dx'.
\end{align}
\end{subequations}
Recall that $s^{A}$ was defined in \eqref{sAdef intro}. For compactly supported data $u_{0},v_{0}$, all entries of $X$, $X^{A}$, $Y$, $Y^{A}$, $s$, $s^{A}$ are well defined for all $k \in \mathbb{C}\setminus \hat{\mathcal{Q}}$, and we have $X^{A}=(X^{-1})^T$, $Y^{A}=(Y^{-1})^T$, and $s^{A}=(s^{-1})^T$. For general $u_{0},v_{0} \in \mathcal{S}(\mathbb{R})$, the properties of $X^{A}$, $Y^{A}$ are established in the following proposition.

\begin{proposition}\label{XAYAprop}
Suppose $u_0, v_0 \in \mathcal{S}(\R)$. 
Then the equations (\ref{XAYAdef}) uniquely define two $3 \times 3$-matrix valued functions $X^A$ and $Y^A$ with the following properties:
\begin{enumerate}[$(a)$]
\item The function $X^A(x, k)$ is defined for $x \in \R$ and $k \in (-\omega^2 \hat{\mathcal{S}}, -\omega \hat{\mathcal{S}}, -\hat{\mathcal{S}}) \setminus \hat{\mathcal{Q}}$. For each $k \in (-\omega^2 \hat{\mathcal{S}}, -\omega \hat{\mathcal{S}}, -\hat{\mathcal{S}}) \setminus \hat{\mathcal{Q}}$, $X^A(\cdot, k)$ is smooth and satisfies 
\begin{align}\label{xpartA}
(X^A)_x + [\mathcal{L}, X^A] = -\mathsf{U}^TX^A.
\end{align}

\item The function $Y^A(x, k)$ is defined for $x \in \R$ and $k \in (\omega^2 \hat{\mathcal{S}}, \omega \hat{\mathcal{S}}, \hat{\mathcal{S}}) \setminus \hat{\mathcal{Q}}$. For each $k \in (\omega^2 \hat{\mathcal{S}}, \omega \hat{\mathcal{S}}, \hat{\mathcal{S}}) \setminus \hat{\mathcal{Q}}$, $Y^A(\cdot, k)$ is smooth and satisfies (\ref{xpartA}).

\item For each $x \in \R$, $X^A(x,\cdot)$ is continuous for $k \in (-\omega^2 \hat{\mathcal{S}}, -\omega \hat{\mathcal{S}}, -\hat{\mathcal{S}})\setminus \mathcal{Q}$ and analytic for $k \in (-\omega^2 \mathcal{S}, -\omega \mathcal{S}, -\mathcal{S}) \setminus \hat{\mathcal{Q}}$, while $Y^A(x,\cdot)$ is continuous for $k \in (\omega^2 \hat{\mathcal{S}}, \omega \hat{\mathcal{S}}, \hat{\mathcal{S}})\setminus \mathcal{Q}$ and analytic for $k \in (\omega^2 \mathcal{S}, \omega \mathcal{S}, \mathcal{S}) \setminus \hat{\mathcal{Q}}$.

\item For each $x \in \R$ and each $j = 1, 2, \dots$, $\partial_k^j X^A(x, \cdot)$ is well defined for $(-\omega^2 \hat{\mathcal{S}}, -\omega \hat{\mathcal{S}}, -\hat{\mathcal{S}})\setminus \hat{\mathcal{Q}}$, while $\partial_k^j Y^A(x, \cdot)$ is well defined for $(\omega^2 \hat{\mathcal{S}}, \omega \hat{\mathcal{S}}, \hat{\mathcal{S}})\setminus \hat{\mathcal{Q}}$.

\item For each $n \geq 1$ and $\epsilon > 0$, there are bounded smooth positive functions $f_+(x)$ and $f_-(x)$ of $x \in \R$ with rapid decay as $x \to +\infty$ and $x \to -\infty$, respectively, such that the following estimates hold for $x \in \R$ and $j = 0, 1, \dots, n$:
\begin{align*}
& |\partial_k^j(X^A(x,k) - I) | \leq
f_+(x), \qquad k \in (-\omega^2 \hat{\mathcal{S}}, -\omega \hat{\mathcal{S}}, 
-\hat{\mathcal{S}}), \ \dist(k,\hat{\mathcal{Q}}) > \epsilon,
	\\ 
& |\partial_k^j(Y^A(x,k) - I) | \leq
f_-(x), \qquad k \in (\omega^2 \hat{\mathcal{S}}, \omega \hat{\mathcal{S}}, \hat{\mathcal{S}}), \ \dist(k,\hat{\mathcal{Q}}) > \epsilon.
\end{align*}

\item $X^A$ and $Y^A$ obey the following symmetries for each $x \in \R$:
\begin{align*}
&  X^A(x, k) = \mathcal{A} X^A(x,\omega k)\mathcal{A}^{-1} = \mathcal{B} X^A(x,k^{-1})\mathcal{B}, \qquad k \in (-\omega^2 \hat{\mathcal{S}}, -\omega \hat{\mathcal{S}}, -\hat{\mathcal{S}})\setminus \hat{\mathcal{Q}},
	\\\nonumber
&  Y^A(x, k) = \mathcal{A} Y^A(x,\omega k)\mathcal{A}^{-1} = \mathcal{B} Y^A(x,k^{-1})\mathcal{B}, \qquad k \in (\omega^2 \hat{\mathcal{S}}, \omega \hat{\mathcal{S}}, \hat{\mathcal{S}})\setminus \hat{\mathcal{Q}}.
\end{align*}
\item If $u_0(x), v_0(x)$ have compact support, then, for each  $x \in \R$,  $X^A(x, k)$ and $Y^A(x, k)$ are defined and analytic for $k \in \C \setminus \hat{\mathcal{Q}}$ and $\det X^A = \det Y^A = 1$.
\item $X$, $X^{A}$, $Y$ and $Y^{A}$ satisfy the following symmetries:
\begin{align}
& \overline{(X^{A})(x,\bar{k})} = \bigg\{ \frac{u_{0}(x)}{2}\begin{pmatrix}
1 & 1 & 1 \\
1 & 1 & 1 \\
1 & 1 & 1 
\end{pmatrix} + R(k)^{-1} \bigg\}X(x,k)R(k), \qquad k \in (\omega^2 \hat{\mathcal{S}}, \omega \hat{\mathcal{S}}, \hat{\mathcal{S}})\setminus \hat{\mathcal{Q}}, \label{XA X symmetry relation} 
	\\
& \overline{(Y^{A})(x,\bar{k})} = \bigg\{ \frac{u_{0}(x)}{2}\begin{pmatrix}
1 & 1 & 1 \\
1 & 1 & 1 \\
1 & 1 & 1 
\end{pmatrix} + R(k)^{-1} \bigg\}Y(x,k)R(k), \qquad k \in (-\omega^2 \hat{\mathcal{S}}, -\omega \hat{\mathcal{S}}, -\hat{\mathcal{S}})\setminus \hat{\mathcal{Q}},
\end{align}
where $R$ is given by
\begin{align}\label{def of R}
R(k) := -4k^{2} \begin{pmatrix}
0 & \frac{\omega}{(k^{2}-1)(k^{2}-\omega^{2})} & 0 \\
\frac{\omega^{2}}{(k^{2}-1)(k^{2}-\omega)} & 0 & 0 \\
0 & 0 & \frac{1}{(k^{2}-\omega)(k^{2}-\omega^{2})}
\end{pmatrix}.
\end{align}
\end{enumerate}
\end{proposition}
\begin{proof}
Assertions $(a)$--$(g)$ follow in the same way as the analogous assertions of Proposition \ref{XYprop}. 
To prove $(h)$, we define, for $k \in (\omega^2 \hat{\mathcal{S}}, \omega \hat{\mathcal{S}}, \hat{\mathcal{S}})\setminus \hat{\mathcal{Q}}$, 
\begin{align*}
\tilde{X}(x,k) := P(k)X(x,k)e^{\mathcal{L}(k)x}, \qquad \overline{\tilde{X}^{A}(x,\bar{k})} := \overline{P^{A}(\bar{k})} \, \overline{X^{A}(x,\bar{k})}e^{-\overline{\mathcal{L}(\bar{k})}x}, \qquad P^A(k) := (P(k)^{-1})^T.
\end{align*}
The functions $\tilde{X}$ and $\tilde{X}^{A}$ satisfy the equations
\begin{align*}
\begin{cases}
\tilde{X}_{x}(x,k) = \tilde{L}(x,k)\tilde{X}(x,k), & \lim_{x \to + \infty}\tilde{X}(x,k) \big( P(k)e^{\mathcal{L}(k)x} \big)^{-1} = I, \\
\overline{\tilde{X}_{x}^{A}(x,\bar{k})} = -\overline{\tilde{L}^{T}(x,\bar{k})}\,\overline{\tilde{X}^{A}(x,\bar{k})}, & \lim_{x \to + \infty} \overline{\tilde{X}^{A}(x,\bar{k})} \big( \overline{P^{A}(\bar{k})}e^{-\overline{\mathcal{L}(\bar{k})}x} \big)^{-1} = I.
\end{cases}
\end{align*}
We also define
\begin{align*}
\mathfrak{X}(x,k) := \overline{\begin{pmatrix}
\tilde{X}^{A}_{31} & \tilde{X}^{A}_{32} & \tilde{X}^{A}_{33} \\
(\tilde{X}^{A}_{31})' & (\tilde{X}^{A}_{32})' & (\tilde{X}^{A}_{33})' \\
(\tilde{X}^{A}_{31})'' & (\tilde{X}^{A}_{32})'' & \tilde{X}^{A}_{33})''
\end{pmatrix}(x,\bar{k})} = \begin{pmatrix}
0 & 0 & 1 \\
0 & -1 & 0 \\
1 & 0 & \overline{\tilde{L}(x,\bar{k})}_{32}
\end{pmatrix}\overline{\tilde{X}^{A}(x,\bar{k})}.
\end{align*}
Using the relations
\begin{align*}
\overline{\tilde{L}(x,\bar{k})}_{32} = \tilde{L}(x,k)_{32}, \qquad \partial_{x}\overline{\tilde{L}(x,\bar{k})}_{32} = \tilde{L}(x,k)_{31}+\overline{\tilde{L}(x,\bar{k})}_{31},
\end{align*}
a direct computation shows that $\mathfrak{X}_{x}(x,k) = \tilde{L}(x,k)\mathfrak{X}(x,k)$.
Therefore, there exists a matrix $R(k)$, independent of $x$, such that $\mathfrak{X}(x,k) = \tilde{X}(x,k)R(k)$. Letting $x \to + \infty$, we obtain
\begin{align*}
R(k) & = \lim_{x \to + \infty}\tilde{X}(x,k)^{-1}\mathfrak{X}(x,k) 
 = \lim_{x \to + \infty} \big(P(k)e^{\mathcal{L}(k)x}\big)^{-1}\begin{pmatrix}
0 & 0 & 1 \\
0 & -1 & 0 \\
1 & 0 & -\frac{1}{4}
\end{pmatrix}\overline{P^{A}(\bar{k})}e^{-\overline{\mathcal{L}(\bar{k})}x},
\end{align*}
and we find \eqref{def of R}. We have proved that
\begin{align*}
\overline{\tilde{X}^{A}(x,\bar{k})} = \begin{pmatrix}
0 & 0 & 1 \\
0 & -1 & 0 \\
1 & 0 & -\frac{1+2u_{0}(x)}{4}
\end{pmatrix}^{-1}\tilde{X}(x,k)R(k),
\end{align*}
or equivalently that
\begin{align*}
\overline{P^{A}(\bar{k})} \, \overline{X^{A}(x,\bar{k})}e^{-\overline{\mathcal{L}(\bar{k})}x} = \begin{pmatrix}
0 & 0 & 1 \\
0 & -1 & 0 \\
1 & 0 & -\frac{1+2u_{0}(x)}{4}
\end{pmatrix}^{-1}P(k)X(x,k)e^{\mathcal{L}(k)x}R(k).
\end{align*}
Finally, straightforward calculations show that $e^{\mathcal{L}(k)x}R(k)e^{\overline{\mathcal{L}(\bar{k})}x} = R(k)$ and 
\begin{align*}
\overline{P^{A}(\bar{k})}^{-1}\begin{pmatrix}
0 & 0 & 1 \\
0 & -1 & 0 \\
1 & 0 & -\frac{1+2u_{0}(x)}{4}
\end{pmatrix}^{-1}P(k) = \frac{u_{0}(x)}{2}\begin{pmatrix}
1 & 1 & 1 \\
1 & 1 & 1 \\
1 & 1 & 1 
\end{pmatrix} + R(k)^{-1}
\end{align*}
which finishes the proof.
\end{proof}

The following two propositions establish the behavior of $X^A$ and $Y^A$ as $k \to \infty$ (Proposition \ref{XAYAprop2}) and as $k \to \pm 1$ (Proposition \ref{XAYAat1prop}). The proofs are analogous to the proofs of Propositions \ref{XYprop2} and \ref{XYat1prop}.

\begin{proposition}\label{XAYAprop2}
Suppose $u_0, v_0 \in \mathcal{S}(\R)$. 
As $k \to \infty$, $X^A$ and $Y^A$ coincide to all orders with $X^A_{formal}$ and $Y^A_{formal}$, respectively. More precisely, let $p \geq 1$ be an integer and let $X^A_{(p)}(x,k)$ and $Y^A_{(p)}(x,k)$ be the cofactor matrices of the functions in (\ref{Xpdef}). Then, for each integer $j \geq 0$,
\begin{subequations}\label{XAasymptotics}
\begin{align}\label{XAasymptoticsa}
& \bigg|\frac{\partial^j}{\partial k^j}\big(X^A - X^A_{(p)}\big) \bigg| \leq
\frac{f_+(x)}{|k|^{p+1}}, \qquad x \in \R, \  k \in (-\omega^2 \bar{\mathcal{S}}, -\omega \bar{\mathcal{S}}, -\bar{\mathcal{S}}), \ |k| \geq 2,
	\\ \label{XAasymptoticsb}
& \bigg|\frac{\partial^j}{\partial k^j}\big(Y^A - Y^A_{(p)}\big) \bigg| \leq
\frac{f_-(x)}{|k|^{p+1}}, \qquad x \in \R, \  k \in (\omega^2 \bar{\mathcal{S}}, \omega \bar{\mathcal{S}}, \bar{\mathcal{S}}), \ |k| \geq 2,
\end{align}
\end{subequations}
where $f_+(x)$ and $f_-(x)$ are bounded smooth positive functions of $x \in \R$ with rapid decay as $x \to +\infty$ and $x \to -\infty$, respectively.
\end{proposition}

\begin{proposition}\label{XAYAat1prop}
Suppose $u_0, v_0 \in \mathcal{S}(\R)$ and let $p \geq 0$ be an integer. 
Then there are $3 \times 3$-matrix valued functions $D_i^{(l)}(x)$, $i = 1,2,3,4$, $l = -1,0, \dots, p$, with the following properties:
\begin{itemize}
\item For $x \in \R$ and $k \in (-\omega^2 \bar{\mathcal{S}}, -\omega \bar{\mathcal{S}}, -\bar{\mathcal{S}})$, the function $X^A$ satisfies
\begin{subequations}\label{XAYAat1}
\begin{align}\label{XAYAat1a}
& \bigg|\frac{\partial^j}{\partial k^j}\big(X^A - I - \sum_{l=-1}^p D_1^{(l)}(x)(k-1)^l\big) \bigg| \leq
f_+(x)|k-1|^{p+1-j}, \qquad |k-1| \leq \frac{1}{2}, \\
& \bigg|\frac{\partial^j}{\partial k^j}\big(X^A - I - \sum_{l=-1}^p D_2^{(l)}(x)(k+1)^l\big) \bigg| \leq
f_+(x)|k+1|^{p+1-j}, \qquad |k+1| \leq \frac{1}{2},
\end{align}
while, for $x \in \R$ and $k \in (\omega^2 \bar{\mathcal{S}}, \omega \bar{\mathcal{S}}, \bar{\mathcal{S}})$, the function $Y^{A}$ satisfies
\begin{align}
& \bigg|\frac{\partial^j}{\partial k^j}\big(Y^A - I - \sum_{l=-1}^p D_3^{(l)}(x)(k-1)^l\big) \bigg| \leq
f_-(x)|k-1|^{p+1-j}, \qquad |k-1| \leq \frac{1}{2}, \\
& \bigg|\frac{\partial^j}{\partial k^j}\big(Y^A - I - \sum_{l=-1}^p D_4^{(l)}(x)(k+1)^l\big) \bigg| \leq
f_-(x)|k+1|^{p+1-j}, \qquad |k+1| \leq \frac{1}{2},
\end{align}
\end{subequations}
where $f_+(x)$ and $f_-(x)$ are smooth positive functions of $x \in \R$ with rapid decay as $x \to +\infty$ and $x \to -\infty$, respectively, and $j \geq 0$ is any integer.

\item For each $l\geq -1$, $D_1^{(l)}(x),D_2^{(l)}(x)$ and $D_3^{(l)}(x),D_4^{(l)}(x)$ are smooth functions of $x \in \R$ which have rapid decay as $x \to +\infty$ and $x \to -\infty$, respectively.

\item The leading coefficients have the form
\begin{align}
D_{i}^{(-1)}(x)
= & \;
\begin{pmatrix}
\tilde{\alpha}_{i}(x) & -\tilde{\alpha}_{i}(x) & 0 \\
\tilde{\alpha}_{i}(x) & -\tilde{\alpha}_{i}(x) & 0 \\
\tilde{\beta}_{i}(x) & -\tilde{\beta}_{i}(x) & 0
\end{pmatrix}, \label{Djpm1p} \\
D_{i}^{(0)}(x) = & \; -I+ \begin{pmatrix}
\tilde{\gamma}_{i,3}(x) & \tilde{\gamma}_{i,4}(x)+(-1)^{i}\tilde{\alpha}_{i}(x) & \tilde{\gamma}_{i,1}(x) \\
\tilde{\gamma}_{i,4}(x) & \tilde{\gamma}_{i,3}(x)+(-1)^{i}\tilde{\alpha}_{i}(x) & \tilde{\gamma}_{i,1}(x) \\
\tilde{\gamma}_{i,5}(x) & \tilde{\gamma}_{i,5}(x)+(-1)^{i}\tilde{\beta}_{i}(x) & \tilde{\gamma}_{i,2}(x)
\end{pmatrix}  \label{Djp0p}
\end{align}
where $\tilde{\alpha}_i(x)$, $\tilde{\beta}_i(x)$, $\tilde{\gamma}_{i,j}(x)$, $i = 1,2,3,4$, $j=1,2,3,4,5$ are complex-valued functions of $x \in \R$. Furthermore, the functions $\{\tilde{\alpha}_{i},\tilde{\beta}_{i},\tilde{\gamma}_{i,1},\tilde{\gamma}_{i,2}-1,\tilde{\gamma}_{i,3}-1,\tilde{\gamma}_{i,4},\tilde{\gamma}_{i,5}\}$ have rapid decay at $+\infty$ for $i=1,2$, and rapid decay at $-\infty$ for $i=3,4$.
\end{itemize}
\end{proposition}

\subsection{The spectral function $s^A(k)$}

\begin{proposition}\label{sAprop}
Suppose $u_0,v_0 \in \mathcal{S}(\R)$. 
Then the spectral function $s^A(k)$ defined in (\ref{sAdef intro}) has the following properties:
\begin{enumerate}[$(a)$]
\item $s^A(k)$ is defined and continuous for $k$ in
\begin{align}\label{sAdomainofdefinition}
 \begin{pmatrix}
 -\omega^2 \hat{\mathcal{S}} & \hat{\Gamma}_{4} & \hat{\Gamma}_{6} \\
\hat{\Gamma}_{4} & -\omega \hat{\mathcal{S}} & \hat{\Gamma}_{2} \\
\hat{\Gamma}_{6} & \hat{\Gamma}_{2} & -\hat{\mathcal{S}}
 \end{pmatrix}\setminus \hat{\mathcal{Q}}.
\end{align}
 
\item The diagonal entries of $s^A(k)$  are analytic in the interior of their domains of definition as given in (\ref{sAdomainofdefinition}). 
 
\item For $j = 1, 2, \dots$, the derivative $\partial_k^js^A(k)$ is well-defined and continuous for $k$ in (\ref{sAdomainofdefinition}).

\item $s^A(k)$ obeys the symmetries
\begin{align}\label{symmetries of sA}
&  s^A(k) = \mathcal{A} s^A(\omega k)\mathcal{A}^{-1} = \mathcal{B} s^A(k^{-1})\mathcal{B}.
\end{align}

\item $s^{A}(k)$ approaches the identity matrix as $k \to \infty$. More precisely, there are diagonal matrices $\{s^{A}_j\}_1^\infty$ such that, for any $N \geq 1$,
\begin{align*}\nonumber
\begin{cases}
\big|\partial_k^j \big(s^{A}(k) - I - \sum_{j=1}^N \frac{s^{A}_j}{k^j}\big)\big| = O(k^{-N-1}), & k \to \infty,
\\
\big|\partial_k^j \big(s^{A}(k) - I - \sum_{j=1}^N \mathcal{B}s^{A}_j\mathcal{B}k^{j}\big)\big| = O(k^{N+1}), & k \to 0,
\end{cases}
\end{align*}
uniformly for $k$ as in \eqref{sAdomainofdefinition} and $j = 0, 1, \dots, N$. In particular, the off-diagonal entries of $s^{A}(k)$ have rapid decay as $k \to \infty$ and as $k \to 0$.

\item For $k$ as in \eqref{sAdomainofdefinition}, we have
\begin{align}
& s^{A}(k) = \frac{s_{1}^{A(-1)}}{k-1} + s_{1}^{A(0)} + s_{1}^{A(1)}(k-1) + \dots & & \mbox{as } k \to 1, \label{sA at 1} \\
& s^{A}(k) = \frac{s_{-1}^{A(-1)}}{k+1} + s_{-1}^{A(0)} + s_{-1}^{A(1)}(k+1) + \dots & & \mbox{as } k \to -1, \label{sA at -1}
\end{align}
where
\begin{align}
& s_{1}^{A(-1)} = \begin{pmatrix}
\mathfrak{s}_{1}^{A} & -\mathfrak{s}_{1}^{A} & 0 \\
\mathfrak{s}_{1}^{A} & -\mathfrak{s}_{1}^{A} & 0 \\
\mathfrak{s}_{2}^{A} & -\mathfrak{s}_{2}^{A} & 0
\end{pmatrix}, & & s_{-1}^{A(-1)} = \begin{pmatrix}
\mathfrak{s}_{3}^{A} & -\mathfrak{s}_{3}^{A} & 0 \\
\mathfrak{s}_{3}^{A} & -\mathfrak{s}_{3}^{A} & 0 \\
\mathfrak{s}_{4}^{A} & -\mathfrak{s}_{4}^{A} & 0
\end{pmatrix}, \label{sA1pm1p and sAm1pm1p} \\
& (s_{1}^{A(0)})_{13} = (s_{1}^{A(0)})_{23}, & & (s_{-1}^{A(0)})_{13} = (s_{-1}^{A(0)})_{23},
\end{align}
for certain constants $\mathfrak{s}_{1}^{A},\mathfrak{s}_{2}^{A},\mathfrak{s}_{3}^{A},\mathfrak{s}_{4}^{A} \in  \mathbb{C}$, and the expansions can be differentiated termwise any number of times.
\item If $u_0(x), v_0(x)$ have compact support, then $s^A(k)$ is defined and equals the cofactor matrix of $s(k)$ for all $k \in \C \setminus \hat{\mathcal{Q}}$.
\item For $k$ as in \eqref{sdomainofdefinition}, we have
\begin{align}\label{sRsymm}
\overline{s^{A}(\bar{k})} = R(k)^{-1}s(k)R(k).
\end{align}
\end{enumerate}
\end{proposition}
\begin{proof}
The proofs of $(a)$--$(g)$ are similar to the proofs of the analogous assertions of Proposition \ref{sprop}. To prove $(h)$, we may assume that $u_{0},v_{0}$ are compactly supported (cf. \cite[Lemma 4.5]{CLgoodboussinesq}). Then, by \eqref{XYs}, we have $s(k) = e^{-x \mathcal{L}(k)}Y(x,k)^{-1}X(x,k)e^{x\mathcal{L}(k)}$. Using also \eqref{XA X symmetry relation}, we obtain
\begin{align*}
\overline{s^{A}(\bar{k})} & = \overline{s(\bar{k})^{-1}}^{T} = e^{x \overline{\mathcal{L}(\bar{k})}}\overline{Y(x,\bar{k})}^{T}\overline{X^{A}(x,\bar{k})}e^{-x\overline{\mathcal{L}(\bar{k})}} \\
& = e^{x \overline{\mathcal{L}(\bar{k})}}\overline{Y(x,\bar{k})}^{T}\bigg\{ \frac{u_{0}(x)}{2}\begin{pmatrix}
1 & 1 & 1 \\
1 & 1 & 1 \\
1 & 1 & 1 
\end{pmatrix} + R(k)^{-1} \bigg\}X(x,k)R(k)e^{-x\overline{\mathcal{L}(\bar{k})}}.
\end{align*}
Since $u_{0},v_{0}$ are compactly supported, we have $u_{0}(x)=0$ and $Y(x,k)=I$ for all $x<-M$, where $M$ is a large positive constant. Hence,
\begin{align*}
\overline{s^{A}(\bar{k})} = e^{x \overline{\mathcal{L}(\bar{k})}} R(k)^{-1} X(x,k)R(k)e^{-x\overline{\mathcal{L}(\bar{k})}} \quad \mbox{for any } x<-M.
\end{align*}
Since $e^{x\mathcal{L}(k)}R(k) = R(k)e^{-\overline{x\mathcal{L}(\bar{k})}}$, we find 
$$\overline{s^{A}(\bar{k})} = R(k)^{-1} e^{-x \mathcal{L}(k)}X(x,k)e^{x\mathcal{L}(k)}R(k) = R(k)^{-1}s(k)R(k),$$ 
which is (\ref{sRsymm}).
\end{proof}

\subsection{Proof of Theorem \ref{directth}}
Recall from (\ref{r1r2def}) that $r_1 = s_{12}/s_{11}$ and $r_2 = s_{12}^A/s_{11}^A$. Assertions $(a)$ and $(c)$ of Propositions \ref{sprop} and \ref{sAprop} imply that $s_{12}(k)$ and $s_{11}(k)$ are smooth on $\hat{\Gamma}_{1}\setminus \hat{\mathcal{Q}}$, while $s_{12}^A(k)$ and $s_{11}^A(k)$ are smooth on $\hat{\Gamma}_{4}\setminus \hat{\mathcal{Q}}$. By Assumption \ref{solitonlessassumption}, $s_{11}$ is nonzero for $k \in \omega^{2}\hat{\mathcal{S}} \setminus \hat{\mathcal{Q}}$ and
$s_{11}^A(k)$ is nonzero for $k \in -\omega^{2}\hat{\mathcal{S}} \setminus \hat{\mathcal{Q}}$. It follows that $r_1 \in C^\infty(\hat{\Gamma}_{1}\setminus \hat{\mathcal{Q}})$ and $r_2 \in C^\infty(\hat{\Gamma}_{4}\setminus \hat{\mathcal{Q}})$. Assumption \ref{originassumption} implies that the coefficients $\mathfrak{s}_{1}$, $\mathfrak{s}_{3}$ and $\mathfrak{s}_{1}^{A}$, $\mathfrak{s}_{3}^{A}$ in (\ref{s1pm1p and sm1pm1p}) and (\ref{sA1pm1p and sAm1pm1p}) are all nonzero, ensuring that all four functions $s_{12}$, $s_{11}$, $s_{12}^A$, $s_{11}^A$ are of order $(k-1)^{-1}$ as $k \to 1$ and of order $(k+1)^{-1}$ as $k \to -1$. Assumption \ref{originassumption} together with (\ref{s1pm1p and sm1pm1p}) and (\ref{sA1pm1p and sAm1pm1p}) imply that all entries in the third row (resp. column) of $s_{\pm 1}^{(0)}$ (resp. $s_{\pm1}^{A(0)}$) are nonzero. For $k \in \partial \mathbb{D}\setminus \hat{\mathcal{Q}}$, all entries of $s$ and $s^{A}$ are well-defined and thus $s^{A}=(s^{-1})^{T}$. Using these observations, the symmetries \eqref{symmetries of s} and \eqref{symmetries of sA} together with the asymptotics of $s$ and $s^{A}$ as $k \to 1$ and as $k \to -1$ (see statement $(f)$ of Propositions \ref{sprop} and \ref{sAprop}) imply assertions $(\ref{Theorem2.3itemi})$, $(\ref{Theorem2.3itemii})$, and $(\ref{Theorem2.3itemiv})$. Moreover, statements $(e)$ of Propositions \ref{sprop} and \ref{sAprop} imply that $r_1(k)$ and $r_2(k)$ satisfy (\ref{r1r2rapiddecay}), which proves assertion $(\ref{Theorem2.3itemiii})$.

Let us prove $(\ref{Theorem2.3itemv})$. Since $s(k)=R(k)\overline{s^{A}(\bar{k})}R^{-1}(k)$ for $k \in \hat{\Gamma}_{1}$ by (\ref{sRsymm}), we infer that
\begin{align*}
r_{1}(k) = \frac{(s(k))_{12}}{(s(k))_{11}} = \frac{k^{2}-\omega}{\omega k^{2}-1} \frac{\overline{s^{A}(\bar{k})}_{21}}{\overline{s^{A}(\bar{k})}_{22}}.
\end{align*}
On the other hand, using the symmetry $s^A(k) = \mathcal{B} s^A(k^{-1})\mathcal{B}$, we obtain
\begin{align*}
\frac{\overline{s^{A}(\bar{k})}_{21}}{\overline{s^{A}(\bar{k})}_{22}} = \frac{\overline{s^{A}(\bar{k}^{-1})}_{12}}{\overline{s^{A}(\bar{k}^{-1})}_{11}} = \overline{r_{2}(\bar{k}^{-1})},
\end{align*}
which gives \eqref{r1r2 relation with kbar symmetry} and hence completes the proof of $(\ref{Theorem2.3itemv})$.

\subsection{Proof of Lemma \ref{inequalitieslemma}}\label{inequalitiessubsec}

In order to prove $(\ref{inequalitieslemmaitemi})$, we will first show that
\begin{align}\label{fs11expression}
f(k) = \frac{1}{|s_{11}(k)|^2} \qquad \text{for $k \in \partial \mathbb{D}$}.
\end{align}
The definition \eqref{def of f} of $f$ together with the symmetry \eqref{r1r2 relation with kbar symmetry} implies that, for $k \in \partial \D$,
\begin{align}\label{fr1}
f(k) = 1+\tilde{r}(k)|r_{1}(k)|^{2}+\tilde{r}(\tfrac{1}{\omega^{2}k})|r_{1}(\tfrac{1}{\omega^{2}k})|^{2},
\end{align}
where we recall that $\tilde{r}(k) = \frac{\omega^{2}-k^{2}}{1-\omega^{2}k^{2}}$. Moreover, by the symmetries (\ref{symmetries of s}) of $s$, for $k \in \partial \D$, 
\begin{align*}
& \overline{s_{12}(k)} = \overline{s_{21}(\bar{k})}, \quad \overline{s_{11}(k)} = \overline{s_{22}(\bar{k})}, \quad s_{12}(\omega \bar{k}) = s_{23}(\bar{k}) = s_{13}(k), \quad 
s_{11}(\omega \bar{k}) = s_{22}(\bar{k}) = s_{11}(k).
\end{align*}
Thus, we can write, for $k \in \partial \mathbb{D}$,
\begin{align}\nonumber
f(k) & = 1 + \tilde{r}(k) \frac{s_{12}(k) \overline{s_{12}(k)}}{s_{11}(k) \overline{s_{11}(k)}} 
+ \tilde{r}(\omega \bar{k}) \frac{s_{12}(\omega \bar{k}) \overline{s_{12}(\omega \bar{k})}}{s_{11}(\omega \bar{k}) \overline{s_{11}(\omega \bar{k})}}
	\\ \label{frs}
& = 1 + \tilde{r}(k) \frac{s_{12}(k) s_{21}^*(k)}{s_{11}(k) s_{22}^*(k)} 
+ \tilde{r}(\omega \bar{k}) \frac{s_{13}(k) s_{23}^*(k)}{s_{11}(k) s_{22}^*(k)},
\end{align}
where we recall that $s^{*}(k):=\overline{s(\bar{k})}$.
Since all entries of $s(k)$ are defined for $k \in \partial \mathbb{D} \setminus \mathcal{Q}$ we have, according to (\ref{sRsymm}), 
\begin{align}\label{sRsymm2}
\overline{s^{A}(\bar{k})} = R(k)^{-1}s(k)R(k), \qquad k \in \partial \mathbb{D} \setminus \mathcal{Q}.
\end{align}
In particular, taking the inverse transpose,
$$s^*(k) = R(k)^T (s(k)^{-1})^T (R(k)^{-1})^T, \qquad k \in \partial \mathbb{D} \setminus \mathcal{Q}.$$
Using this relation to eliminate all entries of $s^*(k)$ in (\ref{frs}) and simplifying, we obtain
$$f(k) = \frac{1}{(s(k))_{11} (s^A(k))_{11}} \qquad \text{for $k \in \partial \mathbb{D}$}.$$
Since $(s^A(k))_{11} = s_{22}^*(k)$ by (\ref{sRsymm2}) and $s_{22}^*(k) = \overline{s_{11}(k)}$ by (\ref{symmetries of s}), we arrive at (\ref{fs11expression}). 

It follows immediately from (\ref{fs11expression}) that $f(k) \geq 0$ for all $k \in \partial \mathbb{D}$. Furthermore, by Proposition \ref{sprop}, $s_{11}(k)$ is a continuous function of $k \in \partial \mathbb{D} \setminus \{\pm 1, \pm \omega,\pm \omega^2\}$.
In view of (\ref{fs11expression}), this means that $f(k) > 0$ for $k \in \partial \mathbb{D} \setminus \{\pm 1, \pm \omega,\pm \omega^2\}$.
On the other hand, Assumption \ref{originassumption} implies that $f(\pm 1) = 0$. The symmetry $f(k) = f(\frac{1}{\omega^{2}k})$ then implies that $f(\pm \omega) = 0$.
Moreover, using that $s_{11}(\omega^2 k) = s_{33}(k)$, we can write
$f(k) = \frac{1}{|s_{33}(\omega k)|^2}$ for $k \in \partial \mathbb{D}$.
By Proposition \ref{sprop}, $s_{33}(k)$ does not have a singularity at $k = \pm 1$, and hence $f(\pm \omega^2 ) > 0$. 
Finally, since, for $k \in \partial \D$,
\begin{align}\label{rtildesign}
\begin{cases}
\tilde{r}(k) < 0 \quad \text{for} \arg k \in (-\pi/3, \pi/3) \cup (2\pi/3, 4\pi/3),
	\\
 \tilde{r}(k) > 0 \quad \text{for} \arg k \in (\pi/3, 2\pi/3) \cup (4\pi/3, 5\pi/3),
 \end{cases}
\end{align}
both $\tilde{r}(k)$ and $\tilde{r}(\frac{1}{\omega^2 k})$ are strictly negative for all $k \in \partial \mathbb{D}$ with $\arg k \in (2\pi/3, \pi) \cup (5\pi/3, 2\pi)$, so we see from (\ref{fr1}) that $f(k) \leq 1$ for all such $k$. This completes the proof of $(\ref{inequalitieslemmaitemi})$.

We next prove $(\ref{inequalitieslemmaitemii})$.
We know from assertion $(\ref{inequalitieslemmaitemi})$ that $f(k) > 0$ for all $k \in \partial \mathbb{D}$ with $\arg k \in (\pi/3, 2\pi/3) \cup (2\pi/3, \pi) \cup (4\pi/3, 5\pi/3)\cup (5\pi/3, 2\pi)$. Since $\tilde{r}(\tfrac{1}{\omega^{2}k}) < 0$ for all $k$ in this range, it follows that
$$f(k) - \tilde{r}(\tfrac{1}{\omega^{2}k})|r_{1}(\tfrac{1}{\omega^{2}k})|^2 > 0, \qquad k \in \partial \mathbb{D}, \; \; \arg k \in (\pi/3, \pi) \cup (4\pi/3, 2\pi),$$
where we have used that $f(\pm \omega) = 0, \tilde{r}(\pm 1) = -1$, and $r_1(\pm 1) = 1$ to see that the inequality holds also for $k = \pm \omega$. Using that $r_2(k) = \tilde{r}(k) \overline{r_1(k)}$ for $k \in \partial \mathbb{D}$ as well as the definition \eqref{def of f} of $f(k)$, this implies that $1+r_{1}(k)r_{2}(k) > 0$ for $k \in \partial \mathbb{D}$ with $\arg k \in (\pi/3, \pi) \cup (4\pi/3, 2\pi)$, which proves $(\ref{inequalitieslemmaitemii})$.

We next prove $(\ref{inequalitieslemmaitemiii})$.
Let $k \in \partial \mathbb{D}$ with $\arg k \in (5\pi/3, 2\pi)$. 
By $(\ref{inequalitieslemmaitemii})$, $1+r_{1}(k)r_{2}(k) > 0$. Moreover, by (\ref{r1r2 relation with kbar symmetry}) and (\ref{rtildesign}),
$$1+r_{1}(k)r_{2}(k) = 1+ \tilde{r}(k) |r_{1}(k)|^2  \leq 1,$$
so it follows that $-\frac{1}{2\pi}\ln(1+r_{1}(k)r_{2}(k)) \geq 0$, which proves $(\ref{inequalitieslemmaitemiii})$.

Let us finally prove $(\ref{inequalitieslemmaitemiv})$. 
Let $k \in \partial \mathbb{D}$ with $\arg k \in (5\pi/3, 2\pi)$. Then $1+r_{1}(\omega k)r_{2}(\omega k) > 0$ and $f(\omega k) > 0$ by $(\ref{inequalitieslemmaitemii})$ and $(\ref{inequalitieslemmaitemi})$, and hence
$$\hat{\nu}_1(k) = \frac{1}{2\pi}\ln\bigg(\frac{1+r_{1}(\omega k)r_{2}(\omega k)}{f(\omega k)}\bigg) \in \R.$$
In fact, using (\ref{def of f}) and (\ref{r1r2 relation with kbar symmetry}) as well as the fact that $\tilde{r}(k^{-1}) < 0$, we obtain 
$$f(\omega k) - (1+r_{1}(\omega k)r_{2}(\omega k))
=  \tilde{r}(k^{-1}) |r_1(k^{-1})|^2 \leq 0,$$
so it follows that $\hat{\nu}_1(k) \geq 0$.
We next consider $\hat{\nu}_2$. 
Let $\mathcal{I}$ denote the subset of the unit circle consisting of all $k \in \partial \mathbb{D}$ with $\arg k \in (\pi, 4\pi/3)$. For $k \in \mathcal{I}$, we have $f(\omega^2 k) > 0$, $f(\omega k) > 0$, and $1+r_{1}(\omega^{2} k)r_{2}(\omega^{2} k) > 0$ by $(\ref{inequalitieslemmaitemi})$ and $(\ref{inequalitieslemmaitemii})$. Hence $\nu_2, \nu_3$, and $\nu_4$ are real-valued and well defined on $\mathcal{I}$ by (\ref{nu12345def}), and 
$$\hat{\nu}_2(k) = - \frac{1}{2\pi}\ln\frac{(1+r_{1}(\omega^{2} k)r_2(\omega^{2} k)) f(\omega k)}{f(\omega^{2} k)} \qquad \text{for} \quad k \in \mathcal{I},$$
so the inequality $\hat{\nu}_2 \geq 0$ will follow if we can show that
$$\frac{(1+r_{1}(\omega^{2} k)r_2(\omega^{2} k)) f(\omega k)}{f(\omega^{2} k)} \in (0, 1] \qquad \text{for} \quad  k \in  \mathcal{I}.$$
Recalling the definition \eqref{def of f} of $f(k)$, we can write this as
\begin{align*}
\bigg(1 - \frac{\tilde{r}(\frac{1}{\omega k})}{f(\omega^{2} k)}|r_{1}(\tfrac{1}{\omega k})|^2\bigg) f(\omega k) \in (0, 1] \qquad \text{for} \quad k \in \mathcal{I}.
\end{align*}
In view of (\ref{fs11expression}) and $s_{11}(\omega^{2}k) = s_{11}(\frac{1}{\omega k})$, it follows that it is sufficient to prove that
\begin{align}\label{1tilderfin01}
\frac{1 - \tilde{r}(\tfrac{1}{\omega k}) |s_{12}(\tfrac{1}{\omega k})|^2}{|s_{11}(\omega k)|^2} \in (0, 1] \qquad \text{for} \quad k \in \mathcal{I}.
\end{align}
Since $\tilde{r}(\frac{1}{\omega k}) < 0$ for $k \in \mathcal{I}$, the expression in (\ref{1tilderfin01}) is strictly positive for $k \in \mathcal{I}$.
To see that it is also bounded above by $1$, we first observe that the symmetries (\ref{sRsymm2}) and (\ref{symmetries of s}) imply that
\begin{align}\label{sRBconstraint}
\overline{s(k)}^T \mathcal{B} R(k)^{-1} s(k) R(k) \mathcal{B} = I, \qquad k \in \partial \mathbb{D},
\end{align}
where
$$ \mathcal{B} R(k)^{-1} s(k) R(k) \mathcal{B}
= \begin{pmatrix}
 s_{11}(k) & \tilde{r}(k) s_{12}(k)  &  \tilde{r}(\frac{1}{\omega^2 k}) s_{13}(k) 
 	\\
 \tilde{r}(\frac{1}{k}) s_{21}(k)  & s_{22}(k) & \tilde{r}(\omega k) s_{23}(k)  
 	\\
 \tilde{r}(\omega^2 k) s_{31}(k) & \tilde{r}(\frac{1}{\omega k}) s_{32}(k)  & s_{33}(k)
\end{pmatrix}.$$
In particular, the $(11)$-entry of (\ref{sRBconstraint}) yields
$$|s_{11}(k)|^2 + \tilde{r}(1/k) |s_{21}(k)|^2  + \tilde{r}(\omega^2 k) |s_{31}(k)|^2 = 1, \qquad k \in \partial \mathbb{D}.$$
Using the identity $s_{21}(k) = s_{12}(1/k)$ and then replacing $k$ by $\omega k$, this becomes after a simple rearrangement
$$1 - \tilde{r}(\tfrac{1}{\omega k}) |s_{12}(\tfrac{1}{\omega k})|^2  = |s_{11}(\omega k)|^2 +  \tilde{r}(k) |s_{31}(\omega k)|^2, \qquad k \in \partial \mathbb{D}.$$
Dividing by $|s_{11}(\omega k)|^2$ and noting that $\tilde{r}(k) < 0$ for $k \in \mathcal{I}$, we obtain the upper bound in (\ref{1tilderfin01}). This completes the proof of $(\ref{inequalitieslemmaitemiv})$.

\subsection{Construction of $n$}\label{sec: n at t=0}
In this subsection, we show how to construct a solution $n(x,t,k)$ of RH problem \ref{RHn} from a given solution $\{u,v\}$ of \eqref{boussinesqsystem}. 
More precisely, assuming that $\{u(x,t), v(x,t)\}$ is a Schwartz class solution of \eqref{boussinesqsystem} on $\R \times [0,T)$ with initial data $u_0, v_0 \in \mathcal{S}(\R)$, we will construct a solution $M(x,t,k)$ of a $3 \times 3$-matrix RH problem for $(x,t) \in \R \times [0,T)$. The solution $n$ of RH problem \ref{RHn} will then be obtained by premultiplying $M$ by the constant row vector $(1,1,1)$, i.e., $n=(1,1,1)M$. The results of this subsection will be used in Sections \ref{IVPsec} and \ref{blowupsec}; they also serve as motivation for the constructions in Section \ref{inversesec}.

We begin by considering the construction of $M$ at time $t = 0$. The restriction of $M$ to the open subset $D_n$, $n = 1,\dots, 6$, will be denoted by $M_n$, and we will slightly abuse notation and write $M(x,k)$ for $M(x,0,k)$. For each $n = 1, \dots, 6$, let $M_n(x,k)$ be a $3\times 3$-matrix valued solution of (\ref{xpart}) defined for $k \in \bar{D}_n\setminus \hat{\mathcal{Q}}$ by the following system of Fredholm integral equations: 
\begin{align}\label{Mndef}
(M_n)_{ij}(x,k) = \delta_{ij} + \int_{\gamma_{ij}^n} \left(e^{(x-x')\widehat{\mathcal{L}(k)}} (\mathsf{U}M_n)(x',k) \right)_{ij} dx', \qquad  i,j = 1, 2,3,
\end{align}
where $\gamma^n_{ij}$, $n = 1, \dots, 6$, $i,j = 1,2,3$, are given by
 \begin{align} \label{gammaijndef}
 \gamma_{ij}^n =  \begin{cases}
 (-\infty,x),  & \re  l_i(k) < \re  l_j(k), 
	\\
(+\infty,x),  \quad & \re  l_i(k) \geq \re  l_j(k),
\end{cases} \quad \text{for} \quad k \in D_n.
\end{align}
The definition of $\gamma_{ij}^n$ ensures that the exponential $e^{(l_i - l_j)(x-x')}$ appearing in the equation for $(M_n)_{ij}$ is bounded for $k \in D_n$ and $x' \in \gamma_{ij}^n$. 
The next proposition collects some basic properties of $M_{n}$. In particular, all entries of $M_n$ are well-defined for $k \in \bar{D}_n\setminus (\hat{\mathcal{Q}}\cup\mathcal{Z})$, where  $\mathcal{Z}$ denotes the set of zeros of the Fredholm determinants associated with (\ref{Mndef}). 
The proof is similar to the proof of \cite[Proposition 4.1]{CLgoodboussinesq} and is omitted.

\begin{proposition}\label{Mnprop}
If $u_0, v_0 \in \mathcal{S}(\R)$, then (\ref{Mndef}) uniquely defines six $3 \times 3$-matrix valued solutions $\{M_n\}_1^6$ of (\ref{xpart}) with the following properties:
\begin{enumerate}[$(a)$]
\item The function $M_n(x, k)$ is defined for $x \in \R$ and $k \in \bar{D}_n \setminus (\hat{\mathcal{Q}}\cup\mathcal{Z})$. For each $k \in \bar{D}_n  \setminus (\hat{\mathcal{Q}}\cup\mathcal{Z})$, $M_n(\cdot, k)$ is smooth and satisfies (\ref{xpart}).

\item For each $x \in \R$, $M_n(x,\cdot)$ is continuous for $k \in \bar{D}_n \setminus (\hat{\mathcal{Q}}\cup\mathcal{Z})$ and analytic for $k \in D_n\setminus (\hat{\mathcal{Q}}\cup\mathcal{Z})$.

\item For each $\epsilon > 0$, there exists a $C = C(\epsilon)$ such that $|M_n(x,k)| \leq C$ for $x \in \R$ and for all $k \in \bar{D}_n$ with $\dist(k, \hat{\mathcal{Q}}\cup\mathcal{Z}) \geq \epsilon$.

\item For each $x \in \R$ and each integer $j \geq 1$, $\partial_k^j M_n(x, \cdot)$ has a continuous extension to $\bar{D}_n \setminus (\hat{\mathcal{Q}}\cup\mathcal{Z})$.

\item $\det M_n(x,k) = 1$ for $x \in \R$ and $k \in \bar{D}_n \setminus (\hat{\mathcal{Q}}\cup\mathcal{Z})$.

\item For each $x \in \R$, the sectionally analytic function $M(x,k)$ defined by $M(x,k) = M_n(x,k)$ for $k \in D_n$ satisfies the symmetries
\begin{subequations}\label{Msymm}
\begin{align}\label{Msymma}
 & M(x, k) = \mathcal{A} M(x,\omega k)\mathcal{A}^{-1} = \mathcal{B} M(x,k^{-1}) \mathcal{B}, \qquad k \in \C \setminus (\hat{\mathcal{Q}}\cup\mathcal{Z}),
	\\
& \overline{(M^{A})(x,\bar{k})} = \bigg\{ \frac{u_{0}(x)}{2}\begin{pmatrix}
1 & 1 & 1 \\
1 & 1 & 1 \\
1 & 1 & 1 
\end{pmatrix} + R(k)^{-1} \bigg\}M(x,k)R(k), \qquad k \in \C \setminus (\hat{\mathcal{Q}}\cup\mathcal{Z}),
\end{align}
\end{subequations}
where $M^A := (M^{-1})^T$.
\end{enumerate}
\end{proposition}


The large $k$ asymptotics of $M_n$ can be derived by considering formal power series solutions of (\ref{xpart}). These solutions are of the form
\begin{align*}
& M_{n,formal}(x,k) = I + \frac{M_{n,1}(x)}{k} + \frac{M_{n,2}(x)}{k^2} + \cdots, \qquad n = 1, \dots, 6,
\end{align*}
and satisfy
\begin{align}\label{Mnlnormalization}
\begin{cases}
\displaystyle{\lim_{x\to -\infty}} (M_{n,l}(x))_{ij} = 0 & \text{if} \ \gamma_{ij}^n = (-\infty,x), 
	\\
\displaystyle{\lim_{x\to +\infty}} (M_{n,l}(x))_{ij} = 0 & \text{if} \ \gamma_{ij}^n = (+\infty,x),
\end{cases} \quad l \geq 1.
\end{align}
In the same way as for $X_{formal}$ and $Y_{formal}$, the coefficients $M_{n,j}$ are uniquely determined from (\ref{Mnlnormalization}), the recursive relations (\ref{xrecursive}) (with $X_{j}$ replaced by $M_{n,j}$), and the initial values $M_{n,-1} = 0$, $M_{n,0} = I$.
Because $\gamma_{ii}^n = (+\infty,x)$ for $i = 1,2,3$ and $n = 1, \dots,6$, we have in fact that $M_{n,j}(x) = X_j(x)$ for all $n$ and $j$. 
The following lemma is similar to \cite[Lemma 4.3]{CLgoodboussinesq} and we omit the proof.

\begin{lemma}\label{Matinftylemma}
Suppose $u_0, v_0 \in \mathcal{S}(\R)$ and $u_0,v_0 \not\equiv 0$. Given an integer $p \geq 1$, let $X_{(p)}(x,k)$ be the function defined in (\ref{Xpdef}). Then there exists an $R > 0$ such that
\begin{align*}
& \big|M(x,k) - X_{(p)}(x,k) \big| \leq
\frac{C}{|k|^{p+1}}, \qquad x \in \R, \  k \in \C \setminus \Gamma, \ |k| \geq R.
\end{align*}
\end{lemma}

Our next lemma relates $M_n$ to $X$ and $Y$; the proof is nearly identical to the proof of \cite[Lemma 4.4]{CLgoodboussinesq} and is therefore omitted. 

\begin{lemma}\label{Snexplicitlemma}
If $u_0,v_0 \in \mathcal{S}(\R)$ have compact support, then
\begin{align*}
   M_n(x,k) & = Y(x, k) e^{x\widehat{\mathcal{L}(k)}} S_n(k)	
  = X(x, k) e^{x\widehat{\mathcal{L}(k)}} T_n(k), \qquad x \in \R, \ k \in \bar{D}_n\setminus (\hat{\mathcal{Q}}\cup\mathcal{Z}), \ n = 1, \dots, 6,
\end{align*}
where $S_n(k)$ and $T_n(k)$ are given in terms of the entries of $s(k)$ by
\begin{subequations}\label{SnTnexplicit}
\begin{align}\nonumber
&  S_1(k) = \begin{pmatrix}
 s_{11} & 0 & 0 \\
 s_{21} & \frac{m_{33}(s)}{s_{11}} & 0 \\
 s_{31} & \frac{m_{23}(s)}{s_{11}} & \frac{1}{m_{33}(s)} \\
  \end{pmatrix},
&&
  S_2(k) =  \begin{pmatrix}
 s_{11} & 0 & 0 \\
 s_{21} & \frac{1}{m_{22}(s)} & \frac{m_{32}(s)}{s_{11}} \\
 s_{31} & 0 & \frac{m_{22}(s)}{s_{11}} \\
\end{pmatrix},
	\\ \nonumber
&  S_3(k) = \begin{pmatrix}
 \frac{m_{22}(s)}{s_{33}} & 0 & s_{13} \\
 \frac{m_{12}(s)}{s_{33}} & \frac{1}{m_{22}(s)} & s_{23} \\
 0 & 0 & s_{33} \\
\end{pmatrix},
&&
  S_4(k) =  \begin{pmatrix}
  \frac{1}{m_{11}(s)} & \frac{m_{21}(s)}{s_{33}} & s_{13} \\
 0 & \frac{m_{11}(s)}{s_{33}} & s_{23} \\
 0 & 0 & s_{33} \\
\end{pmatrix},
	\\ \label{Snexplicit}
&  S_5(k) = \begin{pmatrix}
 \frac{1}{m_{11}(s)} & s_{12} & -\frac{m_{31}(s)}{s_{22}} \\
 0 & s_{22} & 0 \\
 0 & s_{32} & \frac{m_{11}(s)}{s_{22}} \\
  \end{pmatrix},
&&
  S_6(k) =  \begin{pmatrix}
 \frac{m_{33}(s)}{s_{22}} & s_{12} & 0 \\
 0 & s_{22} & 0 \\
 -\frac{m_{13}(s)}{s_{22}} & s_{32} & \frac{1}{m_{33}(s)} \\
 \end{pmatrix},	
\end{align}
and
\begin{align}\nonumber
&  T_1(k) = \begin{pmatrix}
  1 & -\frac{s_{12}}{s_{11}} & \frac{m_{31}(s)}{m_{33}(s)} \\
 0 & 1 & -\frac{m_{32}(s)}{m_{33}(s)} \\
 0 & 0 & 1
   \end{pmatrix},
&&
  T_2(k) =  \begin{pmatrix}
 1 & -\frac{m_{21}(s)}{m_{22}(s)} & -\frac{s_{13}}{s_{11}} \\
 0 & 1 & 0 \\
 0 & -\frac{m_{23}(s)}{m_{22}(s)} & 1 
\end{pmatrix},
	\\ \nonumber
&  T_3(k) = \begin{pmatrix}
 1 & -\frac{m_{21}(s)}{m_{22}(s)} & 0 \\
 0 & 1 & 0 \\
 -\frac{s_{31}}{s_{33}} & -\frac{m_{23}(s)}{m_{22}(s)} & 1 
\end{pmatrix},
&&
  T_4(k) =  \begin{pmatrix}
 1 & 0 & 0 \\
 -\frac{m_{12}(s)}{m_{11}(s)} & 1 & 0 \\
 \frac{m_{13}(s)}{m_{11}(s)} & -\frac{s_{32}}{s_{33}} & 1 
\end{pmatrix},
	\\ \label{Tnexplicit}
&  T_5(k) = \begin{pmatrix}
 1 & 0 & 0 \\
 -\frac{m_{12}(s)}{m_{11}(s)} & 1 & -\frac{s_{23}}{s_{22}} \\
 \frac{m_{13}(s)}{m_{11}(s)} & 0 & 1 
  \end{pmatrix},
&&
  T_6(k) =  \begin{pmatrix}
 1 & 0 & \frac{m_{31}(s)}{m_{33}(s)} \\
 -\frac{s_{21}}{s_{22}} & 1 & -\frac{m_{32}(s)}{m_{33}(s)} \\
 0 & 0 & 1 
 \end{pmatrix},
\end{align}
\end{subequations}
where $m_{ij}(s)$ denotes the $(ij)$th minor of the matrix $s$.
\end{lemma}

Let $\eta \in C_c^\infty(\R)$ be a bump function such that $\eta(x) = 1$ for $|x| \leq 1$ and $\eta(x) = 0$ for $|x| \geq 2$. For $j \geq 1$, let $\eta_j(x) = \eta(x/j)$. If $f \in \mathcal{S}(\R)$, then $\eta_jf$ is a sequence in $C_c^\infty(\R)$ which converges to $f$ in $\mathcal{S}(\R)$ as $j \to \infty$. The following lemma is obtained in the same way as \cite[Lemma 4.5]{CLgoodboussinesq}. 

\begin{lemma}\label{sequencelemma}
Let $u_0, v_0 \in \mathcal{S}(\R)$.
Let $\{s(k), M_n(x,k)\}$ and $\{s^{(i)}(k), M_n^{(i)}(x,k)\}$ be the spectral functions and eigenfunctions associated with $(u_0, v_0)$ and 
\begin{align}\label{uvsequence}
(u_0^{(i)}(x), v_0^{(i)}(x)) := (\eta_iu_0, \eta_iv_0) \in C_c^\infty(\R) \times C_c^\infty(\R), 
\end{align} 
respectively. Then 
\begin{align}\label{slimiti}
& \lim_{i\to\infty} s^{(i)}(k) = s(k), \qquad k \in \begin{pmatrix}
 \omega^2 \hat{\mathcal{S}} & \hat{\Gamma}_{1} & \hat{\Gamma}_{3} \\
 \hat{\Gamma}_{1} & \omega \hat{\mathcal{S}} & \hat{\Gamma}_{5} \\
\hat{\Gamma}_{3} & \hat{\Gamma}_{5} & \hat{\mathcal{S}}
 \end{pmatrix}\setminus \hat{\mathcal{Q}},
 	\\\label{sAlimiti}
& \lim_{i\to\infty} (s^A)^{(i)}(k) = s^A(k), \qquad k \in	 \begin{pmatrix}
 -\omega^2 \hat{\mathcal{S}} & \hat{\Gamma}_{4} & \hat{\Gamma}_{6} \\
\hat{\Gamma}_{4} & -\omega \hat{\mathcal{S}} & \hat{\Gamma}_{2} \\
\hat{\Gamma}_{6} & \hat{\Gamma}_{2} & -\hat{\mathcal{S}}
 \end{pmatrix}\setminus \hat{\mathcal{Q}},
	\\ \label{Xlimiti}
& \lim_{i\to \infty} X^{(i)}(x,k) = X(x,k), \qquad x \in \R, \ k \in (\omega^2 \hat{\mathcal{S}}, \omega \hat{\mathcal{S}}, \hat{\mathcal{S}}) \setminus \hat{\mathcal{Q}},
	\\ \label{Ylimiti}
& \lim_{i\to \infty} Y^{(i)}(x,k) = Y(x,k), \qquad x \in \R, \ k \in (-\omega^2 \hat{\mathcal{S}}, -\omega \hat{\mathcal{S}}, -\hat{\mathcal{S}}) \setminus \hat{\mathcal{Q}},
	\\ \label{Mnlimiti}
& \lim_{i\to \infty} M_n^{(i)}(x,k) = M_n(x,k), \qquad x \in \R, \ k \in \bar{D}_n\setminus (\hat{\mathcal{Q}}\cup\mathcal{Z}), \ n = 1, \dots, 6.
\end{align}
\end{lemma}


\begin{lemma}[Jump condition for $M$]\label{Mjumplemma}
Let $u_0, v_0 \in \mathcal{S}(\R)$.
For each $x \in \R$, $M(x,k)$ satisfies the jump condition
\begin{align*}
  M_+(x,k) = M_-(x, k) v(x, 0, k), \qquad k \in \Gamma \setminus (\hat{\mathcal{Q}}\cup\mathcal{Z}),
\end{align*}
where $v$ is the jump matrix defined in (\ref{vdef}).
\end{lemma}
\begin{proof}
The proof that
\begin{align}\label{M1M6v1}
M_{n} = M_{n-1} v_n, \qquad k \in \Gamma_{n}\setminus (\hat{\mathcal{Q}}\cup\mathcal{Z}), \; n=1,\ldots,6 \qquad (M_{0}:=M_{6})
\end{align}
is similar to the proof of \cite[Lemma 4.6]{CLgoodboussinesq} (the only difference is that $r_{1}^{\star}(k)$ and $r_{2}^{\star}(k)$ in \cite{CLgoodboussinesq} should here be replaced by $r_{1}(\frac{1}{k})$ and $r_{2}(\frac{1}{k})$, respectively). 
Let us prove that 
\begin{align}\label{M1M4v7}
M_{1}(x,k) = M_{4}(x,k) v_7(x,0,k), \qquad k \in \Gamma_{7}\setminus (\hat{\mathcal{Q}}\cup\mathcal{Z}).
\end{align}
For each $k\in \bar{D}_{n}\setminus (\hat{\mathcal{Q}}\cup\mathcal{Z})$, $M_n(x,k)$ is a smooth function of $x \in \R$ which satisfies (\ref{xpart}). Hence there exists a matrix $J_7(k)$ independent of $x$  such that
\begin{align}\label{M1M4J7}
M_1(x,k) = M_4(x,k) e^{x\widehat{\mathcal{L}(k)}}J_7(k), \qquad k \in \Gamma_{7}\setminus (\hat{\mathcal{Q}}\cup\mathcal{Z}).
\end{align}
If $u_0, v_0$ are compactly supported, then for large negative $x$ we have $M_n(x,k) = e^{x\widehat{\mathcal{L}(k)}}S_n(k)$, where $S_n(k)$ is the matrix defined in (\ref{Snexplicit}). Hence, evaluating (\ref{M1M4J7}) at a large negative $x$ gives
\begin{align}\label{J7 computation}
J_7(k) = S_4(k)^{-1}S_1(k) = e^{-x\widehat{\mathcal{L}(k)}}v_7(x,0,k),
\end{align}
where we have used (\ref{vdef}) for the last equality.
An application of Lemma \ref{sequencelemma} shows that (\ref{J7 computation}) remains true when $u_0,v_0 \in \mathcal{S}(\R)$ are not compactly supported. This proves (\ref{M1M4v7}); the proofs that $M_5 = M_2 v_8$ for $k \in \Gamma_{8}\setminus (\hat{\mathcal{Q}}\cup\mathcal{Z})$ and $M_3 = M_6 v_9$ for $k \in \Gamma_{9}\setminus (\hat{\mathcal{Q}}\cup\mathcal{Z})$ are similar. 
\end{proof}

\begin{lemma}\label{M1XYlemma}
Let $u_0, v_0 \in \mathcal{S}(\R)$.
The functions $M_2$ and $M_2^A = (M_2^{-1})^T$ can be expressed in terms of the entries of $X,Y,X^A, Y^A, s$, and $s^A$ as follows:
\begin{align*}
M_2 = \begin{pmatrix} 
X_{11} & \frac{Y_{12}}{s_{22}^A} & \frac{Y_{21}^AX_{32}^A - Y_{31}^AX_{22}^A}{s_{11}}  \\
X_{21} & \frac{Y_{22}}{s_{22}^A} & \frac{Y_{31}^AX_{12}^A - Y_{11}^AX_{32}^A}{s_{11}}  \\
X_{31} & \frac{Y_{32}}{s_{22}^A} & \frac{Y_{11}^AX_{22}^A - Y_{21}^AX_{12}^A}{s_{11}} 
\end{pmatrix}, \qquad
M_2^A = \begin{pmatrix} 
\frac{Y_{11}^A}{s_{11}} & X_{12}^A & \frac{X_{21}Y_{32} - X_{31}Y_{22}}{s_{22}^A}  \\
\frac{Y_{21}^A}{s_{11}} & X_{22}^A & \frac{X_{31}Y_{12} - X_{11}Y_{32}}{s_{22}^A}  \\
\frac{Y_{31}^A}{s_{11}} & X_{32}^A  & \frac{X_{11}Y_{22} - X_{21}Y_{12}}{s_{22}^A}
\end{pmatrix},
\end{align*}
for all $x \in \R$ and $k \in \bar{D}_2 \setminus (\hat{\mathcal{Q}}\cup\mathcal{Z})$.
\end{lemma}
\begin{proof}
Let $(u_0^{(i)}(x), v_0^{(i)}(x))$ be as in (\ref{uvsequence}). By Lemma \ref{Snexplicitlemma} we have
\begin{align*}
   M_n^{(i)}(x,k) & = Y^{(i)}(x, k) e^{x\widehat{\mathcal{L}(k)}} S_n^{(i)}(k)	
= X^{(i)}(x, k) e^{x\widehat{\mathcal{L}(k)}} T_n^{(i)}(k), \qquad k \in \bar{D}_n\setminus (\hat{\mathcal{Q}}\cup\mathcal{Z}).
\end{align*}
In particular, the first two columns of $M_2^{(i)}$ are given by
\begin{align*}
\begin{cases}
 [M_2^{(i)}(x,k)]_1 = [X^{(i)}(x,k)]_1,
	\\
 [M_2^{(i)}(x,k)]_2 = \frac{[Y^{(i)}(x,k)]_2}{m_{22}(s^{(i)})},
\end{cases} \qquad x \in \R, \; k \in \bar{D}_2\setminus (\hat{\mathcal{Q}}\cup\mathcal{Z}), \; i \geq 1,
\end{align*}
where we recall that $[A]_j$ denotes the $j$th column of a matrix $A$. 
Using Lemma \ref{sequencelemma}  to take $i \to \infty$, we obtain
\begin{align}\label{M1XY}
\begin{cases}
 [M_2(x,k)]_1 = [X(x,k)]_1,
	\\
 [M_2(x,k)]_2 = \frac{[Y(x,k)]_2}{(s^A(k))_{22}},
 \end{cases} \quad x \in \R, \ k \in \bar{D}_2\setminus (\hat{\mathcal{Q}}\cup\mathcal{Z}).
\end{align} 
Analogous arguments using that
\begin{align*}
(S_2(k)^{-1})^T = \begin{pmatrix}
 \frac{1}{s_{11}} & s^{A}_{12} & -\frac{s_{31}}{s^{A}_{22}} \\
 0 & s^{A}_{22} & 0 \\
 0 & s^{A}_{32} & \frac{s_{11}}{s^{A}_{22}}
 \end{pmatrix}, \qquad
(T_{2}(k)^{-1})^T = \begin{pmatrix}
  1 & 0 & 0 \\
 \frac{s_{12}}{s_{11}} & 1 & -\frac{s^{A}_{23}}{s^{A}_{22}} \\
 \frac{s_{13}}{s_{11}} & 0 & 1 
 \end{pmatrix},
\end{align*} 
show that
\begin{align}\label{M1AXAYA}
\begin{cases}
 [M_2^A(x,k)]_1 = \frac{[Y^A(x,k)]_1}{s_{11}(k)},
	\\
 [M_2^A(x,k)]_2 = [X^A(x,k)]_2,
 \end{cases} \quad x \in \R, \ k \in \bar{D}_2\setminus (\hat{\mathcal{Q}}\cup\mathcal{Z}).
\end{align} 
The third columns of $M_2$ and $M_2^{A}$ can then be expressed as in the statement using that $M = ((M^A)^{-1})^T$, $M^A = (M^{-1})^T$, and $\det M = 1$.
\end{proof}

\begin{lemma}\label{QtildeQlemma}
Suppose $u_0,v_0 \in \mathcal{S}(\R)$ are such that Assumption \ref{solitonlessassumption} holds. Then the statements of Proposition \ref{Mnprop} and Lemmas \ref{Snexplicitlemma}-\ref{M1XYlemma} hold with $\mathcal{Z}$ replaced by the empty set.
\end{lemma}
\begin{proof}
Lemma \ref{M1XYlemma} implies that if $u_0, v_0$ satisfy Assumption \ref{solitonlessassumption}, then $M$ has no singularities apart from $k=\kappa_{j}$, $j=1,\ldots,6$. Indeed, $s_{11}^{A}(k) \neq 0$ for $k \in \bar{D}_{5}\setminus \hat{\mathcal{Q}}$ so $s_{22}^{A}(k)=s_{11}^{A}(1/k) \neq 0$ for $k \in \bar{D}_{2}\setminus \hat{\mathcal{Q}}$ and thus $M_2$ has no singularities in $\bar{D}_2 \setminus \mathcal{Q}$. The symmetries in (\ref{Msymma}) then imply that $M_n$ has no singularities in $\bar{D}_n \setminus \mathcal{Q}$ for any $n$.
Hence, for any $n$, $M_{n}(x,k)$ can be extended to any $k_j \in (\mathcal{Z} \cap \bar{D}_n)  \setminus \mathcal{Q}$ by continuity, from which the claim follows.
\end{proof}

We now turn to the behavior of $M$ as $k \to \pm 1$.

\begin{lemma}\label{Mat1lemma}
Suppose $u_0,v_0 \in \mathcal{S}(\R)$ are such that Assumptions \ref{solitonlessassumption} and \ref{originassumption} hold.
Let $p \geq 1$ be an integer.
Then there are $3 \times 3$-matrix valued functions $\{\mathcal{M}_2^{(l)}(x),\widetilde{\mathcal{M}}_2^{(l)}(x), \mathcal{N}_2^{(l)}(x), \widetilde{\mathcal{N}}_2^{(l)}(x)\}$, $l = -1,0, \dots, p$, with the following properties:
\begin{enumerate}[$(a)$]
\item The function $M$ satisfies, for $x \in \R$,
\begin{align*}
\begin{cases}
& \big|M_2(x,k) - \sum_{l=-1}^p \mathcal{M}_{2}^{(l)}(x)(k-1)^l\big| \leq C
|k-1|^{p+1}, \qquad |k-1| \leq \frac{1}{2}, \ k \in \bar{D}_2, \\
& \big|M_2(x,k) - \sum_{l=-1}^p \widetilde{\mathcal{M}}_2^{(l)}(x)(k+1)^l\big| \leq C
|k+1|^{p+1}, \qquad |k+1| \leq \frac{1}{2}, \ k \in \bar{D}_2
\end{cases}
\end{align*}
\item The function $M^{-1}$ satisfies, for $x \in \R$,
\begin{align*}
\begin{cases}
& \big|M_2(x,k)^{-1} - \sum_{l=-1}^p \mathcal{N}_{2}^{(l)}(x)(k-1)^l\big| \leq C
|k-1|^{p+1}, \qquad |k-1| \leq \frac{1}{2}, \ k \in \bar{D}_2, \\
& \big|M_2(x,k)^{-1} - \sum_{l=-1}^p \widetilde{\mathcal{N}}_2^{(l)}(x)(k+1)^l\big| \leq C
|k+1|^{p+1}, \qquad |k+1| \leq \frac{1}{2}, \ k \in \bar{D}_2.
\end{cases}
\end{align*}
\item For each $l \geq -1$, $\{\mathcal{M}_2^{(l)}(x),\widetilde{\mathcal{M}}_2^{(l)}(x), \mathcal{N}_2^{(l)}(x), \widetilde{\mathcal{N}}_2^{(l)}(x)\}$ are smooth functions of $x \in \R$.
\item The first coefficients are given by
\begin{align*}
& \mathcal{M}_{2}^{(-1)}(x) = \begin{pmatrix}
\alpha(x) & 0 & \beta(x) \\
-\alpha(x) & 0 & -\beta(x) \\
0 & 0 & 0
\end{pmatrix}, & & \mathcal{M}_{2}^{(0)}(x) = \begin{pmatrix}
\star & \gamma(x) & \star \\
\star & -\gamma(x) & \star \\
\star & 0 & \star
\end{pmatrix}, \\
& \widetilde{\mathcal{M}}_{2}^{(-1)}(x) = \begin{pmatrix}
\tilde{\alpha}(x) & 0 & \tilde{\beta}(x) \\
-\tilde{\alpha}(x) & 0 & -\tilde{\beta}(x) \\
0 & 0 & 0
\end{pmatrix}, & & \widetilde{\mathcal{M}}_{2}^{(0)}(x) = \begin{pmatrix}
\star & \tilde{\gamma}(x) & \star \\
\star & -\tilde{\gamma}(x) & \star \\
\star & 0 & \star
\end{pmatrix}.
\end{align*}
Here $\star$ denotes an entry whose value is irrelevant for us and
\begin{align*}
\alpha = \alpha_{1}, \quad \beta = \frac{-\tilde{\alpha}_{3}\tilde{\beta}_{1}+\tilde{\alpha}_{1}\tilde{\beta}_{3}}{\mathfrak{s}_{1}}, \quad \gamma = -\frac{\alpha_{3}}{\mathfrak{s}_{1}^{A}}, \quad \tilde{\alpha} = \alpha_{2}, \quad \tilde{\beta} = \frac{-\tilde{\alpha}_{4}\tilde{\beta}_{2}+\tilde{\alpha}_{2}\tilde{\beta}_{4}}{\mathfrak{s}_{3}}, \quad \tilde{\gamma} = -\frac{\alpha_{4}}{\mathfrak{s}_{3}^{A}},
\end{align*}
where the functions $\{\alpha_{j},\tilde{\alpha}_{j},\tilde{\beta}_{j}\}_{j=1}^{4}$ are defined in \eqref{Cjpm1p} and \eqref{Djpm1p}, and the constants $\{\mathfrak{s}_{j},\mathfrak{s}^{A}_{j}\}_{j=1,3}$ are defined in \eqref{s1pm1p and sm1pm1p} and \eqref{sA1pm1p and sAm1pm1p}.
\item For each $x \in \mathbb{R}$, the function $k \mapsto \begin{pmatrix}
1 & 1 & 1
\end{pmatrix}M_{2}(x,k)$ is bounded as $k \to \pm 1$, $k \in \bar{D}_{2}$.
\end{enumerate}
\end{lemma}
\begin{proof}
Lemma \ref{M1XYlemma} expresses $M_2$ and $M_2^A$ for $k \in \bar{D}_2\setminus \hat{\mathcal{Q}}$ in terms of $X, Y, s, X^A, Y^A, s^A$. Moreover, since $u_0,v_0$ satisfy Assumption \ref{originassumption}, $\mathfrak{s}_{1},\mathfrak{s}_{3} \neq 0$ and $\mathfrak{s}_{1}^{A},\mathfrak{s}_{3}^{A} \neq 0$. Hence, using the expansions of $X, Y, s, X^A, Y^A, s^A$ as $k \to \pm 1$ given in Propositions \ref{XYat1prop}, \ref{sprop}, \ref{XAYAat1prop}, and \ref{sAprop} and recalling Assumptions \ref{solitonlessassumption} and \ref{originassumption}, the lemma follows. 
\end{proof}

\begin{remark}\upshape
The statement of Lemma \ref{Mat1lemma} describes the asymptotics of $M_{2}(k)$ as $k \to \pm 1$, $k \in \bar{D}_{2}$. The asymptotics in the other sectors near the points $\kappa_{j}$, $j=1,...,6$, can then be obtained from the symmetries $M(x,k) = \mathcal{A}M(x,\omega k) \mathcal{A}^{-1} = \mathcal{B} M(x,1/k)\mathcal{B}$.
\end{remark}

We next consider the time dependence of $M$. Given a Schwartz class solution of (\ref{boussinesqsystem}) on $\R \times [0,T)$, we define time-dependent eigenfunctions $\{M_n(x,t,k)\}_{n=1}^6$ by replacing $\mathsf{U}(x,k)$ with the time-dependent matrix $\mathsf{U}(x,t,k)$ in the integral equations (\ref{Mndef}). We then define the sectionally analytic function $M(x,t,k)$ by setting $M(x,t,k) = M_n(x,t,k)$ for $k \in D_n$. 
We will show that $M(x,t,k)$ satisfies a $3 \times 3$-matrix RH problem. 
Due to the singularities of $M$ at the sixth roots of unity, the formulation of this RH problem is rather complicated. 

\begin{RHproblem}[RH problem for $M$]\label{RH problem for M}
Find $M(x,t,k)$ with the following properties:
\begin{enumerate}[(a)]
\item $M(x,t,\cdot) : \mathbb{C}\setminus \Gamma \to \mathbb{C}^{3 \times 3}$ is analytic.

\item The limits of $M(x,t,k)$ as $k$ approaches $\Gamma\setminus (\Gamma_\star \cup \mathcal{Q})$ from the left and right exist, are continuous on $\Gamma\setminus (\Gamma_\star \cup \mathcal{Q})$, and satisfy
\begin{align}\label{Mjumpcondition}
& M_{+}(x,t,k) = M_{-}(x,t,k)v(x,t,k), \qquad k \in \Gamma \setminus (\Gamma_\star \cup \mathcal{Q}),
\end{align}
where $v$ is defined by \eqref{vdef}.

\item As $k \to \infty$, 
\begin{align}\label{asymp for M at infty in RH def}
M(x,t,k) = I + \frac{M^{(1)}(x,t)}{k} + \frac{M^{(2)}(x,t)}{k^{2}} + O\bigg(\frac{1}{k^3}\bigg),
\end{align}
where the matrices $M^{(1)}$ and $M^{(2)}$ depend on $x$ and $t$ but not on $k$, and satisfy
\begin{align}\label{singRHMatinftyb}
M_{12}^{(1)} = M_{13}^{(1)} = M_{12}^{(2)} + M_{21}^{(2)} = 0.
\end{align}

\item There exist matrices $\{\mathcal{M}_2^{(l)}(x,t),\widetilde{\mathcal{M}}_2^{(l)}(x,t)\}_{l=-1}^{+\infty}$ depending on $x$ and $t$ but not on $k$ such that, for any $N \geq -1$,
\begin{align}\label{singRHMat0}
& M(x,t,k) = \sum_{l=-1}^{N} \mathcal{M}_2^{(l)}(x,t)(k-1)^{l} + O((k-1)^{N+1}) \qquad \text{as}\ k \to 1, \ k \in \bar{D}_2, \\
& M(x,t,k) = \sum_{l=-1}^{N} \widetilde{\mathcal{M}}_2^{(l)}(x,t)(k+1)^{l} + O((k+1)^{N+1}) \qquad \text{as}\ k \to -1, \ k \in \bar{D}_2.
\end{align}
Furthermore, there exist scalar coefficients $\alpha, \beta, \gamma, \tilde{\alpha}, \tilde{\beta}, \tilde{\gamma}$ depending on $x$ and $t$, but not on $k$, such that
\begin{align}\nonumber
& \mathcal{M}_{2}^{(-1)}(x,t) = \begin{pmatrix}
\alpha(x,t) & 0 & \beta(x,t) \\
-\alpha(x,t) & 0 & -\beta(x,t) \\
0 & 0 & 0
\end{pmatrix}, & & \mathcal{M}_{2}^{(0)}(x,t) = \begin{pmatrix}
\star & \gamma(x,t) & \star \\
\star & -\gamma(x,t) & \star \\
\star & 0 & \star
\end{pmatrix}, 
	\\ \label{mathcalMcoefficients}
& \widetilde{\mathcal{M}}_{2}^{(-1)}(x,t) = \begin{pmatrix}
\tilde{\alpha}(x,t) & 0 & \tilde{\beta}(x,t) \\
-\tilde{\alpha}(x,t) & 0 & -\tilde{\beta}(x,t) \\
0 & 0 & 0
\end{pmatrix}, & & \widetilde{\mathcal{M}}_{2}^{(0)}(x,t) = \begin{pmatrix}
\star & \tilde{\gamma}(x,t) & \star \\
\star & -\tilde{\gamma}(x,t) & \star \\
\star & 0 & \star
\end{pmatrix}.
\end{align}

\item $M$ satisfies the symmetries $M(x,t, k) = \mathcal{A} M(x,t,\omega k)\mathcal{A}^{-1} = \mathcal{B} M(x,t,\tfrac{1}{k})\mathcal{B}$ and
\begin{align*}
\overline{(M^{A})(x,t,\bar{k})} = \bigg\{ \frac{u(x,t)}{2}\begin{pmatrix}
1 & 1 & 1 \\
1 & 1 & 1 \\
1 & 1 & 1 
\end{pmatrix} + R(k)^{-1} \bigg\}M(x,t,k)R(k), \qquad k \in \C \setminus \Gamma,
\end{align*}
where $M^{A}=(M^{-1})^{T}$ and $R$ is given by \eqref{def of R}. 
\end{enumerate}
\end{RHproblem}

It can be shown that the conditions (\ref{singRHMatinftyb}) make the solution of RH problem \ref{RH problem for M} unique, but we will not need this fact.

\begin{proposition}\label{RHth}
Suppose $\{u(x,t), v(x,t)\}$ is a Schwartz class solution of (\ref{boussinesqsystem}) on $\R \times [0,T)$ and initial data $u_0, v_0 \in \mathcal{S}(\R)$ for some $T \in (0, +\infty]$ such that Assumptions \ref{solitonlessassumption} and \ref{originassumption} hold. Define $\{r_j(k)\}_1^2$ in terms of $u_0, v_0$ by (\ref{r1r2def}). Define the sectionally analytic function $M$ by $M(x,t,k) = M_n(x,t,k)$ for $k \in D_n$.
Then $M(x,t,k)$ satisfies RH problem \ref{RH problem for M} for each $(x,t) \in \R \times [0,T)$ and the formulas 
\begin{align}\label{recoveruv}
\begin{cases}
 \displaystyle{u(x,t) = -i\sqrt{3}\frac{\partial}{\partial x}\lim_{k\to \infty}k\big[(M(x,t,k))_{33} - 1\big] = \frac{1-\omega}{2} \lim_{k\to \infty}k^{2}(M(x,t,k))_{32},}
	\vspace{.1cm}\\
 \displaystyle{v(x,t) = -i\sqrt{3}\frac{\partial}{\partial t}\lim_{k\to \infty}k\big[(M(x,t,k))_{33} - 1\big],}
\end{cases}
\end{align}
 expressing $\{u(x,t), v(x,t)\}$ in terms of $M$ are valid for all $(x,t) \in \R \times [0,T)$.
\end{proposition}
\begin{proof}

By Lemma \ref{Matinftylemma} and the definition (\ref{Xpdef}) of $X_{(p)}$, we have
\begin{align*}
& M^{(1)}(x,t)_{33} = \frac{i}{\sqrt{3}}  \int_{\infty}^{x} u(x^{\prime}, t) dx', \qquad  M^{(2)}(x,t)_{32} = \frac{2u(x,t)}{1-\omega}.
\end{align*}
Recalling that $u,v$ have rapid decay as $x \to \infty$ and that $u_t = v_x$, the formulas (\ref{recoveruv}) for $u$ and $v$ follow. Property $(c)$ of RH problem \ref{RH problem for M} related to the asymptotics of $M$ as $k \to \infty$ follows from Lemma \ref{Matinftylemma} and the definition (\ref{Xpdef}) of $X_{(p)}$.
Property $(d)$ related to the asymptotics of $M$ as $k \to \pm 1$ follows from Lemma \ref{Mat1lemma}, whereas property $(e)$ follows from the symmetries (\ref{Msymm}) of $M$ and Lemma \ref{QtildeQlemma}. It remains to prove that $M$ satisfies properties $(a)$ and $(b)$ of RH problem \ref{RH problem for M}.

The analyticity and the existence of continuous boundary values are a consequence of Proposition \ref{Mnprop} and Lemma \ref{QtildeQlemma}.
Moreover, as in \cite[Lemma 5.1]{CLgoodboussinesq}, one can show that $M_n(x,t,k)$, $n = 1, \dots, 6$, is a smooth function of $(x,t) \in \R \times [0,T)$ satisfying the Lax pair equations \eqref{Xlax} for each $k \in \bar{D}_n \setminus \hat{\mathcal{Q}}$. These equations imply that the functions $M_n$ are related by
\begin{align}\label{relation between M at time t and time 0}
M_{+}(x,t,k) = M_-(x,t,k) e^{\hat{\mathcal{L}}x + \hat{\mathcal{Z}}t}\big(M_-(0,0,k)^{-1}M_{+}(0,0,k)\big),
\end{align}
for $(x,t) \in \R \times [0,T)$ and $k \in \Gamma \setminus \hat{\mathcal{Q}}$. Equation (\ref{Mjumpcondition}) now follows from Lemma \ref{Mjumplemma}.
\end{proof}

\begin{proposition}[Time evolution of the scattering data]\label{reflectionprop}
Let $T \in (0, \infty]$ and suppose $\{u, v\}$ is a Schwartz class solution of \eqref{boussinesqsystem} on $\R \times [0,T)$ with initial data $u_0, v_0 \in \mathcal{S}(\R)$ such that Assumptions \ref{solitonlessassumption} and \ref{originassumption} hold.
Let $\{r_j(k)\}_1^2$ be the reflection coefficients associated to $\{u_0, v_0\}$ via (\ref{r1r2def}) and let $\{r_j(k; t)\}_1^2$ be the reflection coefficients associated to $\{u(\cdot,t), v(\cdot, t)\}$. Then
\begin{align}\label{rel between rj at different time}
r_1(k; t) = r_1(k) e^{-\theta_{21}(0,t,k)} \quad \text{and} \quad r_2(k; t) = r_2(k) e^{\theta_{21}(0,t,k)}.
\end{align}
The first and second identities in \eqref{rel between rj at different time} are valid for $k \in \hat{\Gamma}_{1}$ and $k \in \hat{\Gamma}_{4}\setminus \{\omega^{2}, -\omega^{2}\}$, respectively.
\end{proposition}
\begin{proof}
Define $M(x,t,k)$ as in Proposition \ref{RHth}, and let $\tilde{v}(x, t, k)$ be the $3\times 3$ matrix obtained by replacing $\{r_j(k)\}_1^2$ by $\{r_j(k; t)\}_1^2$ in $v(x, 0, k)$. By following the proof of Lemma \ref{Mjumplemma} with $(u_{0}(x),v_{0}(x))$ replaced by $(u(x,t),v(x,t))$, we find $M_+ = M_- \tilde{v}$ for $k \in \Gamma \setminus (\Gamma_\star \cup \mathcal{Q})$.
On the other hand, we know from Proposition \ref{RHth} that $M$ satisfies RH problem \ref{RH problem for M}, and in particular \eqref{relation between M at time t and time 0}. This implies that $\tilde{v} = v$, where $v$ is given by \eqref{vdef}. Looking for example at the $(12)$ and $(21)$ entries of $v_{8}(x,t,k)$, we obtain the relations
\begin{align*}
& r_{1}(k,t)e^{-\theta_{21}(x,0,k)} = r_{1}(k)e^{-\theta_{21}(x,t,k)}, \quad r_{2}(k,t)e^{\theta_{21}(x,0,k)} = r_{2}(k)e^{\theta_{21}(x,t,k)}
\end{align*}
for $k \in \Gamma_{8}$. By \eqref{def of Phi ij}, we have $e^{\theta_{21}(x,t,k)} = e^{\theta_{21}(x,0,k)} e^{\theta_{21}(0,t,k)}$, which proves \eqref{rel between rj at different time}  for $k \in \Gamma_{8}$. The proof of \eqref{rel between rj at different time} for $k$ on the other indicated parts of $\Gamma$ can be proved similarly.
\end{proof}

In light of (\ref{mathcalMcoefficients}), the singularities of $M$ at the points $\kappa_j$ disappear when $M$ is multiplied from the left by $(1,1,1)$. This suggests considering the row vector $n =(1,1,1)M$.

\begin{proposition}\label{prop:construction of n}
Let $T \in (0, \infty]$ and suppose $\{u(x,t), v(x,t)\}$ is a Schwartz class solution of \eqref{boussinesqsystem} on $\R \times [0,T)$ with initial data $u_0, v_0 \in \mathcal{S}(\R)$ such that Assumptions \ref{solitonlessassumption} and \ref{originassumption} hold.
Let $M(x,t,k)$ be defined as in the proof of Proposition \ref{RHth}, and let $n(x,t,k):=(1,1,1)M(x,t,k)$. Then $n$ solves RH problem \ref{RHn}. Moreover,
\begin{align}\nonumber
& u(x,t) := -i\sqrt{3} \frac{\partial}{\partial x} n_3^{(1)}(x,t),
	\\ \label{uvdef}
& v(x,t) := -i\sqrt{3} \frac{\partial}{\partial x}\bigg(n_{3}^{(2)}(x,t) + u(x,t) + \frac{1}{6} \bigg( \int_{\infty}^{x}u(x',t)dx' \bigg)^{2}\bigg),
\end{align}
where $n^{(1)}$ and $n^{(2)}$ are defined through the expansion 
\begin{align}\label{expansion of n at inf}
n(x,t,k) = (1,1,1) + n^{(1)}(x,t) \, k^{-1} + n^{(2)}(x,t) \, k^{-2} + O(k^{-3}) \qquad \mbox{as } k \to \infty.
\end{align}
\end{proposition}
\begin{proof}
We know from Proposition \ref{RHth} that $M$ solves RH problem \ref{RH problem for M}. The fact that $n$ solves RH problem \ref{RHn} follows directly from the definition $n(x,t,k)=(1,1,1)M(x,t,k)$. The existence of the expansion \eqref{expansion of n at inf} is a consequence of \eqref{asymp for M at infty in RH def}. The formulas \eqref{uvdef} follow directly from the relations $n^{(1)}(x,t)=(1,1,1)M^{(1)}(x,t)$, $n^{(2)}(x,t)=(1,1,1)M^{(2)}(x,t)$, and the fact that $M^{(1)}(x,t)$ and $M^{(2)}(x,t)$ are given by \eqref{def of X1}--\eqref{def of X2} with $(u_{0},v_{0})$ replaced by $(u(\cdot,t),v(\cdot,t))$ (recall that $M^{(1)}(x,t)=X_{1}(x,t)$ and $M^{(2)}(x,t)=X_{2}(x,t)$).
\end{proof}

\section{The inverse problem}\label{inversesec}
Throughout this section, we assume that $r_1:\hat{\Gamma}_{1}\setminus \hat{\mathcal{Q}} \to \C$ and $r_2: \hat{\Gamma}_{4}\setminus \hat{\mathcal{Q}} \to \C$ are two functions satisfying $(\ref{Theorem2.3itemi})$--$(\ref{Theorem2.3itemv})$ of Theorem \ref{directth}. We assume that $v$ is the jump matrix in (\ref{vdef}) and that $T \in (0, \infty]$ is defined by (\ref{Tdef}). The main goal of this section is to prove Theorem \ref{inverseth}. Direct calculations using (\ref{r1r2 relation on the unit circle}) and (\ref{r1r2 relation with kbar symmetry}) yield the following lemma. Recall that $R(k)$ is defined in (\ref{def of R}). 

\begin{lemma}[Complex conjugation symmetry of $v$]\label{vsymmlemma}
The jump matrix $v$ defined in (\ref{vdef}) satisfies $(\overline{v(\bar{k})^{-1}})^T = R(k)^{-1}  v(k) R(k)$ for $k \in \partial \D \setminus \Gamma_\star$ and $(\overline{v(\bar{k})^{-1}})^T = R(k)^{-1}  v(k)^{-1} R(k)$ for $k \in \Gamma \setminus \partial \D$.
\end{lemma}

\subsection{A vanishing lemma}\label{vanishinglemmasubsec}
The following vanishing lemma will be used to show existence of a solution $n(x,t,k)$ of RH problem \ref{RHn}. 

\begin{lemma}\label{vanishinglemma}
Let $(x,t) \in \R \times [0,T)$. Suppose $n^h(x,t,k)$ is a solution of the homogeneous version of RH problem \ref{RHn}, that is, suppose $n^h$ satisfies $(\ref{RHnitema})$--$(\ref{RHnitemd})$ of RH problem \ref{RHn} together with the homogeneous condition $n^h(x,t,k) = O(k^{-1})$ as $k \to \infty$.
Then $n^h$ vanishes identically. 
\end{lemma}
\begin{proof}
To shorten notation, let us write $n^h(x,t,k) \equiv \mathbf{u}(k) = (u_1(k), u_2(k), u_3(k))$. Recall the definition (\ref{lambdadef}) of the spectral parameter $\lambda = (k^3 + k^{-3})/2$.
The mapping $k \mapsto \lambda$ is six-to-one except at isolated points and is represented in Figure \ref{klambdamapfig}.
\begin{figure}
\begin{center}
\begin{tikzpicture}[scale=0.8]
\node at (0,0) {};
\draw[blue,line width=0.65 mm,-<-=0.2,->-=0.8] (0,0)--(0:4);
\draw[red,line width=0.65 mm,->-=0.3,-<-=0.7] (0,0)--(60:4);
\draw[blue,line width=0.65 mm,-<-=0.2,->-=0.8] (0,0)--(120:4);
\draw[red,line width=0.65 mm,->-=0.3,-<-=0.7] (0,0)--(180:4);
\draw[blue,line width=0.65 mm,-<-=0.2,->-=0.8] (0,0)--(-120:4);
\draw[red,line width=0.65 mm,->-=0.3,-<-=0.7] (0,0)--(-60:4);

\draw[black,line width=0.65 mm] ([shift=(-180:2cm)]0,0) arc (-180:-150:2cm);
\draw[green,line width=0.65 mm] ([shift=(-150:2cm)]0,0) arc (-150:-120:2cm);
\draw[green,line width=0.65 mm] ([shift=(-120:2cm)]0,0) arc (-120:-90:2cm);
\draw[black,line width=0.65 mm] ([shift=(-90:2cm)]0,0) arc (-90:-60:2cm);
\draw[black,line width=0.65 mm] ([shift=(-60:2cm)]0,0) arc (-60:-30:2cm);
\draw[green,line width=0.65 mm] ([shift=(-30:2cm)]0,0) arc (-30:0:2cm);
\draw[green,line width=0.65 mm] ([shift=(0:2cm)]0,0) arc (0:30:2cm);
\draw[black,line width=0.65 mm] ([shift=(30:2cm)]0,0) arc (30:60:2cm);
\draw[black,line width=0.65 mm] ([shift=(60:2cm)]0,0) arc (60:90:2cm);
\draw[green,line width=0.65 mm] ([shift=(90:2cm)]0,0) arc (90:120:2cm);
\draw[green,line width=0.65 mm] ([shift=(120:2cm)]0,0) arc (120:150:2cm);
\draw[black,line width=0.65 mm] ([shift=(150:2cm)]0,0) arc (150:180:2cm);
\draw[green,arrows={-Triangle[length=0.27cm,width=0.22cm]}]
($(10:2)$) --  ++(-75:0.001);
\draw[black,arrows={-Triangle[length=0.27cm,width=0.22cm]}]
($(40:2)$) --  ++(-45:0.001);
\draw[black,arrows={-Triangle[length=0.27cm,width=0.22cm]}]
($(80:2)$) --  ++(165:0.001);
\draw[green,arrows={-Triangle[length=0.27cm,width=0.22cm]}]
($(110:2)$) --  ++(195:0.001);
\draw[green,arrows={-Triangle[length=0.27cm,width=0.22cm]}]
($(130:2)$) --  ++(45:0.001);
\draw[black,arrows={-Triangle[length=0.27cm,width=0.22cm]}]
($(160:2)$) --  ++(75:0.001);
\draw[black,arrows={-Triangle[length=0.27cm,width=0.22cm]}]
($(200:2)$) --  ++(-75:0.001);
\draw[green,arrows={-Triangle[length=0.27cm,width=0.22cm]}]
($(230:2)$) --  ++(-45:0.001);
\draw[green,arrows={-Triangle[length=0.27cm,width=0.22cm]}]
($(250:2)$) --  ++(165:0.001);
\draw[black,arrows={-Triangle[length=0.27cm,width=0.22cm]}]
($(280:2)$) --  ++(195:0.001);
\draw[black,arrows={-Triangle[length=0.27cm,width=0.22cm]}]
($(320:2)$) --  ++(45:0.001);
\draw[green,arrows={-Triangle[length=0.27cm,width=0.22cm]}]
($(350:2)$) --  ++(75:0.001);

\node at (45:2.35) {\small{$10$}};

\node at (15:2.3) {\small{$9$}};

\node at (-15:2.3) {\small{$8$}};

\node at (-45:2.3) {\small{$7$}};

\node at (-5:3) {\small $1$};
\node at (55:3) {\small $2$};
\node at (115:3) {\small $3$};
\node at (175:3) {\small $4$};
\node at (235:3) {\small $5$};
\node at (295:3) {\small $6$};

\node at (30:3.5) {\small $U_1$};
\node at (90:3.5) {\small $U_2$};
\node at (150.5:3.5) {\small $U_3$};
\node at (210:3.5) {\small $U_4$};
\node at (-90:3.5) {\small $U_5$};
\node at (-30:3.5) {\small $U_6$};

\node at (30:1.3) {\small $U_7$};
\node at (90:1.3) {\small $U_8$};
\node at (151.3:1.3) {\small $U_9$};
\node at (210:1.3) {\small $U_{10}$};
\node at (-90:1.3) {\small $U_{11}$};
\node at (-30:1.3) {\small $U_{12}$};

\node at (0.6:4.5) {\small $\Gamma^h$};

\end{tikzpicture}
\\
\vspace{0.5cm}
\begin{tikzpicture}
\node at (0,0) {};
\draw[red,line width=0.65 mm,->-=0.6] (-4,0)--(-1,0);
\draw[black,line width=0.65 mm,->-=0.7] (-1,0)--(0,0);
\draw[green,line width=0.65 mm,->-=0.7] (0,0)--(1,0);
\draw[blue,line width=0.65 mm,->-=0.6] (1,0)--(4,0);

\node at (-1,-0.3) {\small $-1$};
\node at (0,-0.3) {\small $0$};
\node at (1,-0.3) {\small $1$};
\node at (4.3,0) {\small $\R$};
\end{tikzpicture}

\end{center}
\begin{figuretext}\label{klambdamapfig}
The map $k \mapsto \lambda = \frac{k^3 + k^{-3}}{2}$ maps the contour $\Gamma^h$ to the real axis. In particular, the unit circle is mapped to the interval $[-1, 1]$.
\end{figuretext}
\end{figure}
Let $\Gamma^h$ be the contour in Figure \ref{klambdamapfig}. Let $\Gamma^h_j$ be the part of $\Gamma^h$ labeled by $j$ in the same figure.
For $\lambda < -1$, let $k_2^h(\lambda) \in \Gamma_2^h$ be the unique solution of $\lambda = \frac{k^3 + k^{-3}}{2}$ in $\Gamma_2^h$.
Define $\delta_3:\C \setminus (-\infty,-1] \to \C$ by
\begin{align}\label{d3def}
\delta_3(\lambda) = e^{\frac{1}{2\pi i} \int_{-\infty}^{-1} \ln\big(\omega^2 \tilde{r}(\omega k_2^h(\lambda'))\big) \frac{d\lambda'}{\lambda' - \lambda}}, \qquad \lambda \in \C \setminus (-\infty,-1],
\end{align}
where the branch of the logarithm is fixed as follows. A calculation shows that if $\lambda < -1$, then
\begin{align}\label{omega2rtilde}
\omega^2 \tilde{r}(\omega k_2^h(\lambda)) = \frac{k_2^h(\lambda)^2 -1}{k_2^h(\lambda)^2 - \omega^2} = e^{i\arctan(\frac{\sqrt{3} (2 r^2+1)}{2 r^4+2 r^2-1})},
\end{align}
where $r = |k_2^h(\lambda)| > 1$. Thus, we can fix the branch of the logarithm in (\ref{d3def}) by requiring that
\begin{align}\label{d3integrandpositive}
\frac{1}{2\pi i} \ln(\omega^2 \tilde{r}(\omega k_2^h(\lambda))) = \frac{1}{2\pi} \arctan\bigg(\frac{\sqrt{3} \left(2 r^2+1\right)}{2 r^4+2 r^2-1}\bigg) \in (0, 1/6)
\end{align}
for $\lambda < -1$.
Define $\delta_2: \C \setminus \Gamma \to \C$ and $F:(U_1 \cup U_6)\setminus \Gamma \to \C$ by
\begin{align}\label{d2def}
& \delta_2(k) := \tilde{r}(\omega^2 k) \delta_3(\lambda), \qquad k \in \C \setminus \Gamma,
	\\ \nonumber
& F(k) := \begin{cases}
\mathbf{u}(k) \Delta_3(k) \overline{\mathbf{u}(\bar{k})}^T, \qquad k \in U_1 \setminus \Gamma,
	\\
\mathbf{u}(k) \Delta_2(k) \overline{\mathbf{u}(\bar{k})}^T, \qquad k \in U_6 \setminus \Gamma,
\end{cases}
\end{align}
where $U_1$ and $U_6$ are the open subsets shown in Figure \ref{klambdamapfig} and
\begin{align}\label{D2D3def}
\Delta_2(k) = \begin{pmatrix} 0 & \delta_2(k) & 0 \\ 0 & 0 & 0 \\ 0 & 0 & 0 \end{pmatrix}, \qquad 
\Delta_3(k) = \begin{pmatrix} 0 & \delta_2(k) & 0 \\ 0 & 0 & 0 \\ 0 & 0 & \delta_3(\lambda) \end{pmatrix}.
\end{align}

For $k \in \Gamma_3 \cap \{|k| > 1\}$ (with $\Gamma$  oriented as in Figure \ref{fig: Dn}), we have
\begin{align*}
F_+(k) - F_-(k) 
& = \mathbf{u}_+(k) \Delta_3(k) \overline{\mathbf{u}_-(\bar{k})}^T
- \mathbf{u}_-(k) \Delta_3(k) \overline{\mathbf{u}_+(\bar{k})}^T
	\\
& = \mathbf{u}_-(k) \Big(v_{3}(k) \Delta_3(k) (\overline{v_{2}(\bar{k})}^{-1})^T  
- \Delta_3(k) \Big)\overline{\mathbf{u}_+(\bar{k})}^T = 0
\end{align*}
as a consequence of (\ref{d2def}), (\ref{D2D3def}), and Lemma \ref{vsymmlemma}.
Similarly, for $k \in \Gamma_2 \cap \{|k| > 1\}$,
\begin{align*}
F_+(k) - F_-(k) =
\mathbf{u}_-(k) \Big(v_{2}(k) \Delta_2(k) (\overline{v_{3}(\bar{k})}^{-1})^T - \Delta_2(k)\Big)\overline{\mathbf{u}_+(\bar{k})}^T = 0.
\end{align*}
This shows that $F$ has no jump across $(\Gamma_3 \cup \Gamma_2) \cap \{|k| > 1\}$. Hence $F$ extends to an analytic function $F:U_1 \cup U_6 \to \C$.
We extend $F$ further to a function $\C \setminus \Gamma^h \to \C$ by means of the symmetries
$$F(k) = F(\omega k), \qquad F(k) = F(1/k).$$
These symmetries imply that $F$ can be viewed as a function of $\lambda$ for $\lambda \in \C \setminus \R$. Define the analytic function $G:\C \setminus \R \to \C$ by
$$G(\lambda) := \frac{F(k)}{(\lambda + 1)^{2/3}}, \qquad \lambda \in \C \setminus \R,$$
where the principal branch is used for the complex power.
The map $\lambda \mapsto (\lambda + 1)^{2/3}$ is analytic $\C \setminus (-\infty,-1] \to \C$ and satisfies the following jump relation on $(-\infty, 1)$:
\begin{align}\label{lambdajump new}
\omega ((\lambda + 1)^{2/3})_+ = ((\lambda + 1)^{2/3})_-, \qquad \lambda < -1.
\end{align}
We next compute the jump of $G$ on the various parts of the real axis. 

\subsubsection*{Jump for $\lambda > 1$}
The jump of $G$ across the contour in the $\lambda$-plane where $\lambda > 1$ can be computed using any of the blue contours in the $k$-plane in Figure \ref{klambdamapfig}. We choose to compute it across $\Gamma_1^h$. Letting $k \in \Gamma_1^h$ correspond to $\lambda > 1$, we find
\begin{align}\nonumber
G_+(\lambda) - G_-(\lambda) & = \frac{1}{(\lambda +1)^{2/3}}\Big(\mathbf{u}(k) \Delta_3(k) \overline{\mathbf{u}(\bar{k})}^T - \mathbf{u}(k) \Delta_2(k) \overline{\mathbf{u}(\bar{k})}^T\Big)
	\\ \label{Gjump1}
& = \frac{1}{|\lambda + 1|^{2/3}} \mathbf{u}(k) (\Delta_3(k) - \Delta_2(k)) \overline{\mathbf{u}(\bar{k})}^T =  \frac{\delta_3(\lambda)}{|\lambda + 1|^{2/3}} |u_3(k)|^2, \qquad \lambda > 1.
\end{align}
It follows from (\ref{d3integrandpositive}) that $\delta_3(\lambda) > 0$ for $\lambda > -1$.
In particular, $\delta_3(\lambda)|\lambda + 1|^{-2/3} > 0$ for $\lambda > 1$.

\subsubsection*{Jump for $\lambda < -1$}
Letting $k \in \Gamma_2^h$ correspond to $\lambda < -1$, we find
\begin{align*}
G_+(\lambda) - G_-(\lambda) & =
\frac{F_+(k)}{((\lambda+1)^{2/3})_+} - \frac{F_-(k)}{((\lambda+1)^{2/3})_-} 
= \frac{F_+(k)}{((\lambda+1)^{2/3})_+} - \frac{F_-(\omega^2 k)}{((\lambda+1)^{2/3})_-} 
 	\\
&
= \frac{\mathbf{u}(k) \Delta_{3+}(k) \overline{\mathbf{u}(\bar{k})}^T}{((\lambda+1)^{2/3})_+}
- \frac{\mathbf{u}(\omega^2 k) \Delta_{2-}(\omega^2 k) \overline{\mathbf{u}(\omega \bar{k})}^T}{((\lambda+1)^{2/3})_-}, \qquad k \in \Gamma_2^h.
\end{align*}
Using \eqref{lambdajump new}, the fact that $\bar{k} = \omega^2 k$ for $k \in \Gamma_2^h$, and the symmetry $\mathbf{u}(k) = \mathbf{u}(\omega k) \mathcal{A}^{-1}$ (which follows from \eqref{nsymm}), we obtain, for $k \in \Gamma_2^h$,
\begin{align*}
G_+(\lambda) - G_-(\lambda) & = \frac{\mathbf{u}(k) \Delta_{3+}(k) \overline{\mathbf{u}(\omega^2 k)}^T}{((\lambda+1)^{2/3})_+}
- \frac{\mathbf{u}(\omega^2 k) \Delta_{2-}(\omega^2 k) \overline{\mathbf{u}(k)}^T}{\omega ((\lambda+1)^{2/3})_+}
	\\
& = \frac{1}{|\lambda + 1|^{2/3}} \bigg(\frac{\mathbf{u}(k) \Delta_{3+}(k) \overline{\mathbf{u}(k)\mathcal{A}^2}^T}{\omega}
- \frac{\mathbf{u}(k)\mathcal{A}^2 \Delta_{2-}(\omega^2 k) \overline{\mathbf{u}(k)}^T}{\omega^2}\bigg)
	\\
& = \frac{1}{|\lambda + 1|^{2/3}} \mathbf{u}(k)\bigg( \frac{\Delta_{3+}(k) \mathcal{A}}{\omega}
- \frac{\mathcal{A}^2 \Delta_{2-}(\omega^2 k) }{\omega^2}\bigg)\overline{\mathbf{u}(k)}^T
	\\
& = \frac{1}{|\lambda + 1|^{2/3}} \mathbf{u}(k)
\begin{pmatrix} \omega^2 \delta_{2+}(k) & 0 &  0 \\ 0 & 0 & 0 \\ 0 & 
\frac{-(k^2 - \omega^2)\delta_{3+}(\lambda) + (k^2 -1)\delta_{3-}(\lambda)}{1 - \omega k^2} & 0 \end{pmatrix}
\overline{\mathbf{u}(k)}^T.
\end{align*}
Since $\delta_3(\lambda)$ satisfies the jump relation
$$\frac{\delta_{3+}(\lambda)}{\delta_{3-}(\lambda)} = \frac{k_2^h(\lambda)^2 -1}{k_2^h(\lambda)^2 - \omega^2}\qquad \text{for $\lambda < -1$},$$
the $(32)$-entry of the above matrix vanishes. Moreover, by the Plemelj formula, the $(11)$-entry is given by
\begin{align*}
\omega^2 \delta_{2+}(k) & = \omega^2 \tilde{r}(\omega^2 k) \delta_{3+}(\lambda)
= \omega^2 \tilde{r}(\omega^2 k) e^{\frac{1}{2} \ln\big(\omega^2 \tilde{r}(\omega k)\big) }
e^{\frac{1}{2\pi i} \dashint_{-\infty}^{-1} \ln\big(\omega^2 \tilde{r}(\omega k_2^h(\lambda'))\big) \frac{d\lambda'}{\lambda' - \lambda}}.
\end{align*}
By (\ref{d3integrandpositive}),
$$e^{\frac{1}{2\pi i} \dashint_{-\infty}^{-1} \ln\big(\omega^2 \tilde{r}(\omega k_2^h(\lambda'))\big) \frac{d\lambda'}{\lambda' - \lambda}} > 0 \qquad \text{for $\lambda < -1$}.$$ 
Furthermore, a direct calculation shows that, if $k = r e^{\frac{\pi i}{3}} \in \Gamma_2^h$ with  $r >1$, then
$$\omega^2 \tilde{r}(\omega^2 k) e^{\frac{1}{2} \ln\big(\omega^2 \tilde{r}(\omega k)\big) } = \sqrt{\frac{(r^2-1)^2}{r^4 + r^2 + 1}} > 0.$$
(To fix the branch, note that the branch of the logarithm was fixed in (\ref{d3integrandpositive}) so that $\arg\big(\omega^2 \tilde{r}(\omega k)\big) \in (0, \pi/3)$ for $k \in \Gamma_2^h$, and that $\arg(\omega^2 \tilde{r}(\omega^2 k)) \in (-\pi/6,0)$ for $k \in \Gamma_2^h$.) 
We conclude that $\omega^2 \delta_{2+}(k) > 0$ for $k \in \Gamma_2^h$. Thus, if $\lambda < -1$ corresponds to  $k \in \Gamma_2^h$, we have
\begin{align}\label{Gjump2}
G_+(\lambda) - G_-(\lambda) & = \frac{ \omega^2 \delta_{2+}(k) }{|\lambda + 1|^{2/3}} |u_1(k)|^2, \qquad \text{where} \quad \omega^2 \delta_{2+}(k) > 0.
\end{align}

\subsubsection*{Jump for $0 < \lambda < 1$}
Letting $k \in \Gamma_9^h$ correspond to $\lambda \in (0,1)$, we find
\begin{align*}
G_+(\lambda) - G_-(\lambda) & =
\frac{F_+(k) - F_-(k)}{(\lambda + 1)^{2/3}} 
= \frac{F_+(k) - F_-(1/k)}{|\lambda + 1|^{2/3}} 
 	\\
&
= \frac{1}{|\lambda + 1|^{2/3}} \bigg(\mathbf{u}^2(k) \Delta_3(k) \overline{\mathbf{u}^2(\bar{k})}^T
- \mathbf{u}^2(1/k) \Delta_2(1/k) \overline{\mathbf{u}^2(1/\bar{k})}^T\bigg), \qquad k \in \Gamma_9^h,
\end{align*}
where $\mathbf{u}^n$ denotes the restriction of $\mathbf{u}$ to $D_n$ with $D_n$ as in Figure \ref{fig: Dn}.
Since $\mathbf{u}(k) = \mathbf{u}(k^{-1})\mathcal{B}$ by \eqref{nsymm} and since $\bar{k} = k^{-1}$ for $k \in \Gamma_{9}^{h}$, we obtain
\begin{align*}
G_+(\lambda) - G_-(\lambda) & = \frac{1}{|\lambda + 1|^{2/3}} \bigg(\mathbf{u}^2(k) \Delta_3(k) (\overline{\mathbf{u}^5(k)\mathcal{B}})^T
- \mathbf{u}^5(k)\mathcal{B} \Delta_2(1/k) \overline{\mathbf{u}^2(k)}^T\bigg)
	\\
& = \frac{1}{|\lambda + 1|^{2/3}} \bigg(\mathbf{u}^2(k) \Delta_3(k) (\overline{\mathbf{u}^2(k) v_8(k)\mathcal{B}})^T
- \mathbf{u}^2(k)v_8(k)\mathcal{B} \Delta_2(1/k) \overline{\mathbf{u}^2(k)}^T\bigg)
	\\
& = \frac{1}{|\lambda + 1|^{2/3}} \mathbf{u}^2(k) \big( \Delta_3(k) (\overline{v_8(k)\mathcal{B}})^T
- v_8(k) \mathcal{B} \Delta_2(1/k) \big)\overline{\mathbf{u}^2(k)}^T.
\end{align*}
Using also that $v_8(k) = \mathcal{B}v_8(\bar{k})^{-1} \mathcal{B}$ for $|k| =1$, the jump becomes
\begin{align}\nonumber
G_+(\lambda) - G_-(\lambda) & = \frac{1}{|\lambda + 1|^{2/3}} \mathbf{u}^2(k) \big( \Delta_3(k) (\overline{\mathcal{B}v_8(\bar{k})^{-1}})^T
- v_8(k) \mathcal{B} \Delta_2(1/k) \big)\overline{\mathbf{u}^2(k)}^T
	\\ \label{Gjump3}
& = \frac{1}{|\lambda + 1|^{2/3}} \mathbf{u}^2(k) Q_1(k) \overline{\mathbf{u}^2(k)}^T,
\end{align}
where, since $(\overline{v_8(\bar{k})^{-1}})^T = R(k)^{-1}  v_8(k) R(k)$ by Lemma \ref{vsymmlemma},
$$Q_1(k) := \Delta_3(k) R(k)^{-1}  v_8(k) R(k) \mathcal{B}
- v_8(k) \mathcal{B} \Delta_2(1/k).$$

Direct calculations using the symmetry \eqref{r1r2 relation on the unit circle} show that if $k = e^{i \alpha}$, then 
\begin{align*}
Q_1(k) = \delta_3(\lambda) \begin{pmatrix}
\frac{\sqrt{3} \cot(\alpha) - 1}{2}f(k) & 0 & r_1(\frac{1}{\omega^2 k}) - r_1(k) r_1(\omega k) \\
0 & \frac{\sqrt{3} \cot(\alpha) + 1}{2} & 0 \\
\overline{r_1(\frac{1}{\omega^2 k}) }- \overline{r_1(k) r_1(\omega k)} & 0 & 1 + r_1(\omega k)r_2(\omega k) 
\end{pmatrix},
\end{align*}
where $f(k)$ is given by (\ref{def of f}). In particular, $Q_1$ is Hermitian. Moreover, a calculation gives
$$\det Q_1(k) = \bigg(\frac{3}{4\sin^2\alpha} -1\bigg) \delta_3(\lambda)^3 f(k^{-1}).$$
Recall that $\delta_3(\lambda) > 0$ for $\lambda > -1$ and that, by Lemma \ref{inequalitieslemma}, $f(k) > 0$, $f(k^{-1}) > 0$, and $1+r_{1}(\omega k)r_{2}(\omega k)>0$ for $\alpha \in (0, \pi/3)$. It follows from these inequalities and Sylvester's criterion that $Q_1(k)$ is positive definite for $\alpha \in (0, \pi/3)$.

\subsubsection*{Jump for $-1 < \lambda < 0$}
Letting $k \in \Gamma_{10}^h$ correspond to $\lambda \in (-1,0)$, we find
\begin{align*}
G_+(\lambda) - G_-(\lambda) & =
\frac{F_+(k)}{(\lambda+1)^{2/3}} - \frac{F_-(k)}{(\lambda+1)^{2/3}} 
 = \frac{F_+(k) - F_-(1/k)}{|\lambda +1|^{2/3}}
 	\\
& = \frac{1}{|\lambda + 1|^{2/3}} \bigg(\mathbf{u}^3(k) \Delta_3(k) \overline{\mathbf{u}^1(\bar{k})}^T
- \mathbf{u}^1(1/k) \Delta_2(1/k) \overline{\mathbf{u}^3(1/\bar{k})}^T \bigg).
\end{align*}
Hence, using that $\bar{k} = k^{-1}$ and the symmetry $\mathbf{u}(k) = \mathbf{u}(k^{-1})\mathcal{B}$,
\begin{align*}
G_+(\lambda) - G_-(\lambda) & = \frac{1}{|\lambda + 1|^{2/3}} \big(\mathbf{u}^3(k) \Delta_3(k) (\overline{\mathbf{u}^6(k)\mathcal{B}})^T
- \mathbf{u}^6(k)\mathcal{B} \Delta_2(1/k) \overline{\mathbf{u}^3(k)}^T\big)
	\\
& = \frac{1}{|\lambda + 1|^{2/3}} \big(\mathbf{u}^6(k) v_9(k) \Delta_3(k) (\overline{\mathbf{u}^6(k)\mathcal{B}})^T
- \mathbf{u}^6(k)\mathcal{B} \Delta_2(1/k) (\overline{\mathbf{u}^6(k)v_9(k)})^T\big)
	\\
& = \frac{1}{|\lambda +1|^{2/3}} \mathbf{u}^6(k) \big(v_9(k) \Delta_3(k) \mathcal{B}
- \mathcal{B} \Delta_2(1/k) \overline{v_9(k)}^T \big)\overline{\mathbf{u}^6(k)}^T.
\end{align*}
Since $v_9(k) = \mathcal{B}v_7(\bar{k})^{-1}\mathcal{B}$, Lemma \ref{vsymmlemma} gives $\overline{v_9(k)}^T =  \mathcal{B}R(k)^{-1}  v_9(k) R(k)\mathcal{B}$ and hence
\begin{align}\label{Gjump4}
G_+(\lambda) - G_-(\lambda) & = \frac{1}{|\lambda +1|^{2/3}} \mathbf{u}^6(k) Q_2(k) \overline{\mathbf{u}^6(k)}^T,
\end{align}
where
$$Q_2(k) := v_9(k) \Delta_3(k) \mathcal{B}
- \mathcal{B} \Delta_2(1/k) \mathcal{B}R(k)^{-1}  v_9(k) R(k)\mathcal{B}.$$
Direct calculations using (\ref{r1r2 relation on the unit circle}) show that if $k = e^{i \alpha}$, then
$$Q_2(k) = \delta_3(\lambda) \begin{pmatrix} \frac{\sqrt{3}\cot(\alpha) -1}{2}(1 + r_1(\omega^2 k)r_2(\omega^2 k)) & 0 & -r_2(\omega^2 k) \\
0 & \frac{\sqrt{3}\cot(\alpha) + 1}{2}f(\frac{1}{k})
& 0 \\
-\overline{r_2(\omega^2 k)} & 0 & 1 \end{pmatrix}$$
and
$$\det Q_2(k) = \bigg(\frac{3}{4\sin^2\alpha} -1\bigg) \delta_3(\lambda)^3 f(k^{-1}).$$
As above, it follows from Sylvester's criterion and Lemma \ref{inequalitieslemma} that $Q_2(k)$ is positive definite for $k \in \Gamma_{10}^h$.

\subsubsection*{Final steps}
The function $G(\lambda)$ is analytic for $\lambda \in \C \setminus \R$ and has continuous boundary values on $\R \setminus \{-1,0,1\}$. 
As $\lambda \to \infty$, we have $G(\lambda) = O(\lambda^{-4/3})$, so Cauchy's theorem gives
$$\int_\R G_\pm(\lambda) d\lambda = 0 \qquad \text{and hence} \qquad 0 = \int_\R (G_+(\lambda) - G_-(\lambda))d\lambda.$$ 
But we have shown that $G_+ - G_- \geq 0$ on $\R \setminus \{-1,0,1\}$, so recalling the formulas (\ref{Gjump1})--(\ref{Gjump4}) for $G_+ - G_-$, we find that
\begin{align*}
& u_3(k) = 0 \;\; \text{for $k \in \Gamma_1^h$};
\quad u_1(k) = 0 \;\; \text{for $k \in \Gamma_2^h$};
\quad  \mathbf{u}^2(k) = 0 \;\; \text{for $k \in \Gamma_9^h$}; \quad
 \mathbf{u}^6(k) = 0 \;\; \text{for $k \in \Gamma_{10}^h$}.
\end{align*}	
These relations imply that $\mathbf{u}$ is identically zero for all $k \in \C \setminus \Gamma$. To see this, we can for example use the last relation and the jump relation $\mathbf{u}_+ = \mathbf{u}_- v$ to see that the boundary values $\mathbf{u}_\pm$ of $\mathbf{u}$ on $\Gamma_{10}^h$ vanish. An application of Morera's
theorem then shows that $\mathbf{u}$ is analytic at each point in $\Gamma_{10}^h$. But then it follows that $\mathbf{u} \equiv 0$ in the parts of $D_3$ and $D_6$ in the upper half-plane. Similarly, using that $\mathbf{u}^2 = 0$ on $\Gamma_9^h$, we conclude that $\mathbf{u} \equiv 0$ in the parts of $D_2$ and $D_5$ in the right half-plane.
Using the symmetry $\mathbf{u}(k) = \mathbf{u}(\omega k) \mathcal{A}^{-1}$, we see that $\mathbf{u}$ vanishes everywhere. 
\end{proof}

\subsection{Proof of Theorem \ref{inverseth}}\label{inversethsubsec}
Our first goal is to use the vanishing lemma (Lemma \ref{vanishinglemma}) together with Zhou's theory for RH problems in Sobolev spaces $H^N(\Gamma)$ \cite{Z1989} (see also \cite{TO2016}) to deduce existence of a solution $n(x,t,k)$ of RH problem \ref{RHn}. The main difficulty is that our vanishing lemma only applies to solutions satisfying the $\mathcal{A}$- and $\mathcal{B}$-symmetries (\ref{nsymm}). Thus, these symmetries have to be built into the functional analytic framework. This is not so difficult for the $\mathcal{A}$-symmetry, but requires new ideas for the $\mathcal{B}$-symmetry (which involves $k \to 1/k$).

Let $N \geq 1$ be an integer. Let $H^N(\Gamma)$ be the Sobolev space of functions $f \in L^2(\Gamma)$ with $N$ weak derivatives in $L^2$, see \cite[Definition 2.45]{TO2016}. Note that $\Gamma$ is an {\it admissible contour} in the sense of \cite[Definition 2.40]{TO2016}. 
Let $\Omega_+ = D_1 \cup D_3 \cup D_5$ and $\Omega_- = D_2 \cup D_4 \cup D_6$, so that $\C \setminus \Gamma$ is the disjoint union of $\Omega_+$ and $\Omega_-$. From now on in this proof, we assume that $\Gamma$ is oriented so that $\Omega_+$ lies on the left and $\Omega_-$ on the right (this orientation differs from Figure \ref{fig: Dn} only in that $\Gamma_{2}$, $\Gamma_{4}$, and $\Gamma_{6}$ are now oriented towards $0$). 
Define the closed linear subspace $H_z^N(\Gamma)$ of $H^N(\Gamma)$ by 
\begin{align}
  H_z^N(\Gamma) := \{h \in H^N(\Gamma) \, | \, \text{$h$ satisfies the $(N-1)$th-order zero-sum condition} \},
\end{align}
where the $(N-1)$th-order zero-sum condition is defined as in \cite[Definition 2.47]{TO2016}. Define the operator $\mathcal{C} = \mathcal{C}^\Gamma$ by
$$(\mathcal{C} f)(z) = \frac{1}{2\pi i} \int_\Gamma \frac{f(z')dz'}{z' - z}$$
and let $\mathcal{C}_\pm f$ denote the boundary values of $\mathcal{C}f$ on $\Gamma$.
The Cauchy operators $\mathcal{C}_\pm$ are bounded linear maps $H_z^N(\Gamma) \to H_z^N(\Gamma)$ (see \cite[Theorem 2.50]{TO2016}). If $f \in H_z^N(\Gamma)$, then $\mathcal{C}_\pm f \in H_z^N(\partial D)$ for every $D \Subset \Omega_\pm$, where $D \Subset \Omega_\pm$ means that $D$ is a connected component of $\Omega_\pm$.
Following Zhou \cite{Z1989}, define the Sobolev spaces $H_\pm^N(\Gamma)$ by (see also \cite[Eq. (2.41)]{TO2016})
$$H_\pm^N(\Gamma) := \{f \in L^2(\Gamma) \, | \, \text{$f \in H_z^N(\partial D)$ for every $D \Subset \Omega_\pm$}\}.$$
Throughout the proof, we assume that $(x,t) \in \R \times [0,T)$ with $T$ given by (\ref{Tdef}).



\begin{lemma}\label{vpmlemma}
For any integer $N \geq 1$, there exist $3 \times 3$-matrix valued functions $v^\pm$ such that
\begin{enumerate}[$(a)$]
\item $v = (v^-)^{-1}v^+$ on $\Gamma  \setminus \Gamma_\star$, 
\item $v^\pm, (v^\pm)^{-1} \in I + H_{\pm}^N(\Gamma)$,
\item $v^\pm(k) = \mathcal{A} v^\pm(\omega k) \mathcal{A}^{-1}
= \mathcal{B}v^\mp(k^{-1})\mathcal{B}$ for $k \in \Gamma$, and 
\item $w^+ := v^+ - I$ and $w^{-} := I - v^{-}$ are nilpotent.
\end{enumerate}
\end{lemma}
\begin{proof}
Straightforward calculations using the definitions \eqref{vdef}--\eqref{def of f} show that
\begin{align*}
v = \begin{cases} 
(v_8^-)^{-1}v_8^+, & k \in \Gamma_8,
	\\
(v_9^-)^{-1}v_9^+, & k \in \Gamma_9,
\end{cases}
\end{align*}
where
\begin{subequations}\label{v8v9}
\begin{align}
& v_8^- := \begin{pmatrix}
 1 & [-r_1(\frac{1}{\omega^2 k})
   r_2(\omega k )-r_1(k)]e^{-\theta_{21}} &
   -r_1(\frac{1}{\omega^2 k})e^{-\theta_{31}} \\
 0 & 1 & 0 \\
 0 & r_2(\omega k )e^{\theta_{32}} & 1 
 \end{pmatrix}, \\
 & v_8^+ := \begin{pmatrix}
 1 & 0 & 0 \\
 r_2(k)e^{\theta_{21}} & 1 & -r_1(\omega k )e^{-\theta_{32}} \\
 r_2(\frac{1}{\omega^2 k})e^{\theta_{31}} & 0 & 1
\end{pmatrix},
	\\
& v_9^- := \begin{pmatrix}
 1 & 0 & r_2(\omega^2 k)e^{-\theta_{31}} \\
 -r_1(\frac{1}{k})e^{\theta_{21}} & 1 & [-r_1(\frac{1}{k})
   r_2(\omega^2 k)-r_1(\omega k )]e^{-\theta_{32}} \\
 0 & 0 & 1 
 \end{pmatrix}, \\
& v_9^+ := \begin{pmatrix}
 1 & r_2(\frac{1}{k})e^{-\theta_{21}} & 0 \\
 0 & 1 & 0 \\
 -r_1(\omega^2 k)e^{\theta_{31}} & r_2(\omega k )e^{\theta_{32}} & 1 
\end{pmatrix}.
\end{align}
\end{subequations}
Using (\ref{r1r2 relation on the unit circle}), we can write 
\begin{subequations}\label{v8plusv9plus}
\begin{align}
& v_8^+ =  
\begin{pmatrix}
1 & 0 & 0 \\
 [-r_1(\omega k ) r_2(\frac{1}{\omega^2 k})-r_1(\frac{1}{k})]e^{\theta_{21}} & 1 & -r_1(\omega k )e^{-\theta_{32}} \\
 r_2(\frac{1}{\omega^2 k})e^{\theta_{31}} & 0 & 1 \\
\end{pmatrix},
	\\
& v_9^+ = \begin{pmatrix}
 1 & [-r_1(\frac{1}{\omega^2 k})
   r_2(\omega k )-r_1(k)]e^{-\theta_{21}} & 0 \\
 0 & 1 & 0 \\
 -r_1(\omega^2 k)e^{\theta_{31}} & r_2(\omega k)e^{\theta_{32}} & 1 
 \end{pmatrix}.
 \end{align}
\end{subequations}
Let us define $v_3^\pm$ and $v_6^\pm$ by
$$v_3^+ = v_9^+, \quad v_3^- = v_8^-, \quad v_6^+ = v_8^+, \quad v_6^- = v_9^-,$$
where $v_8^-, v_9^-$ are given by (\ref{v8v9}) and $v_8^+, v_9^+$ are given by (\ref{v8plusv9plus}). Then it follows that (recall that $\Gamma_6$ is now oriented toward $0$)
\begin{align*}
v = \begin{cases}
(v_3^-)^{-1}v_3^+, & k \in \Gamma_3,
	\\
(v_6^-)^{-1}v_6^+, & k \in \Gamma_6,
\end{cases}
\end{align*}
as desired.
Moreover, using again the identity (\ref{r1r2 relation on the unit circle}), we find that $v_8^\pm$ obey the symmetry
$$v_8^\pm(k) = \mathcal{B}v_8^\mp(k^{-1})\mathcal{B}, \qquad k \in \Gamma_8.$$
By construction, $v^\pm|_{\partial D}$ is zero-sum (to all orders) at $e^{\pi i/6}$ if $D \Subset \C \setminus \Gamma$ is one of the four open connected components with a corner at $e^{\pi i/6}$.
Moreover, if we define 
\begin{align*}
& v_1^\pm(k) = \mathcal{A}v_3^\pm(\omega k)\mathcal{A}^{-1}, \qquad k \in \Gamma_1;
&& v_2^\pm(k) = \mathcal{B}v_6^\mp(k^{-1})\mathcal{B}, \qquad k \in \Gamma_2;
	\\
& v_4^\pm(k) = \mathcal{A}v_6^\pm(\omega k)\mathcal{A}^{-1}, \qquad k \in \Gamma_4;
&& v_5^\pm(k) = \mathcal{B}v_3^\mp(k^{-1})\mathcal{B}, \qquad k \in \Gamma_5;
	\\
& v_7^\pm(k) = \mathcal{A}v_9^\pm(\omega k)\mathcal{A}^{-1}, \qquad k \in \Gamma_7,
\end{align*}
then 
$$v = \begin{cases}
(v_1^-)^{-1} v_1^+,\qquad k \in \Gamma_1, 
	\\
(v_2^-)^{-1} v_2^+,\qquad k \in \Gamma_2, 
	\\
(v_4^-)^{-1} v_4^+,\qquad k \in \Gamma_4, 
	\\
(v_7^-)^{-1} v_7^+,\qquad k \in \Gamma_7, 
\end{cases}$$
as desired. We have shown that $v = (v^-)^{-1} v^+$ everywhere on $\Gamma \setminus \Gamma_\star$. Moreover, $\det v^\pm = 1$ on $\Gamma$, so by construction we have $v^\pm, (v^\pm)^{-1} \in I + H_{\pm}^N(\Gamma; \C^{3 \times 3})$ for every $N \geq 1$.
It is now straightforward to verify that $v^\pm(k) = \mathcal{A} v^\pm(\omega k) \mathcal{A}^{-1}
= \mathcal{B}v^\mp(k^{-1})\mathcal{B}$ for $k \in \Gamma$ and that $w^\pm = \pm v^\pm \mp I$ are nilpotent. 
\end{proof}

We define the closed linear subspaces $H_{z, \sym}^N(\Gamma; \C^{1 \times 3})$ and $H_{z, \sym}^N(\Gamma; \C^{3 \times 3})$ of $H_z^N(\Gamma; \C^{1 \times 3})$ and $H_z^N(\Gamma; \C^{ 3 \times 3})$, respectively, by
\begin{align*}
 & H_{z, \sym}^N(\Gamma; \C^{1 \times 3}) := \{f \in H_z^N(\Gamma; \C^{1 \times 3}) \, | \, \text{$f(k) = f(\omega k) \mathcal{A}^{-1} = f(k^{-1}) \mathcal{B}$ for all $k \in \Gamma$}\},
	\\
&  H_{z, \sym}^N(\Gamma; \C^{3 \times 3}) := \{f \in H_z^N(\Gamma; \C^{3 \times 3}) \, | \, \text{$f(k) = \mathcal{A} f(\omega k) \mathcal{A}^{-1} = \mathcal{B} f(k^{-1}) \mathcal{B}$ for all $k \in \Gamma$}\}.
\end{align*}
Similarly, we define
\begin{align*}
& H_{\pm, \sym}^N(\Gamma; \C^{1 \times 3}) := \{f \in H_{\pm}^N(\Gamma; \C^{1 \times 3})  \, | \, \text{$f(k) = f(\omega k) \mathcal{A}^{-1} = f(k^{-1}) \mathcal{B}$ for all $k \in \Gamma$}\},
	\\
& H_{\pm, \sym}^N(\Gamma; \C^{3 \times 3}) := \{f \in H_{\pm}^N(\Gamma; \C^{3 \times 3})  \, | \, \text{$f(k) = \mathcal{A} f(\omega k) \mathcal{A}^{-1} = \mathcal{B} f(k^{-1}) \mathcal{B}$ for all $k \in \Gamma$}\}.
\end{align*}
Let $v^\pm$ and $w^\pm = \pm v^\pm \mp I$ be as in Lemma \ref{vpmlemma}. Define the operator $\mathcal{C}_w$ by
$$\mathcal{C}_wf := \mathcal{C}_-(fw^+) + \mathcal{C}_+(fw^-).$$
Since $\mathcal{C}_\pm$ and $f \mapsto f w^\pm$ are bounded linear maps $H_z^N(\Gamma) \to H_z^N(\Gamma)$, $\mathcal{C}_w$ is a bounded linear map $H_z^N(\Gamma) \to H_z^N(\Gamma)$.

\begin{lemma}\label{BCwlemma}
If $f \in H_{z, \sym}^1(\Gamma; \C^{1 \times 3})$, then
\begin{align}\label{BCwf}
(\mathcal{C}_wf)(k^{-1})\mathcal{B} = \frac{1}{2\pi i} \int_\Gamma \frac{k}{u} \frac{f(u)w^+(u)}{u - k_-} du
+ \frac{1}{2\pi i} \int_\Gamma \frac{k}{u}\frac{f(u)w^-(u)}{u - k_+} du, \qquad k \in \Gamma \setminus \Gamma_\star,
\end{align}
and
\begin{align}\label{intfww0}
\int_\Gamma \frac{f(k)(w^+(k) + w^-(k))}{k} dk = 0.
\end{align}
\end{lemma}
\begin{proof}
It follows from Lemma \ref{vpmlemma} that
\begin{align}\label{sym of wpm}
w^\pm(k) = \mathcal{A} w^\pm(\omega k) \mathcal{A}^{-1}
= - \mathcal{B}w^\mp(k^{-1})\mathcal{B}, \qquad k \in \Gamma \setminus \Gamma_\star.
\end{align}
Hence,
\begin{align*}
(\mathcal{C}_-(fw^+))(k^{-1})\mathcal{B} = &\; \frac{1}{2\pi i} \int_\Gamma \frac{f(s)w^+(s)}{s - (k^{-1})_-}ds \mathcal{B}
= - \frac{k}{2\pi i} \int_\Gamma \frac{f(1/s)w^-(1/s)}{k_+ s - 1}ds 
	\\
= &\; - \frac{k}{2\pi i} \int_{-\Gamma} \frac{f(u)w^-(u)}{k_+/u - 1} \bigg(-\frac{du}{u^2}\bigg) 
= \frac{1}{2\pi i} \int_\Gamma \frac{k}{u} \frac{f(u)w^-(u)}{u - k_+} du
\end{align*}
and, similarly,
\begin{align*}
(\mathcal{C}_+(fw^-))(k^{-1})\mathcal{B} = \frac{1}{2\pi i} \int_\Gamma \frac{k}{u}\frac{f(u)w^+(u)}{u - k_-} du.
\end{align*}
This gives (\ref{BCwf}). Moreover, the symmetries $f(k) = f(\omega k) \mathcal{A}^{-1} = f(k^{-1}) \mathcal{B}$ and (\ref{sym of wpm}) imply
$$\int_\Gamma \frac{f(k)  w^\pm(k) dk}{k} = \int_\Gamma \frac{f(\omega s)  w^\pm(\omega s) ds}{s}
= \int_\Gamma \frac{f(s) w^\pm(s) ds}{s} \mathcal{A}$$
and, using also that the map $k \mapsto k^{-1}$ reverses the orientation of $\Gamma$,
$$\int_\Gamma \frac{f(k) w^\pm(k) dk}{k} 
= -\int_{-\Gamma} \frac{f(s^{-1}) w^\pm(s^{-1}) ds}{s}
= -\int_{\Gamma} \frac{f(s) w^\mp(s) ds}{s} \mathcal{B}.$$
In terms of the column vectors $\mathbf{v}_\pm := (\int_\Gamma \frac{f(k) w^\pm(k) dk}{k})^T$, we can write these equations as
$$(I - \mathcal{A})^T \mathbf{v}_\pm = 0, \qquad \mathbf{v}_\pm + \mathcal{B} \mathbf{v}_\mp = 0.$$
In particular, 
$$(I - \mathcal{A})^T(\mathbf{v}_+ + \mathbf{v}_-) = (I + \mathcal{B})(\mathbf{v}_+ + \mathbf{v}_-) = 0.$$
But the nullspaces of the matrices $(I - \mathcal{A})^T$ and $I + \mathcal{B}$ are one-dimensional and are spanned by the linearly independent vectors $(1,1,1)$ and $(1, -1, 0)$, respectively. Hence $\mathbf{v}_+ + \mathbf{v}_- = 0$, which is (\ref{intfww0}).
\end{proof}

The next lemma shows that $\mathcal{C}_w$ preserves the $\mathcal{A}$- and $\mathcal{B}$-symmetries and thus defines a bounded linear operator on $H_{z, \sym}^N(\Gamma; \C^{1 \times 3})$.

\begin{lemma}\label{lemma:prop of B and Cw}
If $f \in H_{z, \sym}^1(\Gamma; \C^{1 \times 3})$, then 
\begin{align}\label{Cwsymm}
(\mathcal{C}_w f)(k) = (\mathcal{C}_w f)(\omega k) \mathcal{A}^{-1} = (\mathcal{C}_w f)(k^{-1}) \mathcal{B}, \qquad k \in \Gamma \setminus \Gamma_\star.
\end{align}
In particular, $\mathcal{C}_w$ is a bounded linear operator $H_{z, \sym}^N(\Gamma; \C^{1 \times 3}) \to H_{z, \sym}^N(\Gamma; \C^{1 \times 3})$ for each $N \geq 1$.
\end{lemma}
\begin{proof}
Let $f \in H_{z, \sym}^1(\Gamma; \C^{1 \times 3})$. Then, by \eqref{sym of wpm},
\begin{align}\nonumber
(\mathcal{C}_-(fw^+))(k) 
& = \frac{1}{2\pi i} \int_\Gamma \frac{f(s)w^+(s)}{s - k_-} ds 
= \frac{1}{2\pi i} \int_\Gamma \frac{f(\omega s)w^+(\omega s)}{s - k_-} ds \mathcal{A}^{-1}
= \frac{1}{2\pi i} \int_\Gamma \frac{f(u)w^+(u)}{\omega^{-1} u - k_-} \frac{du}{\omega} \mathcal{A}^{-1}
	\\
& = \frac{1}{2\pi i} \int_\Gamma \frac{f(u)w^+(u)}{u - \omega k_-} du \mathcal{A}^{-1}
= (\mathcal{C}_-(fw^+))(\omega k) \mathcal{A}^{-1},
\end{align}
and, similarly, $(\mathcal{C}_+(fw^-))(k) = (\mathcal{C}_+(fw^-))(\omega k) \mathcal{A}^{-1}$, which proves that $\mathcal{C}_w f$ obeys the $\mathcal{A}$-symmetry. 
On the other hand, the identity (\ref{BCwf}) implies that, for $k \in \Gamma \setminus \Gamma_\star$,
\begin{align*}
(\mathcal{C}_wf)(k^{-1})\mathcal{B} = & -\frac{1}{2} f(k)w^+(k) + \frac{1}{2\pi i} \dashint_\Gamma \frac{k}{u} \frac{f(u)w^+(u)}{u - k} du
	\\
& + \frac{1}{2} f(k)w^-(k)
+ \frac{1}{2\pi i} \dashint_\Gamma \frac{k}{u}\frac{f(u)w^-(u)}{u - k} du.
\end{align*}
Utilizing (\ref{intfww0}), we can write this as
\begin{align*}
(\mathcal{C}_wf)(k^{-1})\mathcal{B} = & -\frac{1}{2} f(k)w^+(k) + \frac{1}{2\pi i} \dashint_\Gamma \frac{f(u)w^+(u)}{u}  \Big(\frac{k}{u - k}  + 1\Big)du
	\\
& + \frac{1}{2} f(k)w^-(k)
+ \frac{1}{2\pi i} \dashint_\Gamma \frac{f(u)w^-(u) }{u}\Big(\frac{k}{u - k}  + 1\Big) du
	\\
= & -\frac{1}{2} f(k)w^+(k) + \frac{1}{2\pi i} \dashint_\Gamma \frac{f(u)w^+(u)}{u-k} du
	\\
& + \frac{1}{2} f(k)w^-(k)
+ \frac{1}{2\pi i} \dashint_\Gamma \frac{f(u)w^-(u) }{u-k}du
=  (\mathcal{C}_wf)(k) \qquad \text{for} \quad k \in \Gamma \setminus \Gamma_\star,
\end{align*}
which shows that $\mathcal{C}_w f$ obeys also the $\mathcal{B}$-symmetry. 
Since $\mathcal{C}_w$ is a bounded linear map $H_z^N(\Gamma) \to H_z^N(\Gamma)$ which, by (\ref{Cwsymm}), preserves the linear subspace $H_{z, \sym}^N(\Gamma)$, it follows that $\mathcal{C}_w$ is a bounded linear operator on $H_{z, \sym}^N(\Gamma; \C^{1 \times 3})$ for each $N \geq 1$. 
\end{proof}

\begin{lemma}\label{fredholmlemma}
The linear operator $I - \mathcal{C}_w: H_{z, \sym}^N(\Gamma; \C^{1 \times 3}) \to H_{z, \sym}^N(\Gamma; \C^{1 \times 3})$ is Fredholm for any $N \geq 1$.
\end{lemma}
\begin{proof}
Recall that $w^+ = v^+ - I$ and $w^- = I-v^-$ correspond to the factorization $v = (v^-)^{-1}v^+$ of $v$. Let 
\begin{align}\label{tildewdef}
\tilde{w}^+ = (v^+)^{-1} - I, \qquad \tilde{w}^- = I - (v^-)^{-1}.
\end{align}
By Lemma \ref{vpmlemma} $(b)$, $\tilde{w}^\pm \in  H_{\pm}^N(\Gamma; \C^{3 \times 3})$.
Using that $\mathcal{C}_+ - \mathcal{C}_- = I$, we obtain, for $h \in H_z^N(\Gamma; \C^{1 \times 3})$,
\begin{align*}
  \mathcal{C}_{\tilde{w}}\mathcal{C}_w h
  = &\; \mathcal{C}_+((\mathcal{C}_+(hw^-) + \mathcal{C}_-(hw^+))\tilde{w}^-)
 +  \mathcal{C}_-((\mathcal{C}_+(hw^-) + \mathcal{C}_-(hw^+))\tilde{w}^+)
	\\
 = &\; \mathcal{C}_+((hw^- + \mathcal{C}_-(hw^-) + \mathcal{C}_-(hw^+))\tilde{w}^-)
 +  \mathcal{C}_-((\mathcal{C}_+(hw^-) - hw^+ + \mathcal{C}_+(hw^+))\tilde{w}^+)
	\\
= &\; T_wh + \mathcal{C}_+(hw^-\tilde{w}^-) - \mathcal{C}_-(hw^+\tilde{w}^+),
\end{align*}
where
$$T_w h := \mathcal{C}_+(\mathcal{C}_-(hw^-)\tilde{w}^-) 
+ \mathcal{C}_+(\mathcal{C}_-(hw^+)\tilde{w}^-)
 +  \mathcal{C}_-(\mathcal{C}_+(hw^-)\tilde{w}^+)  
 + \mathcal{C}_-(\mathcal{C}_+(hw^+)\tilde{w}^+).$$
In view of the identities $w^+ \tilde{w}^+ = - w^+ - \tilde{w}^+$ and $w^- \tilde{w}^- = w^- + \tilde{w}^-$, we infer that
$$\mathcal{C}_{\tilde{w}}\mathcal{C}_w h = T_wh + \mathcal{C}_wh + \mathcal{C}_{\tilde{w}} h.$$
Hence
\begin{align}\label{TwCw}
I + T_w = (I - \mathcal{C}_{\tilde{w}})(I - \mathcal{C}_w).
\end{align}
Interchanging $w$ and $\tilde{w}$ in the above argument, we find
\begin{align}\label{TtildewCw}
I + T_{\tilde{w}} = (I - \mathcal{C}_w)(I - \mathcal{C}_{\tilde{w}}),
\end{align}
where
$$T_{\tilde{w}} h := \mathcal{C}_+(\mathcal{C}_-(h\tilde{w}^-)w^-) 
+ \mathcal{C}_+(\mathcal{C}_-(h\tilde{w}^+)w^-)
 +  \mathcal{C}_-(\mathcal{C}_+(h\tilde{w}^-)w^+)  
 + \mathcal{C}_-(\mathcal{C}_+(h\tilde{w}^+)w^+).$$

By standard arguments (see e.g. \cite[Lemma 2.60]{TO2016}), $T_w$ and $T_{\tilde{w}}$ are compact as operators $H_z^N(\Gamma) \to H_z^N(\Gamma)$.
The identities (\ref{TwCw}) and (\ref{TtildewCw}), together with Lemma \ref{lemma:prop of B and Cw}, imply that the restrictions of $T_w$ and $T_{\tilde{w}}$ to the closed linear subspace $H_{z, \sym}^N(\Gamma)$ map into $H_{z, \sym}^N(\Gamma)$.
We conclude that $T_w$ and $T_{\tilde{w}}$ are compact also as operators $H_{z, \sym}^N(\Gamma) \to H_{z, \sym}^N(\Gamma)$.
Thus, (\ref{TwCw}) and (\ref{TtildewCw}) show that $I - \mathcal{C}_w:H_{z, \sym}^N(\Gamma) \to H_{z, \sym}^N(\Gamma)$ is invertible modulo compact operators; hence $I - \mathcal{C}_w$ is Fredholm on $H_{z, \sym}^N(\Gamma)$.
\end{proof}

\begin{lemma}[Endpoint behavior of a Cauchy integral] \label{endpointlemma}
Let $\mathcal{N} \geq 0$ be an integer. Let $[a,b] \subset \R$ be a closed finite interval and suppose $g_0 \in \mathcal{C}^{\mathcal{N}} ([a,b],\C)$ is such that the $\mathcal{N}$th derivative of $g_0$ is uniformly H\"older continuous on $[a,b]$. Then the function
\begin{align}\label{f0def}
  f_0(z) := \int^b_a g_0(s)  \frac{ds}{s-z}, \qquad z \in \C \setminus [a, b],
\end{align} 
satisfies
\begin{align}\label{f0expansion} 
f_0(z) = \sum_{n=0}^\mathcal{N} \frac{g_0^{(n)}(b)}{n!}(z-b)^n\ln\frac{z-b}{z-a} 
+ \sum_{n=0}^{\mathcal{N}-1} C_n(z-b)^{n}
 + (z-b)^\mathcal{N} E(z), \qquad z\to b,
\end{align} 
uniformly for $\arg (z-b) \in (-\pi,\pi)$, where $E(z)$ is a bounded function tending to a definite limit as $z \to b$ along any path in $\C \setminus [a, b]$, and the complex constants $C_n$ are given by 
$$C_n = \int_{a}^{b} g_{n+1}(s) ds,  \qquad n = 0,1, \dots, \mathcal{N}-1,$$
with the functions $g_n\in C^{\mathcal{N}-n} ([a,b],\C)$ defined by
\begin{align}\label{gndef}
g_n(s) = \frac{g_{n-1}(s) - g_{n-1}(b)}{s-b} 
= \frac{g_0(s) - \sum_{l = 0}^{n-1}\frac{g_0^{(l)}(b)}{l!}(s-b)^l}{(s-b)^n}, \qquad n = 1, 2, \dots, \mathcal{N}.
\end{align}
\end{lemma}
\begin{proof}
We first show that $g_\mathcal{N}$ is uniformly H\"{o}lder continuous on $[a,b]$. Let $s_1 \leq s_2$ be any two points in $[a, b]$. Let us first consider the case $\mathcal{N} \geq 2$.
By Taylor's theorem for $g_0$ with the integral form of the remainder, we have
$$(\mathcal{N}-1)! g_\mathcal{N}(s) = \frac{\int_b^s g_0^{(\mathcal{N})}(t) (s - t)^{\mathcal{N}-1} dt}{(s-b)^\mathcal{N}} 
= \frac{\int_b^s (g_0^{(\mathcal{N})}(t) - g_0^{(\mathcal{N})}(b)) (s - t)^{\mathcal{N}-1} dt}{(s-b)^\mathcal{N}} 
+ \frac{g_0^{(\mathcal{N})}(b)}{\mathcal{N}}$$
and hence
\begin{align}\nonumber
  (\mathcal{N}-1)!( g_\mathcal{N}(s_2) - g_\mathcal{N}(s_1))
 = &\; \bigg(\frac{1}{(s_2-b)^\mathcal{N}}  - \frac{1}{(s_1-b)^\mathcal{N}}\bigg) \int_b^{s_2} (g_0^{(\mathcal{N})}(t) - g_0^{(\mathcal{N})}(b)) (s_2 - t)^{\mathcal{N}-1} dt
	\\\nonumber
&  +\frac{1}{(s_1-b)^\mathcal{N}}\int_b^{s_2} (g_0^{(\mathcal{N})}(t) - g_0^{(\mathcal{N})}(b)) [(s_2 - t)^{\mathcal{N}-1} - (s_1 - t)^{\mathcal{N}-1}]dt
	\\\label{gMs1gMs2}
&  - \frac{1}{(s_1-b)^\mathcal{N}} \int_{s_2}^{s_1} (g_0^{(\mathcal{N})}(t) - g_0^{(\mathcal{N})}(b)) (s_1 - t)^{\mathcal{N}-1} dt.
\end{align}
By the uniform H\"older continuity of $g_0^{(\mathcal{N})}$, there are constants $C_1 > 0$ and $\alpha \in (0,1]$ such that $|g_0^{(\mathcal{N})}(t) - g_0^{(\mathcal{N})}(b)| \leq C_1(b-t)^\alpha$ for all $t \in [a,b]$. Using this in (\ref{gMs1gMs2}), straightforward estimates yield 
\begin{align*}
  |g_\mathcal{N}(s_2) - g_\mathcal{N}(s_1)|
 \leq &\; C \bigg|\frac{(s_1-b)^\mathcal{N} - (s_2-b)^\mathcal{N}}{(s_2-b)^\mathcal{N}(s_1-b)^\mathcal{N}} \bigg| 
 \int_{s_2}^b (b-t)^\alpha (t-s_2)^{\mathcal{N}-1} dt
	\\
&  + \frac{C}{(b-s_1)^\mathcal{N}}\int_{s_2}^b (b-t)^\alpha (s_2-s_1) (b-s_1)^{\mathcal{N}-2}dt
	\\
&  + \frac{C}{(b-s_1)^\mathcal{N}} \int_{s_1}^{s_2} (b-t)^\alpha (t-s_1)^{\mathcal{N}-1} dt
	\\
\leq &\; C \frac{(s_2-s_1) (b-s_1)^{\mathcal{N}-1}}{(b-s_2)^\mathcal{N}(b-s_1)^\mathcal{N}} (b-s_2)^{\alpha+\mathcal{N}}
  + \frac{C}{(b-s_1)^\mathcal{N}}(b-s_1)^\alpha (s_2-s_1) (b-s_1)^{\mathcal{N}-1}
	\\
&  + \frac{C}{(b-s_1)^\mathcal{N}} (b-s_1)^\alpha  (s_2 - s_1)^\mathcal{N}
	\\
\leq &\; \frac{C(s_2-s_1)}{(b-s_1)^{1-\alpha}}
\leq  C(s_2-s_1)^\alpha,
\end{align*}
showing that $g_\mathcal{N}$ is uniformly H\"older continuous on $[a,b]$ with the same H\"older exponent $\alpha$ as $g_0^{(\mathcal{N})}$. The same conclusion holds for $\mathcal{N} = 0$ and $\mathcal{N} = 1$ by similar but simpler arguments.

For $n = 0, 1, \dots, \mathcal{N}$, let
\begin{align*}
  f_n(z) := \int_{a}^{b} g_n(s) \frac{ds}{s-z}
\end{align*}
where $\{g_n\}_1^\mathcal{N}$ are defined in \eqref{gndef}. Using the relation
\begin{equation*}
	\frac{g_n(s)}{s-z} = \frac{g_{n+1}(s)}{s-z}(z-b)
	+ \frac{g_n(b)}{s-z} + g_{n+1}(s), \qquad n = 0,1, \dots, \mathcal{N}-1,
\end{equation*}
we see that
\begin{align*}
f_n(z) = f_{n+1}(z)(z-b) + g_n(b) \ln\frac{z-b}{z-a} + C_n, \qquad n = 0,1, \dots, \mathcal{N}-1,
\end{align*}
and recursive use of this equation gives
\begin{align}\nonumber
f_0(z) & = f_1(z) (z-b) + g_0(b)\ln\frac{z-b}{z-a} + C_0
	\\\nonumber
& = \Big(f_2(z) (z-b) + g_1(b)\ln\frac{z-b}{z-a}  + C_1\Big) (z-b) + g_0(b)\ln\frac{z-b}{z-a}  + C_0 = \cdots
	\\ \label{f0recursive}
 & =  f_\mathcal{N}(z)(z-b)^\mathcal{N}
+ \sum_{n=0}^{\mathcal{N}-1}  g_n(b)(z-b)^n\ln\frac{z-b}{z-a} 
+ \sum_{n=0}^{\mathcal{N}-1} C_n (z-b)^n.
\end{align}
Since $g_{\mathcal{N}}$ is uniformly H\"older continuous on $[a,b]$, \cite[Eq. (29.4)]{M1992} implies that
\begin{equation*}
	f_{\mathcal{N}}(z) = g_{\mathcal{N}}(b)\ln \frac{z-b}{z-a} + E(z),
\end{equation*}
where $E(z)$ is a bounded function tending to a definite limit as $z \to b$ along any path. Since
\begin{equation*}
g_n(b) = \frac{g_0^{(n)}(b)}{n!}, \qquad n = 0, 1, \dots, \mathcal{N},	
\end{equation*}
the expansion \eqref{f0expansion} follows. 
\end{proof}

\begin{lemma}\label{Cfsymmlemma}
If $f \in H_{z, \sym}^1(\Gamma; \C^{1 \times 3})$ is a row-vector valued function and $h := f(w^+ + w^-)$, then $\mathcal{C}h$ obeys the $\mathcal{A}$- and $\mathcal{B}$-symmetries, i.e., for all $k \in \C \setminus \Gamma$,
$$(\mathcal{C}h)(k) = (\mathcal{C}h)(\omega k) \mathcal{A}^{-1} = (\mathcal{C}h)(k^{-1}) \mathcal{B}.$$
\end{lemma}
\begin{proof}
Using (\ref{sym of wpm}), we see that $h(k) = h(\omega k)\mathcal{A}^{-1} = -h(k^{-1})\mathcal{B}$ for $k \in \Gamma$.
Thus, for $k \in \C \setminus \Gamma$,
\begin{align*}
(\mathcal{C}h)(\omega k) \mathcal{A}^{-1} 
= \frac{1}{2\pi i} \int_\Gamma \frac{h(u) du}{u - \omega k}\mathcal{A}^{-1} 
= \frac{1}{2\pi i} \int_\Gamma \frac{h(\omega s) \omega ds}{\omega s - \omega k}\mathcal{A}^{-1} 
= \frac{1}{2\pi i} \int_\Gamma \frac{h(s) ds}{s - k}
= (\mathcal{C}h)(k),
\end{align*}
and, recalling that $\int_\Gamma s^{-1}h(s)ds = 0$ by Lemma \ref{BCwlemma},
\begin{align*}
(\mathcal{C}h)(k^{-1}) \mathcal{B}
& = \frac{1}{2\pi i} \int_\Gamma \frac{h(u) du}{u - k^{-1}}\mathcal{B} 
= - \frac{1}{2\pi i} \int_{-\Gamma} \frac{h(s^{-1}) s^{-2} ds}{s^{-1} - k^{-1}}\mathcal{B}
= -\frac{1}{2\pi i} \int_\Gamma \frac{k}{s} \frac{h(s) ds}{k - s}
	\\
& = - \frac{1}{2\pi i} \int_\Gamma \Big(\frac{k}{s} \frac{1}{k - s} - \frac{1}{s}\Big)h(s) ds
= - \frac{1}{2\pi i} \int_\Gamma \frac{1}{k-s} h(s) ds
= (\mathcal{C}h)(k).
\end{align*}
\end{proof}

\begin{lemma}\label{wpmdecaylemma}
  Let $T$ be defined by (\ref{Tdef}). For each $(x,t) \in \R \times [0,T)$, the functions $k \mapsto w^\pm(x,t,k)$ and their derivatives are rapidly decreasing as $\Gamma \ni k \to \infty$ and as $\Gamma \ni k \to 0$. 
\end{lemma}
\begin{proof}
 The matrices $w^\pm$ involve $r_1(k)e^{- \theta_{21}}$ and $r_1(1/k)e^{\theta_{21}}$ on $\Gamma_1$, and $r_2(1/k)e^{- \theta_{21}}$ and $r_2(k)e^{\theta_{21}}$ on $\Gamma_4$. For $k \in i\R$, we have
$$e^{\theta_{21}(x,t,k)} = e^{\pm i\frac{|k|^2 + 1}{2|k|}x} e^{\frac{|k|^4 - 1}{4|k|^2}t},\qquad \mp ik > 0,$$
and hence the definition of $T$ implies that $r_1(1/k)e^{\theta_{21}(x,t,k)}$ and its derivatives have rapid decay as $k \in (-i, -i\infty) \subset \Gamma_1$ tends to infinity for any $(x,t) \in \R \times [0, T)$. For $k \in (0, i) \subset \Gamma_1$, $|e^{\theta_{21}(x,t,k)}| \leq 1$ for $t \geq 0$, and so $r_1(1/k)e^{\theta_{21}(x,t,k)}$ 
and its derivatives have rapid decay as $\Gamma_1 \ni k \to 0$ thanks to (\ref{r1r2rapiddecay}). Since $\theta_{21}(x,t,k) = - \theta_{21}(x,t,1/k)$, it follows that $r_1(k)e^{-\theta_{21}} $ is also rapidly decreasing as $k \in \Gamma_1$ tends to $0$ and $\infty$. By the symmetry (\ref{r1r2 relation with kbar symmetry}), it then follows that $r_2(k) e^{\theta_{21}}$ and $r_2(1/k) e^{-\theta_{21}}$ are rapidly decreasing as $k \in \Gamma_4$ tends to $0$ and $\infty$ for any $(x,t) \in \R \times [0, T)$. By the $\mathcal{A}$-symmetry, these decay properties extend so that $w^\pm$ have the appropriate decay properties as $k \to \infty$ and as $k \to 0$ along any ray in $\Gamma$.
\end{proof}

\begin{lemma}\label{atinfinitylemma}
If $f \in H_{z, \sym}^2(\Gamma; \C^{1 \times 3})$ and $h := f(w^+ + w^-)$, then
$$(\mathcal{C}h)(k) = O(k^{-1}), \qquad k \in \C \setminus \Gamma, \;\; k \to \infty,$$
uniformly with respect to $\arg k \in [0, 2\pi]$. 
More generally, if $f \in H_{z, \sym}^N(\Gamma; \C^{1 \times 3})$ for some integer $N \geq 2$, then there exist constants $\{c_j\}_1^{N-2} \subset \C$  such that
$$(\mathcal{C}h)(k) = \frac{c_1}{k} + \frac{c_2}{k^2} + \cdots + \frac{c_{N-2}}{k^{N-2}} + O\Big(\frac{1}{k^{N-1}}\Big), \qquad k \in \C \setminus \Gamma, \;\; k \to \infty,$$
uniformly with respect to $\arg k \in [0, 2\pi]$. 
\end{lemma}
\begin{proof}
Let $f \in H_{z, \sym}^2(\Gamma; \C^{1 \times 3})$. By Lemma \ref{wpmdecaylemma}, the functions $w^\pm$ vanish to all orders as $k \in \Gamma$ tends to $0$. Furthermore, since $f \in H_z^1(\Gamma)$, $\mathcal{C} f$ is uniformly $1/2$-H\"older continuous in each $D \Subset D_\epsilon(0) \setminus \Gamma$, and hence in $\bar{D}$, whenever $\epsilon > 0$ is small enough (see \cite[Lemma 2.51]{TO2016}). Thus, taking $k \to 0$ and then using (\ref{intfww0}), we obtain
\begin{align}\label{Cpmhat0}
\lim_{k\to 0, k \in \C \setminus \Gamma} (\mathcal{C} h)(k) =  \frac{1}{2\pi i} \int_\Gamma \frac{h(u) du}{u} = 0.
\end{align}
The function $h$ is in $H^2(\Gamma)$, and hence, by standard Sobolev embeddings, it is $C^1$ on $\Gamma \setminus \Gamma_\star$ and its derivative is uniformly H\"older continuous with exponent $1/2$ on each subarc of $\Gamma$. Since $h$ and its derivative vanish at the origin, it follows (using Lemma \ref{endpointlemma} with $\mathcal{N} = 1$ and (\ref{Cpmhat0})) that $(\mathcal{C}h)(k) = O(k)$ as $k \to 0$. The symmetry $(\mathcal{C}h)(k) = (\mathcal{C}h)(k^{-1})\mathcal{B}$, established in Lemma \ref{Cfsymmlemma}, then shows that $(\mathcal{C}h)(k) = O(k^{-1})$ as $k \to \infty$. 

More generally, suppose $f \in H_{z, \sym}^{N}(\Gamma; \C^{1 \times 3})$ with $N \geq 2$.
Then, by Sobolev embeddings, $h$ belongs to the H\"older space $C^{N-1, 1/2}(\Gamma\setminus \Gamma_\star)$. Since $\lim_{k\to 0} h^{(j)}(k) = 0$ for $j = 0, 1, \dots, N-1$, Lemma \ref{endpointlemma} with $\mathcal{N} = N-1$ together with (\ref{Cpmhat0}) shows that there exist complex constants $\{C_j\}_1^{N-2}$ such that $(\mathcal{C}h)(k) = \sum_{j=1}^{N-2} C_j k^j + O(k^{N-1})$ as $k \to 0$, $k \notin \Gamma$, uniformly for $\arg k \in [0,2\pi]$. The desired assertion then follows from the symmetry $(\mathcal{C}h)(k) = (\mathcal{C}h)(k^{-1})\mathcal{B}$.
\end{proof}

\begin{lemma}\label{injectivelemma}
For each integer $N \geq 2$, the operator $I - \mathcal{C}_w: H_{z, \sym}^N(\Gamma; \C^{1 \times 3}) \to H_{z, \sym}^N(\Gamma; \C^{1 \times 3})$ is injective. 
\end{lemma}
\begin{proof}
Suppose $f \in H_{z, \sym}^N(\Gamma; \C^{1 \times 3})$ satisfies $(I - \mathcal{C}_w)f = 0$. Let $n^h = \mathcal{C}(f (w^+ + w^-))$. 
Since $f (w^+ + w^-) \in H_z^1(\Gamma)$, $n^h(k)$ is an analytic function of $k \in \C \setminus \Gamma$ with continuous boundary values on $\Gamma \setminus \Gamma_\star$ such that $n^h(k) = O(1)$ as $k$ approaches any point in $\Gamma_\star$ (see \cite[Lemma 2.51]{TO2016} for the behavior near the intersection points). 
Moreover, since $(I - \mathcal{C}_w)f = 0$,
$$n^h_+ = \mathcal{C}_+ (f (w^+ + w^-)) 
= \mathcal{C}_+ (f w^+) + f -  \mathcal{C}_- (f w^+)
= f w^+ + f = f v^+,$$
$$n^h_- = \mathcal{C}_- (f (w^+ + w^-)) 
= f -  \mathcal{C}_+ (f w^-) + \mathcal{C}_- (f w^-)
= f - fw^-
= fv^-,$$
so it follows that $n^h_+ = n^h_-v$ a.e. on $\Gamma$. 
Lemma \ref{Cfsymmlemma} shows that $n^h$ obeys the $\mathcal{A}$- and $\mathcal{B}$-symmetries, and Lemma \ref{atinfinitylemma} shows that $n^h(k) = O(k^{-1})$ as $k \to \infty$. 
We conclude that $n^h$ satisfies the assumptions of Lemma \ref{vanishinglemma}. Applying Lemma \ref{vanishinglemma}, we infer that $n^h = 0$ and thus $f = n^h_-(v^-)^{-1} = 0$ on $\Gamma$, showing that $I - \mathcal{C}_w$ is injective.
\end{proof}

\begin{lemma}\label{index0lemma}
For each integer $N \geq 2$, the operator $I - \mathcal{C}_w: H_{z, \sym}^N(\Gamma; \C^{1 \times 3}) \to H_{z, \sym}^N(\Gamma; \C^{1 \times 3})$ has Fredholm index $0$.
\end{lemma}
\begin{proof}
Since $w^\pm$ are nilpotent, for every $t \in [0, 1]$, $w_t^\pm := tw^\pm \in H_{\pm}^N(\Gamma; \C^{3 \times 3})$ satisfy $\det(w_t^+ + I) = \det(I - w_t^-) = 1$. 
For each $t \in [0,1]$, define $v_t^\pm$ by $tw^\pm = \pm v_t^\pm \mp I$ and define $\tilde{w}_t^\pm$ by (cf. (\ref{tildewdef}))
$$\tilde{w}_t^+ = (v_t^+)^{-1} - I, \qquad \tilde{w}_t^- = I - (v_t^-)^{-1}.$$
By Lemma \ref{vpmlemma} (b), $\tilde{w}_t^\pm$ belong to $H_{\pm}^N(\Gamma; \C^{3 \times 3})$. 
Hence we can repeat the proof of Lemma \ref{fredholmlemma} to conclude that the bounded linear operator $I - \mathcal{C}_{tw}: H_{z, \sym}^N(\Gamma) \to H_{z, \sym}^N(\Gamma)$ is Fredholm for every $t \in [0,1]$.
Since the Fredholm index is constant on connected components and the Fredholm index of $(I - \mathcal{C}_{tw})|_{t=0} = I$ is $0$, this proves that $I - \mathcal{C}_w: H_{z, \sym}^N(\Gamma; \C^{1 \times 3}) \to H_{z, \sym}^N(\Gamma; \C^{1 \times 3})$ has Fredholm index $0$.
\end{proof}

By Lemma \ref{injectivelemma} and Lemma \ref{index0lemma}, the operator $I - \mathcal{C}_w: H_{z, \sym}^N(\Gamma; \C^{1 \times 3}) \to H_{z, \sym}^N(\Gamma; \C^{1 \times 3})$ is bijective for any $N \geq 2$; by the open mapping theorem, its inverse $(I - \mathcal{C}_w)^{-1}$ is bounded on $H_{z, \sym}^N(\Gamma; \C^{1 \times 3})$. We define $\mu \in (1,1,1) + H_{z, \sym}^2(\Gamma; \C^{1 \times 3})$ by
$$\mu = (1,1,1) + (I - \mathcal{C}_w)^{-1}\mathcal{C}_w (1,1,1)$$
and note that $\mu \in (1,1,1) + H_{z, \sym}^N(\Gamma; \C^{1 \times 3})$ for each $N \geq 2$.

\begin{lemma}\label{nexistencelemma}
The row-vector valued function $n = (1,1,1) + \mathcal{C}(\mu(w^+ + w^-))$ is the unique solution of RH problem \ref{RHn}.
\end{lemma}
\begin{proof}
Since $\mu(w^+ + w^-) \in H_{z}^{1}(\Gamma; \C^{1 \times 3})$, $n$ is analytic in $\C \setminus \Gamma$, has continuous boundary values on $\Gamma \setminus \Gamma_\star$, and satisfies $n(x,t,k) = O(1)$ as $k \to k_{\star} \in \Gamma_\star$.
By Lemma \ref{atinfinitylemma}, $n(x,t,k) = (1,1,1) + O(k^{-1})$ as $k \to \infty$.
Moreover, Lemma \ref{Cfsymmlemma} shows that $n(k) = n(\omega k)\mathcal{A}^{-1} = n(k^{-1}) \mathcal{B}$ for $k \in \C \setminus \Gamma$.
Finally, the definition of $\mu$ implies that $\mathcal{C}_w\mu = \mu - (1,1,1)$. Hence, 
\begin{align}
& n_+ = (1,1,1) + (\mathcal{C}_+ - \mathcal{C}_-)(\mu w^+) + \mathcal{C}_w\mu
  = \mu (w^+ + I) = \mu v^+,
  	\\
& n_- = (1,1,1) - (\mathcal{C}_+ - \mathcal{C}_-)(\mu w^-) + \mathcal{C}_w\mu
  = \mu (I - w^-) = \mu v^-.
\end{align}
Since $v = (v^-)^{-1}v^+$, these relations show that $n$  obeys the jump relation (\ref{njump}).

Uniqueness follows because if $n$ and $\tilde{n}$ are two solutions, then $n^h := n - \tilde{n}$ satisfies the homogeneous RH problem of Lemma \ref{vanishinglemma}, and hence $n^h = 0$.
\end{proof}

Henceforth, let 
$$n(x,t,k) = (1,1,1) + \frac{1}{2\pi i} \int_\Gamma \frac{(\mu(w^+ + w^-))(x,t,s) ds}{s-k}$$ 
be the solution of Lemma \ref{nexistencelemma}.
It follows from Lemma \ref{atinfinitylemma} that there exist functions $n^{(1)}(x,t)$ and $n^{(2)}(x,t)$ such that
\begin{align}\label{nexpansion}
n(x,t,k) = (1,1,1) + \frac{n^{(1)}(x,t)}{k} + \frac{n^{(2)}(x,t)}{k^2} + O(k^{-3}), \qquad k \to \infty.
\end{align}

\begin{lemma}\label{n1n2lemma}
The functions $n^{(1)}(x,t)$ and $n^{(2)}(x,t)$ are smooth functions of $(x,t) \in \R \times [0,T)$ which have rapid decay as $x \to \pm \infty$ in the sense that, for each $\tau \in [0,T)$ and each integer $M \geq 1$,
$$\sup_{(x,t) \in \R \times [0, \tau]} \sum_{i =0}^M (1+|x|)^M(|\partial_x^i n^{(1)}| + |\partial_x^i n^{(2)}| ) < \infty.$$
\end{lemma}
\begin{proof}
It follows from Lemma \ref{wpmdecaylemma} that the maps
\begin{align}\label{xttowpm}
(x,t) \mapsto \big( k \mapsto (|k|^{-j} + |k|^j) w^\pm(x,t,k)\big): \R \times [0,T) \to H_{z}^N(\Gamma)
 \end{align}
are smooth for any $j \geq 0$ and $N \geq 2$. Since
\begin{align*}
n^{(1)}(x,t) = - \frac{1}{2\pi i} \int_\Gamma (\mu(w^+ + w^-))(x,t,s) ds,
	\quad
n^{(2)}(x,t) = - \frac{1}{2\pi i} \int_\Gamma s (\mu(w^+ + w^-))(x,t,s) ds,
\end{align*}
standard estimates now show that $n^{(1)}$ and $n^{(2)}$ are smooth on $\R \times [0,T)$ (see e.g. \cite[pp. 47--51]{LNonlinearFourier} for a detailed proof in a similar situation). Finally, the rapid decay as $x \to \pm \infty$ follows from a nonlinear steepest descent analysis similar to (but simpler than) the one used to establish the asymptotics in Sector I in Theorem  \ref{asymptoticsth}.
\end{proof}

Thanks to Lemma \ref{n1n2lemma}, we may define the functions $u(x,t)$ and $v(x,t)$ by (\ref{uvdef}) and these functions satisfy properties $(i)$ and $(iii)$ of Definition \ref{Schwartzsolutiondef}.
In what follows, we show that $(u,v)$ also satisfy property $(ii)$ of Definition \ref{Schwartzsolutiondef}, and that $v$ can be expressed in terms of $n_3^{(1)}$ by (\ref{recoveruvn}).

\begin{lemma}\label{constructMlemma}
The function $M(x,t,k)$ defined by
\begin{align}\label{Mdef}
M(x,t,k) := P(k)^{-1} \begin{pmatrix}
ne^{\mathcal{L}x + \mathcal{Z}t}  \\
(ne^{\mathcal{L}x + \mathcal{Z}t})_x \\
(ne^{\mathcal{L}x + \mathcal{Z}t})_{xx}
\end{pmatrix} e^{-(\mathcal{L}x+\mathcal{Z}t)}
= P(k)^{-1} \begin{pmatrix}
n  \\
n_x + n\mathcal{L} \\
n_{xx} + 2n_x\mathcal{L} + n\mathcal{L}^2
\end{pmatrix}
\end{align}
satisfies the Lax pair equations (\ref{Xlax}) for $(x,t) \in \R \times [0,T)$ and $k \in \C \setminus \Gamma$, with $\mathsf{U}$ and $\mathsf{V}$ defined in terms of the functions $u,v$ in (\ref{uvdef}). Moreover, as $k \to \infty$,
\begin{align}\label{Mlargek}
M = I + O(k^{-1}).
\end{align}
\end{lemma}
\begin{proof}
The jump matrix $v$ has the form $v(x,t,k) = e^{x\widehat{\mathcal{L}(k)} + t\widehat{\mathcal{Z}(k)}}v_0(k)$, where $v_0(k)$ is independent of $x$ and $t$. 
Hence $(ne^{\mathcal{L}x + \mathcal{Z}t})_x$, $(ne^{\mathcal{L}x + \mathcal{Z}t})_{xx}$, and $(ne^{\mathcal{L}x + \mathcal{Z}t})_t$ have the same jumps as $ne^{\mathcal{L}x + \mathcal{Z}t}$.
In particular, $\tilde{M}$ defined by
\begin{align}\label{Mtildedef}
\tilde{M} := P(k) M e^{\mathcal{L}x + \mathcal{Z}t}
\end{align}
satisfies $\tilde{M}_+ = \tilde{M}_- v_0$ on $\Gamma \setminus \Gamma_\star$ (or, in other words, $M_+ = M_- v$). Substituting the expansion (\ref{nexpansion}), which can be differentiated termwise with respect to $x$, into (\ref{Mdef}), we obtain (\ref{Mlargek}).
We conclude that $\det M$  is entire and tends to $1$ at infinity, so by Liouville's theorem, $\det M \equiv 1$. In particular, $\det \tilde{M} = \det P$, so $\tilde{M}$ is invertible for all $k \in \C \setminus \hat{\mathcal{Q}}$. It follows that the functions $\tilde{M}_x  \tilde{M}^{-1}$ and $\tilde{M}_t \tilde{M}^{-1}$ are analytic for $k \in \C \setminus \hat{\mathcal{Q}}$. 
In fact, using that $n$ has well-defined Taylor expansions as $k \to \kappa_j$ from either side of $\partial \D$, long but straightforward calculations show that
$$\tilde{M}_x  \tilde{M}^{-1}
= \begin{pmatrix}
n_x + n\mathcal{L}  \\
n_{xx} + 2n_x\mathcal{L} + n\mathcal{L}^2 \\
n_{xxx} + 3n_{xx}\mathcal{L} + 3n_x\mathcal{L}^2 + n\mathcal{L}^3
\end{pmatrix}\begin{pmatrix}
n  \\
n_x + n\mathcal{L} \\
n_{xx} + 2n_x\mathcal{L} + n\mathcal{L}^2
\end{pmatrix}^{-1}
$$
satisfies $\tilde{M}_x  \tilde{M}^{-1} = O(1)$ as $k \to \kappa_j$, $j = 1, \dots, 6$; hence $\tilde{M}_x  \tilde{M}^{-1}$ extends to an analytic function of $k \in \C \setminus \{0\}$.
Similar calculations employing the symmetries $n(k) = n(\omega k)\mathcal{A}^{-1} = n(k^{-1}) \mathcal{B}$ and (\ref{nexpansion}) show that
$$\tilde{M}_x  \tilde{M}^{-1} =  
\begin{cases}
\begin{pmatrix}
0 & 1 & 0 \\
0 & 0 & 1 \\
-\frac{ik^3}{24\sqrt{3}} - \frac{u_x}{4}-\frac{iv}{4\sqrt{3}} & -\frac{1}{4}-\frac{u}{2} & 0
\end{pmatrix} 
+ O(k^{-1}), \qquad k \to \infty,
	\\
\begin{pmatrix}
0 & 1 & 0 \\
0 & 0 & 1 \\
-\frac{i}{24\sqrt{3}k^3} - \frac{u_x}{4}-\frac{iv}{4\sqrt{3}} & -\frac{1}{4}-\frac{u}{2} & 0
\end{pmatrix} + O(k), \qquad k \to 0,
\end{cases}$$
where $u$ and $v$ are given by (\ref{uvdef}). Since $\tilde{M}_x  \tilde{M}^{-1}$ is analytic for $k \in \C \setminus \{0\}$, we conclude that $\tilde{M}$ satisfies the $x$-part in (\ref{Lax pair tilde}). Analogous arguments give the $t$-part in (\ref{Lax pair tilde}). Recalling the transformations (\ref{X hat to X tilde}) and (\ref{XhatX}), this shows that $M$ satisfies (\ref{Xlax}).
\end{proof}

We next use the reality condition (\ref{r1r2 relation with kbar symmetry}) to show that $u$ and $v$ are real-valued. 

\begin{lemma}\label{uvreallemma}
 The functions $u$ and $v$ defined in (\ref{uvdef}) are real-valued.
  Moreover, $v$ is given by the expression in (\ref{recoveruvn}).
\end{lemma}
\begin{proof}
The relation $u_t = v_x$ implies that $v$ is given by the expression in (\ref{recoveruvn}).

Let $M$ and $\tilde{M}$ be given by (\ref{Mdef}) and (\ref{Mtildedef}), let $M^A = (M^{-1})^T$ and $\tilde{M}^A = (\tilde{M}^{-1})^T$, and define
$$\hat{n}(x,t,k) := (1,1,1)R(k) \overline{M^A(x,t,\bar{k})} R(k)^{-1}= (0,0,1)\overline{\tilde{M}^A(x,t,\bar{k})} \mathcal{B} e^{-(\mathcal{L}x + \mathcal{Z}t)}\mathcal{B} R(k)^{-1}.$$
Since
$R(k)^{-1}$ is analytic for $k \in \C \setminus \{0\}$, we see from the second expression that $\hat{n}$ is analytic on $\C \setminus \Gamma$, has continuous boundary values on $\Gamma \setminus \Gamma_\star$, and is $O(1)$ as $k \to k_{\star} \in \Gamma_\star$. Furthermore, Lemma \ref{vsymmlemma} implies that the jump matrix $v$ obeys the symmetry $\overline{v^A(\bar{k})} = R(k)^{-1}  v(k) R(k)$ for $k \in \Gamma$ where $v^A = (v^{-1})^T$ (recall that $\Gamma_{2}$, $\Gamma_{4}$, and $\Gamma_{6}$ are oriented towards $0$ in this subsection). We conclude that $\hat{n}$ obeys the jump relation $\hat{n}_+ = \hat{n}_- v$ on $\Gamma \setminus \Gamma_\star$.
Indeed, for $k \in \Gamma \setminus \Gamma_\star$,
\begin{align*}
\hat{n}_+(k) & = (1,1,1) R(k) \overline{M_+^A(\bar{k})} R(k)^{-1} 
= (1,1,1) R(k) \overline{M_-^A(\bar{k})} \overline{v^A(\bar{k})} R(k)^{-1} 
	\\
& = (1,1,1) R(k) \overline{M_-^A(\bar{k})}R(k)^{-1}  v(k)  = \hat{n}_-(k) v(k).
\end{align*}
Straightforward calculations utilizing the symmetries (\ref{nsymm}) of $n$ imply that $M$ obeys the $\mathcal{A}$- and $\mathcal{B}$-symmetries (\ref{XYsymm}). Using also that $R(k) = \mathcal{A} R(\omega k) \mathcal{A} = \mathcal{B} R(k^{-1}) \mathcal{B}$, we find that $\hat{n}$ obeys the symmetries in (\ref{nsymm}).
From (\ref{Mlargek}), we see that $\hat{n}(x,t,k) = (1,1,1) + O(k^{-1})$ as $k \to \infty$.
This shows that $\hat{n}$ solves RH problem \ref{RHn}; since the solution is unique by Lemma \ref{nexistencelemma}, we have $\hat{n} = n$. By comparing the terms of $O(1/k)$ in the large $k$ expansion of $\hat{n}$ and $n$, we conclude that $n_3^{(1)}(x,t) \in i\R$ for all  $x,t$. In particular, $u$ and $v$ are real-valued.
\end{proof}

The next lemma completes the proof of Theorem \ref{inverseth}.

\begin{lemma}
$\{u,v\}$ is a Schwartz class solution of (\ref{boussinesqsystem}) on $\R \times [0,T)$. 
\end{lemma}
\begin{proof}
The claim follows from the compatibility condition of the Lax pair equations (\ref{Xhatlax}) together with Lemma \ref{n1n2lemma}, Lemma \ref{constructMlemma}, Lemma \ref{uvreallemma}, and the definition (\ref{uvdef}) of $u$ and $v$.
\end{proof}

\section{Solution of the initial value problem}\label{IVPsec}
In this section, we prove Theorem \ref{IVPth}.
Let $u_0, v_0 \in \mathcal{S}(\R)$ and $r_j(k)$, $j = 1,2$, be as in the statement of the theorem.
By Theorem \ref{directth} and Theorem \ref{inverseth}, $\{u(x,t), v(x,t)\}$ defined in (\ref{recoveruvn}) is a Schwartz class solution of the system (\ref{boussinesqsystem}) on $\R \times [0,T)$. If we can show that $u(x,0) = u_0(x)$ and $v(x,0) = v_0(x)$ for $x \in \R$, then it follows that $u(x,t)$ is a Schwartz class solution of (\ref{badboussinesq}) on $\R \times [0,T)$ with initial data $u_0, u_1$. 
To see that $u(x,0) = u_0(x)$ and $v(x,0) = v_0(x)$, let $n(x,t,k)$ be the solution of RH problem \ref{RHn} constructed in Lemma \ref{nexistencelemma}.
The functions $u(x,0)$ and $v(x,0)$ are given in terms of $n$ by the formulas (\ref{uvdef}) evaluated at $t = 0$. On the other hand, starting from the initial data $\{u_0, v_0\}$, we can construct another solution $\tilde{n}(x,0,k)$ of RH problem \ref{RHn} as in Proposition \ref{prop:construction of n} such that $u_0$ and $v_0$ are given in terms of $\tilde{n}$ by (\ref{uvdef}) evaluated at $t = 0$.
Since the solution of RH problem \ref{RHn} is unique by Lemma \ref{nexistencelemma}, we must have $n(x,0,k) = \tilde{n}(x,0, k)$ for all $x \in \R$. This shows that $u(x,0) = u_0(x)$ and $v(x,0) = v_0(x)$. Finally, uniqueness of the solution $u(x,t)$ follows in a similar way: if $\tilde{u}(x,t)$ is another solution with the same initial data, then Proposition \ref{prop:construction of n} shows that $u$ and $\tilde{u}$ can be recovered via (\ref{uvdef}) from the associated solutions $n(x,t,k)$ and $\tilde{n}(x,t,k)$ of RH problem \ref{RHn}. By Lemma \ref{nexistencelemma}, we have $n = \tilde{n}$, and hence $u = \tilde{u}$ by (\ref{uvdef}). This completes the proof of Theorem \ref{IVPth}.

\section{Blow-up}\label{blowupsec}

\subsection{Proof of Theorem \ref{blowupth}}
Let $\{u, v\}$ be the Schwartz class solution of (\ref{boussinesqsystem}) on $\R \times [0,T)$ defined in (\ref{recoveruvn}), where $T$ is given by (\ref{Tdef}). Suppose that $\{u, v\}$ can be extended to a Schwartz class solution on $[0,\tilde{T})$ for some $\tilde{T} > T$.
Choose $t \in (T, \tilde{T})$. Then Theorem \ref{directth} applied to $\{u(\cdot, t), v(\cdot, t)\}$ implies that $r_2(k; t)$ and its derivatives have rapid decay as $\Gamma_4 \ni k  \to \infty$, and hence, by (\ref{r1r2 relation with kbar symmetry}), $r_1(1/k; t)$ and its derivatives have rapid decay as $\Gamma_1 \ni k \to \infty$. But Proposition \ref{reflectionprop} shows that 
$$r_1(1/k; t) = r_1(1/k) e^{\theta_{21}(0,t,k)} = r_1(1/k) e^{\frac{|k|^4 - 1}{4|k|^2}t} \quad \text{for} \quad k \in \Gamma_1,$$ 
so then $r_1(1/k) e^{\frac{|k|^2t}{4}}$ and its derivatives have rapid decay as $\Gamma_1 \ni k \to \infty$. Since $t > T$, this contradicts the definition (\ref{Tdef}) of $T$.

\subsection{Proof of Theorem \ref{existenceblowupth}}
Let $T > 0$. If $\{u,v\}$ is a Schwartz class solution of (\ref{boussinesqsystem}) that blows up at time $T$, then $u$ is a Schwartz class solution of (\ref{badboussinesq}) which also blows up at time $T$.
Hence, according to Theorems \ref{inverseth} and \ref{blowupth}, it is enough to construct functions $\{r_j(k)\}_1^2$ such that (a) properties $(i)$--$(v)$ of Theorem \ref{directth} hold and (b) $\exp(|k|^2t/4)r_1(1/k)$ and its derivatives are rapidly decreasing as $\Gamma_1 \ni k \to \infty$ for any $t <T$ but not for $t = T$. 
To construct such functions, let $u_0, v_0$ be any real-valued functions in the Schwartz class. By Theorem \ref{inverseth}, the associated reflection coefficients $\{r_j(k)\}_1^2$ satisfy $(i)$--$(v)$ of Theorem \ref{directth}. By changing $r_1(k)$ for $k \in \Gamma_1$ near $0$ and making the corresponding change in $r_2(k)$ for $k \in \Gamma_4$ as dictated by (\ref{r1r2 relation with kbar symmetry}) (but leaving $r_1$ and $r_2$ unchanged elsewhere), we easily find reflection coefficients fulfilling both (a) and (b).

\section{Asymptotics: Overview}\label{overviewsec}
Since the derivations of the asymptotic formulas of Theorem \ref{asymptoticsth} are rather involved, we first give an overview of the key ideas entering the proof. The main idea is to use Deift--Zhou steepest descent arguments to analyze RH problem \ref{RHn} for $n(x,t,k)$ as $(x,t) \to \infty$. The jump matrix $v(x,t,k)$ for this RH problem is defined in (\ref{vdef}) and it depends on the parameters $(x,t)$ only via the exponentials $e^{\pm \theta_{ij}} = e^{\pm t \Phi_{ij}}$, where the phase functions $\Phi_{ij}$ are given by (\ref{def of Phi ij}). The main contributions to the asymptotics of $n$ will come from the saddle points of these phase functions.

\subsection{Analysis of the phase functions}
The relations 
\begin{align}\label{relations between the different Phi}
\Phi_{31}(\zeta,k) = - \Phi_{21}(\zeta,\omega^{2}k), \qquad \Phi_{32}(\zeta,k) = \Phi_{21}(\zeta,\omega k),
\end{align}
allow us to restrict our attention to the analysis of $\Phi_{21}$. We define the saddle points of $\Phi_{21}$ as the points $k$ for which $\partial_{k}\Phi_{21}(\zeta,k)=0$. It is a simple computation to verify that $\Phi_{21}$ admit four saddle points, which are denoted by $\{k_{j}\}_{j=1}^{4}$ and are given in \eqref{def of kj}. The next proposition summarizes some basic properties of $\Phi_{21}$ and $\smash{\{k_{j}\}_{j=1}^{4}}$.

\begin{proposition}\label{prop: phase of Phi21 for various zeta}
The phase function $\Phi_{21}$ has the following properties (see also Figure \ref{fig: Re Phi 21 for various zeta}):
\begin{enumerate}[$-$]
\item For each $\zeta \in [0,+\infty)$, $\Phi_{21}$ has four saddle points $\{k_{j}=k_{j}(\zeta)\}_{j=1}^{4}$ given by \eqref{def of kj}.

\item $\re \Phi_{21}(\zeta,k) = 0$ for all $k \in \partial \mathbb{D}$ and all $\zeta \in [0,+\infty)$.

\item For each $\zeta \in [0,+\infty)$, $k_{1},k_{2} \in \partial \mathbb{D}$, $k_{2} = \bar{k}_1$, $\re k_{1} \in [-1/\sqrt{2},0)$, $k_{1}k_{2}=1$, and $k_{3}k_{4}=1$. 

\item If $\zeta = 0$, then $k_{1} = e^{\frac{3\pi i}{4}}$, $k_{2} = e^{-\frac{3\pi i}{4}}$, $k_{3}=e^{\frac{\pi i}{4}}$, and $k_{4}=e^{-\frac{\pi i}{4}}$.

\item If $\zeta \in (0,1)$, then $k_{3},k_{4} \in \partial \mathbb{D}$, $k_{4}=\bar{k}_3$, $\re k_{3} \in (1/\sqrt{2},1)$, and $\re k_{1} \in (-1/\sqrt{2},-1/2)$. In particular, $\re \Phi_{21}(\zeta,k_{j})=0$ for $j = 1,\dots,4$. Furthermore, if $\zeta \in (1/\sqrt{3},1)$, then $\re k_{3} \in (\sqrt{3}/2,1)$; if $\zeta = 1/\sqrt{3}$, then $k_{3} = e^{\frac{\pi i}{6}}$, $k_{4} = e^{-\frac{\pi i}{6}}$; and if $\zeta \in (0,1/\sqrt{3})$, then $\re k_{3} \in (1/\sqrt{2},\sqrt{3}/2)$.

\item If $\zeta=1$, then $k_{1} = \omega$, $k_{2} = \omega^{2}$, $k_{3}=k_{4}=1$, and $\Phi_{21}$ has a double saddle at $k=1$, that is, $\partial_{k}\Phi_{21}(1,1)=\partial_{k}^{2}\Phi_{21}(1,1)=0$.
\item If $\zeta > 1$, $k_{3} \in (1,+\infty)$, $k_{4} \in (0,1)$, $\re \Phi_{21}(\zeta,k_{3})>0$, $\re \Phi_{21}(\zeta,k_{4})<0$, and $\re k_{1} \in (-\frac{1}{2},0)$. 
\end{enumerate}
\end{proposition}
\begin{proof}
The claims follow from an analysis of \eqref{def of Phi ij} and \eqref{def of kj}.
\end{proof}

\subsection{Asymptotic sectors}
From (\ref{relations between the different Phi}) and Proposition \ref{prop: phase of Phi21 for various zeta}, we see that there are twelve saddle points that are relevant for the large $(x,t)$ asymptotics of $n$: the four saddles $\{k_{j}\}_{j=1}^{4}$ of $\Phi_{21}$, the four saddles $\{\omega k_{j}\}_{j=1}^{4}$ of $\Phi_{31}$, and the four saddles $\{\omega^2 k_{j}\}_{j=1}^{4}$ of $\Phi_{32}$. For $0 < \zeta < 1$, all these saddle points lie on the unit circle, whereas only six lie on the unit circle for $1 < \zeta < \infty$, see Figure \ref{fig: Re Phi 21 for various zeta}. For $\zeta = 1$, the twelve saddle points merge in three groups of four at $1, \omega$, and $\omega^2$. This gives rise to the asymptotic sectors I--V of Theorem \ref{asymptoticsth} as follows:

\begin{enumerate}[$-$]
\item Sectors I and II correspond to $1 < \zeta < \infty$. In these sectors, it turns out that the six saddle points on the unit circle $\partial \D$ contribute to the large $(x,t)$  behavior of $n(x,t,k)$, whereas the contributions from the six saddle points on $\R \cup \omega \R \cup \omega^2 \R$ are exponentially small. The contributions from the six saddle points on $\partial \D$ can be computed with the help of local parametrices involving parabolic cylinder functions. This generates the leading term $A(\zeta) \cos \alpha(\zeta,t) t^{-1/2} $ in (\ref{uasymptotics}). The reason for splitting the interval $1 < \zeta < \infty$ into two sectors is that $t \to +\infty$ in Sector II, whereas Sector I describes the rapid decay as $x \to +\infty$ even if $t$ stays bounded.

\item Sector III corresponds to $\zeta \approx 1$. In this sector, the twelve saddle points cluster in three groups of four at the third roots of unity $1, \omega, \omega^2$. The local analysis near these three points requires the use of a novel model RH problem, which is too complicated to be solved exactly. The complications stem from the fact that there are two different structures nonlinearly superimposed on each other. Moreover, these structures are associated with different spatial and temporal scales. One of these structure involves the Hastings--McLeod solution of the Painlev\'{e} II equation, while the second structure involves the error function. The jump structure associated with the error function is truly $3\times 3$ in the sense that it involves all three rows and columns in a nontrivial way. The large $t$ behavior of this model RH problem is derived in Appendix \ref{IIIappendix} which constitutes a significant part of this work. 

\item Sector IV corresponds to $1/\sqrt{3} < \zeta < 1$. All twelve saddle points contribute to the leading term in this case: six of them generate the right-moving modulated wave $A_{1}(\zeta)  \cos \alpha_{1}(\zeta,t) t^{-1/2}$, while the other six generate the left-moving modulated wave $A_{2}(\zeta) \cos \alpha_{2}(\zeta,t) t^{-1/2}$ in (\ref{uasymptotics}). 

\item Sector V corresponds to $0 < \zeta < 1/\sqrt{3}$. The asymptotic analysis follows the same general pattern as in Sector IV, but compared to that sector some of the saddle points now belong to other subcontours of $\Gamma$. For example, the saddle point $\omega k_4$ of $\Phi_{31}$ belongs to $\Gamma_7$ for $1/\sqrt{3} < \zeta < 1$ but to $\Gamma_9$ for $0 < \zeta < 1/\sqrt{3}$. This means that a different local model is required.

\end{enumerate}

The derivation of the asymptotic formula in Sectors III is presented in detail in Section \ref{transitionsec}. In Section \ref{othersectorssec}, we comment on the changes that are needed to handle Sectors I, II, IV, and V.

\subsection{Transformations of the RH problem}\label{transsec}
Theorem \ref{inverseth} together with Assumption \ref{nounstablemodesassumption} ensures that the solution $n(x,t,k)$ exists and is unique for each $(x,t) \in \R \times [0, +\infty)$, and that $u(x,t)$ is given by (\ref{recoveruvn}). 
By performing a number of transformations, we will bring RH problem \ref{RHn} for $n$ to a form suitable for determining the asymptotics. More precisely, starting with $n$, we will define functions $n^{(j)}(x,t,k)$, $j = 1,2, \dots$, such that the RH problem satisfied by each $n^{(j)}$ is equivalent to the original RH problem \ref{RHn}. The jump matrix obtained after the final transformation will tend to the identity matrix as $(x,t) \to \infty$ everywhere except near the saddle points, so that the $(x,t)$ asymptotics can be computed by considering the local contributions from these points.
The jump contour for the RH problem for $n^{(j)}$ will be denoted by $\Gamma^{(j)}$ and the jump matrix by $v^{(j)}$, so that the jump relation is $n^{(j)}_+ = n^{(j)}_- v^{(j)}$ on $\Gamma^{(j)} \setminus \Gamma^{(j)}_\star$, where $\Gamma^{(j)}_\star$ denotes the points of self-intersection of $\Gamma^{(j)}$. Just like $n$, each $n^{(j)}$ will obey the normalization condition $n^{(j)} \to (1,1,1)$ as $k \to \infty$. 

The jump matrix $v$ obeys the symmetries
\begin{align}\label{vsymm}
v(x,t,k) = \mathcal{A} v(x,t,\omega k)\mathcal{A}^{-1}
 = \mathcal{B} v(x,t, k^{-1})^{-1}\mathcal{B}, \qquad k \in \Gamma,
\end{align}
where $\mathcal{A}$ and $\mathcal{B}$ are given by \eqref{def of Acal and Bcal}, and accordingly the solution $n$ obeys the symmetries (\ref{nsymm}). 
It is convenient to implement the transformations so that these symmetries are preserved. This means that for each $j$ we will have
\begin{align}\label{vjsymm}
& v^{(j)}(x,t,k) = \mathcal{A} v^{(j)}(x,t,\omega k)\mathcal{A}^{-1}
 = \mathcal{B} v^{(j)}(x,t,k^{-1})^{-1}\mathcal{B}, && k \in \Gamma^{(j)},
	\\ \label{njsymm}
& n^{(j)}(x,t, k) = n^{(j)}(x,t,\omega k)\mathcal{A}^{-1}
 = n^{(j)}(x,t, k^{-1}) \mathcal{B}, && k \in \C \setminus \Gamma^{(j)},
\end{align}
(except in Sector III where we will only preserve the $\mathcal{A}$-symmetry).
These symmetries imply that we only have to specify the transformations of $n$ for $k$ in a part of the complex plane, because they can then be extended to the whole complex $k$-plane by symmetry. 

We mention that for Sectors IV and V, the transformations $n^{(j)}\to n^{(j+1)}$ can only be implemented at the cost of introducing poles in the jump matrices. This, in turn, complicates significantly the construction of the global parametrix.

\subsection{Assumptions}
For the remainder of the paper, we assume that the initial data $u_0, v_0 \in \mathcal{S}(\R)$ are such that Assumptions \ref{solitonlessassumption}, \ref{originassumption}, and \ref{nounstablemodesassumption} hold, where $r_1(k)$ and $r_2(k)$ are defined by (\ref{r1r2def}). 
As can be seen from Figure \ref{fig: Re Phi 21 for various zeta}, Assumption \ref{nounstablemodesassumption} ensures that $v_{1}$ does not to blow up as $t \to \infty$; by symmetry, $v_{3}$ and $v_{5}$ then also remain bounded as $t \to \infty$. Moreover, the symmetry \eqref{r1r2 relation with kbar symmetry} together with Assumption \ref{nounstablemodesassumption} implies that $r_{2}(k) = 0$ for $k \in [i,i\infty)$. This in turn implies that $v_{4}$ does not blow up as $t \to \infty$. By symmetry, $v_{2}$ and $v_{6}$ also remain bounded as $t \to \infty$. The matrices $v_{7},v_{8},v_{9}$ are automatically bounded. 

For $j = 1, \dots, 6$, let us write $\Gamma_j = \Gamma_{j'} \cup \Gamma_{j''}$, where $\Gamma_{j'} = \Gamma_j \setminus \D$ and $\Gamma_{j''} := \Gamma_j \setminus \Gamma_{j'}$ with $\Gamma_j$ as in Figure \ref{fig: Dn}. Let $v_{j'}$ and $v_{j''}$ be the jumps of $n$ on $\Gamma_{j'}$ and $\Gamma_{j''}$, respectively. Assumption \ref{nounstablemodesassumption} implies that the jumps for $n$ on $\cup_{j=1}^{6}\Gamma_{j} = \cup_{j=1}^{6}(\Gamma_{j'} \cup \Gamma_{j''})$ are as follows:
\begin{align*}\nonumber
& v_{1'} = \begin{pmatrix}
1 & -r_{1}(k)e^{-\theta_{21}} & 0 \\
0 & 1 & 0 \\
0 & 0 & 1
\end{pmatrix}, & & \hspace{-.2cm} v_{1''} = \begin{pmatrix}
1 & 0 & 0 \\
r_{1}(\frac{1}{k})e^{\theta_{21}} & 1 & 0 \\
0 & 0 & 1
\end{pmatrix}, & & \hspace{-.2cm}  v_{2'} = \begin{pmatrix}
1 & 0 & 0 \\
0 & 1 & -r_{2}(\frac{1}{\omega k})e^{-\theta_{32}} \\
0 & 0 & 1
\end{pmatrix},  \\
& v_{2''} = \begin{pmatrix}
1 & 0 & 0 \\
0 & 1 & 0 \\
0 & r_{2}(\omega k)e^{\theta_{32}} & 1
\end{pmatrix}, & & \hspace{-.2cm}  v_{3'} = \begin{pmatrix}
1 & 0 & 0 \\
0 & 1 & 0 \\
-r_{1}(\omega^{2}k)e^{\theta_{31}} & 0 & 1
\end{pmatrix}, & & \hspace{-.2cm} v_{3''} = \begin{pmatrix}
1 & 0 & r_{1}(\frac{1}{\omega^{2}k})e^{-\theta_{31}} \\
0 & 1 & 0 \\
0 & 0 & 1
\end{pmatrix}, \\
& v_{4'} = \begin{pmatrix}
1 & -r_{2}(\frac{1}{k})e^{-\theta_{21}} & 0 \\
0 & 1 & 0 \\
0 & 0 & 1
\end{pmatrix}, & & \hspace{-.2cm} v_{4''} = \begin{pmatrix}
1 & 0 & 0 \\
r_{2}(k)e^{\theta_{21}} & 1 & 0 \\
0 & 0 & 1
\end{pmatrix}, & & \hspace{-.2cm} v_{5'} = \begin{pmatrix}
1 & 0 & 0 \\
0 & 1 & -r_{1}(\omega k)e^{-\theta_{32}} \\
0 & 0 & 1
\end{pmatrix}, \\
& v_{5''} = \begin{pmatrix}
1 & 0 & 0 \\
0 & 1 & 0 \\
0 & r_{1}(\frac{1}{\omega k})e^{\theta_{32}} & 1
\end{pmatrix}, & & \hspace{-.2cm}  v_{6'} = \begin{pmatrix}
1 & 0 & 0 \\
0 & 1 & 0 \\
-r_{2}(\frac{1}{\omega^{2} k})e^{\theta_{31}} & 0 & 1
\end{pmatrix}, & & \hspace{-.2cm} v_{6''} = \begin{pmatrix}
1 & 0 & r_{2}(\omega^{2}k)e^{-\theta_{31}} \\
0 & 1 & 0 \\
0 & 0 & 1
\end{pmatrix}.
\end{align*}

\section{Asymptotics: Sector III}\label{transitionsec}
In this section, we establish the asymptotic formula (\ref{uasymptotics}) for $u(x,t)$ as $(x,t) \to \infty$ in Sector \III, which is given by 
$$\III := \{(x,t) \in \R \times [2,\infty) \, | \, |\zeta - 1|\leq M t^{-2/3}\},$$
where $\zeta := x/t$ and $M > 0$ is a constant. For the purpose of the proof, it is necessary to split this sector into two halves and write
$$\III = \III_\geq \cup \III_\leq \quad \text{where} \quad \III_\geq := \III \cap \{\zeta \geq 1\} \quad \text{and} \quad \III_\leq := \III \cap \{\zeta \leq 1\}.$$
This is because the saddle points $\omega k_3$ and $\omega k_4$ lie on the line $\omega \R$ for $\zeta \geq 1$ (see Figure \ref{saddlepointsfig}), while they lie on $\partial \D$ for $\zeta \leq 1$, and this difference changes the contour deformations slightly. Since the proofs for Sectors $\III_\geq$ and $\III_\leq$ are very similar, we will only give the proof for $\III_\geq$.

In Sector $\III_\geq$, we have $\zeta \geq 1$ and $\zeta = 1+ O(t^{-2/3})$ which implies that there are four saddle points merging at $k=\omega$ as follows:
\begin{itemize}
\item $\Phi_{21}$ has one saddle point $k_1$ near $\omega$ such that (blue dot in Figure \ref{saddlepointsfig}) 
$$k_1 = \omega -i\omega \frac{1}{3\sqrt{3}}(\zeta -1) + O((\zeta-1)^2) \qquad \text{as} \;  \zeta \downarrow 1.$$

\item $\Phi_{32}$ has one saddle point $\omega^2 k_2$ near $\omega$ such that (green dot in Figure \ref{saddlepointsfig})
$$\omega^2 k_2 = \omega + i \omega \frac{1}{3\sqrt{3}}(\zeta -1) + O((\zeta-1)^2) \qquad \text{as} \;  \zeta \downarrow 1.$$

\item $\Phi_{31}$ has two saddle points $\omega k_3$ and $\omega k_4$ near $\omega$ such that (red dots in Figure \ref{saddlepointsfig})
$$\omega k_3 = \omega + \omega \sqrt{\frac{2}{3}} \sqrt{\zeta -1} + O(\zeta-1) \;\; \text{and}\;\; \omega k_4 = \omega - \omega \sqrt{\frac{2}{3}} \sqrt{\zeta -1} + O(\zeta-1) \qquad \text{as} \; \zeta \downarrow 1,$$
\end{itemize}
Thus, the saddle points $k_1$ and $\omega^2 k_2$ approach $\omega$ at least as fast as $t^{-2/3}$, while the saddle points 
$\omega k_3$ and $\omega k_4$ of $\Phi_{31}$ approach $\omega$ at least as fast as $t^{-1/3}$, i.e., 
$$|k_1 - \omega| \leq C t^{-2/3}, \quad |\omega^2 k_2 - \omega| \leq C t^{-2/3}, \quad |\omega k_3 - \omega| \leq C t^{-1/3}, \quad |\omega k_4 - \omega| \leq C t^{-1/3}.$$

\begin{figure}
\begin{center}
\begin{tikzpicture}[master,scale=0.8]
\node at (0,0) {};
\draw[black,line width=0.55 mm] (0,0)--(30:4);
\draw[black,line width=0.55 mm] (0,0)--(90:4);
\draw[black,line width=0.55 mm] (0,0)--(150:4);

\draw[black,line width=0.55 mm] ([shift=(30:3cm)]0,0) arc (30:150:3cm);

\draw[blue,fill] (113:3) circle (0.1cm);
\draw[green,fill] (-113+240:3) circle (0.1cm);

\draw[red,fill] (120:3.8) circle (0.1cm);
\draw[red,fill] (120:2.25) circle (0.1cm);

\node at (110:3.3) {\small $k_1$};
\node at (-106.5+240:3.5) {\small $\omega^2 k_2$};
\node at (120:4.2) {\small $\omega k_3$};
\node at (120:1.9) {\small $\omega k_4$};

\end{tikzpicture}
\end{center}
\begin{figuretext}
\label{saddlepointsfig}
The saddle points $k_1$ (blue), $\omega^2 k_2$ (green), and $\omega k_3, \omega k_4$ (red) of $\Phi_{21}$, $\Phi_{32}$, and $\Phi_{31}$, respectively, in the case of Sector $\III_\geq$.
For visual convenience, these dots have been moved slightly from their exact locations here and in many figures below (see Figure \ref{II fig: Re Phi 21 31 and 32 for zeta=1.2} for their true locations).
\end{figuretext}
\end{figure}

We will see that the two saddle points of $\Phi_{31}$ give rise to the leading term in the asymptotics of $u(x,t)$. This term is $O(t^{-2/3})$ and is given in terms of the Hastings--McLeod solution of the Painlev\'e II equation. In Sector IV, the saddle points of $\Phi_{21}$ and $\Phi_{32}$ correspond to parabolic cylinder model problems. In analogy with for example the mKdV equation \cite{DZ1993}, one might expect that this pair of merging saddle points would give rise to another Painlev\'e II model problem in Sector III. However, these two saddle points belong to two different phase functions $\Phi_{21}$ and $\Phi_{32}$. This means that the approximation of the jumps on the contours emanating from these two saddle points does not involve Painlev\'e II. In fact, the contributions from these two saddle points are of subleading order compared to the term involving the Hastings--McLeod solution.

In addition to the four saddle points near $k= \omega$, there are four saddle points near $k = 1$ and four saddle points near $k = \omega^2$; however, by symmetry, it is enough to consider the saddle points near $k = \omega$. Hence, to obtain asymptotics in Sector $\III_\geq$, we need to construct a local parametrix which approximates the (appropriately transformed) RH problem near $k=\omega$. The construction of this local parametrix is complicated by the fact that it involves two different structures nonlinearly superimposed on each other near $\omega$. Furthermore, these structures are associated with different spatial and temporal scales.

The first structure corresponds to the two merging saddle points of $\Phi_{31}$. This part of the RH problem, which involves $\Phi_{31}$, is approximated by the RH problem corresponding to the Hastings--McLeod solution formulated in terms of the variables $y$ and $z$ which scale as
\begin{align*}
y  \sim t^{-1/3}(x-t), \qquad 
z \sim t^{1/3} (k-\omega).
\end{align*}

The second structure corresponds to the saddle points of $\Phi_{21}$ and $\Phi_{32}$. This part of the RH problem is approximated by a model problem of a new type which can be solved exactly in terms of the error function. This model problem involves all three rows and columns in a nontrivial way. It is  naturally formulated in terms of the variables $\tilde{y}$ and $w$ which scale as
\begin{align*}
\tilde{y} \sim t^{-1}(x-t), \qquad w \sim t^{1/2}(k-\omega).
\end{align*}

Thus the local parametrix near $\omega$ can be viewed as the nonlinear superposition of two different local parametrices. This is also essentially how we will construct it: We will first quotient out the solution of the Hastings--McLeod RH problem; we will then quotient out the solution of the $3 \times 3$ error function problem mentioned above appropriately coupled to the Hastings--McLeod solution. 

In this section, we assume that $N \geq 1$ is an integer. The signature tables of $\Phi_{21}$, $\Phi_{31}$, and $\Phi_{32}$ are shown in Figure \ref{II fig: Re Phi 21 31 and 32 for zeta=1.2}. 

\begin{figure}[t]
\begin{center}
\begin{tikzpicture}[master]
\node at (0,0) {\includegraphics[width=4.5cm]{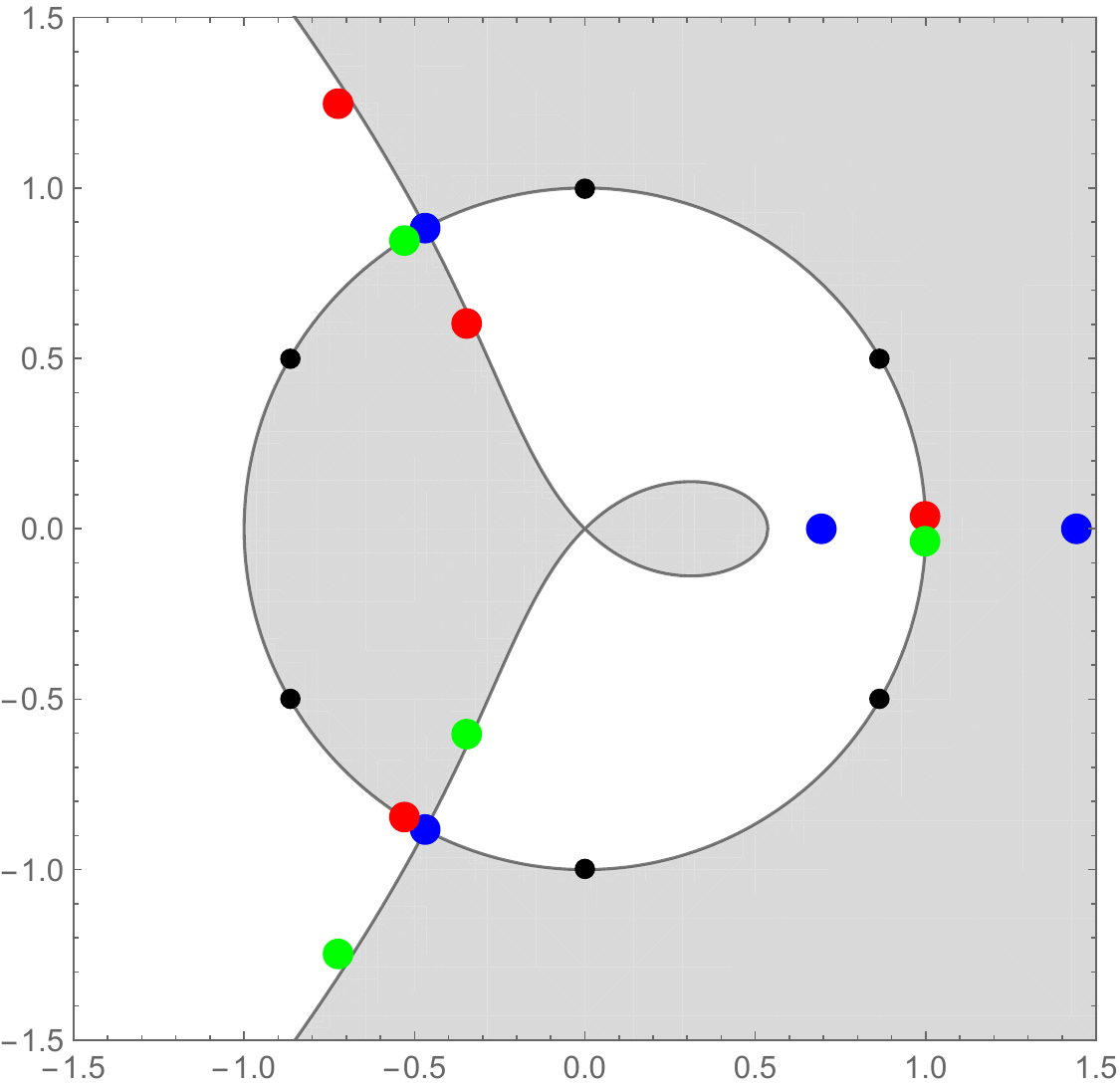}};
\end{tikzpicture} \hspace{0.1cm} \begin{tikzpicture}[slave]
\node at (0,0) {\includegraphics[width=4.5cm]{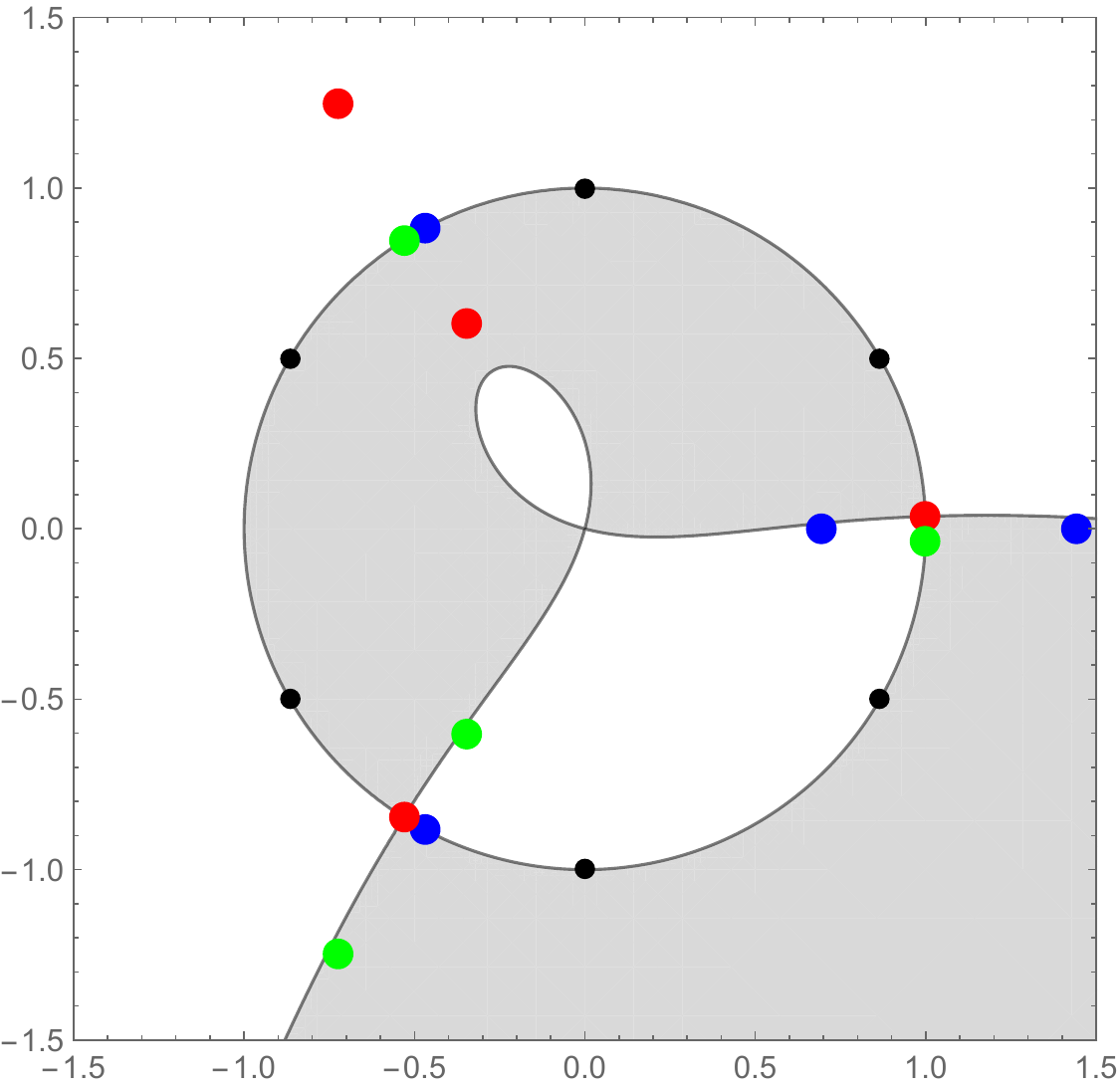}};
\end{tikzpicture} \hspace{0.1cm} \begin{tikzpicture}[slave]
\node at (0,0) {\includegraphics[width=4.5cm]{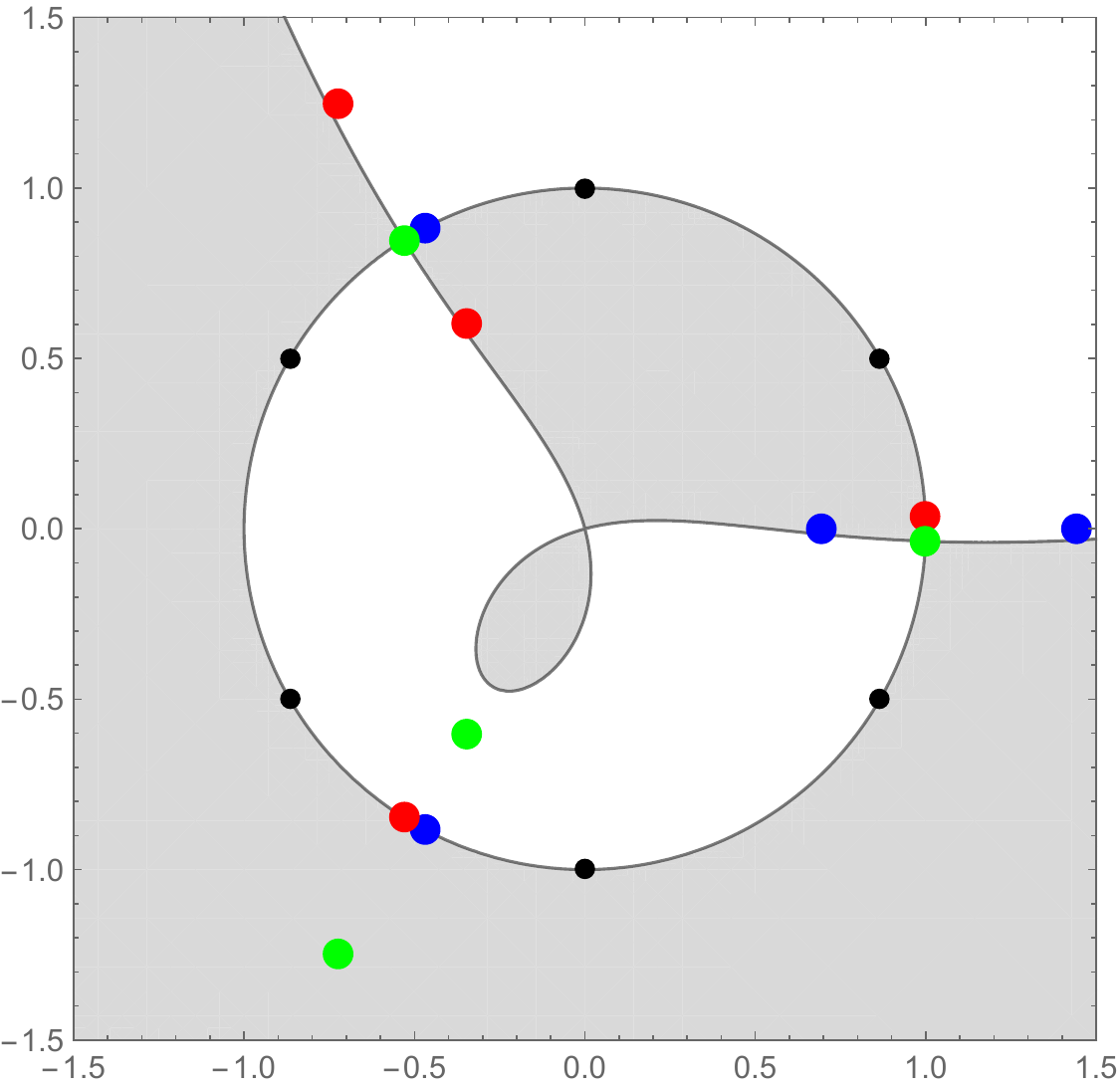}};
\end{tikzpicture}
\end{center}
\begin{figuretext}
\label{II fig: Re Phi 21 31 and 32 for zeta=1.2} From left to right: The signature tables for $\Phi_{21}$, $\Phi_{31}$, and $\Phi_{32}$ for $\zeta=1.2$. In all images, the shaded regions correspond to $\{k \, | \, \re \Phi_{ij}>0\}$ and the white regions to $\{k \, | \, \re \Phi_{ij}<0\}$. The points $k_{1},k_{2}, k_{3}, k_{4}$ are represented in blue, $\omega k_{1},\omega k_{2},\omega k_{3},\omega k_{4}$ in red, and $\omega^{2} k_{1},\omega^{2} k_{2},\omega^{2} k_{3},\omega^{2} k_{4}$ in green. The smaller black dots are the points $i\kappa_j$, $j=1,\ldots,6$.
\end{figuretext}
\end{figure}

\begin{figure}[t]
\begin{center}
\begin{tikzpicture}[master]
\node at (0,0) {\includegraphics[width=5cm]{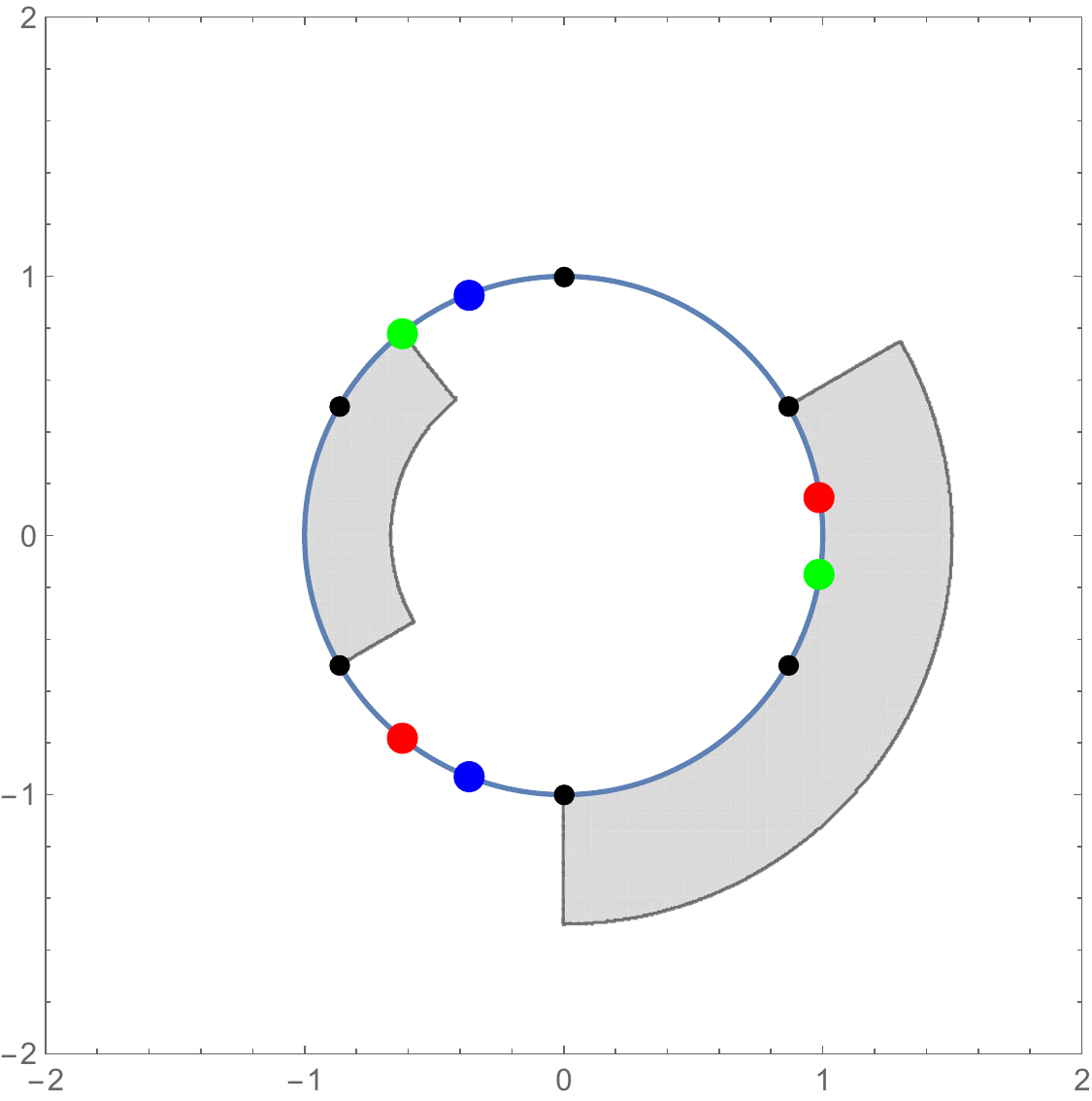}};
\node at (-0.88,0.2) {\small $U_1$};
\node at (1.4,-0.6) {\small $U_1$};
\end{tikzpicture} \hspace{0.1cm} 
\begin{tikzpicture}[slave]
\node at (0,0) {\includegraphics[width=5cm]{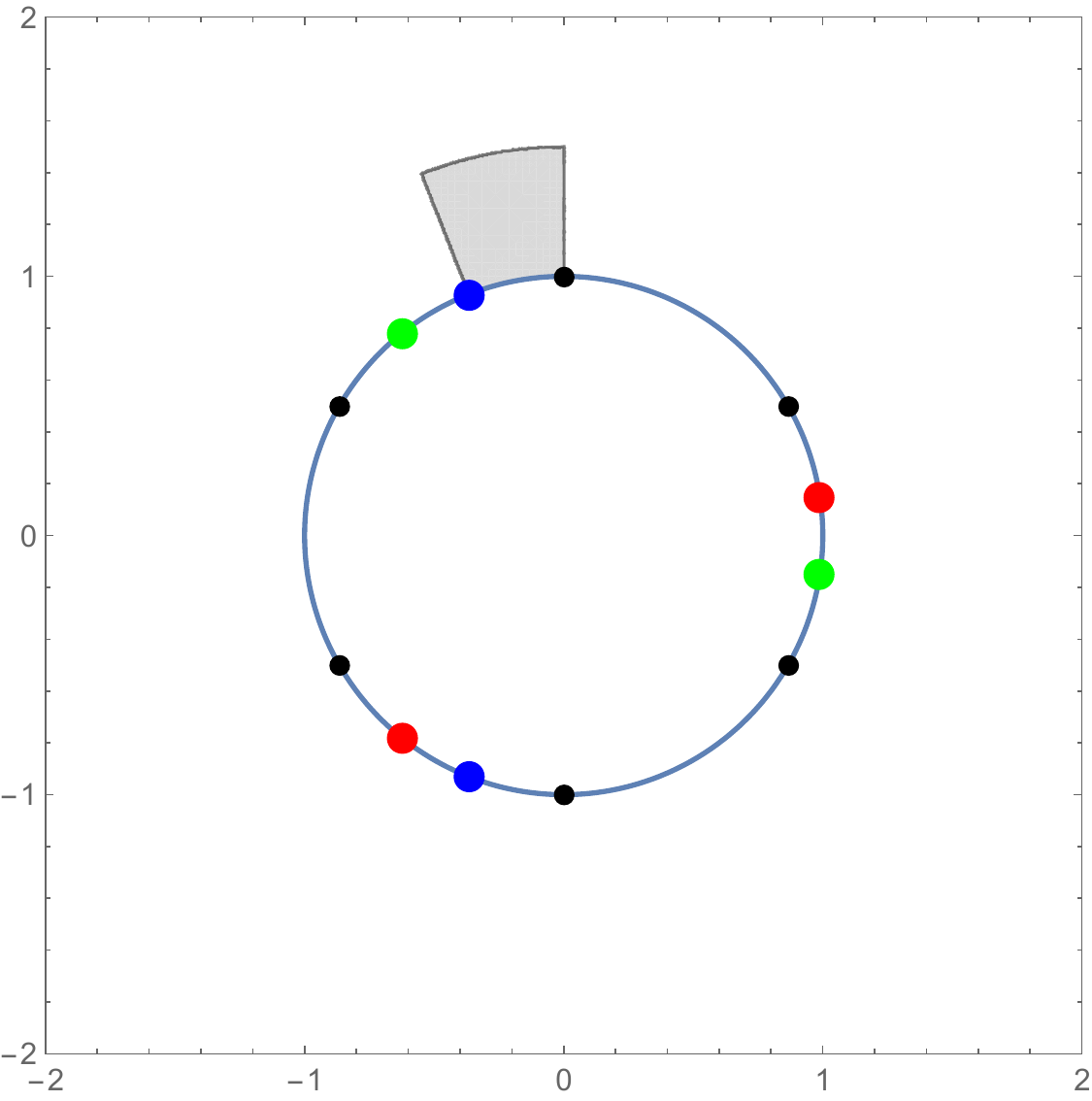}};
\node at (-0.17,1.52) {\small $\hat{U}_1$};
\end{tikzpicture}
\end{center}
\begin{figuretext}
\label{fig: U1 and U2} The open subsets $U_1$ and $\hat{U}_1$ of the complex $k$-plane. 
\end{figuretext}
\end{figure}

\subsection{Decomposition lemma}
Define $\hat{r}_1(k)$  and $\hat{r}_2(k)$ by
\begin{align*}
\hat{r}_{j}(k) := \frac{r_{j}(k)}{1+r_{1}(k)r_{2}(k)} = \frac{r_{j}(k)}{1+\tilde{r}(k)|r_{1}(k)|^{2}}, \qquad j = 1,2.
\end{align*}
Before we can define the first transformation, we need to decompose each of the spectral functions $r_1$, $\hat{r}_1$, $r_2$, and $\hat{r}_2$ into an analytic part and a small remainder. To this end, we let $M>1$ and introduce open sets $U_1 = U_1(\zeta,M)$ and $\hat{U}_1 = \hat{U}_1(\zeta,M)$ as follows (see Figure \ref{fig: U1 and U2}):
\begin{align*}
U_{1} = & \; \big(\{k \,|\, \arg k \in (\arg (\omega^{2}k_{2}),\tfrac{7\pi}{6}), \; M^{-1}<|k|<1 \} 
	\\
& \; \cup \{k \,|\, \arg k \in (-\tfrac{\pi}{2}, \tfrac{\pi}{6}), \; 1<|k|<M\}\big) \cap \{k \,|\,  \re \Phi_{21}>0\}, 
	\\
\hat{U}_{1} = & \; \{k \,|\,  \arg k \in (\tfrac{\pi}{2},\arg k_{1}), \; 1<|k|<M \} \cap \{k \,|\,  \re \Phi_{21}>0\}.
\end{align*}
We will first construct decompositions of $r_{1}$ and $\hat{r}_{1}$; the decompositions of $r_{2}$ and $\hat{r}_{2}$ will then be obtained from the symmetry (\ref{r1r2 relation with kbar symmetry}). 

\begin{lemma}\label{decompositionlemmaIII}
There exist $M>1$ and decompositions
\begin{align}
& r_1(k) = r_{1,a}(x, t, k) + r_{1,r}(x, t, k), & & k \in \partial U_{1} \cap \partial \mathbb{D}, \nonumber \\
& \hat{r}_1(k) = \hat{r}_{1,a}(x, t, k) + \hat{r}_{1,r}(x, t, k), & & k \in \partial \hat{U}_{1} \cap \partial \mathbb{D}, \label{decomposition lemma analytic + remainder}
\end{align}
such that the functions $r_{1,a},r_{1,r},\hat{r}_{1,a},\hat{r}_{1,r}$ have the following properties:
\begin{enumerate}[$(a)$]
\item 
For each $(x,t) \in \III_\geq$, $r_{1,a}(x, t, k)$ is defined and continuous for $k \in \bar{U}_1$ and analytic for $k \in U_1$, while $\hat{r}_{1,a}(x, t, k)$ is defined and continuous for $k \in \bar{\hat{U}}_1$ and analytic for $k \in \hat{U}_1$.

\item For $(x,t) \in \III_\geq$, the functions $r_{1,a}$ and $\hat{r}_{1,a}$ satisfy
\begin{subequations}\label{bounds on analytic part}
\begin{align} \label{lemma: r1a approx at points} 
& \Big| r_{1,a}(x, t, k)-\sum_{j=0}^{N}\frac{r_{1}^{(j)}(k_{\star})}{j!}(k-k_{\star})^{j}  \Big| \leq C |k-k_{\star}|^{N+1} e^{\frac{t}{4}|\re \Phi_{21}(\zeta,k)|}, & & k \in \bar{U}_{1}, \, k_{\star} \in \mathcal{R}, 
	\\ \label{lemma: r1ahat approx at points}
& \Big| \hat{r}_{1,a}(x, t, k)-\sum_{j=0}^{N}\frac{\hat{r}_{1}^{(j)}(k_{\star})}{j!}(k-k_{\star})^{j} \Big| \leq C|k-k_{\star}|^{N+1} e^{\frac{t}{4}|\re \Phi_{21}(\zeta,k)|},  & & k \in \bar{\hat{U}}_{1}, \, k_{\star} \in \hat{\mathcal{R}}, 
\end{align}
\end{subequations}
where $\mathcal{R} := \{1, \omega, -i, \pm e^{\frac{\pi i}{6}}, \pm e^{-\frac{\pi i}{6}}\}$, $\hat{\mathcal{R}} :=\{i, \omega\}$, and the constant $C$ is independent of $\zeta, t, k$. 

\item For each $1 \leq p \leq \infty$, the $L^p$-norm of $r_{1,r}(x,t,\cdot)$ on $\partial U_{1} \cap \partial \mathbb{D}$ is $O(t^{-N})$ and the $L^p$-norm of $\hat{r}_{1,r}(x, t, \cdot)$ on $\partial \hat{U}_{1} \cap \partial \mathbb{D}$ is $O(t^{-N})$ uniformly for $(x,t) \in \III_\geq$ as $t \to \infty$.

\end{enumerate}
\end{lemma}
\begin{proof}
The proof follows the standard approach of \cite{DZ1993}. Indeed, using the fact that $\theta \mapsto \im \Phi_{21}(\zeta,e^{i\theta})$ $= (\zeta - \cos\theta)\sin\theta$ is monotonic on the intervals $(\arg (\omega^{2}k_{2}),\tfrac{7\pi}{6})$, $(-\tfrac{\pi}{2}, \tfrac{\pi}{6})$, and $(\tfrac{\pi}{2},\arg k_{1})$, it is not hard to construct functions with the desired properties using the method of \cite{DZ1993}.
\end{proof}

The decompositions $r_{2} = r_{2,a} + r_{2,r}$ and $\hat{r}_{2} = \hat{r}_{2,a} + \hat{r}_{2,r}$ involve the open sets $U_{2} := \{k \,|\, \bar{k}^{-1} \in U_{1}\}$ and $\hat{U}_{2} := \{k \,|\, \bar{k}^{-1} \in \hat{U}_{1}\}$ and are given by
\begin{align}\label{def of r2a and r2r}
& r_{2,a}(k) := \tilde{r}(k)\overline{r_{1,a}(\bar{k}^{-1})}, \quad k \in U_{2},
& & r_{2,r}(k) := \tilde{r}(k)\overline{r_{1,r}(\bar{k}^{-1})}, \quad k \in \partial U_{2}\cap \partial \mathbb{D},
	\\
& \hat{r}_{2,a}(k) := \tilde{r}(k)\overline{\hat{r}_{1,a}(\bar{k}^{-1})}, \quad k \in \hat{U}_{2},
& & \hat{r}_{2,r}(k) := \tilde{r}(k)\overline{\hat{r}_{1,r}(\bar{k}^{-1})}, \quad k \in \partial \hat{U}_{2}\cap \partial \mathbb{D}.
\end{align}

\subsection{Transformations of the RH problem}

\begin{figure}
\begin{center}
\begin{tikzpicture}[master,scale=0.85]
\node at (0,0) {};
\draw[black,line width=0.55 mm,->-=0.2,->-=0.47,->-=0.65,->-=0.9] (0,0)--(30:6);
\draw[black,line width=0.55 mm,->-=0.26,->-=0.53,->-=0.75,->-=0.98] (0,0)--(90:5.2);
\draw[black,line width=0.55 mm,->-=0.2,->-=0.47,->-=0.65,->-=0.9] (0,0)--(150:6);
\draw[black,line width=0.55 mm,->-=0.2,->-=0.47,->-=0.65,->-=0.9] (0,0)--(-30:6);
\draw[black,line width=0.55 mm,->-=0.26,->-=0.53,->-=0.75,->-=0.98] (0,0)--(-90:5.2);
\draw[black,line width=0.55 mm,->-=0.2,->-=0.47,->-=0.65,->-=0.9] (0,0)--(-150:6);

\draw[black,line width=0.55 mm] ([shift=(-180:3cm)]0,0) arc (-180:180:3cm);
\draw[black,arrows={-Triangle[length=0.27cm,width=0.18cm]}]
($(3:3)$) --  ++(90:0.001);
\draw[black,arrows={-Triangle[length=0.27cm,width=0.18cm]}]
($(57:3)$) --  ++(-30:0.001);
\draw[black,arrows={-Triangle[length=0.27cm,width=0.18cm]}]
($(123:3)$) --  ++(210:0.001);
\draw[black,arrows={-Triangle[length=0.27cm,width=0.18cm]}]
($(177:3)$) --  ++(90:0.001);
\draw[black,arrows={-Triangle[length=0.27cm,width=0.18cm]}]
($(243:3)$) --  ++(330:0.001);
\draw[black,arrows={-Triangle[length=0.27cm,width=0.18cm]}]
($(297:3)$) --  ++(210:0.001);

\node at (60:3.3) {\small $9_r$};
\node at (119:3.25) {\small $7$};

\node at (72:1.23) {\small $1''$};

\node at (82:2.55) {\small $1_{r}''$};
%
\node at (109:2.25) {\small $3$};
\node at (130:2.35) {\small $4$};

\node at (86:4.94) {\small $4'$};

\node at (84.4:3.72) {\small $4_{r}'$};

\node at (104:3.8) {\small $1$};
\node at (136.5:3.8) {\small $6$};

\node at (97:3.24) {\small $2$};
\node at (143:3.27) {\small $5$};

\draw[black,line width=0.55 mm,->-=0.27,->-=0.75] (90:2)--(113.175:3)--(90:4.5);
\draw[black,line width=0.55 mm,->-=0.27,->-=0.75] (90+120:2)--(113.175+120:3)--(90+120:4.5);
\draw[black,line width=0.55 mm,->-=0.27,->-=0.75] (90+240:2)--(113.175+240:3)--(90+240:4.5);

\draw[black,line width=0.55 mm,->-=0.27,->-=0.75] (-90:2)--(-113.175:3)--(-90:4.5);
\draw[black,line width=0.55 mm,->-=0.27,->-=0.75] (-90+120:2)--(-113.175+120:3)--(-90+120:4.5);
\draw[black,line width=0.55 mm,->-=0.27,->-=0.75] (-90+240:2)--(-113.175+240:3)--(-90+240:4.5);

\draw[black,arrows={-Triangle[length=0.27cm,width=0.18cm]}]
($(18:3)$) --  ++(18-90:0.001);
\draw[black,arrows={-Triangle[length=0.27cm,width=0.18cm]}]
($(95:3)$) --  ++(95-90:0.001);
\draw[black,arrows={-Triangle[length=0.27cm,width=0.18cm]}]
($(18+120:3)$) --  ++(18+120-90:0.001);
\draw[black,arrows={-Triangle[length=0.27cm,width=0.18cm]}]
($(95+120:3)$) --  ++(95+120-90:0.001);
\draw[black,arrows={-Triangle[length=0.27cm,width=0.18cm]}]
($(18+240:3)$) --  ++(18+240-90:0.001);
\draw[black,arrows={-Triangle[length=0.27cm,width=0.18cm]}]
($(95+240:3)$) --  ++(95+240-90:0.001);

\draw[black,line width=0.55 mm,-<-=0.47] ([shift=(30:2cm)]0,0) arc (30:90:2cm);
\draw[black,line width=0.55 mm,-<-=0.47] ([shift=(30:4.5cm)]0,0) arc (30:90:4.5cm);

\draw[black,line width=0.55 mm,-<-=0.47] ([shift=(30+120:2cm)]0,0) arc (30+120:90+120:2cm);
\draw[black,line width=0.55 mm,-<-=0.47] ([shift=(30+120:4.5cm)]0,0) arc (30+120:90+120:4.5cm);

\draw[black,line width=0.55 mm,-<-=0.47] ([shift=(30+240:2cm)]0,0) arc (30+240:90+240:2cm);
\draw[black,line width=0.55 mm,-<-=0.47] ([shift=(30+240:4.5cm)]0,0) arc (30+240:90+240:4.5cm);

\node at (60:2.3) {\small $9_{L}$};
\node at (60:4.9) {\small $9_{R}$};

\draw[blue,fill] (113.175:3) circle (0.12cm);
\draw[red,fill] (113.175+120:3) circle (0.12cm);
\draw[green,fill] (113.175+240:3) circle (0.12cm);
\draw[blue,fill] (-113.175:3) circle (0.12cm);
\draw[red,fill] (-113.175+120:3) circle (0.12cm);
\draw[green,fill] (-113.175+240:3) circle (0.12cm);

\draw[blue,fill] (0:3.7) circle (0.12cm);
\draw[blue,fill] (0:2.35) circle (0.12cm);
\draw[red,fill] (120:3.7) circle (0.12cm);
\draw[red,fill] (120:2.35) circle (0.12cm);
\draw[green,fill] (240:3.7) circle (0.12cm);
\draw[green,fill] (240:2.35) circle (0.12cm);

\end{tikzpicture}
\end{center}
\begin{figuretext}
\label{IIIgeqGamma1fig}
The contour $\Gamma^{(1)}$ for Sector $\III_\geq$. The saddle points of $\Phi_{21}$, $\Phi_{32}$, and $\Phi_{31}$ are colored blue, green, and red, respectively. 
\end{figuretext}
\end{figure}

\subsubsection{First transformation}
The jump matrices $v_7$  and $v_9$ admit the factorizations
$$v_7^{-1} = (v_3^{(1)})^{-1} v_2^{(1)} v_1^{(1)}, \qquad
v_9 = v_{9_L}^{(1)} v_{9_r}^{(1)} v_{9_R}^{(1)},$$
where
\begin{align}\nonumber
& v_{1}^{(1)} 
= \begin{pmatrix}
1 & \hat{r}_{1,a}(k)e^{-\theta_{21}} & 0 \\
0 & 1 & 0 \\
r_{2,a}(\frac{1}{\omega^{2}k})e^{\theta_{31}} & -r_{1,a}(\frac{1}{\omega k})e^{\theta_{32}} & 1
\end{pmatrix}, 
\quad v_{2}^{(1)} = \begin{pmatrix}
1+r_{1}(k)r_{2}(k) & 0 & 0 \\
0 & \frac{1}{1+r_{1}(k)r_{2}(k)} & 0 \\
0 & 0 & 1
\end{pmatrix} + v_{2,r}^{(1)},
	\\ \nonumber
& v_3^{(1)} 
= \begin{pmatrix}
1 & 0 & -r_{1,a}(\frac{1}{\omega^{2} k})e^{-\theta_{31}} \\
-\hat{r}_{2,a}(k)e^{\theta_{21}} & 1 & (r_{2,a}(\frac{1}{\omega k}) +r_{1,a}(\frac{1}{\omega^2 k})\hat{r}_{2,a}(k))e^{-\theta_{32}} \\
0 & 0 & 1
\end{pmatrix},
	\\ \nonumber
& v_{9_L}^{(1)} = \begin{pmatrix}
1 & 0 & -r_{2,a}(\omega^{2}k)e^{-\theta_{31}} \\
r_{1,a}(\frac{1}{k})e^{\theta_{21}} & 1 & r_{1,a}(\omega k)e^{-\theta_{32}} \\
0 & 0 & 1
\end{pmatrix}, \quad v_{9_R}^{(1)} = \begin{pmatrix}
1 & r_{2,a}(\frac{1}{k})e^{-\theta_{21}} & 0 \\
0 & 1 & 0  \\
-r_{1,a}(\omega^{2}k)e^{\theta_{31}} & r_{2,a}(\omega k)e^{\theta_{32}} & 1
\end{pmatrix}, 
	\\ \label{IIIv1jumps}
& v_{9_r}^{(1)} = \begin{pmatrix}
1+r_{1,r}(\omega^{2}k)r_{2,r}(\omega^{2}k) & g_{2}(\omega k)e^{-\theta_{21}} & -r_{2,r}(\omega^{2}k)e^{-\theta_{31}} \\
g_{1}(\omega k)e^{\theta_{21}} & g(\omega k) & h_{1}(\omega k)e^{-\theta_{32}} \\
-r_{1,r}(\omega^{2}k)e^{\theta_{31}} & h_{2}(\omega k)e^{\theta_{32}} & 1
\end{pmatrix},
\end{align}
where $f$ is defined in \eqref{def of f} and $h_1, h_2, g_1, g_2, g$ are defined by
\begin{align*}
h_{1}(k) = & \; r_{1,r}(k) + r_{1,a}(\tfrac{1}{\omega^{2}k})r_{2,r}(\omega k), 
	\qquad
h_{2}(k) = r_{2,r}(k) + r_{2,a}(\tfrac{1}{\omega^{2}k})r_{1,r}(\omega k), 
	\\
g_{1}(k) = & \; r_{1,r}(\tfrac{1}{\omega^{2}k})-r_{1,r}(\omega k) \big( r_{1,r}(k)+r_{1,a}(\tfrac{1}{\omega^{2}k})r_{2,r}(\omega k) \big), 
	\\
g_{2}(k) = & \; r_{2,r}(\tfrac{1}{\omega^{2}k})-r_{2,r}(\omega k) \big( r_{2,r}(k)+r_{2,a}(\tfrac{1}{\omega^{2}k})r_{1,r}(\omega k) \big), 
	\\
g(k) = & \; 1+r_{1,r}(k)\big(r_{1,r}(\omega k)r_{2,a}(\tfrac{1}{\omega^{2}k})+r_{2,r}(k)\big) \\
& +r_{1,a}(\tfrac{1}{\omega^{2}k})r_{2,r}(\omega k)\big( r_{1,r}(\omega k)r_{2,a}(\tfrac{1}{\omega^{2}k})+r_{2,r}(k) \big) + r_{1,r}(\tfrac{1}{\omega^{2}k})r_{2,r}(\tfrac{1}{\omega^{2}k}).
\end{align*}
The matrix $v_{2,r}^{(1)}$ in (\ref{IIIv1jumps}) is given by
\begin{align}\label{IIIv2r}
& v_{2,r}^{(1)} := \resizebox{0.5\hsize}{!}{$\left( \begin{array}{c}
c_{13,r}c_{31,r} \\
\big( c_{21,r}+(c_{23,r}-c_{13,r}c_{21,a})c_{31,r}+c_{21,r}r_{1}(k)r_{2}(k) \big)e^{\theta_{21}}  \\
c_{31,r}e^{\theta_{31}}
\end{array} \right. \dots $} 
	\\ \nonumber
& \resizebox{0.91\hsize}{!}{$ \dots \left. \begin{array}{c c}
\big( c_{12,r}+c_{13,r}(c_{32,r}-c_{12,a}c_{31,r}) + c_{12,r}r_{1}(k)r_{2}(k) \big)e^{-\theta_{21}} & c_{13,r}e^{-\theta_{31}} \\
c_{12,r}c_{21,r}+(c_{13,r}c_{21,a}-c_{23,r})(c_{12,a}c_{31,r}-c_{32,r})+c_{12,r}c_{21,r}r_{1}(k)r_{2}(k) & (c_{23,r}-c_{13,r}c_{21,a})c^{-\theta_{32}} \\
(c_{32,r}-c_{12,a}c_{31,r})e^{\theta_{32}} & 0
\end{array} \right),$}
\end{align}
where
\begin{align*}
& c_{12} := \hat{r}_{1}(k), & & c_{13} := r_{1}(\tfrac{1}{\omega^{2}k}), & & c_{23} := -r_{2}(\tfrac{1}{\omega k}), 
& & c_{21} := \hat{r}_{2}(k), & & c_{31} := r_{2}(\tfrac{1}{\omega^{2}k}), & & c_{32} := -r_{1}(\tfrac{1}{\omega k}).
\end{align*}
Here $c_{12, a} := \hat{r}_{1,a}(k)$ and $c_{12,r} := \hat{r}_{1,r}(k)$ denote the analytic approximation and the remainder, respectively, in a decomposition $c_{12} = c_{12, a} + c_{12, r}$ of $c_{12}$, and a similar notation is used also for the other coefficients. Let $\Gamma^{(1)}$ be the contour shown in Figure \ref{IIIgeqGamma1fig}, and let $\Gamma_{j}^{(1)}$ be the subcontour of $\Gamma^{(1)}$ labeled by $j$ in Figure \ref{IIIgeqGamma1fig}.
As a consequence of Lemma \ref{decompositionlemmaIII}, we have $\|v_{2,r}^{(1)}\|_{(L^{1}\cap L^{\infty})(\Gamma_2^{(1)})}=O(t^{-1})$ as $ t \to \infty$.

Let $\Gamma^{(1)}_\star$ be the self-intersection points of $\Gamma^{(1)}$.
In what follows, we define the jump matrix $v^{(1)}$ on $\Gamma^{(1)} \setminus \Gamma^{(1)}_\star$. For $j = 1, 2, 3, 9_L, 9_r, 9_R$, let $v_j^{(1)}$ be given by (\ref{IIIv1jumps}), and let
\begin{align}\label{v1oniRplusdef}
& v_7^{(1)} = v_{7}, \quad v_{1''}^{(1)} = v_{1''}, \quad 
v_{1_r''}^{(1)} = (v_{9_L}^{(1)})^{-1}v_{1''} (v_3^{(1)})^{-1}, \quad 
v_{4_r'}^{(1)} = v_{9_R}^{(1)} v_{4'} (v_1^{(1)})^{-1}, \quad 
v_{4'}^{(1)} = v_{4'},
\end{align}
where $v_j^{(1)}$ denotes the restriction of $v^{(1)}$ to $\Gamma_j^{(1)}$.
We extend $v^{(1)}$ to $\Gamma_4^{(1)} \cup \Gamma_5^{(1)} \cup \Gamma_6^{(1)}$ by means of the symmetries in (\ref{vjsymm}); more precisely, the jump matrices $v_j^{(1)}$, $j = 4,5,6$, are defined by
\begin{align}\nonumber
& v_4^{(1)} = \mathcal{A} \mathcal{B} v_1^{(1)}\Big(\frac{1}{\omega k}\Big)^{-1}\mathcal{B}\mathcal{A}^{-1} 
= \begin{pmatrix} 1 & (r_{1,a}(k) + \hat{r}_{1,a}(\frac{1}{\omega k})r_{2,a}(\omega^2 k))e^{-\theta_{21}} & -r_{2,a}(\omega^2 k)e^{-\theta_{31}} \\
0 & 1 & 0 \\
0 & -\hat{r}_{1,a}(\frac{1}{\omega k})e^{\theta_{32}} & 1 \end{pmatrix},
	\\ \nonumber
& v_5^{(1)} = \mathcal{A} \mathcal{B} v_2^{(1)}\Big(\frac{1}{\omega k}\Big)^{-1}\mathcal{B} \mathcal{A}^{-1} =
\begin{pmatrix}
1 & 0 & 0 \\
0 & 1+r_1(\frac{1}{\omega k})r_2(\frac{1}{\omega k}) & 0 \\
0 & 0 & \frac{1}{1+r_1(\frac{1}{\omega k})r_2(\frac{1}{\omega k})}
\end{pmatrix} + v_{5,r}^{(1)},
	\\ \label{IIIv1456}
& v_6^{(1)} = \mathcal{A} \mathcal{B} v_3^{(1)}\Big(\frac{1}{\omega k}\Big)^{-1} \mathcal{B} \mathcal{A}^{-1} 
= \begin{pmatrix} 1 & 0 & 0 \\
-r_{2,a}(k)e^{\theta_{21}}  & 1 & \hat{r}_{2,a}(\frac{1}{\omega k})e^{-\theta_{32}} \\
r_{1,a}(\omega^2 k)e^{\theta_{31}} & 0 & 1 \end{pmatrix},
\end{align}
where
$$v_{5,r}^{(1)} := \mathcal{A} \mathcal{B} \bigg\{v_{2}^{(1)}\Big(\frac{1}{\omega k}\Big)^{-1} - \Big[(v_{2}^{(1)}-v_{2,r}^{(1)})\Big(\frac{1}{\omega k}\Big)\Big]^{-1} \bigg\} \mathcal{B} \mathcal{A}^{-1}.$$
We define the sectionally analytic function $n^{(1)}$ by
\begin{align}
n^{(1)}(x,t,k) = n(x,t,k)G^{(1)}(x,t,k),
\end{align}
where $G^{(1)}$ is defined for $k \in \{k \in \C \, | \, \arg k \in [\frac{\pi}{6},\frac{5\pi}{6}]\} \setminus \Gamma^{(1)}$ by
\begin{align}\label{G1def}
G^{(1)} = \begin{cases} 
v_{9_L}^{(1)}, & k \mbox{ on the $-$ side of }\Gamma_{9_r}^{(1)}, 
	 \\
(v_{9_R}^{(1)})^{-1}, & k \mbox{ on the $+$ side of }\Gamma_{9_r}^{(1)}, 
	 \\
(v_{3}^{(1)})^{-1}, & k \mbox{ on the $-$ side of }\Gamma_2^{(1)}, 
	 \\
(v_{1}^{(1)})^{-1}, & k \mbox{ on the $+$ side of }\Gamma_2^{(1)}, 
\end{cases}
\qquad
G^{(1)} = \begin{cases} 
v_{4}^{(1)}, & k \mbox{ on the $-$ side of }\Gamma_5^{(1)}, 
	 \\
v_{6}^{(1)}, & k \mbox{ on the $+$ side of }\Gamma_5^{(1)}, 
	\\
I, & \mbox{otherwise},
\end{cases}
\end{align}
and is extended to $\C \setminus \Gamma^{(1)}$ by means of the symmetry
\begin{align}\label{G1symmIII}
G^{(1)}(x,t, k) = \mathcal{A} G^{(1)}(x,t,\omega k)\mathcal{A}^{-1}.
\end{align}
By construction, $n^{(1)}$ obeys the jump relation $n^{(1)}_+ = n^{(1)}_-v^{(1)}$ on the part of $\Gamma^{(1)}$ that lies in the open sector $\{k \in \C \, | \, \arg k \in (\frac{\pi}{6},\frac{5\pi}{6})\}$. We extend the definition of $v^{(1)}$ by setting $v^{(1)} := (n^{(1)})_-^{-1}n^{(1)}_+$ on all of $\Gamma^{(1)}\setminus \Gamma^{(1)}_\star$. Then $v^{(1)}$ satisfies the $\mathcal{A}$-symmetry in (\ref{vjsymm}). 

\begin{remark}\upshape
The jump matrix $v^{(1)}$ does in general not satisfy the $\mathcal{B}$-symmetry in (\ref{vjsymm}). An attempt to preserve the $\mathcal{B}$-symmetry introduces additional jumps on parts of the lines $\omega \R \cup \omega^2 \R \cup \R$ that are not desirable in Sector \III. In the other asymptotic sectors, it is convenient to preserve the $\mathcal{B}$-symmetry at each stage of the transformations, but in Sector III it is not.
\end{remark}

It follows immediately from the definition of $G^{(1)}$ and Lemma \ref{decompositionlemmaIII} that $G^{(1)}$ is analytic on $\C \setminus \Gamma^{(1)}$. 
The next lemma follows from Lemma \ref{decompositionlemmaIII} and the signature tables of $\Phi_{21}$, $\Phi_{31}$, $\Phi_{32}$ (see Figure \ref{II fig: Re Phi 21 31 and 32 for zeta=1.2}).

\begin{lemma}
$G^{(1)}(x,t,k)$ and $G^{(1)}(x,t,k)^{-1}$ are uniformly bounded for $k \in \mathbb{C}\setminus \Gamma^{(1)}$ and $(x,t)\in \III_\geq$. Furthermore, $G^{(1)}(x,t,k)=I$ whenever $|k|$ is large enough.
\end{lemma}

The function $n^{(1)}$ satisfies the following RH problem for $j = 1$. 

\begin{RHproblem}[RH problem for $n^{(j)}$]\label{RHnj}
Find $n^{(j)}(x,t,k)$ with the following properties:
\begin{enumerate}[$(a)$]
\item\label{RHnjitema} $n^{(j)}(x,t,\cdot) : \C \setminus \Gamma^{(j)} \to \mathbb{C}^{1 \times 3}$ is analytic.

\item\label{RHnjitemb} On $\Gamma^{(j)} \setminus \Gamma^{(j)}_\star$, the boundary values of $n^{(j)}$ exist, are continuous, and satisfy $n^{(j)}_+ = n^{(j)}_-v^{(j)}$.

\item\label{RHnjitemc} $n^{(j)}(x,t,k) = O(1)$ as $k \to k_{\star} \in \Gamma^{(j)}_\star$.

\item\label{RHnjitemd} $n^{(j)}$ obeys the symmetry $n^{(j)}(x,t, k) = n^{(j)}(x,t,\omega k)\mathcal{A}^{-1}$ for $k \in \C \setminus \Gamma^{(j)}$.

\item\label{RHnjiteme} $n^{(j)}(x,t,k) = (1,1,1) + O(k^{-1})$ as $k \to \infty$.
\end{enumerate}
\end{RHproblem}

\subsubsection{Second transformation}
The purpose of the second transformation is to isolate the jumps that will later be approximated by the Painlev\'e II model problem corresponding to the Hastings--McLeod solution. 
Let $V_1$ and $V_2$ be the open subsets displayed in Figure \ref{IIIgeqGamma2fig}. 
Define $v_{1_Y}^{(2)}$ and $v_{4_Y}^{(2)}$ by
\begin{align}\label{v1Yv4Ydef}
v_{1_Y}^{(2)} = \begin{pmatrix}
1 & 0 & 0 \\
0 & 1 & 0 \\
r_{2,a}(\frac{1}{\omega^{2}k})e^{\theta_{31}} & 0 & 1
\end{pmatrix},
\qquad
v_{4_Y}^{(2)} = \begin{pmatrix}
1 & 0 & -r_{1,a}(\frac{1}{\omega^{2} k})e^{-\theta_{31}} \\
0 & 1 & 0 \\
0 & 0 & 1
\end{pmatrix}.
\end{align} 
Let $\Gamma^{(2)}$ be the contour shown for $\arg k \in [\pi/6, 5\pi/6]$ in Figure \ref{IIIgeqGamma2fig} and extended to $\C$ by symmetry. We define the sectionally analytic function $n^{(2)}$ by
\begin{align}\label{n2defIIIgeq}
n^{(2)}(x,t,k) = n^{(1)}(x,t,k)G^{(2)}(x,t,k), \qquad  k \in \mathbb{C}\setminus \Gamma^{(2)},
\end{align}
where $G^{(2)}$ is defined for $k \in \{k \in \C \, | \, \arg k \in [\frac{\pi}{6},\frac{5\pi}{6}]\}  \setminus \Gamma^{(2)}$ by
\begin{align*}
G^{(2)}(x,t,k) = \begin{cases}
(v_{1_Y}^{(2)})^{-1}, & k \in V_1,
	\\
(v_{4_Y}^{(2)})^{-1}, & k \in V_2,
	\\	
I, &  \mbox{otherwise},
\end{cases}
\end{align*}
and is extended to $\mathbb{C}\setminus \Gamma^{(2)}$ by the $\mathcal{A}$-symmetry (as in (\ref{G1symmIII})). 
Then $n^{(2)}$ satisfies RH problem \ref{RHnj} for $j=2$ with $v^{(2)} = (G^{(2)}_-)^{-1} v^{(1)} G^{(2)}_+$. More explicitly, $v_{1_Y}^{(2)}$ and $v_{4_Y}^{(2)}$ are given by (\ref{v1Yv4Ydef}), and
\begin{align*}
& v_{2_Y}^{(2)} = (v_{1_Y}^{(2)})^{-1}, \qquad  v_{3_Y}^{(2)} = (v_{4_Y}^{(2)})^{-1},
\qquad
 v_{5_Y}^{(2)} =  v_1^{(1)} (v_{1_Y}^{(2)})^{-1}
= \begin{pmatrix}
1 & \hat{r}_{1,a}(k)e^{-\theta_{21}} & 0 \\
0 & 1 & 0 \\
0 & -r_{1,a}(\frac{1}{\omega k})e^{\theta_{32}} & 1
\end{pmatrix},
	\\
& v_{6_Y}^{(2)} =  v_{1_Y}^{(2)} v_6^{(1)}
= \begin{pmatrix}
1 & 0 & 0 \\
-r_{2,a}(k)e^{\theta_{21}}  & 1 & \hat{r}_{2,a}(\frac{1}{\omega k})e^{-\theta_{32}} \\
(r_{1,a}(\omega^2 k) + r_{2,a}(\frac{1}{\omega^2 k}))e^{\theta_{31}} & 0 & 1
\end{pmatrix},
	\\ 
& v_{7_Y}^{(2)} =  v_{4_Y}^{(2)} v_4^{(1)}
=  \resizebox{0.82\hsize}{!}{$\begin{pmatrix}
1 & (r_{1,a}(k) + \hat{r}_{1,a}(\frac{1}{\omega k})(r_{1,a}(\frac{1}{\omega^2 k}) + r_{2,a}(\omega^2 k)))e^{-\theta_{21}} & -(r_{1,a}(\frac{1}{\omega^2 k}) + r_{2,a}(\omega^2 k))e^{-\theta_{31}} \\
0 & 1 & 0 \\
0 & -\hat{r}_{1,a}(\frac{1}{\omega k})e^{\theta_{32}}  & 1
\end{pmatrix}$},
	\\
& v_{8_Y}^{(2)} =  v_3^{(1)} (v_{4_Y}^{(2)})^{-1}
= \begin{pmatrix}
1 & 0 & 0 \\
-\hat{r}_{2,a}(k)e^{\theta_{21}}  & 1 & r_{2,a}(\frac{1}{\omega k})e^{-\theta_{32}} \\
0 & 0 & 1
\end{pmatrix},
\qquad v_{9_Y}^{(2)} = v_{4_Y}^{(2)} (v_7^{(1)})^{-1} (v_{1_Y}^{(2)})^{-1}.
\end{align*}

Our next lemma shows that the jumps of $n^{(2)}$ on $\Gamma_{4_r'}^{(2)}$ and $\Gamma_{1_r''}^{(2)}$ are small as a consequence of Assumption \ref{nounstablemodesassumption}. 

\begin{lemma}\label{vsmallnearilemma}
The $L^\infty$-norm of $v^{(2)} - I$ on $\Gamma_{4_r'}^{(2)} \cup \Gamma_{1_r''}^{(2)}$ is $O(t^{-N-1})$ as $t \to \infty$ uniformly for $(x,t) \in \III_\geq$.
\end{lemma}
\begin{proof}
We give the proof for $j = 4_r'$; the proof for $j = 1_r''$ is similar. By \eqref{v1oniRplusdef} and \eqref{n2defIIIgeq}, we have
\begin{align*}
& v_{4_r'}^{(2)} = v_{4_r'}^{(1)} = v_{9_R}^{(1)} v_{4'} (v_1^{(1)})^{-1}
 = \begin{pmatrix}
1 & -\hat{r}_{1,a}(k)e^{-t \Phi_{21}} & 0 \\
0 & 1 & 0 \\
-\big(r_{1,a}(\omega^2 k) + r_{2,a}(\frac{1}{\omega^2 k})\big)e^{t \Phi_{31}} 
&  (v_{4_r'}^{(2)})_{32} & 1
\end{pmatrix},
	\\
& (v_{4_r'}^{(2)})_{32} := \big(r_{1,a}(\tfrac{1}{\omega k}) + r_{2,a}(\omega k) + r_{1,a}(\omega^2 k)(\hat{r}_{1,a}(k) + r_{2,a}(\tfrac{1}{k})) + \hat{r}_{1,a}(k) r_{2,a}(\tfrac{1}{\omega^2 k}) \big)e^{t \Phi_{32}}.
\end{align*}
In particular, each entry of $v_{4_r'}^{(2)} - I$ is suppressed by an exponential of the form $e^{-t |\re \Phi_{ij}|}$ where $|\re \Phi_{ij}(\zeta, k)| \geq c |k - i|$ on $\Gamma_{4_r'}^{(2)}$. 
Thus, $v_{4_r'}^{(2)} - I$ is small as $t \to \infty$ on $\Gamma_{4_r'}^{(2)}$ except near $i$. 
Let us temporarily assume that the spectral functions $r_1$ and $r_2$ have analytic continuations to a neighborhood of $\partial \D$, so that we can choose $r_{1,a} = r_1$ and $r_{2,a} = r_2$. Then the $L^\infty$-norm of the $(12)$-entry $(v_{4_r'}^{(2)})_{12} = -\hat{r}_{1}(k)e^{-t \Phi_{21}}$ is $O(t^{-N})$, because $r_1(k)$ vanishes to all orders at $i$ due to Assumption \ref{nounstablemodesassumption}. By (\ref{r1r2 relation with kbar symmetry}), $r_2(k)$ also vanishes to all orders at $i$. Thus, using the identity (\ref{r1r2 relation on the unit circle}) to write
$(v_{4_r'}^{(2)})_{31} = r_{1}(\frac{1}{\omega k}) r_{2}(k) e^{t \Phi_{31}}$ we see that the $L^\infty$-norm of the $(31)$-entry also is $O(t^{-N})$. Applying (\ref{r1r2 relation on the unit circle}) again, we conclude that $\smash{(v_{4_r'}^{(2)})_{32}}$ also vanishes to all orders at $i$, so that $\smash{(v_{4_r'}^{(2)})_{32}}$ also is $O(t^{-N})$ (for any $N$). This completes the proof of the lemma in the case when $r_{1,a} = r_1$ and $r_{2,a} = r_2$. In the general case, analogous arguments together with Lemma \ref{decompositionlemmaIII} show that the coefficients of the exponentials $e^{-t \Phi_{ij}}$ in $v_{4_r'}^{(2)} - I$ are $O((k-i)^{N+1} e^{\frac{t}{2}|\re \Phi_{21}(\zeta,k)|})$ as $k \to i$.
Thus 
$$\| v_{4_r'}^{(2)} - I \|_{L^\infty(\Gamma_{4_r'}^{(2)})} \leq C \sup_{k \in \Gamma_{4_r'}^{(2)}} |k-i|^{N+1} e^{-\frac{c}{2}t|k-i|} \leq C t^{-N-1},$$
uniformly for $(x,t) \in \III_\geq$, and the lemma follows.
\end{proof}

\begin{figure}
\begin{center}
\begin{tikzpicture}[master,scale=0.9]
\node at (0,0) {};
\draw[black,line width=0.65 mm,->-=0.25,->-=0.53,->-=0.80,->-=0.98] (0,0)--(30:7.5);
\draw[black,line width=0.65 mm,->-=0.25,->-=0.53,->-=0.80,->-=0.98] (0,0)--(90:7.5);
\draw[black,line width=0.65 mm,->-=0.25,->-=0.53,->-=0.80,->-=0.98] (0,0)--(150:7.5);

\draw[black,line width=0.65 mm] ([shift=(30:3*1.5cm)]0,0) arc (30:150:3*1.5cm);
\draw[black,arrows={-Triangle[length=0.27cm,width=0.2cm]}]
($(57:3*1.5)$) --  ++(-32:0.001);
\draw[black,arrows={-Triangle[length=0.27cm,width=0.2cm]}]
($(118:3*1.5)$) --  ++(30:0.001);

\node at (60:3.2*1.5) {\small $9_r$};
\node at (118.5:3.15*1.5) {\small $9_Y$};

\node at (77:1.15*1.5) {\small $1''$};
\node at (84:2.5*1.5) {\small $1_r''$};



\node at (86.2:3.85*1.5) {\small $4_{r}'$};
\node at (87.1:4.78*1.5) {\small $4'$};
	
\node at (96:3.56*1.5) {\small $5_Y$};
\node at (143.5:3.6*1.5) {\small $6_Y$};
\node at (140:2.6*1.5) {\small $7_Y$};
\node at (101:2.6*1.5) {\small $8_Y$};

\node at (97:3.2*1.5) {\small $2$};
\node at (142:3.18*1.5) {\small $5$};

\draw[black,line width=0.65 mm,->-=0.25,->-=0.75] (90:2*1.5)--(113.175:3*1.5)--(90:4.5*1.5);

\draw[black,line width=0.65 mm,->-=0.25,->-=0.75] (-90+240:2*1.5)--(-113.175+240:3*1.5)--(-90+240:4.5*1.5);

\draw[black,arrows={-Triangle[length=0.27cm,width=0.2cm]}]
($(95:3*1.5)$) --  ++(95-90:0.001);
\draw[black,arrows={-Triangle[length=0.27cm,width=0.2cm]}]
($(18+120:3*1.5)$) --  ++(18+120-90:0.001);

\draw[black,line width=0.65 mm,-<-=0.47] ([shift=(30:2*1.5cm)]0,0) arc (30:90:2*1.5cm);
\draw[black,line width=0.65 mm,-<-=0.47] ([shift=(30:4.5*1.5cm)]0,0) arc (30:90:4.5*1.5cm);

\node at (60:2.2*1.5) {\small $9_{L}$};
\node at (60:4.75*1.5) {\small $9_{R}$};

\draw[blue,fill] (113.175:3*1.5) circle (0.12cm);
\draw[green,fill] (-113.175+240:3*1.5) circle (0.12cm);

\draw[black,line width=0.65 mm,->-=0.5] (120:3.61*1.5)--(90:4.5*1.5);
\draw[black,line width=0.65 mm,->-=0.5] (120:3.61*1.5)--(150:4.5*1.5);
\draw[black,line width=0.65 mm,->-=0.6] (90:2*1.5)--(120:2.41*1.5);
\draw[black,line width=0.65 mm,->-=0.6] (150:2*1.5)--(120:2.41*1.5);

\node at (106:4.06*1.5) {\small $1_Y$};
\node at (134:4.06*1.5) {\small $2_Y$};
\node at (129:1.95*1.5) {\small $3_Y$};
\node at (107:1.95*1.5) {\small $4_Y$};

\draw[red,fill] (120:3.61*1.5) circle (0.12cm);
\draw[red,fill] (120:2.41*1.5) circle (0.12cm);

\node at (127:3.38*1.5) {\small $V_1$};
\node at (125:2.65*1.5) {\small $V_2$};

\end{tikzpicture}
\end{center}
\begin{figuretext}
\label{IIIgeqGamma2fig}
The contour $\Gamma^{(2)} \cap \{k \,|\, \arg k \in [\frac{\pi}{6},\frac{5\pi}{6}]\}$ for Sector $\III_\geq$. The blue dot is the saddle point $k_1$ of $\Phi_{21}$; the green dot is the saddle point $\omega^2 k_2$ of $\Phi_{32}$; the red dots are the saddle points $\omega k_3$ and $\omega k_4$ of $\Phi_{31}$.
\end{figuretext}
\end{figure}

\subsubsection{Third transformation}
Let us define $\delta(k)$ by
\begin{align}\label{deltadefIII}
\delta(k) = e^{\frac{1}{2\pi i} \int_{\gamma_{(\omega, i)}} \frac{\ln(1 + r_1(s)r_{2}(s))}{s - k} ds}, \qquad k \in \C \setminus \gamma_{(\omega, i)},
\end{align}
where $\gamma_{(\omega, i)}$ denotes the clockwise subarc of the unit circle going from $\omega$ to $i$. 

\begin{lemma}
The function $\delta(k)$ has the following properties:
\begin{enumerate}[$(a)$]

\item $\delta(k)$ and $\delta(k)^{-1}$ are bounded and analytic functions of $k \in \C \setminus \gamma_{(\omega, i)}$ with continuous boundary values on $\gamma_{(\omega, i)} \setminus \{\omega, i\}$.

\item On $\gamma_{(\omega, i)} \setminus \{\omega, i\}$, $\delta$ satisfies the jump condition 
\begin{align}
\delta_+(k) = \delta_-(k)(1 + r_1(k)r_2(k)), \qquad k \in \gamma_{(\omega, i)}\setminus \{\omega, i\}.
\end{align}

\item As $k \to \infty$, $\delta(k) = 1 + O(k^{-1})$.

\end{enumerate}
\end{lemma}
\begin{proof}
All assertions follow easily from (\ref{deltadefIII}).
\end{proof}

We define $n^{(3)}$ by
\begin{align*}
n^{(3)}(x,t,k) = n^{(2)}(x,t,k)\Delta(k), \qquad  k \in \mathbb{C}\setminus \Gamma^{(3)},
\end{align*}
where $\Gamma^{(3)}=\Gamma^{(2)}$ and
$$\Delta(k) = \begin{pmatrix}
\frac{\delta(\omega^{2}k)}{\delta(k)}\frac{\delta(\frac{1}{k})}{\delta(\frac{1}{\omega^{2}k})} & 0 & 0 \\
0 & \frac{\delta(k)}{\delta(\omega k)}\frac{\delta(\frac{1}{\omega k})}{\delta(\frac{1}{k})} & 0 \\
0 & 0 & \frac{\delta(\omega k)}{\delta(\omega^{2} k)}\frac{\delta(\frac{1}{\omega^{2} k})}{\delta(\frac{1}{\omega k})}
\end{pmatrix}.$$
Note that $\Delta(k) = \mathcal{A} \Delta(\omega k) \mathcal{A}^{-1}$. It follows that $n^{(3)}$ satisfies RH problem \ref{RHnj} for $j = 3$, where the jump matrix is given by $v^{(3)} =  \Delta_-^{-1} v^{(2)} \Delta_+$.
This transformation ensures that 
$$v_{2}^{(3)} = I + v_{2,r}^{(3)}, \qquad v_{5}^{(3)} = I + v_{5,r}^{(3)},$$
where $v_{2,r}^{(3)}$ and $v_{5,r}^{(3)}$ are small.
Moreover, we have
\begin{align}\label{vjY3}
v_{j_Y}^{(3)} = e^{x\widehat{\mathcal{L}(k)} + t\widehat{\mathcal{Z}(k)}}v_{j_Y,0}^{(3)}, \qquad j = 1, \dots, 9,
\end{align}
where
\begin{align*}
& v_{1_Y,0}^{(3)} = (v_{2_Y,0}^{(3)})^{-1} = \begin{pmatrix}
1 & 0 & 0 \\
0 & 1 & 0 \\
\frac{\Delta_{11}}{\Delta_{33}}r_{2,a}(\frac{1}{\omega^{2}k}) & 0 & 1
\end{pmatrix},
\qquad v_{4_Y,0}^{(3)} = (v_{3_Y,0}^{(3)})^{-1} = \begin{pmatrix}
1 & 0 & -\frac{\Delta_{33}}{\Delta_{11}}r_{1,a}(\frac{1}{\omega^{2} k}) \\
0 & 1 & 0 \\
0 & 0 & 1
\end{pmatrix},
	\\
& v_{5_Y,0}^{(3)} = \begin{pmatrix}
1 & \frac{\Delta_{22}}{\Delta_{11}}\hat{r}_{1,a}(k) & 0 \\
0 & 1 & 0 \\
0 & -\frac{\Delta_{22}}{\Delta_{33}}r_{1,a}(\frac{1}{\omega k}) & 1
\end{pmatrix},
\qquad v_{6_Y,0}^{(3)} 
= \begin{pmatrix}
1 & 0 & 0 \\
-\frac{\Delta_{11}}{\Delta_{22}}r_{2,a}(k)  & 1 & \frac{\Delta_{33}}{\Delta_{22}}\hat{r}_{2,a}(\frac{1}{\omega k}) \\
\frac{\Delta_{11}}{\Delta_{33}}(r_{1,a}(\omega^2 k) + r_{2,a}(\frac{1}{\omega^2 k})) & 0 & 1
\end{pmatrix},
	\\
& v_{7_Y,0}^{(3)} =  \begin{pmatrix}
1 & \frac{\Delta_{22}}{\Delta_{11}}(r_{1,a}(k) + \hat{r}_{1,a}(\frac{1}{\omega k})(r_{1,a}(\frac{1}{\omega^2 k}) + r_{2,a}(\omega^2 k))) & -\frac{\Delta_{33}}{\Delta_{11}}(r_{1,a}(\frac{1}{\omega^2 k}) + r_{2,a}(\omega^2 k)) \\
0 & 1 & 0 \\
0 & -\frac{\Delta_{22}}{\Delta_{33}}\hat{r}_{1,a}(\frac{1}{\omega k})  & 1
\end{pmatrix},
	\\
& v_{8_Y,0}^{(3)}
= \begin{pmatrix}
1 & 0 & 0 \\
-\frac{\Delta_{11}}{\Delta_{22}}\hat{r}_{2,a}(k)  & 1 & \frac{\Delta_{33}}{\Delta_{22}}r_{2,a}(\frac{1}{\omega k}) \\
0 & 0 & 1
\end{pmatrix}, \qquad v_{9_Y}^{(3)} = \Delta_-^{-1} v_{9_Y}^{(2)} \Delta_+.
\end{align*}

\subsection{Local parametrix}
The RH problem for $n^{(3)}$ has the property that the matrix $v^{(3)} - I$ is uniformly small as $(x,t) \to \infty$ everywhere on $\Gamma^{(3)}$ except near the points $1, \omega$, and $\omega^2$. 
Hence we only have to consider neighborhoods of these three points when computing the long-time asymptotics of $n^{(3)}$. In fact, by the $\mathcal{A}$-symmetry, it is enough to compute the contribution from a neighborhood of $\omega$. In this subsection, we find a local solution $m^\omega$ which approximates $n^{(3)}$ near $\omega$. 

The first step consists of finding approximations of the exponentials $e^{\pm \theta_{ij}}$ appearing in the jump matrices.
Introduce the variables $y$ and $z$ by
\begin{align}\label{yzdef}
y = \frac{2}{3} \bigg(\frac{3t}{2}\bigg)^{2/3} (\zeta -1) 
= \bigg(\frac{2}{3t}\bigg)^{1/3}(x-t), \qquad 
z = \bigg(\frac{3t}{2}\bigg)^{1/3} \frac{i  (k-\omega ) (k+\omega )}{4 k \omega}.
\end{align}
These definitions of $y$ and $z$ are chosen such that
\begin{align}\label{tPhiminusyzIII}
x\mathcal{L}(k) + t\mathcal{Z}(k) - (x\mathcal{L}(\omega) + t\mathcal{Z}(\omega))
 = \alpha \tau + \beta \sigma + O((k-\omega)^4), \qquad k \to \omega,
\end{align}
where $\alpha = \alpha(y,t,z)$ and $\beta = \beta(y,z)$ are defined by
\begin{align}\label{alphabetadef}
\alpha := \frac{i}{\sqrt{3}}\bigg(2\Big(\frac{3 t}{2}\Big)^{1/3} + \Big(\frac{3 t}{2}\Big)^{-1/3}y\bigg)z^2, \qquad
\beta := - i \bigg(yz + \frac{4z^3}{3}\bigg),
\end{align}
and $\tau$ and $\sigma$ are the constant matrices
\begin{align}\label{def of tau and sigma}
\tau := \begin{pmatrix} 1 & 0 & 0 \\ 0 & -2 & 0 \\ 0 & 0 & 1 \end{pmatrix}, \qquad
\sigma := \begin{pmatrix} 1 & 0 & 0 \\ 0 & 0 & 0 \\ 0 & 0 & -1 \end{pmatrix}.
\end{align}

Fix $\epsilon > 0$ small. The inverse image of the open disk $|z| < \epsilon$ under the map $k \mapsto \frac{i  (k-\omega ) (k+\omega )}{k \omega}$ consists of two components: one containing $\omega$ and one containing $-\omega$; we let $D(\epsilon)$ denote the component containing $\omega$. 
Then $k \mapsto z = z(t,k)$ is a biholomorphism from $D(\epsilon)$ onto the open disk of radius $(\frac{3t}{2})^{1/3}\frac{\epsilon}{4}$ centered at the origin.

The map $k \mapsto z$ maps the unit circle onto the real interval $(\frac{3t}{2})^{1/3}[-1/2,1/2] \subset \R$. More precisely, if $|k| = 1$, then
$$z = \frac{(\frac{3t}{2})^{1/3}}{4}(\sqrt{3}\cos(\arg k) + \sin(\arg k)).$$
In particular, the saddle point $k_1$ is mapped to a point $z_1 \geq 0$. Since $k_1 = 1/k_2$, it follows that $\omega^2 k_2$ is mapped to $-z_1 \leq 0$. As $t\to \infty$, since $|k_1 - \omega| = O(t^{-2/3})$, we have $z_1 = O(t^{-1/3})$.
Similarly, the map $k \mapsto z$ maps the ray $\omega \R_{> 0}$ onto the imaginary axis. More precisely, if $\arg k = 2\pi/3$, then
$$z = \frac{i (\frac{3t}{2})^{1/3}}{4}(|k| - |k|^{-1}).$$
Thus $\omega k_3 \in \omega [1,+\infty)$ is mapped to a point $iz_2$ with $z_2 := \frac{(\frac{3t}{2})^{1/3}}{4}(k_3 - k_3^{-1}) \geq 0$.
Since $k_3 = 1/k_4$, $\omega k_4$ is mapped to $-iz_2$.
As $t\to \infty$, since $|k_3 - 1| = O(t^{-1/3})$, we have $z_2 = O(1)$.

As a consequence of (\ref{tPhiminusyzIII}), $\theta_{ij}(k) - \theta_{ij}(\omega)$ agrees with $(\alpha \tau + \beta \sigma)_{ii} - (\alpha \tau + \beta \sigma)_{jj}$ to order three as $k \to \omega$. 
The next lemma shows that by including higher order powers of $z$, we can approximate $\theta_{ij}(k) - \theta_{ij}(\omega)$ to any order at $k = \omega$.

\begin{lemma}
We have
\begin{align}\label{Phi21expansion}
& t(\Phi_{21}(\zeta, k) - \Phi_{21}(\zeta, \omega)) = 
-3\alpha - \beta  + z^3(\zeta q_o(zt^{-1/3}) + q_e(zt^{-1/3})),
	\\\label{Phi31expansion}
& t(\Phi_{31}(\zeta, k) - \Phi_{31}(\zeta, \omega)) = -2\beta + z^3p_e(zt^{-1/3}),
	\\\label{Phi32expansion}
& t(\Phi_{32}(\zeta, k) - \Phi_{32}(\zeta, \omega)) = 3\alpha - \beta 
+ z^3(-\zeta q_o(zt^{-1/3}) + q_e(zt^{-1/3})),
\end{align}
where $q_o(\tilde{z})$ is an odd analytic function of $\tilde{z} = zt^{-1/3}$ in a neighborhood of $\tilde{z} = 0$, and $q_e(\tilde{z})$, $p_e(\tilde{z})$ are even analytic functions of $\tilde{z} = zt^{-1/3}$ in a neighborhood of $\tilde{z} = 0$ such that $q_e(0) = p_e(0) = 0$. 
\end{lemma}
\begin{proof}
A direct calculation shows that
$$p_e(zt^{-1/3}) := \frac{t(\Phi_{31}(\zeta, k) - \Phi_{31}(\zeta, \omega)) - 2i(y z + \frac{4z^3}{3})}{z^3} = -\frac{8 i (k-\omega )^2}{3 (k+\omega )^2}.$$
Since the right-hand side is invariant under $k \mapsto \omega^2/k$, $p_e$ is even. The assertion (\ref{Phi31expansion}) for $\Phi_{31}$ follows because $\tilde{z}=0$ corresponds to $k = \omega$ and the right-hand side has a double zero at $k = \omega$. 
Similarly, letting 
$$q_o(zt^{-1/3}) := \frac{8 k \omega  (k-\omega )}{\sqrt{3} (k+\omega )^3}, \qquad q_e(zt^{-1/3}) := -\frac{4 i (k-\omega )^2}{3 (k+\omega )^2},$$
a direct calculation shows that (\ref{Phi21expansion}) and (\ref{Phi32expansion}) hold. Since $q_e$ is invariant under $k \mapsto \omega^2/k$ while $q_o$ changes sign under this transformation, $q_e$ is even and $q_o$ is odd. Clearly, $q_e$ vanishes for $k = \omega$. 
\end{proof}

Let $N \geq 1$ be an integer. Let us write
$$q_e(\tilde{z}) = q_{e,N}(\tilde{z}) + q_{e,err}(\tilde{z}), \quad
q_o(\tilde{z}) = q_{o,N}(\tilde{z}) + q_{o,err}(\tilde{z}), \quad p_e(\tilde{z}) = p_{e,N}(\tilde{z}) + p_{e,err}(\tilde{z}),$$
where $q_{e,N}(\tilde{z})$, $q_{o,N}(\tilde{z})$, and $p_{e,N}(\tilde{z})$ denote the Taylor expansions of $q_e$, $q_o$, and $p_e$ to order $N$, i.e., 
\begin{align}
& q_{e,N}(\tilde{z}) = \frac{q_e''(0)}{2}\tilde{z}^2 + \cdots + \frac{q_e^{(N)}(0)}{N!}\tilde{z}^N, \qquad
q_{o,N}(\tilde{z}) = q_o'(0)\tilde{z} + \cdots + \frac{q_o^{(N)}(0)}{N!}\tilde{z}^N, 
	\\
& p_{e,N}(\tilde{z}) = \frac{p_e''(0)}{2}\tilde{z}^2 + \cdots + \frac{p_e^{(N)}(0)}{N!}\tilde{z}^N.
\end{align}
Then, for $\tilde{z}$ in a small and fixed neighborhood of $0$,
$$|q_{e,err}(\tilde{z})| \leq C|\tilde{z}|^{N+1}, \quad
|q_{o,err}(\tilde{z})| \leq C|\tilde{z}|^{N+1}, \quad
|p_{e,err}(\tilde{z})| \leq C|\tilde{z}|^{N+1}.$$
Define the approximations $e_{21,N}(x,t,z)$, $e_{31,N}(t,z)$, and $e_{32,N}(x,t,z)$ of the three exponentials
$$e^{-\beta + z^3(\zeta q_o(zt^{-1/3}) + q_e(zt^{-1/3}))}, \quad e^{z^3p_e(zt^{-1/3})}, \quad e^{-\beta + z^3(-\zeta q_o(zt^{-1/3}) + q_e(zt^{-1/3}))},$$
respectively, by
\begin{subequations}\label{eNdefIII}
\begin{align}
&e_{21,N}(x,t,z) = \sum_{j=0}^{N} \frac{[-\beta + z^3(\zeta q_{o,N}(zt^{-1/3}) + q_{e,N}(zt^{-1/3}))]^j}{j!},
	\\
&e_{31,N}(t,z) = \sum_{j=0}^{N} \frac{[z^3p_{e,N}(zt^{-1/3})]^j}{j!},
	\\
&e_{32,N}(x,t,z) = \sum_{j=0}^{N} \frac{[-\beta + z^3(-\zeta q_{o,N}(zt^{-1/3}) + q_{e,N}(zt^{-1/3}))]^j}{j!}.
\end{align}
\end{subequations}

\begin{lemma}\label{elemmaIII}
For any $a > 0$, the functions $e_{21,N}$, $e_{31,N}$, $e_{32,N}$ obey the following estimates uniformly for $\zeta \in [1, 2]$ and $t \geq 1$:
\begin{subequations}
\begin{align}\label{eeNestimatea}
&  \big|e^{-\beta + z^3(\zeta q_o(zt^{-1/3}) + q_e(zt^{-1/3}))} - e_{21,N}(x,t,z)\big| e^{-at^{1/3}|z|^2}
  \leq C|z|^{N+1}, && |z| \leq t^{\frac{1}{12}},
	\\\label{eeNestimateb}
&  \big|e^{z^3 p_e(zt^{-1/3})} - e_{31,N}(t,z)\big| e^{-a|z|^3}
  \leq C|zt^{-1/3}|^{N+1}, && |z| \leq t^{\frac{2}{15}},
	\\\label{eeNestimatec}
&  \big|e^{-\beta + z^3(-\zeta q_o(zt^{-1/3}) + q_e(zt^{-1/3}))} - e_{32,N}(x,t,z)\big| e^{-at^{1/3}|z|^2}
  \leq C|z|^{N+1}, && |z| \leq t^{\frac{1}{12}},
\end{align}
\end{subequations}
and
\begin{subequations}
\begin{align}\label{eNestimatea}
& |e_{21,N}(x,t,z)|e^{-at^{1/3}|z|^2} \leq C, && |z| \leq t^{1/3},
	\\\label{eNestimateb}
& |e_{31,N}(t,z)|e^{-a|z|^3} \leq C, && |z| \leq t^{1/3},
	\\\label{eNestimatec}
& |e_{32,N}(x,t,z)|e^{-at^{1/3}|z|^2} \leq C, &&  |z| \leq t^{1/3}.
\end{align}
\end{subequations}
\end{lemma}
\begin{proof}
Let us first show (\ref{eeNestimateb}).
We have
\begin{align}\nonumber
& \big|e^{z^3 p_e(zt^{-1/3})} - e_{31,N}(t,z)\big| = \bigg|e^{z^3 p_{e,err}(zt^{-1/3})} \bigg(e_{31,N} + \sum_{j=N +1}^\infty \frac{(z^3p_{e,N}(zt^{-1/3}))^j}{j!}\bigg)- e_{31,N}\bigg|
	\\ \label{ez3pe}
& \leq  |e^{z^3 p_{e,err}(zt^{-1/3})} - 1| |e_{31,N}| + \bigg|e^{z^3 p_{e,err}(zt^{-1/3})} \sum_{j=N +1}^\infty \frac{(z^3p_{e,N}(zt^{-1/3}))^j}{j!}\bigg|.
\end{align}
For $\delta \in (0, 2/15]$, it holds that
$$
\begin{cases}
|z^3 p_{e,N}(zt^{-1/3})| \leq C|z|^3 |zt^{-1/3}|^2 \leq Ct^{5\delta - \frac{2}{3}} \leq C,
	\\
|z^3 p_{e,err}(zt^{-1/3})| \leq C|z|^3 |zt^{-1/3}|^{N+1} \leq Ct^{(N+4)\delta - \frac{N+1}{3}} \leq C,
\end{cases} \ t \geq 1, \  |z| \leq t^\delta.$$
Using the general inequality
\begin{align}\label{ewminus1estimate}  
|e^w - 1| \leq |w| \max(1, e^{\re w}), \qquad w \in \C,
\end{align}
we obtain
\begin{align}\nonumber
 & \big|e^{z^3 p_e(zt^{-1/3})}  - e_{31,N}(t,z)\big|
 \leq C |z^3 p_{e,err}(zt^{-1/3})| + C \bigg|\sum_{j = N +1}^\infty \frac{(z^3p_{e,N}(zt^{-1/3}))^j}{j!}\bigg|
 	\\\nonumber
& \leq C |z|^3 |zt^{-1/3}|^{N+1} + C |z^3p_{e,N}(zt^{-1/3})|^{N +1} e^{|z^3p_{e,N}(zt^{-1/3})|} 
 	\\ 
& \leq C |z|^3 |zt^{-1/3}|^{N+1} + C |z|^{3(N +1)} |zt^{-1/3}|^{2(N+1)}, \qquad t \geq 1, \ |z| \leq t^\delta.
\end{align}
The estimate (\ref{eeNestimateb}) follows.
If $|z| \leq t^{1/3}$, then $|p_{e,N}(zt^{-1/3})| \leq C$, and so 
$$|e_{31,N}(t,z)|e^{-a|z|^3} \leq C\sum_{j=0}^{N} |z|^{3j} e^{-a|z|^3} \leq C, \qquad |z| \leq t^{1/3},$$
which proves (\ref{eNestimateb}).

Let us now show (\ref{eeNestimatea}).
In analogy with (\ref{ez3pe}), we have
\begin{align}\label{eqoqeJ1J2}
& \big|e^{-\beta + z^3(\zeta q_o(zt^{-1/3}) + q_e(zt^{-1/3}))} - e_{21,N}(x,t,z)\big| 	
 \leq  J_1 + J_2,
\end{align}
where
\begin{align*}
& J_1 := |e^{z^3(\zeta q_{o,err}(zt^{-1/3}) + q_{e,err}(zt^{-1/3}))} - 1| |e_{21,N}|,
	\\
& J_2 := \bigg|e^{z^3(\zeta q_{o,err}(zt^{-1/3}) + q_{e,err}(zt^{-1/3}))} \sum_{j=N +1}^\infty \frac{[-\beta + z^3(\zeta q_{o,N}(zt^{-1/3}) + q_{e,N}(zt^{-1/3}))]^j}{j!}\bigg|.
\end{align*}
Suppose $\delta \in (0, 1/12]$. Then
$$
\begin{cases}
|z^3 q_{e,N}(zt^{-1/3})| \leq C|z|^3 |zt^{-1/3}|^2 \leq Ct^{5\delta - \frac{2}{3}} \leq C,
	\\
|z^3 q_{e,err}(zt^{-1/3})| \leq C|z|^3 |zt^{-1/3}|^{N+1} \leq Ct^{(N+4)\delta - \frac{N+1}{3}} \leq C,
	\\
|z^3 q_{o,N}(zt^{-1/3})| \leq C|z|^3 |zt^{-1/3}| \leq Ct^{4\delta - \frac{1}{3}} \leq C,
	\\
|z^3 q_{o,err}(zt^{-1/3})| \leq C|z|^3 |zt^{-1/3}|^{N+1} \leq Ct^{(N+4)\delta - \frac{N+1}{3}} \leq C,
\end{cases} \ t \geq 1, \  |z| \leq t^\delta.$$
Hence, 
$$|e_{21,N}| \leq C(1 + |z|^{3N}), \qquad t \geq 1, \  |z| \leq t^\delta,$$
and we find
\begin{align}\label{J1estimate}
J_1 \leq C|z|^3 |zt^{-1/3}|^{N+1} (1 + |z|^{3N}) , \qquad t \geq 1, \  |z| \leq t^\delta.
\end{align}
Moreover,
\begin{align*}
J_2 & \leq C|-\beta + z^3(\zeta q_{o,N}(zt^{-1/3}) + q_{e,N}(zt^{-1/3}))|^{N+1} e^{|-\beta + z^3(\zeta q_{o,N}(zt^{-1/3}) + q_{e,N}(zt^{-1/3}))|}
	\\
& \leq C(|yz| + |z|^3 + |z|^3|zt^{-1/3}|)^{N+1} e^{|yz| + \frac{4}{3}|z|^3 + C|z|^3|zt^{-1/3}|},
\end{align*}
so
\begin{align}\label{J2estimate}
J_2 e^{-a t^{1/3}|z|^2} 
\leq C |z|^{N+1}e^{(Ct^{\delta} - \frac{a}{2} t^{1/3})|z|^2} 
\leq C|z|^{N+1}, \qquad t \geq 1, \  |z| \leq t^\delta.
\end{align}
The estimate (\ref{eeNestimatea}) follows from (\ref{eqoqeJ1J2}), (\ref{J1estimate}), and (\ref{J2estimate}).
If $|z| \leq t^{1/3}$, then 
$$|-\beta + z^3(\zeta q_{o,N}(zt^{-1/3}) + q_{e,N}(zt^{-1/3}))| \leq C(|z|+ |z|^3),$$
and so 
$$|e_{21,N}(x,t,z)|e^{-a t^{1/3}|z|^2} \leq C\sum_{j=0}^{N} (|z|+ |z|^3)^{j} e^{-a t^{1/3}|z|^2} \leq C, \qquad |z| \leq t^{1/3},$$
which proves (\ref{eNestimatea}). The proofs of (\ref{eeNestimatec}) and (\ref{eNestimatec}) are analogous.
\end{proof}

In the same way as in (\ref{eNdefIII}), we define approximations $e_{12,N}(x,t,z)$, $e_{13,N}(t,z)$, and $e_{23,N}(x,t,z)$ of the three exponentials
$$e^{\beta - z^3(\zeta q_o(zt^{-1/3}) + q_e(zt^{-1/3}))}, \quad e^{-z^3p_e(zt^{-1/3})}, \quad e^{\beta - z^3(-\zeta q_o(zt^{-1/3}) + q_e(zt^{-1/3}))}$$
by
\begin{subequations}\label{eNdef2}
\begin{align}\label{eNdef2a}
&e_{12,N}(x,t,z) = \sum_{j=0}^{N} \frac{[\beta - z^3(\zeta q_{o,N}(zt^{-1/3}) + q_{e,N}(zt^{-1/3}))]^j}{j!},
	\\
&e_{13,N}(t,z) = \sum_{j=0}^{N} \frac{[-z^3p_{e,N}(zt^{-1/3})]^j}{j!},
	\\
&e_{23,N}(x,t,z) = \sum_{j=0}^{N} \frac{[\beta - z^3(-\zeta q_{o,N}(zt^{-1/3}) + q_{e,N}(zt^{-1/3}))]^j}{j!}.
\end{align}
\end{subequations}
Obvious analogs of the estimates in Lemma \ref{elemmaIII} hold for the functions $e_{12,N}$, $e_{13,N}$, and $e_{23,N}$; for example, $e_{12,N}$ satisfies
\begin{subequations}\label{e12Nestimates}
\begin{align}
&  \big|e^{\beta - z^3(\zeta q_o(zt^{-1/3}) + q_e(zt^{-1/3}))} - e_{12,N}(x,t,z)\big| e^{-at^{1/3}|z|^2}
  \leq C|z|^{N+1}, && |z| \leq t^{\frac{1}{12}},
	\\
& |e_{12,N}(x,t,z)|e^{-at^{1/3}|z|^2} \leq C, && |z| \leq t^{1/3}.
\end{align}
\end{subequations}

Our next goal is to derive approximations of the off-diagonal entries of the matrices $v_{j_Y}^{(3)}$, $j = 1, \dots, 8$, for $k$ near $\omega$. 
The next lemma establishes the existence of expansions at $k = \omega$ to all orders of the quotients $\Delta_{ll}/\Delta_{jj}$, $1 \leq j < l \leq 3$.

\begin{lemma}\label{deltaat1lemmaIII}
There exist complex constants $\{c_{21,j}, d_{21,j}^\pm, c_{31,j}, d_{31,j}^\pm, c_{32,j}, d_{32,j}^\pm \}_{j=1}^N \subset \C$ such that
\begin{align}\nonumber
\frac{\Delta_{22}(k)}{\Delta_{11}(k)} =&\; \frac{\delta(\omega)^3}{\delta(\omega^2)^3} + \sum_{j=1}^N (c_{21,j}  \ln_\omega(k-\omega) + d_{21,j}^\pm) (k-\omega)^j 
	\\
&+ O((k-\omega)^{N+1}\ln(k-\omega)), \qquad k \to \omega, \ |k| \gtrless 1,
	\\\nonumber
\frac{\Delta_{33}(k)}{\Delta_{11}(k)} = &\; 1 + \sum_{j=1}^N (c_{31,j} \ln_\omega(k-\omega)
+ d_{31,j}^\pm) (k-\omega)^j 
	\\
& + O((k-\omega)^{N+1}\ln(k-\omega)), \qquad k \to \omega, \ |k| \gtrless 1,
	\\\nonumber
\frac{\Delta_{33}(k)}{\Delta_{22}(k)} = &\; \frac{\delta(\omega^2)^3}{\delta(\omega)^3} + \sum_{j=1}^N (c_{32,j} \ln_\omega(k-\omega)
+  d_{32,j}^\pm) (k-\omega)^j 
	\\
& + O((k-\omega)^{N+1}\ln(k-\omega)), \qquad k \to \omega, \ |k| \gtrless 1,
\end{align}
where $\ln_{\omega}(\cdot -\omega)$ has a cut along $\{e^{i \theta}: \theta \in [\frac{\pi}{2},\arg \omega]\}\cup(i,i\infty)$ and satisfies $\ln_{\omega}(1)=2\pi i$.
The leading coefficients in the expansion of $\Delta_{33}/\Delta_{11}$ are given by
\begin{align}\label{c311}
  c_{31,1} = & -\frac{r_1(\omega) r_2'(\omega)}{\pi i},
  	\\\nonumber
  d_{31,1}^\pm = & \mp \frac{r_1(\omega) r_2'(\omega)}{2}
  + 2\omega\bigg(\frac{\delta'(\omega^2)}{\delta(\omega^2)} - 2\omega \frac{\delta'(1)}{\delta(1)}\bigg)
  + r_1(\omega) r_2'(\omega) \Big(\frac{\ln(2 - \sqrt{3})}{2\pi i} + \frac{13}{12}\Big) 
  	\\ \label{d311pm}
&  +\frac{1}{\pi i}  \int_{\gamma_{(\omega, i)}}  \frac{\frac{\ln(1 + r_1(s)r_2(s))}{s-\omega} - r_1(\omega) r_2'(\omega)}{s-\omega} ds.
\end{align}
\end{lemma}
\begin{proof}
The function $\ln(1 + r_1(s)r_2(s))$ is smooth on $\gamma_{(\omega,i)}$ and vanishes at $s = \omega$. Hence, using Lemma \ref{endpointlemma},\footnote{Lemma \ref{endpointlemma} only treats the case where the integral is taken on $[a,b]\subset \mathbb{R}$, but it can easily be adapted to prove \eqref{asymp of delta as k to omega}.} there exist complex constants $\{\tilde{c}_j, \tilde{d}_j \}_1^N \subset \C$ such that
\begin{align}\label{asymp of delta as k to omega}
\delta(k) = &\; \delta(\omega) + \sum_{j=1}^N (\tilde{c}_j  \ln_\omega(k-\omega) +  \tilde{d}_j) (k-\omega)^j 
+ O((k-\omega)^{N+1}\ln(k-\omega))
\end{align}
as $k \to \omega$. 
The first half of the lemma follows. To find explicit expressions for $c_{31,1}$ and $d_{31,1}^\pm$, we observe that
\begin{align*}
\frac{\Delta_{33}(k)}{\Delta_{11}(k)}
= \frac{\delta(k)}{\delta(\frac{1}{\omega k})}h(k), \qquad h(k) := \frac{\delta(\frac{1}{\omega^2 k})^2 \delta(\omega k)}{\delta(\frac{1}{k})\delta(\omega^2 k)^2},
\end{align*}
where $h(k)$ is analytic at $k = \omega$ and satisfies
$$h(k) = 1 + 2\omega\bigg(\frac{\delta'(\omega^2)}{\delta(\omega^2)} - 2\omega \frac{\delta'(1)}{\delta(1)}\bigg)(k-\omega) + O((k-\omega)^2), \qquad k \to \omega.$$
Letting
$$g_0(s) = \ln(1+r_1(s)r_{2}(s)), \qquad g_1(s) = \frac{g_0(s) - g_0(\omega)}{s-\omega}, \qquad g_2(s) = \frac{g_1(s) - g_1(\omega)}{s-\omega},$$
we obtain
$$\frac{g_0(s)}{s-k} = \frac{g_0(\omega)}{s-k} + g_1(s) +  (k-\omega)\frac{g_1(s)}{s-k}, \quad \frac{g_1(s)}{s-k} = \frac{g_1(\omega)}{s-k} + g_2(s) + (k-\omega)\frac{g_2(s)}{s-k},$$
and so, since $r_2(\omega) = 0$ and hence also $g_0(\omega) = 0$,
\begin{align*}
\int_{\gamma_{(\omega, i)}} \frac{g_0(s)ds}{s-k}
 =&\; \int_{\gamma_{(\omega, i)}}g_1(s)ds + (k-\omega)\int_{\gamma_{(\omega, i)}} \frac{g_1(s)ds}{s-k}
	\\
= &\; \int_{\gamma_{(\omega, i)}}g_1(s)ds 
+ (k-\omega) \Big(g_1(\omega)(\ln_i(k-i) - \ln_\omega(k-\omega))
+ \int_{\gamma_{(\omega, i)}} g_2(s)ds\Big)
	\\
& + (k-\omega)^2\int_{\gamma_{(\omega, i)}} \frac{g_2(s)}{s-k}ds,
\end{align*}
where $\ln_{i}(\cdot -i)$ has a cut along $[i,i\infty)$ and satisfies $\ln_{i}(1)=2\pi i$. 
It follows that
\begin{align*}
\frac{\delta(k)}{\delta(\frac{1}{\omega k})} 
=&\; \frac{e^{\frac{k-\omega}{2\pi i} \big(g_1(\omega)(\ln_i(\omega-i) - \ln_\omega(k-\omega))
+ \int_{\gamma_{(\omega, i)}} g_2(s)ds\big)
 + O((k-\omega)^2\ln(k-\omega))}}{e^{\frac{\frac{1}{\omega k}-\omega}{2\pi i} \big(g_1(\omega)(\ln_i(\omega-i) - \ln_\omega(\frac{1}{\omega k}-\omega))
+ \int_{\gamma_{(\omega, i)}} g_2(s)ds\big)
 + O((k-\omega)^2\ln(k-\omega))}}
 	\\
= &\; 1 + \frac{k-\omega}{2\pi i} \bigg(g_1(\omega)\big[2\ln_i(\omega-i) - \ln_\omega(k-\omega) - \ln_\omega(\omega-k)\big]
+ 2\int_{\gamma_{(\omega, i)}} g_2(s)ds\bigg)
	\\
& + O((k-\omega)^2\ln(k-\omega)).
 \end{align*}
Noting that $g_1(\omega) = r_1(\omega) r_2'(\omega)$,
$\ln_i(\omega-i) = \frac{1}{2}\ln(2 - \sqrt{3}) + \frac{13\pi}{12}i$, and $\ln_{\omega}(\omega-k) =  \ln_{\omega}(k-\omega) \pm \pi i$ for $|k|\gtrless 1$ (valid for $|k-\omega|$ small), the expressions in (\ref{c311}) and (\ref{d311pm}) follow.
\end{proof}

Recall from Theorem \ref{directth} that 
$$r_2(\omega) = r_2(-\omega) = 0, \qquad
r_{1}(1) = r_{1}(-1) = 1, \qquad r_{2}(1) = r_{2}(-1) = -1.$$
Together with Lemma \ref{deltaat1lemmaIII}, this implies that to leading order
\begin{align}\label{vjY03leadingorder}
v_{j_Y,0}^{(3)}(k) = v_{j,0}^Y + O((k-\omega)\ln(k-\omega)), \qquad k\to \omega, \;\; j = 1, \dots, 8,
\end{align}
where $v_{j,0}^Y$ are the matrices defined in (\ref{vYj0def}) with $s$ given by
\begin{align}\label{sdef}
s := \frac{\delta(\omega)^3}{\delta(\omega^2)^3} r_1(\omega).
\end{align}

In what follows, we consider the higher-order terms in the expansions (\ref{vjY03leadingorder}) as well as suitable approximations of the exponentials $e^{\theta_{ij}}$ appearing in the jump matrices. We will focus on the $(31)$-entry of $v_{1_Y}^{(3)}$, the $(12)$-entry of $v_{5_Y}^{(3)}$, and the $(13)$-entry of $v_{7_Y}^{(3)}$; we denote these three entries by
\begin{align}\label{fgdef}
& f(k) := \frac{\Delta_{11}(k)}{\Delta_{33}(k)}r_{2,a}\Big(\frac{1}{\omega^{2}k}\Big),\quad  g(k) := \frac{\Delta_{22}(k)}{\Delta_{11}(k)}\hat{r}_{1,a}(k), 
	\\ \label{hdef}
& h(k) := -\frac{\Delta_{33}(k)}{\Delta_{11}(k)}\Big(r_{1,a}\Big(\frac{1}{\omega^2 k}\Big) + r_{2,a}(\omega^2 k)\Big).
\end{align}
The other entries can be treated similarly.
(The function $f(k)$ in (\ref{fgdef}) should not be confused with the function in (\ref{def of f}).)
The next lemma provides higher-order approximations of $f$, $g$, and $h$ near $k = \omega$. 

\begin{lemma}
There exist complex coefficients $\{a_j, b_j, c_j, d_j, c_j', d_j'\}_{j=1}^N$ such that the functions
\begin{subequations}\label{fNgNdef}
\begin{align}\label{fNgNdefa}
& f_N(\tilde{z}) := -1 + \sum_{j=1}^N (a_{j}\ln_0 \tilde{z} + b_{j}) \tilde{z}^j,
	\\ \label{fNgNdefb}
& g_N(\tilde{z}) := s + \sum_{j=1}^N (c_{j}\ln_0 \tilde{z} + d_{j}) \tilde{z}^j,
	\\ \label{fNgNdefc}
& h_N(\tilde{z}) := \omega^2(r_1'(1) - r_2'(1)) 2^{\frac{4}{3}}3^{-\frac{1}{3}}e^{\frac{\pi i}{6}} \tilde{z}
+ \sum_{j=2}^N (c_j' \ln_0 \tilde{z} + d_j') \tilde{z}^j,
\end{align}
\end{subequations}
where $\ln_0$ has a branch cut along $(0,+\infty)$ and $\im \ln_0 \tilde{z} \in [0,2\pi)$, satisfy
\begin{subequations}\label{fgestimates}
\begin{align}\label{fgestimatesa}
|f(k) - f_N(zt^{-1/3})| \leq C(1+|\ln|k-\omega||) |k - \omega|^{N+1} e^{\frac{t}{4}|\re \Phi_{31}(\zeta,k)|}, \quad k \in \Gamma^{(3)}_{1_Y},
	\\\label{fgestimatesb}
|g(k) - g_N(zt^{-1/3})| \leq C(1+|\ln|k-\omega||) |k - \omega|^{N+1} e^{\frac{t}{4}|\re \Phi_{21}(\zeta,k)|}, \quad k \in \Gamma^{(3)}_{5_Y},
	\\\label{fgestimatesc}
|h(k) - h_N(zt^{-1/3})| \leq C(1+|\ln|k-\omega||) |k - \omega|^{N+1} e^{\frac{t}{4}|\re \Phi_{31}(\zeta,k)|}, \quad k \in \Gamma^{(3)}_{7_Y},
\end{align}
\end{subequations}
uniformly for $t \geq 1$, $\zeta \in [1, 2]$, and $k$ in the given ranges.
The leading coefficients in \eqref{fNgNdefa} are given by
\begin{align}\label{a1b1explicit}
a_1 = 2^{\frac{4}{3}}3^{-\frac{1}{3}}e^{\frac{\pi i}{6}}c_{31,1}, \quad 
b_1 = 2^{\frac{4}{3}}3^{-\frac{1}{3}}e^{\frac{\pi i}{6}}\Big(c_{31,1} (\tfrac{4}{3}\ln{2} - \tfrac{1}{3}\ln{3} + \tfrac{\pi i}{6}) + d_{31,1}^+ - \omega^2 r_2'(1)\Big).
\end{align}
\end{lemma}
\begin{proof}
Lemma \ref{decompositionlemmaIII} and Lemma \ref{deltaat1lemmaIII} imply that there exist coefficients $\{\hat{a}_j, \hat{b}_j\}_{j=1}^N$ such that the function
$$\hat{f}_N(k) = -1 + \sum_{j=1}^N (\hat{a}_{j}\ln_\omega(k-\omega) + \hat{b}_{j}) (k-\omega)^j$$
satisfies
$$|f(k) - \hat{f}_N(k)| \leq C (1 + |\ln|k-\omega||)  |k-\omega|^{N+1} e^{\frac{t}{4}|\re \Phi_{31}(\zeta,k)|}$$
for $t \geq 1$, $k \in \Gamma^{(3)}_{1_Y}$, and $\zeta \in [1, 2]$.
Since $zt^{-1/3} = (\frac{3}{2})^{1/3} \frac{i  (k-\omega ) (k+\omega )}{4 k \omega}$, the existence of the function $f_N(\tilde{z})$ follows.
The proof of the existence of $g_N(\tilde{z})$ is similar. On the other hand,
$$r_{2,a}\Big(\frac{1}{\omega^{2}k}\Big) = -1 - \omega^2 r_2'(1) (k-\omega) + O((k-\omega)^2), \qquad k \to \omega,$$
and so
$$f(k) = -1 + \big(c_{31,1} \ln_\omega(k-\omega) + d_{31,1}^+ - \omega^2 r_2'(1)\big)(k-\omega) + O((k-\omega)^2), \qquad k \to \omega, \;\; k \in \Gamma^{(3)}_{1_Y}.$$
Letting $\tilde{z} = zt^{-1/3}$, we have
$$k-\omega = 2^{\frac{4}{3}}3^{-\frac{1}{3}}e^{\frac{\pi i}{6}} \tilde{z} + O(\tilde{z}^2), \qquad k \to \omega,$$
and hence (\ref{a1b1explicit}) follows.
Using that $r_1(1) + r_2(1) = 0$, we find, for $k \in \Gamma^{(3)}_{7_Y}$ approaching $\omega$,
\begin{align*}
h(k) = & - \Big(1 + (c_{31,1} \ln_\omega(k-\omega) + d_{31,1}^-)(k-\omega) + O((k-\omega)^2\ln_\omega(k-\omega))\Big)
	\\
&\times \Big(\omega^2(-r_1'(1) + r_2'(1))(k-\omega) + O((k-\omega)^2e^{\frac{t}{4}|\re \Phi_{31}(\zeta,k)|})\Big)
\end{align*}
and hence similar arguments also give the existence of $h_N(\tilde{z})$.
\end{proof}

Define $p_1(t,z)$, $q_1(t,z)$, and $\hat{p}_2(t,z)$ by
\begin{align}\nonumber
& - 1 - p_1(t,z) = f_N(zt^{-1/3}) e_{31,N}(t,z), \qquad
s -isyz + q_1(t,z) = g_N(zt^{-1/3}) e_{12,N}(x,t,z),
	\\ \label{pNdefIII}
& \hat{p}_2(t,z) = h_N(zt^{-1/3}) e_{13,N}(t,z),
\end{align}
where $f_N$, $g_N$, $h_N$, $e_{31,N}$, $e_{12,N}$, and  $e_{13,N}$ are the functions in (\ref{fNgNdef}), (\ref{eNdefIII}), and (\ref{eNdef2}). 
Then 
\begin{align}\label{p1explicit}
p_1(t,z) = -1 - \bigg(-1 +\sum_{j=1}^N (a_{j}\ln_0(zt^{-1/3}) + b_{j}) z^jt^{-j/3}\bigg)\bigg(1  + \sum_{j=1}^{N} \frac{[z^3p_{e,N}(zt^{-1/3})]^j}{j!}\bigg)
\end{align}
\begin{align*}
q_1(t,z) = & -s + isyz + \bigg(s +\sum_{j=1}^N (c_{j}\ln_0(zt^{-1/3}) + d_{j}) z^jt^{-j/3}\bigg)
	\\
& \times \bigg(1  + \sum_{j=1}^{N} \frac{[\beta - z^3(\zeta q_{o,N}(zt^{-1/3}) + q_{e,N}(zt^{-1/3}))]^j}{j!}\bigg),
\end{align*}
and
\begin{align*}
\hat{p}_2(t,z) = & \bigg(\omega^2(r_1'(1) - r_2'(1)) 2^{\frac{4}{3}}3^{-\frac{1}{3}}e^{\frac{\pi i}{6}} z t^{-1/3}
+ \sum_{j=2}^N (c_j' \ln_0(z t^{-1/3}) + d_j') z^jt^{-j/3}\bigg)
	\\
& \times \bigg(1  + \sum_{j=1}^{N} \frac{[-z^3p_{e,N}(zt^{-1/3})]^j}{j!}\bigg).
\end{align*}
It follows that $p_1$ has the form (\ref{psumIII}) for some sufficiently large $n$. 
Moreover, expressing $\zeta$ in terms of $y$ and noting that $isyz + s\beta = -is 4z^3/3$, it follows that $q_1$ has the form (\ref{qsumIII}) for some sufficiently large $n$. Similarly, we see that $\hat{p}_2$ has the structure required by (\ref{phatsumIII}).
This shows that the $(31)$-entry of $v_{1_Y}^{(3)}$, the $(12)$-entry of $v_{5_Y}^{(3)}$, and the $(13)$-entry of $v_{7_Y}^{(3)}$ have the appropriate approximations required for an application of Lemma \ref{YlemmaIII}. Analogous arguments apply to the other entries of the jump matrices $v_{j_Y}^{(3)}$, $j = 1, \dots, 8$ and the expression (\ref{v9Ydef}) for $v_9^Y$ then follows by consistency; detailed estimates will be provided in Lemma \ref{momegalemma} below.

Let $\mathcal{Y}^\epsilon = (\Gamma^{(3)} \cap D(\epsilon))\setminus (\Gamma^{(3)}_2 \cup \Gamma^{(3)}_5)$. Deforming the contour $\Gamma^{(3)}$ slightly if necessary, we may assume that the map $k \mapsto z$ takes $\mathcal{Y}^\epsilon$ onto $Y \cap \{|z| < (\frac{3t}{2})^{1/3}\frac{\epsilon}{4}\}$, where $Y$ is the contour defined in (\ref{YdefIII}).
We write $\mathcal{Y}^\epsilon = \cup_{j=1}^9 \mathcal{Y}_j^\epsilon$, where $\mathcal{Y}_j^\epsilon$ denotes the part of $\mathcal{Y}^\epsilon$ that maps into $Y_j$.

Let $m^Y(y, t, z)$ be the solution of RH problem \ref{RHmYIII} of Lemma \ref{YlemmaIII} with jump contour $Y = Y(z_1, z_2)$, where $y,z$ are defined in terms of $x,t,k$ by (\ref{yzdef}), $s$ is defined by (\ref{sdef}), $z_1$ and $iz_2$ are the images of the two saddle points $k_1$ and $\omega k_3$ under the map $k \mapsto z$, and the functions $p_k$, $q_k$, and $\hat{p}_k$ are given by the expansions of the entries of $v_{j_Y}^{(3)}$, $j = 1, \dots, 8$, as described in detail above for $p_1$, $q_1$, and $\hat{p}_2$.
The local parametrix $m^\omega(x,t,k)$ at $\omega$ is defined for $k \in D(\epsilon)$ by
\begin{align}\label{m0defIII}
m^\omega(x, t, k) = Y(x,t)m^Y(y, t, z)Y(x,t)^{-1}, \qquad k \in D(\epsilon),
\end{align}
where
\begin{align}\label{Ydef}
Y(x,t) := e^{x\mathcal{L}(\omega) + t\mathcal{Z}(\omega)} = \begin{pmatrix}
 e^{-\frac{i (t+2 x)}{4 \sqrt{3}}} & 0 & 0 \\
 0 & e^{\frac{i (t+2 x)}{2 \sqrt{3}}} & 0 \\
 0 & 0 & e^{-\frac{i (t+2 x)}{4 \sqrt{3}}}
\end{pmatrix}.
\end{align}
By Lemma \ref{YlemmaIII}, we can choose $T \geq 1$ such that $m^\omega$ is well-defined whenever $(x,t) \in \III_\geq^T$, where $\III_\geq^T$ denotes the part of III where $t\geq T$:
$$\III_\geq^T := \III_\geq \cap \{t \geq T\}.$$ 

In order to show that $(1,1,1)m^\omega$ approximates $n^{(3)}$ well in $D(\epsilon)$, we need the following lemma. 

\begin{lemma} 
If $a > 0$ and $\epsilon > 0$ are sufficiently small, then
\begin{subequations}
\begin{align}\label{rePhiIIIa}
& |\re \Phi_{21}(\zeta, k)| \geq c |k-k_1|^2, && \zeta \in [1, 1+a], \ k \in \mathcal{Y}_5^\epsilon \cup  \mathcal{Y}_8^\epsilon,
	\\\label{rePhiIIIb}
& |\re \Phi_{21}(\zeta, k)| \geq c |k-\omega^2 k_2|^2, && \zeta \in [1, 1+a], \ k \in \mathcal{Y}_6^\epsilon \cup  \mathcal{Y}_7^\epsilon,
	\\\label{rePhiIIIc}
& |\re \Phi_{31}(\zeta, k)| \geq c |k-\omega|^3, && \zeta \in [1, 1+a], \ k \in \cup_{j=1}^4 \mathcal{Y}_j^\epsilon,
	\\\label{rePhiIIIe}
& |\re \Phi_{32}(\zeta, k)| \geq c  |k-k_1|^2, && \zeta \in [1, 1+a], \ k \in \mathcal{Y}_5^\epsilon \cup  \mathcal{Y}_8^\epsilon,
	\\\label{rePhiIIIf}
& |\re \Phi_{32}(\zeta, k)| \geq c  |k- \omega^2 k_2|^2, && \zeta \in [1, 1+a], \ k \in \mathcal{Y}_6^\epsilon \cup  \mathcal{Y}_7^\epsilon,
\end{align}
\end{subequations}
where $c > 0$ is independent of $\zeta$ and $k$.
\end{lemma}
\begin{proof}
Since $\partial_k\Phi_{21}(\zeta, k_1) = 0$, Taylor's formula implies that
$$\Phi_{21}(\zeta, k) = \Phi_{21}(\zeta, k_1) + \frac{\partial_k^2\Phi_{21}(\zeta, k_1)}{2!}(k-k_1)^2 +  \frac{\partial_k^3\Phi_{21}(\zeta, k_\star)}{3!}(k-k_1)^3$$
for all $\zeta \in [1, 1+a]$ and $k \in \mathcal{Y}_5^\epsilon$, where $k_\star$ (which depends on $\zeta$ and $k$) stays close to $\omega$ if $a$ and $\epsilon$ are small. Since $\partial_k^3\Phi_{21}(\zeta, k)$ depends continuously on $(\zeta, k)$ for $\zeta$ near $1$ and $k$ near $\omega$, we conclude that $\partial_k^3\Phi_{21}(\zeta, k_\star)$ is uniformly bounded for all $\zeta \in [1, 1+a]$ and $k \in \mathcal{Y}_5^\epsilon$.
Using also that $\re \Phi_{21}(\zeta, k_1) = 0$, we conclude that if $a > 0$ and $\epsilon > 0$ are sufficiently small, then
$$\re \Phi_{21}(\zeta, k) = \re\bigg( \frac{\partial_k^2\Phi_{21}(\zeta, k_1)}{4}(k-k_1)^2\bigg) + O((k-k_1)^3),$$
uniformly for $\zeta \in [1, 1+a]$ and $k \in \mathcal{Y}_5^\epsilon$.
Straightforward calculations show that
$$|\partial_k^2\Phi_{21}(\zeta, k_1)|
= \frac{\sqrt{(\zeta^2+8) (-\zeta^2+ \zeta \sqrt{\zeta^2+8} +4)}}{2   \sqrt{2}} \geq \frac{3^{3/2}}{2}$$
for $\zeta \in [1,2]$, and that
$$\arg \partial_k^2\Phi_{21}(\zeta, k_1) 
= \arctan\bigg(\frac{\sqrt{2} \zeta }{\sqrt{\zeta (\sqrt{\zeta^2+8}-\zeta)+4}}\bigg) -\pi \in \bigg[-\frac{5\pi}{6}, -\frac{3\pi}{4}\bigg]$$
for $\zeta \in [1, 3/2]$. Since $\arg(k -k_1) \approx \pi/3$ for $k \in \mathcal{Y}_5^\epsilon$ if $a > 0$ and $\epsilon > 0$ are small, the estimate (\ref{rePhiIIIa}) follows for $k \in \mathcal{Y}_5^\epsilon$. The other assertions of the lemma follow by similar arguments.
\end{proof}

\begin{lemma}\label{momegalemma}
Let $N \geq 1$ be an integer. For each $(x,t) \in \III_\geq^T$, the function $m^\omega(x,t,k)$ defined in (\ref{m0defIII}) is an analytic function of $k \in D(\epsilon) \setminus \mathcal{Y}^\epsilon$ such that
\begin{align}\label{m0boundIII}
\begin{cases}
|m^\omega(x,t,k) - I| \leq C, 
	\\
|\partial_x m^\omega(x,t,k)| \leq C, 
\end{cases} (x,t) \in \III_\geq^T, \ k \in D(\epsilon) \setminus \mathcal{Y}^\epsilon.
\end{align}
On $\mathcal{Y}^\epsilon$, $m^\omega$ obeys the jump condition $m_+^\omega =  m_-^\omega v^\omega$, where the jump matrix $v^\omega$ satisfies, for each $1 \leq p \leq \infty$,
\begin{subequations}\label{v3vomegaestimates}
\begin{align}\label{v3vomegaestimatesa}
& \|v^{(3)} - v^\omega\|_{L^p(\mathcal{Y}^\epsilon)} \leq Ct^{-\frac{N+1}{6}},
	\\ 
& \|\partial_x(v^{(3)} - v^\omega)\|_{L^p(\mathcal{Y}^\epsilon)} \leq Ct^{-\frac{N+1}{6}},
\end{align}	
\end{subequations}
uniformly for $(x,t) \in \III_\geq^T$. Furthermore, as $t \to \infty$,
\begin{align}\nonumber
  m^\omega(x,t,k) = &\; I + \frac{m_1^P(y)}{t^{1/3}f_\omega(k)} +  \frac{Ym^P(y,0)m_1^W(\tilde{y})m^P(y,0)^{-1}Y^{-1}}{\frac{2^{1/3}}{3^{1/12}} t^{1/2} f_\omega(k)}
+ \frac{m_2^P(y)}{t^{2/3}f_\omega(k)^2}
	\\\label{momegaexpansion}
&+ \frac{\hat{m}_{1, \ln}^Y(y,\tilde{y})\ln t + \hat{m}_1^Y(y,\tilde{y})}{t^{2/3}f_\omega(k)}
+ \frac{Y\hat{m}_{2, \ln}^Y(y,\tilde{y})Y^{-1} \ln t}{t^{5/6}f_\omega(k)} + O\big(t^{-\frac{5}{6}}\big),
\end{align}
uniformly for $(x,t) \in \III_\geq^T$ and $k \in \partial D(\epsilon)$, where $m_1^P$, $m^P$, $m_1^W$, $m_2^P$, $\hat{m}_{1,\ln}^Y$, $\hat{m}_1^Y$, $\hat{m}_{2, \ln}^Y$ are as in Lemma \ref{YlemmaIII}, $\tilde{y}$ is defined by (\ref{tildeywdef}), and
\begin{align}\label{fomegadef}
f_\omega(k) := \Big(\frac{3}{2}\Big)^{1/3} \frac{i  (k-\omega ) (k+\omega )}{4 k \omega}.
\end{align}
In particular, 
\begin{align}\label{m0LinftyestimateIII}
& \|m^\omega(x,t,\cdot) - I\|_{L^\infty(\partial D(\epsilon))} = O\big(t^{-1/3}\big), 
	\\\label{m0xLinftyestimateIII}
& \|\partial_x m^\omega(x,t,\cdot)\|_{L^\infty(\partial D(\epsilon))} = O\big(t^{-1/2}\big).
\end{align}
The expansion (\ref{momegaexpansion}) can be differentiated termwise with respect to $x$ without increasing the error term.
\end{lemma}
\begin{proof}
The analyticity of $m^\omega$ and the bounds (\ref{m0boundIII}) are a consequence of Lemma \ref{YlemmaIII} and (\ref{Ydef}).
Using (\ref{vjY3}), we see that, for $k \in \mathcal{Y}_j^\epsilon$ with $j = 1, \dots, 8$,
\begin{align}\nonumber
 v^{(3)} - v^\omega = &\; e^{x\widehat{\mathcal{L}(k)} + t\widehat{\mathcal{Z}(k)}}v_{j_Y,0}^{(3)} - e^{x\widehat{\mathcal{L}(\omega)} + t\widehat{\mathcal{Z}(\omega)}} v^Y 
	\\ \label{v3minusvomega}
= &\; e^{x\widehat{\mathcal{L}(\omega)} + t\widehat{\mathcal{Z}(\omega)}}\Big(e^{x\widehat{\mathcal{L}(k)} + t\widehat{\mathcal{Z}(k)} - x\widehat{\mathcal{L}(\omega)} - t\widehat{\mathcal{Z}(\omega)}}v_{j_Y,0}^{(3)} -  v^Y\Big).
\end{align}
Let us first prove (\ref{v3vomegaestimates}) on $\mathcal{Y}_1^\epsilon$. For $k \in \mathcal{Y}_1^\epsilon$, we find 
\begin{align*}
 v^{(3)} - v^\omega = &\; e^{x\widehat{\mathcal{L}(\omega)} + t\widehat{\mathcal{Z}(\omega)}}
 \begin{pmatrix}
0 & 0 & 0 \\
0 & 0 & 0 \\
f(k)e^{t(\Phi_{31}(\zeta,k) - \Phi_{31}(\zeta, \omega))} - (-1-p_1(t,z))e^{-2\beta}& 0 & 0
\end{pmatrix},
\end{align*}
where $f$ is the $(31)$-entry of $v_{1_Y,0}^{(3)}$ as defined in (\ref{fgdef}) and $p_1(t,z)$ is given by (\ref{psumIII}). The matrices $\mathcal{L}(\omega)$ and $\mathcal{Z}(\omega)$ are pure imaginary. Moreover, by (\ref{pNdefIII}), 
$$- 1 - p_1(t,z) = f_N(zt^{-1/3}) e_{31,N}(t,z).$$
Hence, utilizing (\ref{fgestimatesa}) and (\ref{Phi31expansion}), we obtain
\begin{align*}
 |v^{(3)} - v^\omega| = &\; \big| f(k)e^{t\Phi_{31}(\zeta,k) - t\Phi_{31}(\zeta, \omega)} - f_N(zt^{-1/3}) e_{31,N}(t,z)e^{-2\beta}\big|
 	\\
\leq &\; \big| f(k) - f_N(zt^{-1/3}) \big| \big| e^{t\Phi_{31}(\zeta,k) - t\Phi_{31}(\zeta, \omega)} \big|
	\\
& +\big|f_N(zt^{-1/3}) \big| \big|e^{t\Phi_{31}(\zeta,k) - t\Phi_{31}(\zeta, \omega)} 
- e_{31,N}(t,z)e^{-2\beta}\big|
 	\\
\leq &\; C(1+|\ln|k-\omega||) |k - \omega|^{N+1} e^{-\frac{3t}{4}|\re \Phi_{31}(\zeta,k)|}
	\\
& +\big|f_N(zt^{-1/3}) \big| \big|e^{-2\beta + z^3p_e(zt^{-1/3})} 
- e_{31,N}(t,z)e^{-2\beta}\big|
, \qquad (x,t) \in \III_\geq^T, \;\; k \in \mathcal{Y}_1^\epsilon.
\end{align*}
For $k \in \mathcal{Y}_1^\epsilon$, we have $|f_N(zt^{-1/3})| \leq C$ and 
$|e^{-2\beta}| = |e^{2 i (yz + \frac{4z^3}{3})}| \leq e^{-c |z|^3}$ uniformly for $(x,t) \in \III_\geq^T$.
Recalling also the estimate (\ref{rePhiIIIc}), we conclude that, for $k \in \mathcal{Y}_1^\epsilon$, 
\begin{align*}
 |v^{(3)} - v^\omega| \leq &\; C |k - \omega|^{N} e^{-c t |k-\omega|^3}
 + C\big|e^{z^3p_e(zt^{-1/3})} 
- e_{31,N}(t,z)\big| e^{-c |z|^3}.
\end{align*}
Applying the inequalities (\ref{eeNestimateb}) and (\ref{eNestimateb}) with $a = c/2$, this gives
\begin{align*}
 |v^{(3)} - v^\omega| \leq &\; \begin{cases}
 C |k - \omega|^{N} e^{-c t |k-\omega|^3} + C|zt^{-1/3}|^{N+1} e^{-\frac{c}{2} |z|^3}, & |z| \leq t^{\frac{2}{15}}, 
	\\
C |k - \omega|^{N} e^{-c t |k-\omega|^3}
 + C e^{-\frac{c}{2} |z|^3}, & |z| \geq t^{\frac{2}{15}},
\end{cases} \quad k \in \mathcal{Y}_1^\epsilon.
\end{align*}
Recalling that $z \sim t^{1/3} (k-\omega)$, we deduce that 
\begin{align*}
 |v^{(3)} - v^\omega| \leq &\; \begin{cases}
 C |zt^{-1/3}|^{N} e^{-c |z|^3}, & |z| \leq t^{\frac{2}{15}}, 
	\\
C e^{-c |z|^3}, & |z| \geq t^{\frac{2}{15}},
\end{cases} \quad k \in \mathcal{Y}_1^\epsilon,
\end{align*}
uniformly for $(x,t) \in \III_\geq^T$. 
It follows that
$$\|v^{(3)} - v^\omega\|_{L^\infty(\mathcal{Y}_1^\epsilon)} \leq C t^{-N/3}
$$
and
\begin{align*}
\|v^{(3)} - v^\omega\|_{L^1(\mathcal{Y}_1^\epsilon)} 
& \leq C t^{-N/3} \int_0^{t^{\frac{2}{15}}} u^N e^{-c u^3} \frac{du}{t^{1/3}} 
+ C \int_{t^{\frac{2}{15}}}^\infty e^{-c u^3} \frac{du}{t^{1/3}} 
\leq Ct^{-\frac{N+1}{3}},
\end{align*}
which proves (\ref{v3vomegaestimatesa}) on $\mathcal{Y}_1^\epsilon$; similar arguments apply on the contours $\mathcal{Y}_j^\epsilon$ with $j = 2,3,4$ and for the $x$-derivatives.

We next prove (\ref{v3vomegaestimates}) on $\mathcal{Y}_5^\epsilon$. For $k \in \mathcal{Y}_5^\epsilon$, (\ref{v3minusvomega}) yields
\begin{align*}
|v^{(3)} - v^\omega| = &\; \left|
 \begin{pmatrix}
0 & \frac{\Delta_{22}}{\Delta_{11}}\hat{r}_{1,a}(k)e^{-t(\Phi_{21}(\zeta,k) - \Phi_{21}(\zeta, \omega))}- (s-isyz + q_1)e^{3\alpha } & 0 \\
0 & 0 & 0 \\
0 & -\frac{\Delta_{22}}{\Delta_{33}}r_{1,a}(\frac{1}{\omega k})e^{t(\Phi_{32}(\zeta,k) - \Phi_{32}(\zeta, \omega))} - (-s-isyz + q_2)e^{3\alpha } & 0
\end{pmatrix}\right|.
\end{align*}
We will estimate the $(12)$-element of the above matrix; the $(32)$-element can be estimated in a similar way. In view of (\ref{Phi21expansion}) and \eqref{pNdefIII}, the $(12)$-element, which we denote by $X_{12}$, can be written as
\begin{align}\nonumber
X_{12} :=&\; g(k)e^{-t(\Phi_{21}(\zeta,k) - \Phi_{21}(\zeta, \omega))} - g_N(zt^{-1/3}) e_{12,N}(x,t,z) e^{3\alpha }
	\\\nonumber
=&\; (g(k) - g_N(zt^{-1/3}))e^{-t(\Phi_{21}(\zeta,k) - \Phi_{21}(\zeta, \omega))}
	\\\label{X12def}
& + g_N(zt^{-1/3})e^{3\alpha}\Big(e^{\beta  - z^3(\zeta q_o(zt^{-1/3}) + q_e(zt^{-1/3}))} - e_{12,N}(x,t,z)\Big),
\end{align}
where $g$ is the function in (\ref{fgdef}) and $g_{N}$ is given in \eqref{fNgNdefb}. 
For $k \in \mathcal{Y}_5^\epsilon$, we have $|g_N(zt^{-1/3})| \leq C$ and 
$|e^{3\alpha}| \leq e^{-c t^{1/3} |z|^2}$ uniformly for $(x,t) \in \III_\geq^T$.
Employing the estimates (\ref{fgestimatesb}) and (\ref{e12Nestimates}), we infer that
\begin{align}\label{X12est}
|X_{12}| \leq C |k - \omega|^{N} e^{-\frac{3t}{4}|\re \Phi_{21}(\zeta,k)|}
+ Ce^{-\frac{c}{2} t^{1/3} |z|^2}\times 
\begin{cases}
|z|^{N+1}, &  |z| \leq t^{\frac{1}{12}},
	\\
1, & |z| \geq t^{\frac{1}{12}},
\end{cases} \quad k \in \mathcal{Y}_5^\epsilon.
\end{align}
In view of (\ref{rePhiIIIa}) and the estimate
$$|k - \omega|^{N} \leq C(|k - k_1|^{N} + |k_1 - \omega|^{N}) \leq C(|k - k_1|^{N} + t^{-2N/3}),$$
this implies that the following estimates hold uniformly for $(x,t) \in \III_\geq^T$ and $k \in \mathcal{Y}_5^\epsilon$:
\begin{align*}
|X_{12}| \leq C (|k - k_1|^N + t^{-2N/3}) e^{-c t |k-k_1|^2}
+ Ce^{-c t^{1/3} |z|^2}\times 
\begin{cases}
|z|^{N+1}, &  |z| \leq t^{\frac{1}{12}},
	\\
1, & |z| \geq t^{\frac{1}{12}},
\end{cases}
\end{align*}
and hence
\begin{align*}
\|X_{12}\|_{L^\infty(\mathcal{Y}_5^\epsilon)} 
\leq C t^{-N/2} + C t^{-\frac{N+1}{6}} \leq C t^{-\frac{N+1}{6}}
\end{align*}
and
\begin{align*}
\|X_{12}\|_{L^1(\mathcal{Y}_5^\epsilon)} 
\leq &\; 
C \int_0^\infty  (u^{N}+ t^{-2N/3}) e^{-c t u^2} du
+ C\int_0^{t^{\frac{1}{12}}} e^{-c t^{1/3} v^2} v^{N+1} \frac{dv}{t^{1/3}}
	\\
& + C\int_{t^{\frac{1}{12}}}^\infty e^{-c t^{1/3} v^2} \frac{dv}{t^{1/3}}
	\\
 \leq 
&\; Ct^{-N/2}
+ Ct^{-\frac{N+1}{6}} \int_0^\infty e^{-c w^2} w^{N+1} \frac{dw}{t^{1/2}}
+ Ce^{-c t^{1/2}}
\leq Ct^{-\frac{N+1}{6}},
\end{align*}
which proves (\ref{v3vomegaestimatesa}) on $\mathcal{Y}_5^\epsilon$; similar arguments apply on the contours $\mathcal{Y}_j^\epsilon$ with $j = 6,7,8$ and for the $x$-derivatives.

We next prove (\ref{v3vomegaestimates}) on $\mathcal{Y}_9^\epsilon$. On $\mathcal{Y}_9^\epsilon \cap \{0 < \arg k < 2\pi/3\}$, we have
$$v^{(3)} = (\check{v}_{8_Y}^{(3)})_-^{-1} (\check{v}_{5_Y}^{(3)})_+,$$
where $\check{v}_{5_Y}^{(3)}$ and $\check{v}_{8_Y}^{(3)}$ are given by the same expressions as $v_{5_Y}^{(3)}$ and $v_{8_Y}^{(3)}$ except that $r_{1,a}$, $\hat{r}_{1,a}$, $r_{2,a}$, $\hat{r}_{2,a}$ are replaced by $r_1$, $\hat{r}_1$, $r_2$, $\hat{r}_2$, respectively, i.e., 
\begin{align*}
\check{v}_{5_Y}^{(3)}
= \begin{pmatrix}
1 & \frac{\Delta_{22}}{\Delta_{11}}\hat{r}_1(k)e^{-\theta_{21}} & 0 \\
0 & 1 & 0 \\
0 & -\frac{\Delta_{22}}{\Delta_{33}}r_1(\frac{1}{\omega k})e^{\theta_{32}} & 1
\end{pmatrix}, \qquad
\check{v}_{8_Y}^{(3)}
= \begin{pmatrix}
1 & 0 & 0 \\
-\frac{\Delta_{11}}{\Delta_{22}}\hat{r}_2(k) e^{\theta_{21}} & 1 & \frac{\Delta_{33}}{\Delta_{22}}r_2(\frac{1}{\omega k})e^{-\theta_{32}} \\
0 & 0 & 1
\end{pmatrix}.
\end{align*}
Thus, for $k \in \mathcal{Y}_9^\epsilon$ with $\arg k < 2\pi/3$, we have
\begin{align}\label{v3minusvomegaY9}
 v^{(3)} - v^\omega = &\; (\check{v}_{8_Y}^{(3)})_-^{-1} (\check{v}_{5_Y}^{(3)})_+ - (v_8^\omega)_-^{-1}(v_5^\omega)_+,
 \end{align}
where we have used (\ref{v9Ydef}).
We claim that
\begin{align}\label{vecheckminusvomega85}
\big|(\check{v}_{8_Y}^{(3)})_- - (v_8^\omega)_-\big| \leq C |k - \omega|^N, \qquad
\big|(\check{v}_{5_Y}^{(3)})_+ - (v_5^\omega)_+\big| \leq C |k - \omega|^N
\end{align}
uniformly for $(x,t) \in \III_\geq^T$ and for $k \in \mathcal{Y}_9^\epsilon$ with $\arg k < 2\pi/3$.
Indeed, consider for example the $(12)$-entry of $(\check{v}_{5_Y}^{(3)})_+ - (v_5^\omega)_+$ whose absolute value is bounded above by $|\check{X}_{12}|$, where $\check{X}_{12}$ is given by the same expression as $X_{12}$ above except that $\hat{r}_{1,a}$ is replaced by $\hat{r}_1$, i.e.,
$$\check{X}_{12} := \frac{\Delta_{22+}}{\Delta_{11+}}\hat{r}_1(k)
e^{-t(\Phi_{21}(\zeta,k) - \Phi_{21}(\zeta, \omega))} - g_N(zt^{-1/3}) e_{12,N}(x,t,z) e^{3\alpha }.$$
Estimating as in (\ref{X12est}), we obtain the claimed estimate.
From (\ref{v3minusvomegaY9}) and (\ref{vecheckminusvomega85}), we get (\ref{v3vomegaestimates}) on $\mathcal{Y}_9^\epsilon \cap \{0 < \arg k < 2\pi/3\}$; similar arguments apply to the part of $\mathcal{Y}_9^\epsilon$ with $\arg k > 2\pi/3$.

We next apply Lemma \ref{YlemmaIII} to determine the large $t$ asymptotics of the solution $m^\omega$. 
The variable $z = (\frac{3t}{2})^{1/3} \frac{i  (k-\omega ) (k+\omega )}{4 k \omega}$ tends to infinity as $t\to \infty$ if $k \in \partial D(\epsilon)$. In fact, for $k \in \partial D(\epsilon)$, we have $|z| = (\frac{3t}{2})^{1/3}\frac{\epsilon}{4}$.
In terms of the function $f_\omega$ defined in (\ref{fomegadef}), we can write 
$$z = t^{1/3}f_\omega(k).$$
Thus equation (\ref{mYasymptoticsIII}) yields, as $t \to \infty$,
\begin{align}\nonumber
m^\omega(x,t,k) = &\; I + \frac{Ym_1^P(y)Y^{-1}}{t^{1/3}f_\omega(k)} +  \frac{Ym^P(y,0)m_1^W(\tilde{y})m^P(y,0)^{-1}Y^{-1}}{\frac{2^{1/3}}{3^{1/12}} t^{1/2} f_\omega(k)}
+ \frac{Ym_2^P(y)Y^{-1}}{t^{2/3}f_\omega(k)^2}
	\\ \label{m0expansionmYIII}
& + Y\frac{\hat{m}_{1, \ln}^Y(y,\tilde{y})\ln t + \hat{m}_1^Y(y,\tilde{y})}{t^{2/3}f_\omega(k)} Y^{-1}
+ \frac{Y\hat{m}_{2, \ln}^Y(y,\tilde{y})Y^{-1} \ln t}{t^{5/6}f_\omega(k)} + O\big(t^{-\frac{5}{6}}\big),
\end{align}
uniformly for $(x,t) \in \III_\geq^T$ and $k \in \partial D(\epsilon)$. From (\ref{mPasymptotics}) and Lemma \ref{hatmYlemma}, we infer that the coefficients $m_1^P$, $m_2^P$, $\hat{m}_{1, \ln}^Y$, and $\hat{m}_1^Y$ all have the form 
$\begin{pmatrix} 
* & 0 & * \\
0 & 0 &0 \\
* & 0 & * 
\end{pmatrix}$, i.e., their (12), (21), (22), (23), and (32) entries vanish, so that they all commute with $Y$. Hence (\ref{momegaexpansion}) follows from (\ref{m0expansionmYIII}).
The estimate (\ref{m0LinftyestimateIII}) is a consequence of (\ref{momegaexpansion}).  
By part ($\ref{YlemmaIIIitemd}$) of Lemma \ref{YlemmaIII}, the expansion (\ref{momegaexpansion}) can be differentiated termwise with respect to $x$, which also leads to (\ref{m0xLinftyestimateIII}).
\end{proof}

\begin{figure}
\begin{center}
\begin{tikzpicture}[master, scale=0.9]
\node at (0,0) {};
\draw[black,line width=0.65 mm,->-=0.25,->-=0.53,->-=0.80,->-=0.98] (0,0)--(30:7.5);
\draw[black,line width=0.65 mm,->-=0.25,->-=0.53,->-=0.80,->-=0.98] (0,0)--(90:7.5);
\draw[black,line width=0.65 mm,->-=0.25,->-=0.53,->-=0.80,->-=0.98] (0,0)--(150:7.5);

\draw[black,line width=0.65 mm] ([shift=(30:3*1.5cm)]0,0) arc (30:150:3*1.5cm);
\draw[black,arrows={-Triangle[length=0.27cm,width=0.2cm]}]
($(57:3*1.5)$) --  ++(-32:0.001);
\draw[black,arrows={-Triangle[length=0.27cm,width=0.2cm]}]
($(118:3*1.5)$) --  ++(30:0.001);

\node at (60:3.2*1.5) {\small $9_r$};
\node at (118.5:3.15*1.5) {\small $9_Y$};

\node at (77:1.15*1.5) {\small $1''$};
\node at (84:2.5*1.5) {\small $1_r''$};



\node at (86.2:3.85*1.5) {\small $4_{r}'$};
\node at (87.1:4.78*1.5) {\small $4'$};
	
\node at (96:3.56*1.5) {\small $5_Y$};
\node at (143.5:3.6*1.5) {\small $6_Y$};
\node at (140:2.55*1.5) {\small $7_Y$};
\node at (99:2.5*1.5) {\small $8_Y$};

\node at (97:3.2*1.5) {\small $2$};
\node at (142:3.18*1.5) {\small $5$};

\draw[black,line width=0.65 mm,->-=0.21,->-=0.75] (90:2*1.5)--(113.175:3*1.5)--(90:4.5*1.5);

\draw[black,line width=0.65 mm,->-=0.21,->-=0.75] (-90+240:2*1.5)--(-113.175+240:3*1.5)--(-90+240:4.5*1.5);

\draw[black,arrows={-Triangle[length=0.27cm,width=0.2cm]}]
($(95:3*1.5)$) --  ++(95-90:0.001);
\draw[black,arrows={-Triangle[length=0.27cm,width=0.2cm]}]
($(18+120:3*1.5)$) --  ++(18+120-90:0.001);

\draw[black,line width=0.65 mm,-<-=0.47] ([shift=(30:2*1.5cm)]0,0) arc (30:90:2*1.5cm);
\draw[black,line width=0.65 mm,-<-=0.47] ([shift=(30:4.5*1.5cm)]0,0) arc (30:90:4.5*1.5cm);

\node at (60:2.2*1.5) {\small $9_{L}$};
\node at (60:4.75*1.5) {\small $9_{R}$};

\draw[blue,fill] (113.175:3*1.5) circle (0.12cm);
\draw[green,fill] (-113.175+240:3*1.5) circle (0.12cm);

\draw[black,line width=0.65 mm,->-=0.5] (120:3.61*1.5)--(90:4.5*1.5);
\draw[black,line width=0.65 mm,->-=0.5] (120:3.61*1.5)--(150:4.5*1.5);
\draw[black,line width=0.65 mm,->-=0.6] (90:2*1.5)--(120:2.41*1.5);
\draw[black,line width=0.65 mm,->-=0.6] (150:2*1.5)--(120:2.41*1.5);

\node at (106:4.06*1.5) {\small $1_Y$};
\node at (134:4.06*1.5) {\small $2_Y$};
\node at (129:1.95*1.5) {\small $3_Y$};
\node at (107:1.95*1.5) {\small $4_Y$};

\draw[red,fill] (120:3.61*1.5) circle (0.12cm);
\draw[red,fill] (120:2.41*1.5) circle (0.12cm);

\draw[black,line width=0.65 mm, ->=0.35] (120:1.5*3) circle (1.22cm);
\draw[black,arrows={-Triangle[length=0.3cm,width=0.2cm]}]
($(118:1.5*3+1.22)$) --  ++(120-93:0.001);

\end{tikzpicture}
\end{center}
\begin{figuretext}
\label{Gammahatfig}
The contour $\hat{\Gamma} \cap \{k \,|\, \arg k \in [\frac{\pi}{6},\frac{5\pi}{6}]\}$ in the case of Sector $\III_\geq$. 
\end{figuretext}
\end{figure}

\subsection{The small-norm solution $\hat{n}$} 
Let $\mathcal{D}$ denote the union of the open set $D(\epsilon)$ and the sets obtained by letting the symmetry $k \mapsto \omega k$ act repeatedly on $D(\epsilon)$. 
Let $\hat{\Gamma} = \Gamma^{(3)} \cup \partial \mathcal{D}$ and assume that the boundary of each of the three components of $\mathcal{D}$ is oriented clockwise, see Figure \ref{Gammahatfig}. 
Using the $\mathcal{A}$-symmetry, we extend the definition of $m^\omega$ to all of $\mathcal{D}$.
The function $\hat{n}(x,t,k)$ defined by 
\begin{align}\label{mhatdef}
\hat{n}(x,t,k) = \begin{cases}
n^{(3)}(x, t, k)m^\omega(x,t,k)^{-1}, & k \in \mathcal{D},\\
n^{(3)}(x, t, k), & k \in \C \setminus \mathcal{D},
\end{cases} \quad (x,t) \in \III_\geq^T,
\end{align}
satisfies the jump relation $\hat{n}_+ = \hat{n}_- \hat{v}$ on $\hat{\Gamma} \setminus \hat{\Gamma}_\star$, where the jump matrix $\hat{v}$ is given by 
\begin{align}\label{vhatdef}
\hat{v}(x, t, k) 
=  \begin{cases}
 m_-^\omega(x, t, k) v^{(3)}(x, t, k) m_+^\omega(x,t,k)^{-1}, & k \in \hat{\Gamma} \cap \mathcal{D}, \\
m^\omega(x, t, k), & k \in \partial \mathcal{D}, \\
v^{(3)}(x, t, k),  & k \in \hat{\Gamma} \setminus \bar{\mathcal{D}}.
\end{cases}
\end{align}
For a function $h:\hat{\Gamma}\to \mathbb{C}$, we define $(\hat{\mathcal{C}}h)(k) := \frac{1}{2\pi i}\int_{\hat{\Gamma}}\frac{h(k')dk'}{k'-k}$ and $\hat{\mathcal{C}}_{\hat{w}}h := \hat{\mathcal{C}}_{-}(h \hat{w})$.
By construction, $\hat{n}$ satisfies a RH problem with jump contour $\hat{\Gamma}$ in the $L^2$-sense. Hence, standard arguments show that, for $(x,t) \in \III_\geq^T$ and $k \in \C \setminus \hat{\Gamma}$,
\begin{align}\label{mhatrepresentation}
\hat{n}(x, t, k) = (1,1,1) + \hat{\mathcal{C}}(\hat{\mu} \hat{w}) = (1,1,1) + \frac{1}{2\pi i}\int_{\hat{\Gamma}} (\hat{\mu} \hat{w})(x, t, s) \frac{ds}{s - k},
\end{align}
where $\hat{\mu} := \hat{n}_- \in (1,1,1) + L^2(\hat{\Gamma})$ solves the singular integral equation
$$\hat{\mu} = (1,1,1) + \hat{\mathcal{C}}_{\hat{w}}\hat{\mu} \quad \text{in} \quad L^2(\hat{\Gamma}).$$

We write $\hat{\Gamma}$ as the union of three subcontours as follows:
$$\hat{\Gamma} = \partial \mathcal{D} \cup \hat{\mathcal{Y}}^\epsilon \cup \hat{\Gamma}',$$
where $\hat{\mathcal{Y}}^\epsilon := \cup_{j=0}^2 \omega^j \mathcal{Y}^\epsilon$ and $\hat{\Gamma}' := \hat{\Gamma}\setminus (\partial \mathcal{D} \cup \hat{\mathcal{Y}}^\epsilon)$.

\begin{lemma}\label{wlemmaIII}
Let $\hat{w} = \hat{v} - I$. For each $1 \leq p \leq \infty$ and $j =0,1$, the following estimates hold uniformly for $(x,t) \in \III_\geq^T$:
\begin{subequations}\label{westimateIII}
\begin{align}\label{westimateIIIa}
& \|\partial_x^j\hat{w}\|_{L^p(\partial \mathcal{D})} \leq Ct^{-\frac{1}{3}-\frac{j}{6}},
	\\\label{westimateIIIb}
& \|\partial_x^j\hat{w}\|_{L^p(\hat{\mathcal{Y}}^\epsilon)} \leq Ct^{-\frac{N+1}{6}}, 
	\\\label{westimateIIIc}
& \|\partial_x^j\hat{w}\|_{L^p(\hat{\Gamma}')} \leq Ct^{-N}.
\end{align}
\end{subequations}
\end{lemma}
\begin{proof}
The estimates (\ref{westimateIIIa}) follow from (\ref{m0LinftyestimateIII})--(\ref{m0xLinftyestimateIII}).
For $k \in \mathcal{Y}^\epsilon$, we have
$$\hat{w} = m_-^\omega (v^{(3)} - v^\omega) (m_+^\omega)^{-1},$$
so equations (\ref{m0boundIII}) and (\ref{v3vomegaestimates}) yield (\ref{westimateIIIb}).
On $\hat{\Gamma}'$, the jump matrix $v^{(3)}$ involves either (i) the small remainders $r_{j,r}$ and $\hat{r}_{j,r}$, (ii) the matrices $v_{4_r'}^{(3)}$ and $v_{1_r''}^{(3)}$ or matrices related to these by symmetry, or (iii) exponentials which are uniformly exponentially small. In case (i), the estimate (\ref{westimateIIIc}) holds as a consequence of Lemma \ref{decompositionlemmaIII} and (for the part of $\hat{\Gamma}'$ that lies in $\mathcal{D}$) the boundedness (\ref{m0boundIII}) of $m^\omega$. In case (ii), (\ref{westimateIIIc}) follows from Lemma \ref{vsmallnearilemma}.
In case (iii), (\ref{westimateIIIc}) holds as a consequence of the exponential decay.
\end{proof}

The estimates in Lemma \ref{wlemmaIII} show that, for $1 \leq p \leq \infty$ and $N$ sufficiently large,
\begin{align}\label{hatwLinftyIII}
\begin{cases}
\|\hat{w}\|_{L^p(\hat{\Gamma})} \leq Ct^{-1/3},
	\\
\|\partial_{x}\hat{w}\|_{L^p(\hat{\Gamma})} \leq C t^{-1/2},
\end{cases}	 \qquad (x,t) \in \III_\geq^T.
\end{align}
In particular, $\hat{w}$ obeys $\|\hat{w}\|_{(L^2 \cap L^\infty)(\hat{\Gamma})} \to 0$ uniformly as $t \to \infty$. Thus, increasing $T$ if necessary, we may assume that $\|\hat{\mathcal{C}}_{\hat{w}}\|_{\mathcal{B}(L^2(\hat{\Gamma}))} \leq 1/2$ for all $(x,t) \in \III_\geq^T$, where $\mathcal{B}(L^{2}(\hat{\Gamma}))$ denotes the space of bounded linear operators on $L^{2}(\hat{\Gamma})$. Hence $I - \hat{\mathcal{C}}_{\hat{w}}$ is invertible on $L^2(\hat{\Gamma})$, $\hat{\mu} = (1,1,1) + (I - \hat{\mathcal{C}}_{\hat{w}})^{-1}\hat{\mathcal{C}}_{\hat{w}}(1,1,1)$,
and
\begin{align}\label{muestimateIII}
\|\hat{\mu}(x,t,\cdot) - (1,1,1)\|_{L^2(\hat{\Gamma})} 
\leq  \frac{C\|\hat{w}\|_{L^2(\hat{\Gamma})}}{1 - \|\hat{\mathcal{C}}_{\hat{w}}\|_{\mathcal{B}(L^2(\hat{\Gamma}))}}\leq C t^{-1/3}
\end{align}
for $(x,t) \in \III_\geq^T$. 
On $\partial \mathcal{D}$, we have $\hat{w} = m^\omega - I$. Hence, by (\ref{momegaexpansion}) and (\ref{westimateIII}), 
\begin{align}\nonumber
\hat{w}(x, t,k) = &\; \frac{\hat{w}_1(y,k)}{t^{1/3}} + \frac{\hat{w}_2(x,t,k)}{t^{1/2}}
+ \frac{\hat{w}_3(y,\tilde{y},k)\ln t}{t^{2/3}}  
+ \frac{\hat{w}_4(y,\tilde{y},k)}{t^{2/3}}  
	\\\label{whatexpansion}
& + \frac{\hat{w}_5(x,y,k)\ln t}{t^{5/6}}  
+ \frac{\hat{w}_{err}(x,t,k)}{t^{5/6}}, \quad (x,t) \in \III_\geq^T, \ k \in \hat{\Gamma},
\end{align}
where the coefficients $\hat{w}_j$, $j = 1,\dots, 5$, are nonzero only for $k \in \partial \mathcal{D}$ and the following uniform bounds hold for $1 \leq p \leq \infty$ and $l = 0,1$,
\begin{align}\label{hatwerrestIII}
\begin{cases}
\|\partial_x^l \hat{w}_j\|_{L_k^p(\hat{\Gamma})}  \leq C, & 
	\\
 \|\partial_x^l\hat{w}_{err}\|_{L_k^p(\hat{\Gamma})} \leq C,  &
  \end{cases} \ (x,t) \in \III_\geq^T, \ k \in \hat{\Gamma}.
\end{align}
The function $\hat{w}$ (and hence also all the coefficients $\hat{w}_j$, $j = 1,\dots, 5$) obey the symmetry 
\begin{align}\label{whatsymm}
\hat{w}(x,t,k) = \mathcal{A} \hat{w}(x,t,\omega k)\mathcal{A}^{-1},
\end{align}
and for $k \in \partial D(\epsilon)$, we have 
\begin{align}\label{what1what2explicit}
& \hat{w}_1(y, k) = \frac{m_1^P(y)}{f_\omega(k)}, \qquad
\hat{w}_2(x,t, k) = \frac{Ym^P(y,0)m_1^W(\tilde{y})m^P(y,0)^{-1}Y^{-1}}{\frac{2^{1/3}}{3^{1/12}} f_\omega(k)},
	\\
& \hat{w}_3(y,\tilde{y},k) = \frac{\hat{m}_{1, \ln}^Y(y,\tilde{y})}{f_\omega(k)},
\qquad \hat{w}_4(y,\tilde{y},k) = \frac{m_2^P(y)}{f_\omega(k)^2} + \frac{\hat{m}_1^Y(y,\tilde{y})}{f_\omega(k)}.
	\\
& \hat{w}_5(x,t,k)=\frac{Y\hat{m}_{2, \ln}^Y(y,\tilde{y})Y^{-1}}{f_\omega(k)}.
\end{align}

Since $\hat{\mu} = (1,1,1) + (I - \hat{\mathcal{C}}_{\hat{w}})^{-1}\hat{\mathcal{C}}_{\hat{w}}(1,1,1) = (1,1,1) + \hat{\mathcal{C}}_{\hat{w}}(1,1,1) + \hat{\mathcal{C}}_{\hat{w}}^2 (1,1,1) + \cdots$ and
\begin{align*}
\hat{\mathcal{C}}_{\hat{w}} = &\; \frac{\hat{\mathcal{C}}_{\hat{w}_1}}{t^{1/3}} + \frac{\hat{\mathcal{C}}_{\hat{w}_2}}{t^{1/2}}
+ \frac{\hat{\mathcal{C}}_{\hat{w}_3} \ln t}{t^{2/3}} 
+ \frac{\hat{\mathcal{C}}_{\hat{w}_4}}{t^{2/3}} 
+ \frac{\hat{\mathcal{C}}_{\hat{w}_5} \ln t}{t^{5/6}} 
+ \frac{\hat{\mathcal{C}}_{\hat{w}_{err}}}{t^{5/6}},
\end{align*}
it follows that 
\begin{align}\label{hatmuexpansion}
\hat{\mu}(x, t,k) = &\; (1,1,1) + \frac{\hat{\mu}_1(y,k)}{t^{1/3}} 
+ \frac{\hat{\mu}_{err}(x,t,k)}{t^{1/2}}, \qquad (x,t) \in \III_\geq^T, \ k \in \hat{\Gamma},
\end{align}
where $\hat{\mu}_1(y,k) := (\hat{\mathcal{C}}_{\hat{w}_1(y,\cdot)}(1,1,1))(k)$ and, for $l = 0,1$,
\begin{align}\label{hatmuerrestIII}
\begin{cases}
 \|\partial_x^l \hat{\mu}_1 \|_{L^2(\partial D(\epsilon))} \leq C, &
	\\ 
\|\partial_x^l \hat{\mu}_{err} \|_{L^2(\partial D(\epsilon))} \leq C, &
\end{cases} \quad (x,t) \in \III_\geq^T.
\end{align}

\subsection{Asymptotics of $u$}
Let
\begin{align}\label{limkm}
I(x,t) := \lim_{k\to \infty} k(n(x,t,k)  - (1,1,1)) = -\frac{1}{2\pi i}\int_{\hat{\Gamma}} \hat{\mu}(x, t, k) \hat{w}(x, t, k) dk.
\end{align}
We will obtain the asymptotics of $u(x,t)$ from the first formula in \eqref{recoveruvn}, which can be written in terms of $I(x,t)$ as
\begin{align}
u(x,t) = -i\sqrt{3}\frac{\partial}{\partial x} \lim_{k\to\infty} k (n_3(x,t,k) - 1) = -i\sqrt{3} \frac{\partial  I_{3}(x,t)}{\partial x}.
\end{align} 
Let $I_{\hat{\mathcal{Y}}^\epsilon}$, $I_{\hat{\Gamma}'}$, and $I_{\partial D(\epsilon)}$ denote the contributions to the right-hand side of (\ref{limkm}) from $\hat{\mathcal{Y}}^\epsilon$, $\hat{\Gamma}'$, and $\partial D(\epsilon)$, respectively. 
By (\ref{westimateIIIb}) and (\ref{muestimateIII}), we have\footnote{In the rest of this section, all error terms $O(\cdot)$ are uniform with respect to $(x,t) \in \III_\geq^T$ as $t \to \infty$.}
\begin{align*}
 I_{\hat{\mathcal{Y}}^\epsilon} = -\frac{1}{2\pi i}\int_{\hat{\mathcal{Y}}^\epsilon} \hat{\mu} \hat{w} dk
= O\big(\|\hat{w}\|_{L^1(\hat{\mathcal{Y}}^\epsilon)}
& + \|\hat{\mu} - I\|_{L^2(\hat{\mathcal{Y}}^\epsilon)}\|\hat{w}\|_{L^2(\hat{\mathcal{Y}}^\epsilon)}\big)
  = O(t^{-(N+1)/6}),
\end{align*}
and, by (\ref{westimateIIIc}) and (\ref{muestimateIII}), we have $I_{\hat{\Gamma}'} = O(t^{-N})$, and both these formulas can be differentiated with respect to $x$: 
$$\partial_x I_{\hat{\mathcal{Y}}^\epsilon} =  O(t^{-(N+1)/6}), \qquad \partial_x I_{\hat{\Gamma}'} = O(t^{-N}).$$

The symmetry $\hat{w}(x,t,k) = \mathcal{A} \hat{w}(x,t,\omega k)\mathcal{A}^{-1}$ implies that the contribution from $\partial \mathcal{D}$ to the right-hand side of (\ref{limkm}) is 
$$I_{\partial \mathcal{D}} := \sum_{l=0}^2 \omega^l  I_{\partial D(\epsilon)} \mathcal{A}^l,$$
so it only remains to consider $I_{\partial D(\epsilon)}$.
By (\ref{whatexpansion}) and (\ref{hatmuexpansion}), we have
\begin{align}\nonumber
 I_{\partial D(\epsilon)} = & -\frac{1}{2\pi i} \int_{\partial D(\epsilon)} (1,1,1)\hat{w} dk	
-  \frac{1}{2\pi i} \int_{\partial D(\epsilon)} (\hat{\mu} - (1,1,1)) \hat{w} dk
	\\\nonumber
=& -\frac{1}{2\pi i}\int_{\partial D(\epsilon)} (1,1,1)(m^\omega - I) dk
 -  \frac{1}{2\pi i} \int_{\partial D(\epsilon)} \bigg(\frac{\hat{\mu}_1}{t^{1/3}} 
+ \frac{\hat{\mu}_{err}}{t^{1/2}} \bigg) 
	\\ \label{partialDcontribution}
&\times
 \bigg(\frac{\hat{w}_1}{t^{1/3}} + \frac{\hat{w}_2}{t^{1/2}}
+ \frac{\hat{w}_3\ln t}{t^{2/3}}  
+ \frac{\hat{w}_4}{t^{2/3}}  
+ \frac{\hat{w}_5\ln t}{t^{5/6}}  
+ \frac{\hat{w}_{err}}{t^{5/6}} \bigg)  dk.
\end{align}
Plugging in the expansion (\ref{momegaexpansion}) of $m^\omega$ and using the bounds (\ref{hatwerrestIII}) and (\ref{hatmuerrestIII}), as well as the expressions
$$\underset{k = \omega}{\res} \frac{1}{f_\omega(k)}= -2i\omega \Big(\frac{2}{3}\Big)^{1/3},
\qquad 
\underset{k = \omega}{\res} \frac{1}{f_\omega(k)^2}= -2i \Big(\frac{2}{3}\Big)^{1/3} \underset{k = \omega}{\res} \frac{1}{f_\omega(k)},$$
we obtain
\begin{align}\nonumber
I_{\partial D(\epsilon)} = & -2i\omega \Big(\frac{2}{3}\Big)^{1/3} (1,1,1)\bigg(\frac{m_1^P(y)}{t^{1/3}} 
+ \frac{Ym^P(y,0)m_1^W(\tilde{y})m^P(y,0)^{-1}Y^{-1}}{\frac{2^{1/3}}{3^{1/12}} t^{1/2}}
	\\\nonumber
&-2i \Big(\frac{2}{3}\Big)^{1/3} \frac{m_2^P(y)}{t^{2/3}} + \frac{\hat{m}_{1,\ln}^Y(y,\tilde{y}) \ln t}{t^{2/3}}
+ \frac{\hat{m}_1^Y(y,\tilde{y})}{t^{2/3}}
+ \frac{Y\hat{m}_{2,\ln}^Y(y,\tilde{y}) Y^{-1}\ln t}{t^{5/6}} \bigg)
	\\\label{IpartialDepsilon}
& -  \frac{1}{2\pi i} \frac{\int_{\partial D(\epsilon)} \hat{\mu}_1 \hat{w}_1 dk}{t^{2/3}}
 + O\bigg(\frac{1}{t^{5/6}}\bigg).
\end{align}

We now make two important observations. First, note that the explicit expressions for $m^P(y,0)$ and $m_1^W$ in (\ref{mPat0expression}) and (\ref{m1Wexpression}) show that $(1,1,1)Ym^P(y,0)m_1^W(\tilde{y})m^P(y,0)^{-1}Y^{-1} = 0$. We conclude that the term of order $t^{-1/2}$ in (\ref{IpartialDepsilon}) does not contribute to $u(x,t)$. 
Second, since
\begin{align}\label{partialxpartialy}
\frac{\partial}{\partial x}= \bigg(\frac{2}{3t}\bigg)^{1/3}\frac{\partial}{\partial y}, \qquad
\frac{\partial}{\partial x}=  \frac{1}{3t} \frac{\partial}{\partial \tilde{y}},
\end{align}
the contributions to $u(x,t)$ from the terms in (\ref{IpartialDepsilon}) involving $m_2^P(y)$, $\hat{m}_{1,\ln}^Y(y,\tilde{y})$, $\hat{m}_{1}^Y(y,\tilde{y})$, and $\int_{\partial D(\epsilon)} \hat{\mu}_1 \hat{w}_1 dk$ will all be of order $O(t^{-1})$. 

The above observations imply that for the purposes of computing the asymptotics of $u(x,t)$ up to an error of order $O(t^{-5/6})$, we may replace $I_{\partial D(\epsilon)}$ by 
\begin{align}\label{IpartialDepsilon2}
I_{\partial D(\epsilon)}' = &\; -2i\omega \Big(\frac{2}{3}\Big)^{1/3}\frac{(1,1,1)m_1^P(y)}{t^{1/3}}
-2i\omega \Big(\frac{2}{3}\Big)^{1/3} \frac{(1,1,1)Y\hat{m}_{2,\ln}^Y(y,\tilde{y}) Y^{-1}\ln t}{t^{5/6}}.
\end{align}
To summarize, we have shown that
\begin{align}\label{recoveru2}
u(x,t) = -i\sqrt{3}  \frac{\partial  }{\partial x}\bigg[\sum_{l=0}^2 \omega^l I_{\partial D(\epsilon)}' \mathcal{A}^l\bigg]_3
+ O\bigg(\frac{1}{t^{5/6}}\bigg), \qquad t \to \infty,
\end{align}
with $I_{\partial D(\epsilon)}'$ given by (\ref{IpartialDepsilon2}). 
Our next lemma shows that the second term on the right-hand side of (\ref{IpartialDepsilon2}) actually vanishes.

\begin{lemma}\label{hatm2lnlemma}
$(1,1,1)Y(x,t)\hat{m}_{2,\ln}^Y(y,\tilde{y}) = 0$ for all $(x,t) \in \III_\geq^T$.
\end{lemma}
\begin{proof}
By comparing (\ref{p1explicit}) and (\ref{psumIII}), we see that all coefficients $a_{1l}^{(1)}$ in the expression (\ref{psumIII}) for $p_1(t,z)$ vanish except for $a_{11}^{(1)}$, and that $a_{11}^{(1)} = -a_1$ where $a_1$  is the constant in (\ref{a1b1explicit}). A similar calculation shows that 
all coefficients $a_{1l}^{(2)}$ in the expression (\ref{psumIII}) for $p_2(t,z)$ vanish except for $a_{11}^{(2)}$, and that $a_{11}^{(2)} = a_1$. Thus, by (\ref{wYexpansionPleadingcoeffb}), (\ref{hatm1Ym2Yc}), and (\ref{A1def}),
\begin{align*}
\hat{m}_{2,\ln}^Y =&\; \frac{1}{6\pi i} \bigg(\int_{Y_1 \cup Y_2} 
a_{11}^{(1)} z \big[m_-^P(y,z) m_+^P(y,z)^{-1}, B_1\big] dz
+ \int_{Y_3 \cup Y_4} a_{11}^{(2)} z \big[m_-^P(y,z) m_+^P(y,z)^{-1}, B_1\big]dz\bigg)
	\\
 = &\; -\frac{a_1}{6\pi i} \frac{s e^{\frac{3\pi i}{4}} e^{U_{\HM}(y)}}{\frac{2^{1/3}}{3^{1/12}} \sqrt{12 \pi} \sqrt{1 + \tilde{y}}}  \Bigg\{\int_{Y_1 \cup Y_2} 
 \Bigg[m_-^P(y,z) m_+^P(y,z)^{-1}, \begin{pmatrix} 0 & 1 & 0 \\
0 & 0 & 0 \\
0 & -1 & 0 \end{pmatrix}\Bigg] dz
	\\
& - \int_{Y_3 \cup Y_4}  \Bigg[m_-^P(y,z) m_+^P(y,z)^{-1}, \begin{pmatrix} 0 & 1 & 0 \\
0 & 0 & 0 \\
0 & -1 & 0 \end{pmatrix}\Bigg]dz\Bigg\}.
\end{align*}

\begin{figure}
\begin{center}
\begin{tikzpicture}[master, scale=0.7]
\node at (0,0) {};

\draw[black,line width=0.55 mm,->-=0.6] (0,0.5)--($(0,0.5)+(30:3)$);
\draw[black,line width=0.55 mm,->-=0.6] (0,0.5)--($(0,0.5)+(150:3)$);
\draw[black,line width=0.55 mm,-<-=0.45] (0,-0.5)--($(0,-0.5)+(-30:3)$);
\draw[black,line width=0.55 mm,-<-=0.45] (0,-0.5)--($(0,-0.5)+(-150:3)$);

\node at (1.5,0.85) {\small $Y_1$};
\node at (-1.5,0.85) {\small $Y_2$};
\node at (-1.5,-0.93) {\small $Y_3$};
\node at (1.5,-0.93) {\small $Y_4$};

\node at (0,0.25) {\small $i$};
\node at (0,-0.3) {\small $-i$};

\node at (0,2.05) {\small $\gamma_V$};
\node at (0,-2.05) {\small $-\gamma_V$};

\draw[black,line width=0.55 mm,->-=0.55]  ($(0,0.5)+(150:3)+(0,0.3)$) to[out=-25,in=-155] ($(0,0.5)+(30:3)+(0,0.3)$);
\draw[black,line width=0.55 mm,->-=0.55]  ($(0,-0.5)+(-30:3)+(0,-0.3)$) to[out=155,in=25] ($(0,-0.5)+(-150:3)+(0,-0.3)$);

\end{tikzpicture}
\end{center}
\begin{figuretext}
\label{gammaVfig}
The contours $\gamma_V$ and $-\gamma_V$.
\end{figuretext}
\end{figure}

Let $\gamma_V$ denote a contour from $e^{5\pi i/6} \infty$ to $e^{\pi i/6} \infty$ which lies above the contour $Y_1 \cup Y_2$, see Figure \ref{gammaVfig}.
A calculation using the jump relation for $m^P$ shows that
\begin{align*}
\hat{m}_{2,\ln}^Y(y,\tilde{y}) 
= & -\frac{a_1}{6\pi i} \frac{s e^{\frac{3\pi i}{4}} e^{U_{\HM}(y)}}{\frac{2^{1/3}}{3^{1/12}} \sqrt{12 \pi} \sqrt{1 + \tilde{y}}}  \Bigg\{\int_{\gamma_V}  e^{2iyz + \frac{8iz^3}{3}} (m^P_{13} + m^P_{33})  \begin{pmatrix} 0 & m^P_{13} & 0 \\
0 & 0 & 0 \\
0 & m^P_{33} & 0 \end{pmatrix}dz
	\\
& - \int_{-\gamma_V}  e^{-2iyz - \frac{8iz^3}{3}} (m^P_{11} + m^P_{31})  \begin{pmatrix} 0 & -m^P_{11} & 0 \\
0 & 0 & 0 \\
0 & -m^P_{31} & 0 \end{pmatrix} dz\Bigg\}.
\end{align*}
The symmetry (\ref{mPsymm}) implies that $m^P(y,z)_{11} = m^P(y,-z)_{33}$ and $m^P(y,z)_{13} = m^P(y,-z)_{31}$. Thus changing variables $z \to -z$ in the integral along $-\gamma_V$, we find
\begin{align*}
\hat{m}_{2,\ln}^Y(y,\tilde{y}) 
= & -\frac{a_1}{6\pi i} \frac{s e^{\frac{3\pi i}{4}} e^{U_{\HM}(y)}}{\frac{2^{1/3}}{3^{1/12}} \sqrt{12 \pi} \sqrt{1 + \tilde{y}}}  \int_{\gamma_V}  e^{2iyz + \frac{8iz^3}{3}} (m^P_{13} + m^P_{33})(m^P_{13} - m^P_{33}) dz \begin{pmatrix} 0 & 1 & 0 \\
0 & 0 & 0 \\
0 & -1 & 0 \end{pmatrix}.
\end{align*}
Since $Y$ is diagonal and $Y_{11} = Y_{33}$, the claim follows.
\end{proof}

Substituting (\ref{IpartialDepsilon2}) into (\ref{recoveru2}) and using Lemma \ref{hatm2lnlemma}, we arrive at
\begin{align*}
 u(x,t) =&\; \frac{2^{1/3} 3^{2/3} \partial_x\left(u_{\HM}(y) - \int_{+\infty}^y u_{\HM}(y)^2 dy\right)}{t^{1/3}} + O\bigg(\frac{1}{t^{5/6}}\bigg).
\end{align*}
Using (\ref{partialxpartialy}) to compute the $x$-derivative, we obtain the formula in (\ref{uasymptotics}) for Sector $\III$.

\section{Asymptotics: Sectors I, II, IV, and V}\label{othersectorssec}
In Section \ref{transitionsec}, we gave a detailed derivation of the asymptotic formula of Theorem \ref{asymptoticsth} valid in Sector III. The formulas of Theorem \ref{asymptoticsth} for Sectors I, II, IV,  and V can be derived using similar tools. Since the derivations are rather long, we do not provide full proofs here but refer to the companion papers \cite{CLsectorI, CLsectorII, CLsectorIV, CLsectorV} for detailed derivations; here we only highlight the main differences compared to Sector III.

\subsection{Sectors I and II} 
Only six of the twelve saddle points contribute to the asymptotics in Sectors I and II. Indeed, for $\zeta > c > 1$, the saddle point $k_4$ lies in the domain where $\re \Phi_{21} < 0$, see Figure \ref{fig: Re Phi 21 for various zeta}. This means that the contributions from $k_4$ (and hence, by symmetry, also from $k_3 = k_4^{-1}$, $\omega k_3, \omega^2 k_3, \omega k_4$, and $\omega^2 k_4$) are uniformly exponentially small. 
The contributions from the six saddle points on  $\partial \D$ can be evaluated with the help of a parabolic cylinder local model, and this eventually leads to the formulas of Theorem \ref{asymptoticsth}, see \cite{CLsectorI, CLsectorII} for details.

The derivation in Sector I is different from that of the other sectors in that we use $x$ as the large parameter instead of $t$. As $\zeta \to \infty$, $k_1$ approaches $i$ and by Assumption \ref{nounstablemodesassumption}, the symmetry (\ref{r1r2 relation with kbar symmetry}), and Theorem \ref{directth} $(i)$, both $r_1(k)$ and $r_2(k)$ vanish to all orders at $k = i$. This leads to the rapid decay observed in Sector I as $\zeta \to \infty$, see \cite{CLsectorI}.

\subsection{Sector IV} 
In Sector IV, all twelve saddle points contribute to the asymptotics. 
The local saddle point analysis involves two rather complicated model RH problems built in terms of parabolic cylinder functions. The contributions from six of the saddle points combine with the global parametrix to yield the right-moving modulated wave $A_{1}(\zeta) t^{-1/2}  \cos \alpha_{1}(\zeta,t) $ in the asymptotic formula (\ref{uasymptotics}), while the other six yield the left-moving modulated wave $A_{2}(\zeta) t^{-1/2} \cos \alpha_{2}(\zeta,t)$, see the companion paper \cite{CLsectorIV} for details.
While the global parametrix $\Delta^{-1}$ employed in Sector III is built out of the single $\zeta$-independent function $\delta(k)$ in (\ref{deltadefIII}), the definition of the global parametrix in Sector IV involves five functions $\delta_j(\zeta, k)$, $j = 1, \dots, 5$, given by
\begin{align}\nonumber
& \delta_{1}(\zeta, k) = \exp \bigg\{\hspace{-0.02cm} \frac{-1}{2\pi i} \int_{i}^{\omega k_{4}} \frac{\ln(1 + r_1(s)r_{2}(s))}{s - k} ds \hspace{-0.02cm}\bigg\},  \; \delta_{2}(\zeta, k) \hspace{-0.02cm}=\hspace{-0.02cm} \exp \bigg\{\hspace{-0.02cm} \frac{1}{2\pi i} \int_{\omega k_{4}}^{\omega^{2}k_{2}} \frac{\ln(1 + r_1(s)r_{2}(s))}{s - k} ds \hspace{-0.02cm}\bigg\},
	 \\\nonumber
& \delta_{3}(\zeta, k) = \exp \bigg\{ \frac{1}{2\pi i} \int_{\omega k_{4}}^{\omega^{2}k_{2}} \frac{\ln f(s)}{s - k} ds \bigg\}, \quad
\delta_{4}(\zeta, k) = \exp \bigg\{ \frac{1}{2\pi i} \int_{\omega^{2}k_{2}}^{\omega} \frac{\ln f(s)}{s - k} ds \bigg\},
  	\\ \label{IVdeltadef}
&  \delta_{5}(\zeta, k) = \exp \bigg\{ \frac{1}{2\pi i} \int_{\omega^{2}k_{2}}^{\omega} \frac{\ln f(\omega^{2}s)}{s - k} ds \bigg\},
\end{align}
where the paths follow the unit circle in the counterclockwise direction, the principal branch is used for the logarithms, and $f(k)$ is defined in (\ref{def of f}).
In Sector IV, we have $\arg k_{2} \in (-\frac{3\pi}{4},-\frac{2\pi}{3})$ and $\arg k_{4} \in (-\frac{\pi}{6},0)$, so Lemma \ref{inequalitieslemma} implies that the arguments of all the logarithms appearing in (\ref{IVdeltadef}) are $>0$.
 
The phase shifts $\arg d_{1,0}$ and $\arg d_{2,0}$ that appear in the formula for Sector IV in Theorem \ref{asymptoticsth} are generated by certain quotients of the functions $\{\delta_j\}_{j=1}^5$ evaluated at the saddle points.  More precisely, the definitions of $\arg d_{1,0}$ and $\arg d_{2,0}$ involve the functions
\begin{align}\nonumber
& \mathcal{D}_{1}(k) := \frac{\delta_{1}(\omega k)\delta_{1}(\frac{1}{\omega^{2}k})^{2}}{\delta_{1}(\omega^{2}k)^{2}\delta_{1}(\frac{1}{\omega k})\delta_{1}(\frac{1}{k})} \frac{\delta_{2}(\omega^{2}k)\delta_{2}(\frac{1}{k})^{2}}{\delta_{2}(\omega k)^{2} \delta_{2}(\frac{1}{\omega^{2}k})\delta_{2}(\frac{1}{\omega k})} \frac{\delta_{3}(\omega k) \delta_{3}(\omega^{2} k) \delta_{3}(\frac{1}{\omega k})^{2}}{\delta_{3}(\frac{1}{k})\delta_{3}(\frac{1}{\omega^{2} k})}  
	\\ \nonumber
& \hspace{1.5cm} \times \frac{\delta_{4}(\omega^{2}k)^{2} \delta_{4}(\frac{1}{k})\delta_{4}(\frac{1}{\omega k})}{\delta_{4}(k)\delta_{4}(\omega k)\delta_{4}(\frac{1}{\omega^{2} k})^{2}} \frac{\delta_{5}(\omega k)^{2} \delta_{5}(\frac{1}{\omega k})\delta_{5}(\frac{1}{\omega^{2}k})}{\delta_{5}(k) \delta_{5}(\frac{1}{k})^{2}\delta_{5}(\omega^{2}k)}, 
	\\ \nonumber
& \mathcal{D}_{2}(k) := \frac{\delta_{1}(\omega k)^{2}\delta_{1}(\frac{1}{\omega^{2}k})\delta_{1}(\frac{1}{k})}{\delta_{1}(\omega^{2}k)\delta_{1}(\frac{1}{\omega k})^{2}\delta_{1}(k)} \frac{\delta_{2}(\frac{1}{k})\delta_{2}(\frac{1}{\omega k})}{\delta_{2}(\omega k) \delta_{2}(\frac{1}{\omega^{2}k})^{2}\delta_{2}(\omega^{2} k)} \frac{\delta_{3}(\omega^{2} k)^{2}  \delta_{3}(\frac{1}{\omega k})\delta_{3}(\frac{1}{\omega^{2} k})}{\delta_{3}(\frac{1}{k})^{2}\delta_{3}(\omega k)}  
	\\ \label{mathcalPdef}
& \hspace{1.5cm} \times \frac{\delta_{4}(\omega^{2}k) \delta_{4}(\frac{1}{\omega k})^{2}}{\delta_{4}(\omega k)^{2}\delta_{4}(\frac{1}{\omega^{2} k})\delta_{4}(\frac{1}{k})} \frac{\delta_{5}(\omega k) \delta_{5}(\frac{1}{\omega^{2}k})^{2}\delta_{5}(\omega^{2} k)}{\delta_{5}(\frac{1}{k}) \delta_{5}(\frac{1}{\omega k})},
\end{align}
where we have suppressed the $\zeta$-dependence for conciseness, as well as the functions $\chi_{j}(\zeta,k)$, $j = 1, 2, 3$, and $\tilde{\chi}_{j}(\zeta,k)$, $j = 2,3,4,5$, defined by
\begin{align}
\chi_{1}(\zeta,k) =&\; - \frac{1}{2\pi i} \int_{i}^{\omega k_{4}}  \ln_{s}(k-s) d\ln(1+r_1(s)r_{2}(s)), \nonumber \\
\chi_{2}(\zeta,k) =&\; \frac{1}{2\pi i} \int_{\omega k_{4}}^{\omega^{2}k_{2}}  \ln_{s}(k-s) d\ln(1+r_1(s)r_{2}(s)), \nonumber \\
\chi_{3}(\zeta,k) =&\; \frac{1}{2\pi i} \int_{\omega k_{4}}^{\omega^{2}k_{2}}  \ln_{s}(k-s) d\ln(f(s)), \nonumber \\
\tilde{\chi}_{2}(\zeta,k) =&\; \frac{1}{2\pi i} \int_{\omega k_{4}}^{\omega^{2}k_{2}}  \tilde{\ln}_{s}(k-s) d\ln(1+r_1(s)r_{2}(s)), \nonumber \\
\tilde{\chi}_{3}(\zeta,k) =&\; \frac{1}{2\pi i} \int_{\omega k_{4}}^{\omega^{2}k_{2}}  \tilde{\ln}_{s}(k-s) d\ln(f(s)), \nonumber \\
\tilde{\chi}_{4}(\zeta,k) =&\;\frac{1}{2\pi i} \dashint_{\omega^{2}k_{2}}^{\omega}  \tilde{\ln}_{s}(k-s) d\ln(f(s)) \nonumber \\
:= &\; \frac{1}{2\pi i}\lim_{\epsilon \to 0_{+}} \bigg(  \int_{\omega^{2}k_{2}}^{e^{i(\frac{2\pi}{3}-\epsilon)}} \tilde{\ln}_{s}(k-s) d\ln(f(s)) - \tilde{\ln}_{\omega}(k-\omega)\ln(f(e^{i(\frac{2\pi}{3}-\epsilon)})) \bigg), \nonumber \\
\tilde{\chi}_{5}(\zeta,k) = &\; \frac{1}{2\pi i} \dashint_{\omega^{2}k_{2}}^{\omega} \tilde{\ln}_{s}(k-s) d\ln(f(\omega^{2}s)) \nonumber \\
:= &\; \frac{1}{2\pi i}\lim_{\epsilon \to 0_{+}} \bigg(  \int_{\omega^{2}k_{2}}^{e^{i(\frac{2\pi}{3}-\epsilon)}} \tilde{\ln}_{s}(k-s) d\ln(f(\omega^{2}s)) - \tilde{\ln}_{\omega}(k-\omega)\ln(f(e^{-i\epsilon})) \bigg), \label{def of chij}
\end{align}
where the paths follow $\partial \mathbb{D}$ in the counterclockwise direction. For $s \in \{e^{i\theta}\, | \, \theta \in [\frac{\pi}{2},\frac{2\pi}{3}]\}$, $k \mapsto \ln_{s}(k-s)=\ln |k-s|+i \arg_{s}(k-s)$ has a cut along $\{e^{i \theta} \, | \,  \theta \in [\frac{\pi}{2},\arg s]\}\cup(i,i\infty)$ and the branch is such that $\arg_{s}(1)=2\pi$. For $s \in \{e^{i\theta} \,|\, \theta \in [\frac{\pi}{2},\frac{2\pi}{3}]\}$, $k \mapsto \tilde{\ln}_{s}(k-s):=\ln |k-s|+i \tilde{\arg}_{s}(k-s)$ has a cut along $\{e^{i \theta} \,|\, \theta \in [\arg s,\pi]\}\cup(-\infty,0)$, and satisfies $\tilde{\arg}_{s}(1)=0$. Since $f(1)=f(\omega)=0$ by Lemma \ref{inequalitieslemma}, regularized integrals are needed in the definitions of $\tilde{\chi}_{4}$ and $\tilde{\chi}_{5}$. 
The functions $d_{1,0}$ and $d_{2,0}$ are defined by
\begin{align}
& d_{1,0}(\zeta,t) := e^{-\chi_{1}(\zeta,\omega k_{4}) - \tilde{\chi}_{2}(\zeta,\omega k_{4}) + 2\tilde{\chi}_{3}(\zeta,\omega k_{4}) } \nonumber 	\\ 
& \hspace{1.6cm} \times e^{i(\nu_{2}(k_{2})-2\nu_{4}(k_{2}))\tilde{\ln}_{\omega^{2}k_{2}}(\omega k_{4}-\omega^{2}k_{2})} t^{-i \hat{\nu}_{1}(k_{4})} z_{1,\star}^{-2i\hat{\nu}_{1}(k_{4})} \mathcal{D}_{1}(\omega k_{4}),  \label{def of d10} 
	\\
& d_{2,0}(\zeta,t) := e^{-2\chi_{2}(\zeta,\omega^{2} k_{2}) + \chi_{3}(\zeta,\omega^{2} k_{2}) - \tilde{\chi}_{4}(\zeta,\omega^{2} k_{2}) + 2\tilde{\chi}_{5}(\zeta,\omega^{2} k_{2}) } \nonumber 
	\\
& \hspace{1.6cm} \times e^{i(\nu_{3}(k_{4})-2\nu_{1}(k_{4}))\ln_{\omega k_{4}}(\omega^{2} k_{2}-\omega k_{4})} t^{-i \hat{\nu}_{2}(k_{2})} z_{2,\star}^{-2i\hat{\nu}_{2}(k_{2})} \mathcal{D}_{2}(\omega^{2} k_{2}), \label{def of d20}
\end{align}
where $z_{1,\star}$, $z_{2,\star}$ are as in (\ref{z1starz2stardef}), the branches of the complex powers $z_{j,\star}^{\alpha} = e^{\alpha \ln z_{j,\star}}$ are fixed by
\begin{subequations}\label{lnz1starlnz2star}
\begin{align}
& \ln z_{1,\star} = \ln |z_{1,\star}| + i \arg z_{1,\star} = \ln |z_{1,\star}| + i \big( \tfrac{\pi}{2}-\arg (\omega k_{4}) \big), 
	\\
& \ln z_{2,\star} = \ln |z_{2,\star}| + i \arg z_{2,\star} = \ln |z_{2,\star}| + i \big( \tfrac{\pi}{2}-\arg (\omega^{2} k_2) \big),
\end{align}
\end{subequations}
with $ \arg (\omega k_{4}) \in (\tfrac{\pi}{2},\tfrac{2\pi}{3})$ and $\arg (\omega^{2} k_2) \in (\tfrac{\pi}{2},\tfrac{2\pi}{3})$, and $\{\nu_{j}\}_{j=1}^{4}$, $\hat{\nu}_1$, and $\hat{\nu}_2$ are the functions defined in Lemma \ref{inequalitieslemma}.

\subsection{Sector V} 
The derivation in Sector V follows the same pattern but is different from that of Sector IV. The difference stems from the fact that $\omega k_4$ now belongs to $\Gamma_9$ instead of $\Gamma_7$. We provide details of the derivation of the asymptotic formula in Sector V in \cite{CLsectorV}. 
In what follows, we state the explicit expressions for the functions $\tilde{d}_{1,0}$ and $\tilde{d}_{2,0}$ that appear in the final formula for Sector V in Theorem \ref{asymptoticsth}. 
Suppressing the $\zeta$-dependence for clarity, we define 
\begin{align*}
& \tilde{\mathcal{D}}_{1}(k) := \frac{\delta_{1}(\omega^{2} k)\delta_{1}(\frac{1}{\omega k})^{2}\delta_{1}(\omega k)}{\delta_{1}(\frac{1}{k})\delta_{1}(\frac{1}{\omega^{2}k})} \frac{\delta_{2}(k)\delta_{2}(\omega^{2}k)\delta_{2}(\frac{1}{k})^{2}}{\delta_{2}(\omega k)^{2} \delta_{2}(\frac{1}{\omega^{2}k})\delta_{2}(\frac{1}{\omega k})} \frac{\delta_{3}(\omega k) \delta_{3}(\omega^{2} k) \delta_{3}(\frac{1}{\omega k})^{2}}{\delta_{3}(k)^{2}\delta_{3}(\frac{1}{k})\delta_{3}(\frac{1}{\omega^{2} k})} \nonumber \\
& \hspace{1.5cm} \times \frac{\delta_{4}(\omega^{2}k)^{2} \delta_{4}(\frac{1}{k})\delta_{4}(\frac{1}{\omega k})}{\delta_{4}(k)\delta_{4}(\omega k)\delta_{4}(\frac{1}{\omega^{2} k})^{2}} \frac{\delta_{5}(\omega k)^{2} \delta_{5}(\frac{1}{\omega k})\delta_{5}(\frac{1}{\omega^{2}k})}{\delta_{5}(k) \delta_{5}(\frac{1}{k})^{2}\delta_{5}(\omega^{2}k)}, \\
& \tilde{\mathcal{D}}_{2}(k) := \frac{\delta_{1}(\omega^{2} k)^{2}  \delta_{1}(\frac{1}{\omega k})\delta_{1}(\frac{1}{\omega^{2} k})}{\delta_{1}(k)\delta_{1}(\frac{1}{k})^{2}\delta_{1}(\omega k)} \frac{\delta_{2}(\frac{1}{k})\delta_{2}(\frac{1}{\omega k})}{\delta_{2}(\omega k) \delta_{2}(\frac{1}{\omega^{2}k})^{2}\delta_{2}(\omega^{2} k)} \frac{\delta_{3}(\omega^{2} k)^{2}  \delta_{3}(\frac{1}{\omega k})\delta_{3}(\frac{1}{\omega^{2} k})}{\delta_{3}(\frac{1}{k})^{2}\delta_{3}(\omega k)} \nonumber \\
& \hspace{1.5cm} \times \frac{\delta_{4}(\omega^{2}k) \delta_{4}(\frac{1}{\omega k})^{2}}{\delta_{4}(\omega k)^{2}\delta_{4}(\frac{1}{\omega^{2} k})\delta_{4}(\frac{1}{k})} \frac{\delta_{5}(\omega k) \delta_{5}(\frac{1}{\omega^{2}k})^{2}\delta_{5}(\omega^{2} k)}{\delta_{5}(\frac{1}{k}) \delta_{5}(\frac{1}{\omega k})},
\end{align*}
where $\delta_4$ and $\delta_5$ are defined as in (\ref{IVdeltadef}) but $\{\delta_j\}_{j=1}^3$ are now instead given by
\begin{align*}
& \delta_{1}(\zeta, k) = \exp \bigg\{ \frac{1}{2\pi i} \int_{\omega k_{4}}^{i} \frac{\ln (1+r_{1}(\omega^{2}s)r_{2}(\omega^{2}s))}{s - k} ds \bigg\}, 	
	\\
& \delta_{2}(\zeta, k) = \exp \bigg\{ \frac{1}{2\pi i} \int_{i}^{\omega^{2}k_{2}} \frac{\ln (1+r_{1}(s)r_{2}(s))}{s - k} ds \bigg\}, \quad \delta_{3}(\zeta, k) = \exp \bigg\{ \frac{1}{2\pi i} \int_{i}^{\omega^{2}k_{2}} \frac{\ln f(s)}{s - k} ds \bigg\},
\end{align*}
where the paths follow $\partial \D$ in the counterclockwise direction.
We let $\tilde{\chi}_{4}$, $\tilde{\chi}_{5}$ be as in (\ref{def of chij}) and let $\{\chi_{j}\}_{j=1}^3$ be defined in Sector V by
\begin{align*}
& \chi_{1}(\zeta,k) = \frac{1}{2\pi i} \int_{\omega k_{4}}^{i}  \ln_{s}(k-s) d\ln(1+r_1(\omega^{2}s)r_{2}(\omega^{2}s)), \\
& \chi_{2}(\zeta,k) = \frac{1}{2\pi i} \int_{i}^{\omega^{2}k_{2}}  \ln_{s}(k-s) d\ln(1+r_1(s)r_{2}(s)), \\
& \chi_{3}(\zeta,k) = \frac{1}{2\pi i} \int_{i}^{\omega^{2}k_{2}}  \ln_{s}(k-s) d\ln(f(s)), 
\end{align*}
where the paths follow $\partial \D$ in the counterclockwise direction,
and for $s \in \{e^{i\theta}\, | \, \theta \in [\frac{\pi}{3},\frac{2\pi}{3}]\}$, $k \mapsto \ln_{s}(k-s)=\ln |k-s|+i \arg_{s}(k-s)$ has a cut along $\{e^{i \theta} \, | \,  \theta \in [\frac{\pi}{2},\arg s]\}\cup(i,i\infty)$ if $\arg s> \frac{\pi}{2}$, and a cut along $\{e^{i \theta} \, | \,  \theta \in [\arg s,\frac{\pi}{2}]\}\cup(i,i\infty)$ if $\arg s< \frac{\pi}{2}$, and the branch is such that $\arg_{s}(1)=2\pi$. 
The functions $\tilde{d}_{1,0}$ and $\tilde{d}_{2,0}$ appearing in the statement of Theorem \ref{asymptoticsth} are given by
\begin{align}
& \tilde{d}_{1,0}(\zeta,t) = e^{-4\pi\nu_{1}(\omega^{2}k_{4})}e^{2\chi_{1}(\zeta,\omega k_{4})} e^{-2i \nu_{3}(\omega^{2}i) \ln_{i}(\omega k_{4}-i)} t^{-i \nu_{1}(\omega^{2}k_{4})} z_{1,\star}^{-2i\nu_{1}(\omega^{2}k_{4})} \tilde{\mathcal{D}}_{1}(\omega k_{4}), \label{def of dt10} \\
& \tilde{d}_{2,0}(\zeta,t) = e^{-2\chi_{2}(\zeta,\omega^{2} k_{2}) + \chi_{3}(\zeta,\omega^{2} k_{2}) - \tilde{\chi}_{4}(\zeta,\omega^{2} k_{2}) + 2\tilde{\chi}_{5}(\zeta,\omega^{2} k_{2}) } \nonumber \\
& \hspace{1.6cm} \times e^{i\nu_{3}(\omega^{2}i)\ln_{i}(\omega^{2} k_{2}-i)} t^{i (\nu_{4}(k_{2})-\nu_{3}(k_{2})-\nu_{2}(k_{2}))} z_{2,\star}^{2i(\nu_{4}(k_{2})-\nu_{3}(k_{2})-\nu_{2}(k_{2}))} \tilde{\mathcal{D}}_{2}(\omega^{2} k_{2}), \label{def of dt20}
\end{align}
where the functions $\nu_j(k)$ are given by (\ref{nu12345def}), $z_{1,\star}$, $z_{2,\star}$ are as in (\ref{z1starz2stardef}), and the branches for the complex powers $z_{j,\star}^{\alpha} = e^{\alpha \ln z_{j,\star}}$ are fixed by (\ref{lnz1starlnz2star}) with $\arg (\omega k_{4}) \in (\tfrac{\pi}{3},\tfrac{\pi}{2})$ and $\arg (\omega^{2} k_2) \in (\tfrac{\pi}{2},\tfrac{2\pi}{3})$.

\appendix

\section{Model problem for Sector III}\label{IIIappendix}
For each $z_1 \geq 0$ and each $z_2 \geq 0$, let $Y = Y(z_1, z_2)$ denote the contour $Y = \cup_{j=1}^9 Y_j$ oriented as in Figure \ref{Yfig}, where
\begin{align} \nonumber
&Y_1 = \bigl\{iz_2+ re^{\frac{i\pi}{6}}\, \big| \, 0 \leq r < \infty\bigr\}, && Y_2 = \bigl\{iz_2 + re^{\frac{5i\pi}{6}}\, \big| \, 0 \leq r < \infty\bigr\},  
	\\ \nonumber
&Y_3 = \bigl\{-iz_2 + re^{-\frac{5i\pi}{6}}\, \big| \, 0 \leq r < \infty\bigr\}, && Y_4 = \bigl\{-iz_2 + re^{-\frac{i\pi}{6}}\, \big| \, 0 \leq r < \infty\bigr\},
	\\\nonumber
&Y_5 = \bigl\{z_1+ re^{\frac{i\pi}{6}}\, \big| \, 0 \leq r < \infty\bigr\}, && Y_6 = \bigl\{-z_1 + re^{\frac{5i\pi}{6}}\, \big| \, 0 \leq r < \infty\bigr\},  
	\\ \nonumber
&Y_7 = \bigl\{-z_1 + re^{-\frac{5i\pi}{6}}\, \big| \, 0 \leq r < \infty\bigr\}, && Y_8 = \bigl\{z_1 + re^{-\frac{i\pi}{6}}\, \big| \, 0 \leq r < \infty\bigr\}, 	\\\label{YdefIII}
&Y_9 = [-z_1, z_1].
\end{align}
The asymptotics in Sector III is related to the solution $m^Y(y, t, z)$ of the following RH problem:
\begin{align}\label{RHmYIII}
\begin{cases}
m^Y(y,t,\cdot) \in I + \dot{E}^2(\C \setminus Y),\\
m_+^Y(y,t,z) =  m_-^Y(y,t,z) v^Y(y,t,z) \quad \text{for a.e.} \ z \in Y,
\end{cases}
\end{align}
where the jump matrix $v^Y$ is specified below and $\dot{E}^2(\C\setminus Y)$ denotes the Smirnoff class of analytic functions $f:\C\setminus Y \to \C$ such that for each connected component $D$ of $\C\setminus Y$ there are curves $\{\gamma_n\}_1^\infty$ in $D$ that eventually surround each compact subset of $D$ such that $\sup_{n\geq 1} \|f\|_{L^2(\gamma_n)} < \infty$. 
In addition to $(y,t)$, the RH problem (\ref{RHmYIII}) also depends on the parameters $z_1,z_2 \geq 0$ which specify the contour. We assume that $(y,t,z_1,z_2)$ belongs to the parameter subset $\mathcal{P}_T$ of $\R^4$ defined for a given $T>0$ by
\begin{align}\label{parametersetIV}
\mathcal{P}_T = \{(y,t,z_1,z_2) \in \R^4 \, | \, y, z_1, z_2 \in [0,M], \ t \geq T\},
\end{align}
where $M>0$ is a constant.
We will prove that there exists a unique solution of the RH problem (\ref{RHmYIII}) whenever $(y,t,z_1,z_2) \in \mathcal{P}_T$ and $T$ is sufficiently large by relating it to a small-norm problem. We have formulated (\ref{RHmYIII}) in the $L^2$-sense (as opposed to in the classical sense), because it simplifies the existence proof and is sufficient for our purposes. 

On $\cup_{j=1}^8 Y_j$, the jump matrix $v^Y(y,t,z)$ in (\ref{RHmYIII}) has the form
\begin{align}\label{vYdef}
v_j^Y = \begin{cases} 
e^{\beta \hat{\sigma}}(v_{j,0}^Y + v_{j,1}^Y), & j = 1, \dots, 4, 
	\\
e^{\alpha \hat{\tau}}(v_{j,0}^Y + v_{j,1}^Y) + e^{\beta \hat{\sigma}}v_{j,1'}^Y, & j = 5, \dots, 8, 
\end{cases}
\end{align}
where $\alpha(y,t,z)$ and $\beta(y,z)$ are given by (\ref{alphabetadef}), the matrices $\sigma$ and $\tau$ are given by \eqref{def of tau and sigma}, the leading order behavior of the jump matrix is given in terms of a constant $s \in \C$ by
\begin{align}\label{vYj0def}
& v_{1,0}^Y = \begin{pmatrix}
1 & 0 & 0 \\
0 & 1 & 0 \\
-1 & 0 & 1
\end{pmatrix},
\quad v_{2,0}^Y = \begin{pmatrix}
1 & 0 & 0 \\
0 & 1 & 0 \\
1 & 0 & 1
\end{pmatrix},
\quad
v_{3,0}^Y= \begin{pmatrix}
1 & 0 & 1 \\
0 & 1 & 0 \\
0 & 0 & 1
\end{pmatrix},
	\\\nonumber
& v_{4,0}^Y = \begin{pmatrix}
1 & 0 & -1 \\
0 & 1 & 0 \\
0 & 0 & 1
\end{pmatrix},
\quad v_{5,0}^Y = v_{7,0}^Y = \begin{pmatrix}
1 & s & 0 \\
0 & 1 & 0 \\
0 & -s & 1
\end{pmatrix}, \quad
v_{6,0}^Y
= v_{8,0}^Y = I,
\end{align}
and the subleading behavior is given by  
\begin{align}\nonumber
& v_{1,1}^Y = \begin{pmatrix}
0 & 0 & 0 \\
0 & 0 & 0 \\
-p_1 & 0 & 0
\end{pmatrix},
\quad v_{2,1}^Y = \begin{pmatrix}
0 & 0 & 0 \\
0 & 0 & 0 \\
p_1 & 0 & 0
\end{pmatrix},
\quad
v_{3,1}^Y= \begin{pmatrix}
0 & 0 & p_2 \\
0 & 0 & 0 \\
0 & 0 & 0
\end{pmatrix},
	\\ \label{vYj1def}
& v_{4,1}^Y = \begin{pmatrix}
0 & 0 & -p_2 \\
0 & 0 & 0 \\
0 & 0 & 0
\end{pmatrix},
\quad v_{5,1}^Y = \begin{pmatrix}
0 & -isyz + q_1 & 0 \\
0 & 0 & 0 \\
0 & -isyz + q_2 & 0
\end{pmatrix}, \quad
v_{6,1}^Y = \begin{pmatrix}
0 & 0 & 0 \\
q_3 & 0 & q_4 \\
0 & 0 & 0
\end{pmatrix},
	\\ \nonumber
& v_{7,1}^Y =  \begin{pmatrix}
0 & -isyz + q_5 & 0 \\
0 & 0 & 0 \\
0 & -isyz + q_6  & 0
\end{pmatrix},
\quad v_{8,1}^Y = \begin{pmatrix}
0 & 0 & 0 \\
q_7 & 0 & q_8 \\
0 & 0 & 0
\end{pmatrix},
\end{align}
and
\begin{align}\nonumber
v_{5,1'}^Y = v_{8,1'}^Y = 0, \quad 
v_{6,1'}^Y = \begin{pmatrix}
0 & 0 & 0 \\
0 & 0 & 0 \\
\hat{p}_1 & 0 & 0
\end{pmatrix}, \quad
v_{7,1'}^Y =  \begin{pmatrix}
0 & 0 & \hat{p}_2 \\
0 & 0 & 0 \\
0 & 0  & 0
\end{pmatrix},
\end{align}
with the functions $p_k = p_k(t,z)$, $k = 1, 2$, $q_k = q_k(y,t,z)$, $j = 1, \dots, 8$, and $\hat{p}_k = \hat{p}_k(t,z)$, $k = 1, 2$, specified below. On $Y_9$, the jump matrix $v^Y$ is given by
\begin{align}\label{v9Ydef}
v_9^Y = \begin{cases}
(v_8^Y)_-^{-1}(v_5^Y)_+, & z > 0, \\
(v_7^Y)_-(v_6^Y)_+^{-1}, & z < 0,
\end{cases}
\end{align}
where $(v_5^Y)_+$ denotes the boundary values on $\R$ from the upper half-plane of the analytic continuation of $v_5^Y$; $(v_6^Y)_+^{-1}$ denotes the boundary values on $\R$ from the upper half-plane of the analytic continuation of $(v_6^Y)^{-1}$, etc. 

The next lemma contains all the information we need about $m^Y$.

\begin{figure}
\begin{center}
\begin{tikzpicture}[master,scale=0.7]
\node at (0,0) {};
\draw[black,line width=0.55 mm,->-=0.6] (-1,0)--(1,0);

\draw[black,line width=0.55 mm,->-=0.55] (1,0)--($(1,0)+(30:5)$);
\draw[black,line width=0.55 mm,->-=0.55] (-1,0)--($(-1,0)+(150:5)$);
\draw[black,line width=0.55 mm,-<-=0.45] (-1,0)--($(-1,0)+(-150:5)$);
\draw[black,line width=0.55 mm,-<-=0.45] (1,0)--($(1,0)+(-30:5)$);

\draw[black,line width=0.55 mm,->-=0.55] (0,1.5)--($(0,1.5)+(30:4)$);
\draw[black,line width=0.55 mm,->-=0.55] (0,1.5)--($(0,1.5)+(150:4)$);
\draw[black,line width=0.55 mm,-<-=0.45] (0,-1.5)--($(0,-1.5)+(-30:4)$);
\draw[black,line width=0.55 mm,-<-=0.45] (0,-1.5)--($(0,-1.5)+(-150:4)$);

\node at (1.6,2.8) {\small $1$};
\node at (-1.6,2.8) {\small $2$};
\node at (-1.63,-2.9) {\small $3$};
\node at (1.63,-2.9) {\small $4$};

\node at (3,1.6) {\small $5$};
\node at (-3,1.6) {\small $6$};
\node at (-3,-1.6) {\small $7$};
\node at (3.03,-1.6) {\small $8$};

\node at (0,0.32) {\small $9$};

\node at (0.9,-0.3) {\small $z_1$};
\node at (-0.85,-0.3) {\small $-z_1$};
\node at (0,1.25) {\small $iz_2$};
\node at (0,-1.25) {\small $-iz_2$};

\end{tikzpicture}
\end{center}
\begin{figuretext}
\label{Yfig}
The contour $Y$ defined in (\ref{YdefIII}). 
\end{figuretext}
\end{figure}

\begin{lemma}[Model problem for Sector III]\label{YlemmaIII}
Let $s \in \C$ and suppose that there is an integer $n \geq 0$ such that
\begin{align}\label{psumIII}
& p_k(t,z) = \sum_{j=1}^n \sum_{l = j}^{n} (a_{jl}^{(k)} \ln_0(zt^{-1/3}) + b_{jl}^{(k)}) \frac{z^l}{t^{j/3}}, \qquad k = 1, 2,
	\\\label{qsumIII}
& q_k(y,t,z) = \sum_{j=0}^n \sum_{\substack{l = 1 \\2j + l \geq 2}}^{n} \sum_{m = 0}^{n} (c_{jlm}^{(k)} \ln_0(zt^{-1/3}) + d_{jlm}^{(k)}) \frac{y^m z^l}{t^{j/3}}, \qquad k = 1, \dots, 8,
	\\\label{phatsumIII}
& \hat{p}_k(t,z) = \hat{b}_{11}^{(k)} z t^{-1/3}
+ \sum_{j=2}^n \sum_{l = 2}^{n}  (\hat{a}_{jl}^{(k)} \ln_0(zt^{-1/3}) + \hat{b}_{jl}^{(k)}) \frac{z^l}{t^{j/3}}, \qquad k = 1,2,	
\end{align}
for some complex coefficients $a_{jl}^{(k)}, b_{jl}^{(k)}, c_{jlm}^{(k)}, d_{jlm}^{(k)}, \hat{a}_{jl}^{(k)}, \hat{b}_{jl}^{(k)}\in \C$ with $c_{02m}^{(k)}=0$ for  $m = 0,\dots,n$ and $k = 1,\dots,8$, where $\ln_0$ has a branch cut along $(0,\infty)$ and $\im \ln_0 z \in [0,2\pi)$.

\begin{enumerate}[$(a)$]
\item There is a $T \geq 1$ such that the RH problem (\ref{RHmYIII}) with jump matrix $v^Y$ given by (\ref{vYdef}) has a unique solution $m^Y$ whenever $(y,t,z_1,z_2) \in \mathcal{P}_T$. 

\item For each integer $N \geq 1$, the following asymptotic expansion holds as $ z \to \infty$:
\begin{align}\nonumber
  m^Y(y,t,z) = &\; I + \frac{m_1^P(y)}{z} +  \frac{m^P(y,0)m_1^W(\tilde{y})m^P(y,0)^{-1}}{\frac{2^{1/3}}{3^{1/12}} t^{1/6} z}
+ \frac{m_2^P(y)}{z^2}
	\\ \nonumber
& + \frac{\hat{m}_{1, \ln}^Y(y,\tilde{y})\ln t + \hat{m}_1^Y(y,\tilde{y})}{z t^{1/3}} 
+ \frac{\hat{m}_{2, \ln}^Y(y,\tilde{y})\ln t}{z t^{1/2}} 
	\\\label{mYasymptoticsIII}
& + O\bigg(\frac{1}{z^3}\bigg) + O\bigg(\frac{1}{z t^{1/2}}\bigg) + O\bigg(\frac{\ln t}{z^2 t^{1/3}}\bigg)
\end{align}
uniformly for $(y,t,z_1,z_2) \in \mathcal{P}_T$ and $\arg z \in [0, 2\pi]$, where $\tilde{y}$ is given in terms of $(y,t)$ by (\ref{tildeywdef}), $m_1^P(y)$ is given by (\ref{m1Pexpression}), $m^P(y,0)$ is given by (\ref{mPat0expression}), $m_1^W(\tilde{y})$ is given by (\ref{m1Wexpression}), $m_2^P(y)$ is given by (\ref{mPasymptotics}), and $\hat{m}_1^Y$, $\hat{m}_{1,\ln}^Y$, $\hat{m}_{2,\ln}^Y$ are smooth functions of $y \in [0,\infty)$ and $\tilde{y} \in (-1,\infty)$ given by (\ref{hatm1Ym2Y}).

\item \label{YlemmaIIIitemd}
The expansion (\ref{mYasymptoticsIII}) can be differentiated termwise with respect to $y$ any finite number of times without increasing the error terms. 

\item $m^Y$ obeys the bound
\begin{align}\label{mYboundedIV}
\sup_{(y,t,z_1,z_2) \in \mathcal{P}_T} \sup_{z \in \C\setminus Y} (|m^Y(y, t, z)| + |\partial_y m^Y(y, t, z)|)  < \infty.
\end{align}

\end{enumerate}
\end{lemma}

The remainder of this appendix is devoted to the proof of Lemma \ref{YlemmaIII}.

\subsection{Outline of proof}\label{outline steepest model RHP transition}
Let us begin by outlining the main ideas of the proof. 
We first note that using (\ref{v9Ydef}) to perform an easy contour deformation, we may (and will) henceforth assume that $z_1 = 0$ and $z_2 = 1$. Fixing $z_1 = 0$ and $z_2 = 1$, we write $Y = P \cup W$, where (see Figure \ref{PWfig})
$$P = \cup_{j=1}^4 Y_j, \qquad W = \cup_{j=5}^8 Y_j.$$
The leading order behavior of $m^Y$ as $t \to \infty$ is described by the solution of the RH problem obtained by setting the subleading terms $v_{j,1}^Y$ in the jump matrix (\ref{vYdef}) to zero. Although this leading order problem is much simpler than the original RH problem for $m^Y$, it is still too complicated to be solved exactly. The main difficulty is that the jump matrix involves one exponent, $\beta$, on part of the contour and a different exponent, $\alpha$, on another part of the contour. These exponents naturally involve different scales. While the variables $y$ and $z$ are adapted to make $\beta = - i (yz + \frac{4z^3}{3})$ match the exponent featured in the Painlev\'e II RH problem, they are not adapted to $\alpha$. We will therefore analyze the leading order RH problem in two stages. First, we will consider the RH problem obtained by (i) ignoring the jumps on $W$ and (ii) keeping only the leading order term on $P$. The solution of this RH problem, which is essentially the RH problem associated to Painlev\'e II, will be denoted by $m^P$. Second, we will consider the RH problem obtained by (i) ignoring the jumps on $P$ and (ii) keeping only the leading order term on $W$.  The solution of this RH problem, which we will solve exactly in terms of the error function, will be denoted by $m^W$. To analyze the RH problem for $m^W$, we will introduce new variables $\tilde{y}$ and $w$ whose time-dependence is such that $\alpha$ takes the simple form $\alpha = i(1+\tilde{y})w^2$. 
Finally, a leading order approximation of $m^Y$ will be constructed by combining $m^P$ and $m^W$ in a suitable nonlinear way. After quotienting out this leading order approximation from $m^Y$, a small-norm problem will be obtained which can be expanded in inverse powers and logarithms of $t$.

\begin{figure}
\begin{center}
\begin{tikzpicture}[scale=0.7]
\node at (0,0) {};

\draw[black,line width=0.55 mm,->-=0.55] (0,0)--(30:5);
\draw[black,line width=0.55 mm,->-=0.55] (0,0)--(150:5);
\draw[black,line width=0.55 mm,-<-=0.45] (0,0)--(-150:5);
\draw[black,line width=0.55 mm,-<-=0.45] (0,0)--(-30:5);

\draw[black,line width=0.55 mm,->-=0.55] (0,1.5)--($(0,1.5)+(30:4)$);
\draw[black,line width=0.55 mm,->-=0.55] (0,1.5)--($(0,1.5)+(150:4)$);
\draw[black,line width=0.55 mm,-<-=0.45] (0,-1.5)--($(0,-1.5)+(-30:4)$);
\draw[black,line width=0.55 mm,-<-=0.45] (0,-1.5)--($(0,-1.5)+(-150:4)$);

\node at (1.7,2.85) {\small $1$};
\node at (-1.7,2.85) {\small $2$};
\node at (-1.7,-2.9) {\small $3$};
\node at (1.7,-2.9) {\small $4$};

\node at (2.3,0.95) {\small $5$};
\node at (-2.3,0.95) {\small $6$};
\node at (-2.3,-0.95) {\small $7$};
\node at (2.3,-0.95) {\small $8$};

\node at (0,-0.3) {\small $0$};

\node at (0,1.2) {\small $i$};
\node at (0,-1.27) {\small $-i$};

\end{tikzpicture}

\end{center}
\begin{figuretext}
\label{PWfig}
When $z_1 = 0$ and $z_2 = 1$, the contour $Y$ is the union of the contours $P =\cup_{j=1}^4 Y_j$ and $W=\cup_{j=5}^8 Y_j$. 
\end{figuretext}
\end{figure}

\subsection{Painlev\'e II model problem}
Consider the subcontour $P=\cup_{j=1}^4 Y_j$ of $Y$, see Figure \ref{Pfig}. Define the jump matrix $v^P(y,z)$ by
\begin{align}\nonumber
& v_1^P = 
e^{- i (yz + \frac{4z^3}{3}) \hat{\sigma}}\begin{pmatrix}
1 & 0 & 0 \\ 0 & 1 & 0 \\
-1 & 0 & 1 \end{pmatrix},
&&
 v_2^P = 
e^{- i (yz + \frac{4z^3}{3}) \hat{\sigma}}\begin{pmatrix}
1 & 0 & 0 \\ 0 & 1 & 0 \\ 1 & 0 & 1 \end{pmatrix},
	\\\label{vPdef}
& v_3^P = 
e^{- i (yz + \frac{4z^3}{3}) \hat{\sigma}}\begin{pmatrix}
1 & 0 & 1 \\ 0 & 1 & 0 \\ 0 & 0 & 1 \end{pmatrix},
 && v_4^P = 
e^{- i (yz + \frac{4z^3}{3}) \hat{\sigma}}\begin{pmatrix}
1 & 0 & -1 \\ 0 & 1 & 0 \\
0 & 0 & 1 \end{pmatrix},
\end{align}
where $v_j^P$ denotes the restriction of $v^P$ to $Y_j$. 

\begin{RHproblem}[RH problem for $m^P$]\label{RHmP}
Find a $3 \times 3$-matrix valued function $m^P(y, z)$ with the following properties:
\begin{enumerate}[$(a)$]
\item $m^P(y,\cdot) : \C \setminus \cup_{j=1}^4 Y_j \to \mathbb{C}^{3 \times 3}$ is analytic.

\item $m^P(y,\cdot)$ has continuous boundary values on $\cup_{j=1}^4 Y_j \setminus \{\pm i\}$ satisfying the jump relation
\begin{align*}
  m^P_+(y,z) = m^P_-(y,z) v^P(y,z), \qquad z \in \cup_{j=1}^4 Y_j \setminus \{\pm i\}.
\end{align*}

\item $m^P(y,z) = I + O(z^{-1})$ as $z \to \infty$ and $m^P(y,z) = O(1)$ as $z \to \pm i$.

\end{enumerate}
\end{RHproblem}

\begin{figure}
\begin{center}
\begin{tikzpicture}[scale=0.7]
\node at (0,0) {};

\draw[black,line width=0.55 mm,->-=0.55] (0,1)--($(0,1)+(30:4)$);
\draw[black,line width=0.55 mm,->-=0.55] (0,1)--($(0,1)+(150:4)$);
\draw[black,line width=0.55 mm,-<-=0.45] (0,-1)--($(0,-1)+(-30:4)$);
\draw[black,line width=0.55 mm,-<-=0.45] (0,-1)--($(0,-1)+(-150:4)$);

\node at (1.7,2.35) {\small $1$};
\node at (-1.7,2.35) {\small $2$};
\node at (-1.7,-2.4) {\small $3$};
\node at (1.7,-2.4) {\small $4$};

\node at (0,0.7) {\small $i$};
\node at (0,-0.7) {\small $-i$};

\end{tikzpicture}
\end{center}
\begin{figuretext}
\label{Pfig}
The jump contour $P = \cup_{j=1}^4 Y_j$ for the RH problem for $m^P$. 
\end{figuretext}
\end{figure}

\begin{lemma}[The solution $m^P$]\label{mPlemma}
For each $y \in \R$, RH problem \ref{RHmP} has a unique solution $m^P(y,z)$ with the following properties:

\begin{enumerate}[$(a)$]
\item\label{mPitema} 
There are smooth functions $\{m_j^P(y)\}_{j=1}^\infty$ of $y \in \R$ such that, for each integer $N \geq 0$,
\begin{align}\label{mPasymptotics}
m^P(y, z) = I + \sum_{j=1}^N \frac{m_j^P(y)}{z^j} + O(z^{-N-1}), \qquad z \to \infty,
\end{align}
uniformly for $y$ in compact subsets of $\R$ and for $\arg z \in [0,2\pi]$, and (\ref{mPasymptotics}) can be differentiated any finite number of times with respect to $y$ and $z$ without increasing the error term. 

\item\label{mPitemb}
The leading term is given by
\begin{align}\label{m1Pexpression}
m_1^{P}(y) = \begin{pmatrix} \frac{i}{2} \int_{+\infty}^y u_{\HM}(y')^2 dy'
& 0 & i\frac{u_{\HM}(y)}{2} \\ 
0 & 0 & 0 \\
-i\frac{u_{\HM}(y)}{2} & 0 & -\frac{i}{2}\int_{+\infty}^y u_{\HM}(y')^2 dy'  \end{pmatrix},
\end{align}
where $u_{\HM}$ is the Hasting--McLeod solution of Painlev\'e II, i.e., $u_{\HM}$ is the unique solution of (\ref{painleveII}) satisfying (\ref{uHMasymptotics}).

\item $m^P$ obeys the symmetry
\begin{align}\label{mPsymm}
m^P(y,z) = \Sigma m^P(y,-z) \Sigma, \qquad \Sigma := \begin{pmatrix}0 & 0 & 1 \\ 0 & 1 & 0 \\ 1 & 0 & 0 \end{pmatrix}.
\end{align}

\item At $z = 0$, we have
\begin{align}\label{mPat0expression}
m^P(y,0) = &\; \begin{pmatrix} \cosh{U_{\HM}(y)} & 0 & - \sinh{U_{\HM}(y)} \\
0 & 1 & 0 \\
-\sinh{U_{\HM}(y)}  & 0 & \cosh{U_{\HM}(y)} 
\end{pmatrix},
	\\  \nonumber
\partial_zm^P(y,0) = & -2 i \left(u_{\HM}^2 \cosh{U_{\HM}} 
- u_{\HM}' \sinh{U_{\HM}}\right)\begin{pmatrix} 1 & 0 & 0   \\
0 & 0 & 0 \\
0 & 0 & -1 
\end{pmatrix}
	\\ \label{mPprimeat0expression}
& + 2 i \left(\left(u_{\HM}^2+y\right) \sinh{U_{\HM}} 
- u_{\HM}' \cosh{U_{\HM}} \right) \begin{pmatrix} 0 & 0 & 1   \\
0 & 0 & 0 \\
-1 & 0 & 0 
\end{pmatrix},
\end{align}
where $U_{\HM}(y) := \int_{+\infty}^y u_{\HM}(y') dy'$.

\end{enumerate}

\end{lemma}
\begin{proof}
Letting
\begin{align}\label{mHMmP}
m^{\HM}(y, z) = e^{-\frac{\pi i}{4}\sigma}m^P(y, z) e^{\frac{\pi i}{4}\sigma},
\end{align}
and restricting attention to the submatrix consisting of the first and third row/column, we see that $m^{\HM}$ satisfies the RH problem associated with the Hastings--McLeod solution of the Painlev\'e II equation (see e.g. Proposition 5.2 \& Theorem 11.7 of \cite{FIKN2006} with $s_1 = -s_4 = -i$, $s_2 = s_{5} = 0$, $s_3 = -s_6 = i$, $\alpha = 0$, and $\sigma = 1$).
We conclude that RH problem \ref{RHmP} has a unique solution for each $y \in \R$, and that this solution  satisfies assertion ($\ref{mPitema}$) with $m_{1,13}^{P}(y) = i\frac{u_{\HM}(y)}{2}$.
The symmetry (\ref{mPsymm}) follows from the analogous symmetry of $v^P$ and the uniqueness of the solution of RH problem \ref{RHmP}. 
The symmetry \eqref{mPsymm} implies that $m_{1,33}^P=-m_{1,11}^P$ and $m_{1,31}^P=-m_{1,13}^P$.

Let us derive the diagonal entries in (\ref{m1Pexpression}).
The jumps of the function
$$\Psi(y,z) := m^P(y, z) e^{- i (yz + \frac{4z^3}{3}) \sigma},$$
are independent of $y,z$; hence standard arguments imply that $\Psi$ satisfies the Lax pair equations
$$\Psi_z + 4iz^2 \sigma\Psi = (A_1 z + A_0)\Psi, \qquad \Psi_y + iz\sigma \Psi = U_0 \Psi,$$
where
\begin{align} \nonumber
& A_1 := 8i m_{1,13}^{P}(y) \begin{pmatrix} 0 & 0 & 1 \\ 0 & 0 & 0 \\ 1 & 0 & 0 \end{pmatrix}, \quad A_0 := - i(y - 8m_{1,13}^{P}(y)^2) \sigma - 4 \partial_y m_{1,13}^{P}(y) \begin{pmatrix} 0 & 0 & 1 \\ 0 & 0 & 0 \\ -1 & 0 & 0 \end{pmatrix},
	\\ \label{A1A0U0}
& U_0 := 2im_{1,13}^{P}(y) \begin{pmatrix} 0 & 0 & 1 \\ 0 & 0 & 0 \\ 1 & 0 & 0 \end{pmatrix}.
\end{align}
Thus
\begin{align}\label{mPlax}
\begin{cases}
\partial_z m^P + i(y + 4z^2)[\sigma, m^P] = (A_1 z + A_0 + iy\sigma )m^P, \\
\partial_y m^P + i z [\sigma, m^P] = U_0m^P.
\end{cases}
\end{align}
The $(11)$-term of $O(z^{-1})$ of the second equation in (\ref{mPlax}) gives $\partial_y m_{1,11}^P = -2im_{1,13}^{2}$.
Using that $m_{1,13}^{P}(y) = i\frac{u_{\HM}(y)}{2}$, we obtain
$$m_1^{P}(y) = \begin{pmatrix} \frac{i}{2} \int_a^y u_{\HM}(y')^2 dy' & 0 & i\frac{u_{\HM}(y)}{2} \\ 
0 & 1 & 0 \\
-i\frac{u_{\HM}(y)}{2} & 0 &  -\frac{i}{2}\int_a^y u_{\HM}(y')^2 dy'  \end{pmatrix},$$
where $a$ is an integration constant. To fix $a$, we note that as $y \to +\infty$, the solution $m^P$ tends to $I$ exponentially fast. This shows that $a = +\infty$ and proves (\ref{m1Pexpression}).

We finally derive (\ref{mPat0expression}) and (\ref{mPprimeat0expression}). Information on the behavior of the solution of the Painlev\'e II RH problem near the origin can be found in \cite[Section 5.0]{FIKN2006}, but the  expression given in \cite{FIKN2006} for the solution evaluated at the origin involves an unknown constant. We therefore present a direct proof. The symmetry (\ref{mPsymm}) implies that
$$m^P(y,0) = \begin{pmatrix} m_{11}^P(y,0) & 0 & m_{13}^P(y,0) \\
0 & 1 & 0 \\
m_{13}^P(y,0) & 0 & m_{11}^P(y,0)
\end{pmatrix}.$$
The second equation in (\ref{mPlax}) now gives
$$\partial_ym_{11}^P(y,0) = -u_{\HM}(y) m_{13}^P(y,0), \qquad
\partial_ym_{13}^P(y,0) = -u_{\HM}(y) m_{11}^P(y,0),$$
and hence
$$m_{11}^P(y,0) = a_1 e^{U_{\HM}(y)} + a_2 e^{-U_{\HM}(y)},\quad m_{13}^P(y,0) = -a_1 e^{U_{\HM}(y)} + a_2 e^{-U_{\HM}(y)},$$
for some constants $a_1$ and $a_2$. 
We saw above that $m^P(y,0) \to I$ as $y \to +\infty$; thus we must have $a_1 + a_2 = 1$ and $a_1 - a_2 = 0$, which yields (\ref{mPat0expression}).
For $z = 0$, the first equation in (\ref{mPlax}) reduces to
$$\partial_z m^P(y,0) - iy m^P(y,0)\sigma =  A_0 m^P(y,0).$$
Inserting the expression (\ref{A1A0U0}) for $A_0$ and the expression (\ref{mPat0expression}) for $m^P(y,0)$, we obtain (\ref{mPprimeat0expression}). 
 \end{proof}

\subsection{Error function model problem}\label{errorsubsec}
We expect the leading order behavior of $m^Y$ as $t \to \infty$ to be described by the solution $m^P$ constructed in Lemma \ref{mPlemma}. This suggests that we consider the function $\tilde{m}^Y := m^Y (m^P)^{-1}$ with jump matrix $\tilde{v}^Y = m_-^P v^Y (m_+^P)^{-1}$. To leading order, $\tilde{v}^Y$ equals the identity matrix on $P = \cup_{j=1}^{4} Y_j$. Moreover, the off-diagonal entries of the leading order terms of $\tilde{v}^Y$ on $W$ are suppressed by exponentials of the form $e^{\pm 3 \alpha} \sim e^{\pm c i t^{1/3}z^2}$. This implies that the $L^1$ and $L^2$ norms of $\tilde{v}^Y - I$ on $Y$ decay as $t \to \infty$. However, the RH problem for $\tilde{m}^Y$ is not a small-norm problem, because $|\tilde{v}^Y - I|$ does not tend to zero at $z = 0$ so the $L^\infty$-norm of $\tilde{v}^Y - I$ is not small for large $t$. We therefore need to solve yet another model problem. 

To understand what this second model problem should be, we note that, to leading order, $\tilde{v}^Y_{6} \approx \tilde{v}^Y_{8} \approx I$ while $\tilde{v}^Y_5$ and $\tilde{v}^Y_7$ are approximated by
\begin{align*} 
m_-^P(y,z)
\left(e^{\alpha \hat{\tau}}
\begin{pmatrix}
1 & s & 0 \\ 
0 & 1 & 0 \\
0 & -s & 1 \end{pmatrix} \right) (m_+^P(y,z))^{-1}.
\end{align*}
The matrices $m_\pm^P$ in the above expression are evaluated at $z$, but since the jumps for nonzero $z$ are suppressed by factors of $e^{-c t^{1/3}|z|^2}$ as $t \to \infty$, and since $m^P$ is independent of $t$, we expect that we can replace $m_\pm^P(y,z)$ with $m_\pm^P(y,0)$ and still obtain a good approximation. The matrices  $m_\pm^P(y,0)$ can then be removed by considering $m^P(y,0)^{-1}\tilde{m}^Y m^P(y,0)$ instead of $\tilde{m}^Y$. This leads us to consider the following model problem. 

Define the jump matrix $v^W$ on the contour $Y_5 \cup Y_7 = e^{\frac{\pi i}{6}}\R$ displayed in Figure \ref{Wfig} by
\begin{align}\nonumber
& v^W = 
e^{\alpha \hat{\tau}}
\begin{pmatrix}
1 & s & 0 \\ 
0 & 1 & 0 \\
0 & -s & 1 \end{pmatrix},
\end{align}
where $\alpha = \alpha(y,t,z)$ is given by (\ref{alphabetadef}). The expression for $\alpha$ in terms of $y$ and $z$ is not convenient because it involves $t$ explicitly. We therefore introduce new variables $\tilde{y}$ and $w$ by
\begin{align}\label{tildeywdef}
\tilde{y} := \frac{1}{3} \Big(\frac{3}{2}\Big)^{1/3}t^{-2/3}y, \qquad w := \frac{2^{1/3}}{3^{1/12}} t^{1/6} z.
\end{align}
In terms of these variables, $\alpha$ is given by the simple expression
$$\alpha = i(1+\tilde{y})w^2.$$
It is convenient to view $m^W$ as a function of $(\tilde{y}, w)$.

\begin{RHproblem}[RH problem for $m^W$]\label{RHmW}
Find a $3 \times 3$-matrix valued function $m^W(\tilde{y}, w)$ with the following properties:
\begin{enumerate}[$(a)$]
\item $m^W(\tilde{y},\cdot) : \C \setminus e^{\frac{\pi i}{6}}\R \to \mathbb{C}^{3 \times 3}$ is analytic.

\item $m^W(y,\cdot)$ has continuous boundary values on $W$ satisfying the jump relation
\begin{align*}
  m^W_+(\tilde{y}, w) = m^W_-(\tilde{y}, w) v^W(\tilde{y}, w), \qquad w \in e^{\frac{\pi i}{6}}\R.
\end{align*}

\item $m^W(\tilde{y}, w) = I + O(w^{-1})$ as $w \to \infty$.

\end{enumerate}
\end{RHproblem}

\begin{figure}
\begin{center}
\begin{tikzpicture}[master, scale=0.7]
\node at (0,0) {};

\draw[black,line width=0.55 mm,->-=0.55] (-150:3)--(30:3);

\end{tikzpicture}
\end{center}
\begin{figuretext}
\label{Wfig}
The contour $e^{\frac{\pi i}{6}}\R$ relevant for the RH problem for $m^W$. 
\end{figuretext}
\end{figure}

\begin{lemma}[The solution $m^W$]\label{mWlemma}
For each $\tilde{y} > -1$, RH problem \ref{RHmW} has a unique solution $m^W(\tilde{y},w)$ with the following properties:

\begin{enumerate}[$(a)$]
\item $m^W$ is given explicitly in terms of the error function by
\begin{align}\label{mWexplicit}
m^W(\tilde{y}, w) = 
\begin{cases}
\begin{pmatrix}
 1 & \frac{s}{2} e^{3iw^2 (1 + \tilde{y})} \left(\erf\left(e^{\frac{\pi i}{4}}  w \sqrt{3} \sqrt{1+\tilde{y}}\right) - 1\right) & 0 \\
 0 & 1 & 0 \\
 0 & -\frac{s}{2} e^{3iw^2 (1 + \tilde{y})} \left(\erf\left(e^{\frac{\pi i}{4}}  w \sqrt{3} \sqrt{1+\tilde{y}}\right) - 1\right) & 1 
   \end{pmatrix}, & \arg w \in (-\frac{5\pi}{6}, \frac{\pi}{6}),
   	\\
\begin{pmatrix}
 1 & \frac{s}{2} e^{3iw^2 (1 + \tilde{y})} \left(\erf\left(e^{\frac{\pi i}{4}}  w \sqrt{3} \sqrt{1+\tilde{y}}\right) + 1\right) & 0 \\
 0 & 1 & 0 \\
 0 & -\frac{s}{2}e^{3iw^2 (1 + \tilde{y})} \left(\erf\left(e^{\frac{\pi i}{4}}  w \sqrt{3} \sqrt{1+\tilde{y}}\right) + 1\right) & 1 \\
   \end{pmatrix}, & \arg w \in (\frac{\pi}{6}, \frac{7\pi}{6}).
\end{cases}
\end{align}

\item
For each integer $N \geq 0$,
\begin{align}\label{mWasymptotics}
m^W(\tilde{y}, w) = I + \sum_{j=0}^N \frac{m_{2j+1}^W(\tilde{y})}{w^{2j+1}} + O(w^{-2N-3}), \qquad w \to \infty,
\end{align}
uniformly for $\arg w \in [0,2\pi]$ and for $\tilde{y}$ in compact subsets of $(-1,\infty)$, where the leading coefficient is given by
\begin{align}\label{m1Wexpression}
m_1^W(\tilde{y}) = \frac{s e^{\frac{3\pi i}{4}}}{\sqrt{12\pi}\sqrt{1+\tilde{y}}}  \begin{pmatrix} 
 0 & 1 & 0 \\
 0 & 0 & 0 \\
 0 & -1 & 0 
\end{pmatrix}
\end{align}
and, more generally,
$$m_{2j+1}^W(\tilde{y}) = - \frac{s}{2\sqrt{\pi}} \frac{(\frac{1}{2}-1)(\frac{1}{2}-2) \cdots (\frac{1}{2}-j)}{e^{\frac{\pi i}{4}(2j+1)} (3(1+\tilde{y}))^{j+\frac{1}{2}}}
\begin{pmatrix} 
 0 &  1 & 0 \\
 0 & 0 & 0 \\
 0 & -1 & 0 
\end{pmatrix}, \qquad j \geq 0.$$
Moreover, (\ref{mWasymptotics}) can be differentiated any finite number of times with respect to $\tilde{y}$ and $w$ without increasing the error term. 
 
\item $m^W$ obeys the symmetry
\begin{align}\label{mWsymm}
m^W(\tilde{y},w) = \Sigma m^W(\tilde{y},-w) \Sigma, \qquad \Sigma := \begin{pmatrix}0 & 0 & 1 \\ 0 & 1 & 0 \\ 1 & 0 & 0 \end{pmatrix}.
\end{align}

\end{enumerate}
\end{lemma}
\begin{proof}
It readily follows from RH problem \ref{RHmW} that the first and third columns of $m^{W}$ are constant and equal to the first and third columns of the $3\times 3$ identity matrix, respectively. The second column can then be obtained by an application of Cauchy's formula, and we obtain
\begin{align}
m^{W}(w) = \begin{pmatrix}
1 & \frac{1}{2\pi i}\int_{e^{-\frac{5\pi i}{6}}\infty}^{e^{\frac{\pi i}{6}}\infty} \frac{s e^{3i(1+\tilde{y})x^{2}}}{x-w}dx & 0 \\
0 & 1 & 0 \\
0 & \frac{1}{2\pi i}\int_{e^{-\frac{5\pi i}{6}}\infty}^{e^{\frac{\pi i}{6}}\infty} \frac{-s e^{3i(1+\tilde{y})x^{2}}}{x-w}dx & 1
\end{pmatrix}.
\end{align}
Using the following integral representation for the error function (see \cite[Eq. (7.7.2)]{NIST}): 
$$e^{-z^2}(1 \mp \erf(-iz)) = \pm \frac{1}{\pi i}\int_{-\infty}^\infty \frac{e^{-t^2} dt}{t-z}, \qquad \im z \gtrless 0,$$
we find \eqref{mWexplicit}.  
Furthermore, for each $\delta > 0$, the error function $\erf{z}$ satisfies the asymptotic formula
\begin{align}\label{erfasymptotics}
\erf{z} & \sim 1- \frac{e^{-z^2}}{\sqrt{\pi}} \sum_{j=0}^\infty \frac{(\frac{1}{2}-1)(\frac{1}{2}-2) \cdots (\frac{1}{2}-j)}{z^{2j+1}} 
= 1- \frac{e^{-z^2}}{\sqrt{\pi}} \bigg(\frac{1}{z} - \frac{1}{2z^3} + O(z^{-5})\bigg)
\end{align}
as $z \to \infty$ uniformly for $\arg z \in [-\frac{3\pi}{4} + \delta, \frac{3\pi}{4} - \delta]$.
Using (\ref{mWexplicit}), the asymptotics (\ref{erfasymptotics}) for $\erf z$, and the relation $\erf(-z) = -\erf(z)$, we find the asymptotic expansion (\ref{mWasymptotics}) of $m^W$ as $w \to \infty$ and the symmetry (\ref{mWsymm}).
\end{proof}

\subsection{The approximate solution $M^Y$}
As explained above, we assume without loss of generality that $z_1 = 0$ and $z_2 = 1$, so that $Y = P \cup W$ where $P = \cup_{j=1}^4 Y_j$ and $W = \cup_{j=5}^8 Y_j$, see Figure \ref{PWfig}. We define $M^Y(y, t, z)$ for $z \in \C \setminus Y$ by
\begin{align}\label{m0Ydef}
M^Y(y,t,z) = m^P(y,0)m^W(\tilde{y}, w) m^P(y,0)^{-1} m^P(y,z),
\end{align}
where $\tilde{y}$ and $w$ are given in terms of $y,t,z$ by (\ref{tildeywdef}). The discussion at the beginning of Subsection \ref{errorsubsec} suggests that $m^Y = \tilde{m}^Y m^P$ is well approximated by $M^Y$ for large $t$. In the following, we make this precise. 

According to Lemma \ref{mPlemma} and Lemma \ref{mWlemma}, the function $M^Y$ defined in (\ref{m0Ydef}) is a well-defined analytic function $\C \setminus Y \to \C^{3 \times 3}$ whose boundary values satisfy $M_+^Y = M_-^Y V^Y$ on $Y$, where the jump matrix $V^Y$ is given by
$$V_j^Y = \begin{cases} 
v_j^P, &  j = 1, \dots, 4, 
	\\
m^P(y,z)^{-1} m^P(y,0) v_j^W(\tilde{y}, w) m^P(y,0)^{-1} m^P(y,z), &  j = 5, 7, 
	\\
I, &  j = 6, 8, 
\end{cases}
$$
where $V_j^Y$ denotes the restriction of $V^Y$ to $Y_j$.
In other words,
\begin{align}\label{VjYexpression}
V_j^Y = \begin{cases} 
e^{\beta \hat{\sigma}} v_{j,0}^Y, &  j = 1, \dots, 4, 
	\\
m^P(y,z)^{-1} m^P(y,0) (e^{\alpha \hat{\tau}} v_{j,0}^Y) m^P(y,0)^{-1} m^P(y,z), &  j = 5, 7,
	\\
e^{\alpha \hat{\tau}} v_{j,0}^Y, &  j = 6,8,  
\end{cases}
\end{align}
where $v_{j,0}^Y$ are the matrices in (\ref{vYj0def}) and $m^P(y,0)$ is given by (\ref{mPat0expression}). 
As a consequence of Lemma \ref{mPlemma} and Lemma \ref{mWlemma}, the solution $M^Y$ and its derivatives are bounded:
\begin{align}\label{MYbounded}
\sup_{y, \tilde{y} \in [0,M]} \sup_{z \in \C\setminus Y}  |\partial_y^j \partial_{\tilde{y}}^l M^Y(y, \tilde{y}, z)|  < \infty, \qquad j,l = 0,1, \dots.
\end{align}
Substituting the expansions (\ref{mPasymptotics}) and (\ref{mWasymptotics}) into (\ref{m0Ydef}), we infer that, as $z \to \infty$,
\begin{align}\nonumber
& M^Y(y,t,z) = \bigg(I + \sum_{j=0}^N \frac{m^P(y,0)m_{2j+1}^W(\tilde{y})m^P(y,0)^{-1}}{(\frac{2^{1/3}}{3^{1/12}})^{2j+1}t^{\frac{2j+1}{6}} z^{2j+1}} + O(t^{-\frac{2N+3}{6}} z^{-2N-3})\bigg) 
	\\\nonumber
&\hspace{2.1cm} \times \bigg(I + \sum_{j=1}^N \frac{m_j^P(y)}{z^j} + O(z^{-N-1})\bigg)
	\\ \label{MYasymptotics}
& = I + \frac{m_1^P(y)}{z} +  \frac{m^P(y,0)m_1^W(\tilde{y})m^P(y,0)^{-1}}{\frac{2^{1/3}}{3^{1/12}} t^{1/6} z}
+ \frac{m_2^P(y)}{z^2} 
+ \sum_{j=3}^N \sum_{k=0}^j \frac{M_{jk}^Y(y, \tilde{y})}{z^j t^{k/6}}
 + O\bigg(\frac{1}{z^{N+1}}\bigg),
\end{align}
uniformly with respect to $\arg z \in [0, 2\pi]$, $t \geq 2$, and $y$ in compact subsets of $[0, \infty)$, where the coefficients $M_{jk}^Y(y, \tilde{y})$ depend smoothly on $y$ and $\tilde{y}$.

\subsection{The small-norm solution $\hat{m}^Y$}
The function $m^Y$ satisfies the RH problem (\ref{RHmYIII}) if and only if $\hat{m}^Y = m^Y (M^Y)^{-1}$ satisfies the RH problem
\begin{align}\label{RHmYhatIV}
\begin{cases}
\hat{m}^Y(y,t,\cdot) \in I + \dot{E}^2(\C \setminus Y),\\
\hat{m}_+^Y(y,t,z) =  \hat{m}_-^Y(y, t, z) \hat{v}^Y(y, t, z) \quad \text{for a.e.} \ z \in Y,
\end{cases}
\end{align}
where $\hat{v}^Y = M_-^Y v^Y(M_+^Y)^{-1}$. Letting $\hat{w}^Y := \hat{v}^Y -I$, we have 
\begin{align}\label{hatwYMvVM}
\hat{w}^Y = M_-^Y (v^Y - V^Y)(M_+^Y)^{-1}.
\end{align}
Using (\ref{vYdef}) and (\ref{VjYexpression}), we can write
\begin{align}\label{vYminusVY}
v^Y - V^Y = \begin{cases} 
e^{\beta \hat{\sigma}} v_{j,1}^Y, &  j = 1, \dots, 4, 
	\\
E^W + e^{\alpha \hat{\tau}} v_{j,1}^Y + e^{\beta \hat{\sigma}} v_{j,1'}^Y, &  j = 5, 7,
	\\
e^{\alpha \hat{\tau}} v_{j,1}^Y + e^{\beta \hat{\sigma}} v_{j,1'}^Y, &  j = 6,8,  
\end{cases}
\end{align}
where $E^W = E^W(y,t,z)$ is given for $z \in Y_5 \cup Y_7$ by
\begin{align}\label{EWdef}
E^W := e^{\alpha \hat{\tau}} v_{j,0}^Y - m^P(y,z)^{-1} m^P(y,0) (e^{\alpha \hat{\tau}} v_{j,0}^Y) m^P(y,0)^{-1} m^P(y,z), \qquad j = 5,7.
\end{align}

In the following two lemmas, we use (\ref{hatwYMvVM}) and (\ref{vYminusVY}) to derive large $t$ expansions of $\hat{w}^Y$ for $z$ in $P$ and $W$, respectively. 

\begin{lemma}\label{whatPlemma}
There exist smooth functions 
$$\hat{w}_{j,\ln}^P, \hat{w}_{j}^P: [0,\infty) \times (-1, \infty) \to L^2(P), \qquad j = 2,3, \dots,$$ 
such that the following hold:

\begin{enumerate}[$(a)$]

\item\label{whatPlemmaitema}
 For any $N \geq 1$, $\hat{w}^Y$ satisfies
\begin{align}\label{wYexpansionP}
\hat{w}^Y(y, t, z) = 
\sum_{j=2}^N \frac{\hat{w}_{j,\ln}^P(y,\tilde{y},z)\ln{t} + \hat{w}_j^P(y,\tilde{y},z)}{t^{j/6}} + O\bigg( \frac{e^{-c|z|^3} \ln{t}}{t^{\frac{N+1}{6}}}\bigg), \qquad z \in P,
\end{align}
uniformly for $y$ in compact subsets of $[0, \infty)$, $t \geq 2$, and $z \in P$.

\item\label{whatPlemmaitemb}
The coefficients $\hat{w}_{j,\ln}^P$ and $\hat{w}_j^P$ obey the estimates
\begin{align*}
\begin{cases}
 |\hat{w}_{j,\ln}^P(y,\tilde{y},z)| \leq Ce^{-c|z|^3}, \\
 |\hat{w}_j^P(y,\tilde{y},z)| \leq Ce^{-c|z|^3}, 
 \end{cases} \quad j = 2, 3,\dots, N,
\end{align*}
uniformly for $z \in P$, $y$ in compact subsets of $[0, \infty)$, and $t \geq 2$.

\item The leading coefficients are given by
\begin{subequations}\label{wYexpansionPleadingcoeff}
\begin{align}\label{wYexpansionPleadingcoeffa}
& \resizebox{0.89\hsize}{!}{$\hat{w}_{2,\ln}^P\ln{t} + \hat{w}_2^P
= \begin{cases}
\sum_{l = 1}^{n} (a_{1l}^{(1)} \ln_0(zt^{-1/3}) + b_{1l}^{(1)}) z^l 
\big( I- m_-^P(y,z) m_+^P(y,z)^{-1}\big), & z \in Y_1 \cup Y_2,
	\\
\sum_{l = 1}^{n} (a_{1l}^{(2)} \ln_0(zt^{-1/3}) + b_{1l}^{(2)}) z^l 
\big( I- m_-^P(y,z) m_+^P(y,z)^{-1}\big), & z \in Y_3 \cup Y_4,
\end{cases}$}
	\\\label{wYexpansionPleadingcoeffb}
& \resizebox{0.89\hsize}{!}{$\hat{w}_{3,\ln}^P\ln{t} + \hat{w}_3^P
= \begin{cases}
 \sum_{l = 1}^{n} (a_{1l}^{(1)} \ln_0(zt^{-1/3}) + b_{1l}^{(1)}) z^l  \big[m_-^P(y,z) m_+^P(y,z)^{-1}, B_1\big], & z \in Y_1 \cup Y_2,
 	\\
 \sum_{l = 1}^{n} (a_{1l}^{(2)} \ln_0(zt^{-1/3}) + b_{1l}^{(2)}) z^l  \big[m_-^P(y,z) m_+^P(y,z)^{-1}, B_1\big], & z \in Y_3 \cup Y_4,	
 \end{cases}$}
\end{align}
\end{subequations}
where $B_1$ is short-hand notation for
\begin{align}\label{A1def}
B_1 := \frac{3^{1/12}}{2^{1/3}z} m^P(y,0) m_1^W(\tilde{y}) m^P(y,0)^{-1}
= \frac{s e^{\frac{3\pi i}{4}} e^{U_{\HM}(y)}}{\frac{2^{1/3}}{3^{1/12}} z \sqrt{12 \pi} \sqrt{1 + \tilde{y}}} \begin{pmatrix} 0 & 1 & 0 \\
0 & 0 & 0 \\
0 & -1 & 0 \end{pmatrix}.
\end{align}

\end{enumerate}
\end{lemma}
\begin{proof}
For $z \in Y_1$, we have as a consequence of (\ref{hatwYMvVM}), (\ref{vYminusVY}), (\ref{vYj1def}), and (\ref{psumIII}) that
\begin{align}\nonumber
\hat{w}^Y(y, t, z) & =
M_-^Y (e^{\beta \hat{\sigma}} v_{1,1}^Y)(M_+^Y)^{-1}
= p_1 M_-^Y(v_1^P - I)
(M_+^Y)^{-1}
	\\ \label{hatwYonY1}
& =
 \sum_{j=1}^n \sum_{l = j}^{n} (a_{jl}^{(1)} \ln_0(zt^{-1/3}) + b_{jl}^{(1)}) \frac{z^l e^{2i (yz + \frac{4z^3}{3})}}{t^{j/3}}
M_-^Y \begin{pmatrix}
0 & 0 & 0 \\
0 & 0 & 0 \\
-1 & 0 & 0
\end{pmatrix}(M_+^Y)^{-1}.
\end{align}
On $Y_1$, we have $|z|\geq1$, and hence $|w| \geq ct^{1/6}$. Substituting the expansion (\ref{mWasymptotics}) into (\ref{m0Ydef}), we infer that
$$M^Y_{\pm}(y,t,z) = m^P_{\pm}(y,z) + \sum_{j=0}^N \frac{M_{2j+1,\pm}^Y(y,\tilde{y}, z)}{t^{\frac{2j+1}{6}}} + O(t^{-\frac{2N+3}{6}}), \qquad t \to \infty,$$
uniformly for $z \in Y_1$ and $y$ in compact subsets of $[0,\infty)$, where the matrices $M_{2j+1,\pm}^Y(y,\tilde{y}, z)$ depend smoothly on $y$ and $\tilde{y}$ and are uniformly bounded for $z \in Y_1$.
In light of the estimate 
\begin{align}\label{eonY}
|e^{\pm 2i(y z + \frac{4z^3}{3})}|
\leq Ce^{-c|z|^3}, \qquad y \in [0,M], \;\;  z \in P,
\end{align}
where the plus (minus) sign applies for $z \in Y_1 \cup Y_2$ ($z \in Y_3 \cup Y_4$), the existence of coefficients $\hat{w}_{j,\ln}^Y$ and $\hat{w}_{j}^Y$ with the desired properties on $Y_1$ follows from (\ref{hatwYonY1}). This completes the proof of assertions ($\ref{whatPlemmaitema}$) and ($\ref{whatPlemmaitemb}$) for $z \in Y_1$; the proofs for $z \in Y_2 \cup Y_3 \cup Y_4$ are similar. 

To compute the leading terms, we use (\ref{mWasymptotics}), \eqref{m0Ydef}, and the fact that  $w = \frac{2^{1/3}}{3^{1/12}} t^{1/6} z$ to write
\begin{align*}
M^Y_{\pm}(y,t,z) & = m^P(y,0)\Big(I + \frac{m_1^W(\tilde{y})}{w} + O(w^{-3})\Big)m^P(y,0)^{-1} m^P_{\pm}(y,z)
	\\
& = \left(I
+ \frac{B_1}{t^{1/6}}\right)m_{\pm}^P(y,z)
+ O(t^{-1/2}), \qquad z \in P,
\end{align*}
where $B_1$ is given by (\ref{A1def}).
Using that $m_-^P (v_j^P - I)(m_+^P)^{-1} =  I- m_-^P(m_+^P)^{-1}$,
we obtain, on $Y_1$,
\begin{align*}
 \hat{w}^Y = p_1 M_-^Y (v_1^P - I)(M_+^Y)^{-1}
= &\;
\left( \sum_{j=1}^n \sum_{l = j}^{n} (a_{jl}^{(1)} \ln_0(zt^{-1/3}) + b_{jl}^{(1)}) \frac{z^l}{t^{j/3}}\right)
	\\
& \times 
 \left(I+ \frac{B_1}{t^{1/6}} \right)( I- m_-^P(m_+^P)^{-1})
\left(I- \frac{B_1}{t^{1/6}}\right)+ O(t^{-5/6}\ln t),
\end{align*}
from which (\ref{wYexpansionPleadingcoeff}) follows for $z \in Y_1$; the proofs for $z \in Y_2 \cup Y_3 \cup Y_4$ are analogous.
\end{proof}

We next consider the analog of (\ref{wYexpansionP}) on $W$. On $W$, the variable $w = \frac{2^{1/3}}{3^{1/12}} t^{1/6} z$ appears naturally. 

\begin{lemma}\label{whatWlemma}
There exist smooth functions 
$$\hat{w}_{j,\ln}^W, \hat{w}_{j}^W: [0,\infty) \times (-1, \infty) \to L^2(W), \qquad j = 1,2, \dots,$$ 
such that the following hold:

\begin{enumerate}[$(a)$]

\item\label{whatWlemmaitema}
 For any $N \geq 1$, $\hat{w}^Y$ satisfies, for $z \in W$,
\begin{align}\label{wYexpansionW}
\hat{w}^Y(y, t, z) = &\;\hat{w}_E(y, t, z) + \begin{cases}
\sum_{j=1}^N \frac{\hat{w}_{j,\ln}^W(y,\tilde{y},w)\ln{t} + \hat{w}_j^W(y,\tilde{y},w)}{t^{j/6}} + O\big( \frac{e^{-c|w|^2} \ln{t}}{t^{\frac{N+1}{6}}}\big), & |z| \leq t^{-\frac{1}{12}},
	\\
O\big(e^{-c|w|^2} e^{-c t^{1/6}}\big), & |z| > t^{-\frac{1}{12}},	
\end{cases}
\end{align}
uniformly for $y$ in compact subsets of $[0, \infty)$, $t \geq 2$, and $z \in W$, where $\tilde{y}, w$ are given in terms of $y,t,z$ by (\ref{tildeywdef}) and $\hat{w}_E$ is defined by
\begin{align}\label{whatEdef}
\hat{w}_E(y,t,z) :=  \begin{cases} 
M^Y\begin{pmatrix} 0 & 0 & 0 \\0 & 0 & 0 \\ \hat{p}_1 e^{-2\beta} & 0 & 0 \end{pmatrix} (M^Y)^{-1}  & \text{on $Y_6$},
	\\
M_-^Y \begin{pmatrix}
0 & 0 & \hat{p}_2 e^{2\beta}  \\
0 & 0 & 0 \\
0 & 0  & 0
\end{pmatrix} (M_+^Y)^{-1} & \text{on $Y_7$},
	\\
0 & \text{on $Y_5 \cup Y_8$}.
\end{cases}
\end{align}

\item\label{whatWlemmaitemb}
The coefficients $\hat{w}_{j,\ln}^W$ and $\hat{w}_j^W$ obey the estimates
\begin{align}\label{hatwjWestimates}
\begin{cases}
 |\hat{w}_{j,\ln}^W(y,\tilde{y},w)| \leq Ce^{-c|w|^2}, \\
 |\hat{w}_j^W(y,\tilde{y},w)| \leq Ce^{-c|w|^2}, 
 \end{cases} \quad j = 1, \dots, N,
\end{align}
uniformly for $w \in W$, $y$ in compact subsets of $[0, \infty)$, and $t \geq 2$.

\item\label{whatWlemmaitemc}
The function $\hat{w}_E$ obeys the estimate
\begin{align}\label{whatEestimate}
  | \hat{w}_E(y, t, z)| \leq C e^{-c|z|^3} t^{-1/3}
\end{align}
uniformly for $z \in W$, $y$ in compact subsets of $[0, \infty)$, and $t \geq 2$.

\item The leading coefficients are given by
\begin{align} \label{wYexpansionWleadingcoeff}
\hat{w}_{1,\ln}^W = &\; \hat{w}_{2,\ln}^W=0 \quad \text{on $W$}, 
	\\\nonumber
\hat{w}_1^W = &\; \begin{cases} -i \frac{3^{1/12} s w (-2 u_{\HM}'(y)+2 u_{\HM}(y)^2+y)}{2^{1/3}}  e^{3 i(1+\tilde{y})w^2 +U_{\HM}(y)} \begin{pmatrix} 0 & 1 & 0 \\0 & 0 & 0 \\ 0 & 1 & 0 \end{pmatrix}  & \text{on $Y_5 \cup Y_7$},
	\\
0 & \text{on $Y_6 \cup Y_8$}.
\end{cases}
\end{align}
\end{enumerate}

\end{lemma}
\begin{proof}
Note that $\hat{w}_E(y,t,z) = M_-^Y(e^{\beta \hat{\sigma}} v_{j,1'}^Y)(M_+^Y)^{-1}$ for $z \in Y_j$, $j = 5, \dots, 8$.
Consequently, by \eqref{hatwYMvVM}--(\ref{vYminusVY}),
\begin{align}\label{hatwY}
\hat{w}^Y  = \begin{cases} 
M_-^Y (E^W + e^{\alpha \hat{\tau}} v_{j,1}^Y)(M_+^Y)^{-1} + \hat{w}_E, &  j = 5, 7,
	\\
M_-^Y (e^{\alpha \hat{\tau}} v_{j,1}^Y)(M_+^Y)^{-1} + \hat{w}_E, &  j = 6,8,
\end{cases}
\end{align}
where $E^W$ is defined in (\ref{EWdef}). Recall that $\tilde{y} \sim t^{-2/3}y$ and $w \sim t^{1/6} z$ and note that $\tilde{y} \to 0$ as $t \to \infty$. 
Since
$$e^{\alpha \hat{\tau}} v_{j,0}^Y = I + e^{3\alpha}\begin{pmatrix}
0 & s & 0 \\
0 & 0 & 0 \\
0 & -s & 0
\end{pmatrix}, \qquad j = 5,7,$$
we can write, for $z \in Y_5 \cup Y_7$,
$$E^W =  e^{3\alpha}\left\{\begin{pmatrix}
0 & s & 0 \\
0 & 0 & 0 \\
0 & -s & 0
\end{pmatrix} - m^P(y,z)^{-1} m^P(y,0)  \begin{pmatrix}
0 & s & 0 \\
0 & 0 & 0 \\
0 & -s & 0
\end{pmatrix} m^P(y,0)^{-1} m^P(y,z)\right\}.$$
Using that $e^{3\alpha} = e^{3i(1+\tilde{y})w^2}$, we find that $E^W = O(e^{-2c |w|^2}) = O(e^{-c |w|^2}e^{-c t^{1/6}})$ on the part of $Y_5 \cup Y_7$ where $|z| > t^{-1/12}$ (i.e. where $|w| \gtrsim t^{1/12}$). 
Similarly, since their entries involve the exponentials $e^{\pm 3\alpha}$, the matrices $e^{\alpha \hat{\tau}} v_{j,1}^Y$ are $O(e^{-c |w|^2}e^{-c t^{1/6}})$ on the part of $W = \cup_{j=5}^8 Y_j$ where $|z| > t^{-1/12}$. 
Since the matrix $M^Y$ and its inverse are bounded by (\ref{MYbounded}), the estimate (\ref{wYexpansionW}) follows in the case of $|z| > t^{-1/12}$.

Let us establish (\ref{wYexpansionW}) for $|z| \leq t^{-1/12}$.
On the part of $Y_5 \cup Y_7$ where $|z| \leq t^{-1/12}$, we have $z \to 0$ so we can Taylor expand $m^P(y,z)^{-1}$ and $m^P(y,z)$ around $z=0$ in the expression for $E_W$. Since $z = \frac{3^{1/12}}{2^{1/3}} t^{-1/6} w$, we deduce that there exist matrices $\{E_k^W(y)\}_{k=1}^\infty$ independent of $z$ and $t$ such that, for any $N \geq 1$,
$$E^W = \sum_{k=1}^N E_k^W(y) \frac{w^k}{t^{k/6}} e^{3\alpha}+ O(e^{-c |w|^2}t^{-\frac{N+1}{6}}), \qquad  t \to \infty, \ |z| \leq t^{-1/12},$$
uniformly for $y$ in compact subsets of $[0, \infty)$ and $z \in Y_5 \cup Y_7$.
Moreover, expressing (\ref{qsumIII}) in terms of $w$, we obtain
\begin{align}\label{qkexpansion}
q_k(y, t,z) = \sum_{j=0}^n \sum_{\substack{l = 1 \\2j + l \geq 2}}^{n} (q_{jl,\ln}^{(k)}(y,w) \ln{t} + q_{jl}^{(k)}(y,w)) \frac{1}{t^{l/6}t^{j/3}}, \qquad k = 1, \dots, 8,
\end{align}
where the functions $q_{jl,\ln}^{(k)}(y,w)$ and $q_{jl}^{(k)}(y,w)$ depend smoothly on $y$ and are bounded by $C|\ln(w)||w|^l$ for $w \in W$, and where $q_{02,\ln}^{(k)}=0$ for $k = 1,\dots,8$. It follows that $e^{\alpha \hat{\tau}} v_{j,1}^Y$ has an expansion of the form
\begin{align*}
e^{\alpha \hat{\tau}} v_{j,1}^Y(y,t,z) = &
\begin{cases}
e^{3\alpha} \frac{3^{1/12} w}{2^{1/3} t^{1/6}}\begin{pmatrix} 0 & -isy & 0 \\
0 & 0 & 0 \\
0 & -isy & 0 \end{pmatrix}, & j = 5,7, \\
0, & j = 6,8,	
\end{cases}
	\\
& + \sum_{j=2}^N \frac{a_{j,\ln}(y,\tilde{y},w)\ln{t} + a_j(y,\tilde{y},w)}{t^{j/6}} 
+ O\bigg( \frac{e^{-c|w|^2} \ln{t}}{t^{\frac{N+1}{6}}}\bigg), \qquad t \to \infty,
\end{align*}
uniformly for $y$ in compact subsets of $[0, \infty)$ and $z \in Y_j$, $j = 5, \dots 8$, where the coefficients $a_{j,\ln}(y,\tilde{y},w)$ and $a_{j}(y,\tilde{y},w)$ depend smoothly on $y$ and $\tilde{y}$, are bounded by $C|w \ln(w)|e^{-c|w|^2}$ for $w \in W$, and $a_{2,\ln}=0$.
It remains to consider the functions $M_\pm^Y$ in (\ref{hatwY}). Taylor expanding $m^P(y,z)$ around $z=0$ in (\ref{m0Ydef}), we infer that
$$M^Y(y,t,z) = \sum_{k=0}^N \frac{M_k^Y(y,\tilde{y}, w) w^k}{t^{k/6}} + O(t^{-\frac{N+1}{6}}), \qquad t \to \infty, \ z \in \C \setminus W, \ |z| \leq t^{-1/12},$$
uniformly for $y$ in compact subsets of $[0, \infty)$ and $z$ in the given range, where the matrices $M_k^Y(y,\tilde{y}, w)$ depend smoothly on $y$ and $\tilde{y}$.
Substituting the above expansions into (\ref{hatwY}), we obtain an expansion of $\hat{w}^Y$ valid for $|z| \leq t^{-1/12}$. Defining $\hat{w}_{j,\ln}^W$ and $\hat{w}_j^W$ to be the coefficients of this expansion, we see that assertions ($\ref{whatWlemmaitema}$)--($\ref{whatWlemmaitemb}$) of the lemma hold.
Assertion $(\ref{whatWlemmaitemc})$ follows from the definition (\ref{whatEdef}) of $\hat{w}_E$, the expressions (\ref{phatsumIII}) for $\hat{p}_1$ and $\hat{p}_2$, as well as the fact that $\beta = - i (yz + \frac{4z^3}{3})$.

The coefficients $\hat{w}_{1,\ln}^W$ and $\hat{w}_1^W$ vanish on $Y_6 \cup Y_8$. Let us find expressions for these coefficients for $w \in Y_5 \cup Y_7$. Suppressing the dependence on $y$ and $\tilde{y}$ for brevity, we find that, as $z \to 0$,
\begin{align*}
E^W = &\; e^{\alpha \hat{\tau}} v_{j,0}^Y - m^P(z)^{-1} m^P(0) (e^{\alpha \hat{\tau}} v_{j,0}^Y) m^P(0)^{-1} m^P(z)
	\\
= &\; e^{\alpha \hat{\tau}} v_{j,0}^Y - \Big(m^P(0) + z\partial_z m^P(0)\Big)^{-1} m^P(0) 
 (e^{\alpha \hat{\tau}} v_{j,0}^Y) m^P(0)^{-1} \Big(m^P(0) + z\partial_z m^P(0)\Big)
+ O(z^2)
	\\
= &\; e^{\alpha \hat{\tau}} v_{j,0}^Y 
- \Big(I - zm^P(0)^{-1}\partial_z m^P(0) \Big) 
 (e^{\alpha \hat{\tau}} v_{j,0}^Y) \Big(I + zm^P(0)^{-1}\partial_z m^P(0)\Big)
+ O(z^2)
	\\
= &\; z E_1^W + O(z^2), \qquad z \in Y_5 \cup Y_{7},
\end{align*}
where $E_1^W := \big[m^P(0)^{-1} \partial_{z}m^P(0), (e^{\alpha \hat{\tau}} v_{j,0}^Y) \big]$.
Furthermore, by \eqref{m0Ydef}, as $z \to 0$,
\begin{align*}
M^Y  = m^P(0)m^W(w) m^P(0)^{-1} m^P(z) = m^P(0)m^W(w) + O(z), \qquad z \in \C \setminus W.
\end{align*}
Hence, on $Y_5 \cup Y_7$, we find that
\begin{align*}
 \hat{w}_{1,\ln}^W\ln{t} + \hat{w}_1^W
= &\; \frac{3^{1/12}}{2^{1/3}} w
 m^P(0)m_-^W(w) \bigg\{E_1^W
+ e^{\alpha \hat{\tau}} \begin{pmatrix}
0 & -isy & 0 \\
0 & 0 & 0 \\
0 & -isy & 0
\end{pmatrix}\bigg\} m_+^W(w)^{-1}m^P(0)^{-1}.
\end{align*}
Long but straightforward calculations using (\ref{mPprimeat0expression}) now yield the expressions in (\ref{wYexpansionWleadingcoeff}).
\end{proof}

By combining Lemma \ref{whatPlemma} and Lemma \ref{whatWlemma}, we obtain estimates for $\hat{w}^Y$ on $Y = P \cup W$.

\begin{lemma}\label{whatYlemma}
For any $1 \leq p \leq \infty$ and $k = 0, 1, \dots$, the function $\hat{w}^Y$ satisfies
$$\|z^k \hat{w}^Y(y,t,z)\|_{L_z^p(P)} \leq C\frac{\ln t}{t^{\frac{1}{3}}},
\quad
\|z^k \hat{w}^Y(y,t,z)\|_{L_z^p(W)} \leq C\frac{1}{t^{\frac{1}{6} + \frac{1}{6p}}},$$
uniformly for $y$ in compact subsets of $[0, \infty)$ and $t \geq 2$.
\end{lemma}
\begin{proof}
The estimate of the $L^p$-norm on $P$ follows immediately from Lemma \ref{whatPlemma}. Using that $z = \frac{3^{1/12}}{2^{1/3}} t^{-1/6} w$, Lemma \ref{whatWlemma} implies that 
$$\|z^k \hat{w}^Y(y,t,z)\|_{L_z^1(W)} \leq \frac{C}{t^{1/3}}, \qquad 
\|z^k \hat{w}^Y(y,t,z)\|_{L_z^\infty(W)} \leq \frac{C}{t^{1/6}},$$
and then the estimate of the $L^p$-norm on $W$ follows from the inequality $\|f\|_{L^p} \leq \|f\|_{L^\infty}^{1 - \frac{1}{p}} \|f\|_{L^1}^{\frac{1}{p}}$.
\end{proof}
For a function $h:Y\to \mathbb{C}$, we define $\mathcal{C}^{Y}h$ and $\mathcal{C}^{Y}_{\hat{w}^{Y}}h$ by
\begin{align*}
& (\mathcal{C}^{Y}h)(k) = \frac{1}{2\pi i}\int_{Y}\frac{h(k')dk'}{k'-k}, \quad k \in \mathbb{C}\setminus Y, & & \mathcal{C}^{Y}_{\hat{w}^{Y}}h = \mathcal{C}^{Y}_{-}(h \hat{w}^{Y}).
\end{align*}
We conclude from Lemma \ref{whatYlemma} that $\hat{w}^Y(y,t,\cdot)$ lies in $L^2(Y)$ and that $\|\hat{w}^Y\|_{L^\infty(Y)} \leq C(\ln t) t^{-1/6}$ uniformly for $y$ in compact subsets of $[0, \infty)$ and for $t \geq 2$.
In particular, using the boundedness of $\mathcal{C}^Y_\pm$ as operators on $L^2(Y)$ 
there exists a $T \geq 1$ such that 
$\|\mathcal{C}^Y_{\hat{w}^Y}\|_{\mathcal{B}(L^2(Y))} \leq C \|\hat{w}^Y\|_{L^\infty(Y)}  \leq 1/2$ for $t \geq T$.
It follows that the operator $I - \mathcal{C}_{\hat{w}^Y}^{Y} \in \mathcal{B}(L^2(Y))$ is invertible for $t \geq T$. By standard arguments,
we conclude that the RH problem (\ref{RHmYhatIV}) has a unique solution $\hat{m}^Y \in I + \dot{E}^2(\C \setminus Y)$ whenever $(y,t,z_1,z_2) \in \mathcal{P}_T$. This solution is given by 
\begin{align}\label{mYrepresentationIV}
\hat{m}^Y(y, t, z) = I + \mathcal{C}^Y(\hat{\mu}^Y \hat{w}^Y) = I + \frac{1}{2\pi i}\int_{Y} (\hat{\mu}^Y \hat{w}^Y)(y, t, z') \frac{dz'}{z' - z},
\end{align}
where $\hat{\mu}^Y(x, t, \cdot) \in I + L^2(Y)$ is defined by $\hat{\mu}^Y = I + (I - \mathcal{C}^Y_{\hat{w}^Y})^{-1}\mathcal{C}^Y_{\hat{w}^Y}I$.


\begin{lemma}\label{hatmuYlemma}
As $t\to \infty$, the function $\hat{\mu}^Y$ satisfies
\begin{align}\label{muYleadingterm}
\|\hat{\mu}^Y(y, t,z) - I\|_{L^2(Y)} = \begin{cases} 
O(t^{-1/3}\ln{t}), & z \in P,
	 \\
O(t^{-1/4}), & z \in W,
\end{cases}
\end{align}
uniformly for $y$ in compact subsets of $[0, \infty)$ and $z$ in the given range. 
\end{lemma}
\begin{proof}
For $t \geq T$, we have
\begin{align*}
\|(I - \mathcal{C}^Y_{\hat{w}})^{-1}\|_{\mathcal{B}(L^2(Y))} 
& \leq \sum_{j=0}^\infty \|\mathcal{C}^Y_{\hat{w}}\|_{\mathcal{B}(L^2(Y))}^j
 = \frac{1}{1 - \|\mathcal{C}^Y_{\hat{w}}\|_{\mathcal{B}(L^2(Y))}} \leq \frac{1}{1 - 1/2} = 2,
\end{align*}
and so $\|\hat{\mu}^Y - I\|_{L^2(Y)} \leq 2 \|\mathcal{C}^Y_{\hat{w}^Y}I\|_{L^2(Y)} \leq C\|\hat{w}^Y\|_{L^2(Y)}$. Thus (\ref{muYleadingterm}) follows from Lemma \ref{whatYlemma}.
\end{proof}

We will see that the leading contribution from $\hat{w}_E$ to the large $t$ expansion of $\hat{m}^Y$ is
$$-\frac{1}{2\pi i z} \int_{Y_6 \cup Y_7} \hat{w}_E(y,t,z) dz.$$
Our next lemma estimates this integral.

\begin{lemma}\label{XElemma}
The function $\hat{w}_E$ defined in (\ref{whatEdef}) satisfies
$$-\frac{1}{2\pi i} \int_{Y_6 \cup Y_7} \hat{w}_E(y,t,z) dz
= -\frac{X_E(y)}{2\pi i  t^{1/3}} + O(t^{-1/2}), \qquad t\to \infty,$$
uniformly for $y$ in compact subsets of $[0, \infty)$, where $X_E: [0, \infty) \to \C$ is a smooth function. 
\end{lemma}
\begin{proof}
Using the definition (\ref{m0Ydef}) of $M^Y$, the expression (\ref{mWexplicit}) for $m^W$, and the expression (\ref{mPat0expression}) for $m^P(y,0)$, we can write, for $z \in Y_6$,
\begin{align*}
\hat{w}_E(y,t,z) = &\; m^P(y,z) \begin{pmatrix} 0 & 0 & 0 \\0 & 0 & 0 \\ \hat{p}_1 e^{-2\beta} & 0 & 0 \end{pmatrix} m^P(y,z)^{-1} 
	\\
& - e^{-2\beta} m_{12}^W(\tilde{y}, w) \hat{p}_1(m_{13}^P + m_{33}^P)(y,z) e^{U_{\HM}(y)} \begin{pmatrix} 0 & m_{13}^P(y,z) & 0 \\0 & 0 & 0 \\ 0 & m_{33}^P(y,z) & 0 \end{pmatrix},
\end{align*}
and, for $z \in Y_7$,
\begin{align*}
\hat{w}_E(y,t,z) = &\; m^P(y,z) \begin{pmatrix} 0 & 0 & \hat{p}_2 e^{2\beta} \\0 & 0 & 0 \\ 0 & 0 & 0 \end{pmatrix} m^P(y,z)^{-1} 
	\\
& + e^{2\beta} (m_{12}^W(\tilde{y}, w))_+ \hat{p}_2(m_{11}^P + m_{31}^P)(y,z) e^{U_{\HM}(y)} \begin{pmatrix} 0 & m_{11}^P(y,z) & 0 \\0 & 0 & 0 \\ 0 & m_{31}^P(y,z) & 0 \end{pmatrix}.
\end{align*}
It follows from the form of $\hat{p}_1$ and $\hat{p}_2$ as specified in (\ref{phatsumIII}) that
$$-\frac{1}{2\pi i} \int_{Y_6 \cup Y_7} \hat{w}_E(y,t,z) dz
= -\frac{X_E(y)}{2\pi i  t^{1/3}}
-e^{U_{\HM}(y)} \frac{Y_E(y, t)}{2\pi i  t^{1/3}}   + O(t^{-2/3}\ln{t}), \qquad t\to \infty,$$
where $X_E$ and $Y_E$ are defined by
\begin{align}\nonumber
X_E(y) := &\; \hat{b}_{11}^{(1)} \int_{Y_6} m^P(y,z) \begin{pmatrix} 0 & 0 & 0 \\0 & 0 & 0 \\ z e^{-2\beta} & 0 & 0 \end{pmatrix} m^P(y,z)^{-1}  dz
	\\ \label{XEdef}
& + \hat{b}_{11}^{(2)} \int_{Y_7} m^P(y,z) \begin{pmatrix} 0 & 0 & z e^{2\beta} \\0 & 0 & 0 \\ 0 & 0 & 0 \end{pmatrix} m^P(y,z)^{-1}  dz,
	\\ \nonumber
Y_E(y, t) := & - \hat{b}_{11}^{(1)} \int_{Y_6} z e^{-2\beta} m_{12}^W(\tilde{y}, w) (m_{13}^P + m_{33}^P)(y,z)  \begin{pmatrix} 0 & m_{13}^P(y,z) & 0 \\0 & 0 & 0 \\ 0 & m_{33}^P(y,z) & 0 \end{pmatrix} dz
	\\ \label{YEdef}
&+ \hat{b}_{11}^{(2)} \int_{Y_7} z e^{2\beta} (m_{12}^W(\tilde{y}, w))_+ (m_{11}^P + m_{31}^P)(y,z)  \begin{pmatrix} 0 & m_{11}^P(y,z) & 0 \\0 & 0 & 0 \\ 0 & m_{31}^P(y,z) & 0 \end{pmatrix} dz.
\end{align}
The decay of the exponentials $e^{\pm 2\beta} = e^{\mp 2 i (yz + \frac{4z^3}{3})}$ in (\ref{XEdef}) together with the structure of $m^P$ implies that $X_E$ is a smooth function of $y \in [0, \infty)$. 
To complete the proof of the lemma, it is therefore sufficient to show that  $Y_E(y, t) = O(t^{-1/6})$ as $t\to \infty$
uniformly for $y$ in compact subsets of $[0, \infty)$. To show this, we consider the integral in (\ref{YEdef}) over $Y_6$; the integral over $Y_7$ can be estimated similarly. We have
\begin{align}\label{Y6integral}
& \bigg|\int_{Y_6} z e^{-2\beta} m_{12}^W(\tilde{y}, w) (m_{13}^P + m_{33}^P)(y,z)  \begin{pmatrix} 0 & m_{13}^P(y,z) & 0 \\0 & 0 & 0 \\ 0 & m_{33}^P(y,z) & 0 \end{pmatrix} dz\bigg|
 \leq C\int_{Y_6} |z e^{2 i (yz + \frac{4z^3}{3})} m_{12}^W(\tilde{y}, w)||dz|
\end{align}
uniformly for $y$ in compact subsets of $[0, \infty)$. If $|z| \geq t^{-1/12}$, then $w \sim t^{1/6} z$ tends to infinity, so (\ref{mWasymptotics}) implies the estimate $|m_{12}^W(\tilde{y}, w)| \leq |w|^{-1}$ for $|z| \geq t^{-1/12}$. On the other hand, for $|z| \leq t^{-1/12}$ we have $|m_{12}^W(\tilde{y}, w)| \leq C$. Hence, letting $z = r e^{\frac{5\pi i}{6}}$ and splitting the integral into two, the right-hand side of (\ref{Y6integral}) is bounded above by
\begin{align*}
 C\int_0^\infty r e^{-cr^3} \Big|m_{12}^W\Big(\tilde{y}, \frac{2^{1/3}}{3^{1/12}} t^{1/6} e^{\frac{5\pi i}{6}} r\Big)\Big|dr
  \leq &\; C\int_0^{t^{-\frac{1}{12}}} r dr 
 + C\int_{t^{-\frac{1}{12}}}^\infty r e^{-cr^3} \frac{1}{r t^{1/6}}dr
\leq C t^{-1/6}	
\end{align*}
uniformly for $y$ in compact subsets of $[0, \infty)$.
\end{proof}

The next lemma provides an expansion of $\hat{m}^Y$ for large $t$.

\begin{lemma}\label{hatmYlemma}
The following expansion holds as $ z \to \infty$:
\begin{align}\label{hatmYexpansion}
\hat{m}^Y(y, t, z) 
= &\; I + \frac{\hat{m}_{1, \ln}^Y(y,\tilde{y})\ln t + \hat{m}_1^Y(y,\tilde{y})}{z t^{1/3}} 
+ \frac{\hat{m}_{2, \ln}^Y(y,\tilde{y})\ln t}{z t^{1/2}} 
+ O\bigg(\frac{1}{z t^{1/2}}\bigg) + O\bigg(\frac{\ln t}{z^2 t^{1/3}}\bigg),
\end{align}
where $\tilde{y}$ is given in terms of $y,t$ by (\ref{tildeywdef}) and the error terms are uniform for $t \geq T$, $\arg z \in [0, 2\pi]$,  and $y$ in compact subsets of $[0, \infty)$. The coefficients $\hat{m}_{1,\ln}^Y$, $\hat{m}_1^Y$, $\hat{m}_{2,\ln}^Y$ are smooth functions of $y \in [0,\infty)$ and $\tilde{y} \in (-1,\infty)$ defined by
\begin{subequations}\label{hatm1Ym2Y}
\begin{align} 
& \hat{m}_{1,\ln}^Y(y,\tilde{y}) = -\frac{1}{2\pi i } \int_{P} \hat{w}_{2,\ln}^P(y,\tilde{y},z) dz, 
	\\
& \hat{m}_1^Y(y,\tilde{y}) = -\frac{1}{2\pi i } \int_{P} \hat{w}_2^P(y,\tilde{y},z) dz - \frac{X_E(y)}{2\pi i},
	\\ \label{hatm1Ym2Yc}
& \hat{m}_{2,\ln}^Y(y,\tilde{y}) = -\frac{1}{2\pi i } \int_{P} \hat{w}_{3,\ln}^P(y,\tilde{y},z) dz,
\end{align}
\end{subequations}
where $\hat{w}_{2,\ln}^P, \hat{w}_2^P, \hat{w}_{3,\ln}^P$ are given by (\ref{wYexpansionPleadingcoeff}) and $X_E$ is given by (\ref{XEdef}).
In particular, all entries in the second column and in the second row of $\hat{m}_{1,\ln}^Y(y,\tilde{y})$ and $\hat{m}_1^Y(y,\tilde{y})$ are identically zero.
Moreover, the expansion (\ref{hatmYexpansion}) can be differentiated termwise with respect to $y$ any finite number of times without increasing the error terms, and $\hat{m}^Y$ obeys the bounds
\begin{align}\label{hatmYbounded}
\sup_{y \in [0,M]} \sup_{t \geq T} \sup_{z \in \C\setminus Y}  |\partial_y^j \hat{m}^Y(y, t, z)|  < \infty, \qquad j = 0,1, \dots.
\end{align}
\end{lemma}
\begin{proof}
It follows from the representation (\ref{mYrepresentationIV}) for $\hat{m}^Y$ together with Lemma \ref{whatYlemma} and Lemma \ref{hatmuYlemma} that the following formula holds as $z \in \C \setminus Y$ goes to infinity in any nontangential sector:
\begin{align}\nonumber
\hat{m}^Y&(y, t, z) 
=  I -  \frac{1}{2\pi i z}\int_{Y} \hat{\mu}^Y \hat{w}^Y dz'
+ \frac{1}{2\pi i}\int_{Y} \frac{z'\hat{\mu}^Y \hat{w}^Y}{z(z'-z)}dz'
	\\\nonumber
=&\;  I -  \frac{1}{2\pi i z}\int_{Y} \hat{w}^Y dz
-  \frac{1}{2\pi i z}\int_{Y} (\hat{\mu}^Y - I)\hat{w}^Y dz
	\\\nonumber
& + O(z^{-2}\| \hat{\mu}^Y - I \|_{L^2(Y)} \|(1 + |\cdot|) \hat{w}^Y\|_{L^2(Y)})
+ O(z^{-2}\| (1 + |\cdot|) \hat{w}^Y\|_{L^1(Y)})
	\\ \label{mhatYexpansionproof}
=&\;  I -  \frac{1}{2\pi i z}\int_{Y} \hat{w}^Y dz + O\bigg(\frac{1}{z t^{1/2}}\bigg) + O\bigg(\frac{\ln t}{z^2 t^{1/3}}\bigg)
\end{align}
uniformly for $t \geq T$ and $y$ in compact subsets of $[0,\infty)$. Repeating the above steps with $Y$ replaced by a slightly deformed contour $\tilde{Y}$, we see that in fact the condition that $z$ lies in a nontangential sector can be dropped.
Moreover, recalling that $\hat{w}_{1,\ln}^W = \hat{w}_{2,\ln}^W = 0$  by Lemma \ref{whatWlemma}, we deduce with the help of the expansions (\ref{wYexpansionP}) and (\ref{wYexpansionW}) of $\hat{w}^Y$ that
\begin{align}\nonumber
-  \frac{1}{2\pi i z} \int_{Y} & \hat{w}^Y dz
= -  \frac{\int_{P} (\hat{w}_{2,\ln}^P \ln{t} + \hat{w}_2^P) dz}{2\pi i z t^{1/3}}
-  \frac{\int_{P} (\hat{w}_{3,\ln}^P \ln{t} + \hat{w}_3^P) dz}{2\pi i z t^{1/2}}
-  \frac{\int_{W} \hat{w}_E^Y dz}{2\pi i z} 
	\\\nonumber
&-  \frac{\int_{W \cap \{|w| \leq \frac{2^{1/3}t^{1/12}}{3^{1/12}} \}} \hat{w}_1^W(y, \tilde{y}, w) dw}{2\pi i z t^{1/6} \frac{2^{1/3}}{3^{1/12}} t^{1/6}}
 -  \frac{\int_{W \cap \{|w| \leq \frac{2^{1/3}t^{1/12}}{3^{1/12}}\}} \hat{w}_2^W(y, \tilde{y}, w) dw}{2\pi i z t^{1/3} \frac{2^{1/3}}{3^{1/12}} t^{1/6}}
 + O\Big(\frac{\ln{t}}{z t^{2/3}}\Big).
\end{align}
We infer from the explicit expression in (\ref{wYexpansionWleadingcoeff}) that $\hat{w}_1^W$ is an odd function of $w$, so the integral $\int_{W}  \hat{w}_1^W dw$ vanishes. 
From (\ref{hatwjWestimates}), we see that the term involving $\hat{w}_2^W$ is $O(z^{-1}t^{-1/2})$. Employing also Lemma \ref{XElemma} to estimate the integral of $\hat{w}_E^Y$, we obtain, as $z \to \infty$,
\begin{align}\label{intYhatwYexpansion}
-  \frac{1}{2\pi i z}\int_{Y} \hat{w}^Y dz
=&  -  \frac{\int_{P} (\hat{w}_{2,\ln}^P \ln{t} + \hat{w}_2^P) dz}{2\pi i z t^{1/3}}
-\frac{X_E(y)}{2\pi i  z t^{1/3}} 
-  \frac{\int_{P} \hat{w}_{3,\ln}^P dz \ln{t} }{2\pi i z t^{1/2}}
+ O\bigg(\frac{1}{z t^{1/2}}\bigg)
\end{align}
uniformly for $t \geq T$ and $y$ in compact subsets of $[0,\infty)$.
The desired conclusion follows from (\ref{mhatYexpansionproof}) and (\ref{intYhatwYexpansion}).
\end{proof}

\subsection{Proof of Lemma \ref{YlemmaIII}}
The expansion (\ref{mYasymptoticsIII}) of $m^Y = \hat{m}^Y M^Y$ follows from the expansions (\ref{MYasymptotics}) and (\ref{hatmYexpansion}); since both of these expansions can be differentiated termwise with respect to $y$, the same is true of the expansion in (\ref{mYasymptoticsIII}).
The bound (\ref{mYboundedIV}) follows from (\ref{MYbounded}) and (\ref{hatmYbounded}). This completes the proof of Lemma \ref{YlemmaIII}.

\subsection*{Acknowledgements}
The authors are grateful to Percy Deift and J\"orgen \"Ostensson for valuable discussions. 
Support is acknowledged from the Novo Nordisk Fonden Project, Grant 0064428, the European Research Council, Grant Agreement No. 682537, the Swedish Research Council, Grant No. 2015-05430, Grant No. 2021-04626, and Grant No. 2021-03877, the G\"oran Gustafsson Foundation, and the Ruth and Nils-Erik Stenb\"ack Foundation.

\bibliographystyle{plain}
\bibliography{is}

\end{document}